\documentclass{surv-l}
\pdfoutput=1
\usepackage{amssymb,amsmath,latexsym, amscd, graphicx, multirow, longtable, appendix}
\usepackage{pinlabel}

\usepackage[cmtip,all]{xy}

\newtheorem{theorem}{Theorem}[chapter]
\newtheorem{lemma}[theorem]{Lemma}
\newtheorem{corollary}[theorem]{Corollary}
\theoremstyle{definition}
\newtheorem{definition}[theorem]{Definition}
\newtheorem{example}[theorem]{Example}
\theoremstyle{problem} 
\newtheorem{problem}[theorem]{Problem}   
\theoremstyle{proposition} 
\newtheorem{proposition}[theorem]{Proposition} 
\theoremstyle{conjecture}
\newtheorem{conjecture}[theorem]{Conjecture}
\theoremstyle{remark}

\theoremstyle{question}
\newtheorem{question}[theorem]{Question}

\numberwithin{section}{chapter}
\numberwithin{equation}{chapter}

\def\a{{\alpha}}

\def\e{{\epsilon}}

\def\H{{\mathbb H}}

\def\Q{{\mathbb Q}}
\def\R{{\mathbb R}}
\def\Z{{\mathbb Z}}
\def\C{{\mathbb C}}
\def\square{{\vcenter{\hrule height.4pt
       \hbox{\vrule width.4pt height5pt \hskip5pt
            \vrule width.4pt}
       \hrule height.4pt}}}

\def\qed{\hfill$\square$\bigskip}
\def\pf{{\noindent{\bf Proof: \hspace{2mm}}}}

\makeindex

\begin{document}

\frontmatter

\title{Ordered Groups and Topology}

\author{Adam CLAY}

\address{Department of Mathematics; University of Manitoba; Winnipeg MB, R3T 2N2, Canada}

\email{adam.clay@umanitoba.ca}

\urladdr{server.math.umanitoba.ca/\!\hbox{$\sim$}claya}

\author{Dale ROLFSEN}

\address{Mathematics Department; University of
British Columbia; Vancouver BC, V6T 1Z2, Canada}

\email{rolfsen@math.ubc.ca}

\urladdr{www.math.ubc.ca/\!\hbox{$\sim$}rolfsen}



\date{}


\maketitle


\setcounter{page}{5}

\tableofcontents

%

\chapter*{Preface}
%
%

The inspiration for this book is the remarkable interplay, expecially in the past few decades, between topology and the theory of orderable groups.  Applications go in both directions.  For example, orderability of the fundamental group of a 3-manifold is related to the existence of certain foliations.  On the other hand, one can apply topology to study the space of all orderings of a given group, providing strong algebraic applications.   Many groups of special topological interest are now known to have invariant orderings, for example braid groups, knot groups, fundamental groups of (almost all) surfaces and many interesting manifolds in higher dimensions.

There are several excellent books on orderable groups, and even more so for topology.  The current book emphasizes the {\em connections} between these subjects, leaving out some details that are available elsewhere, although we have tried to include enough to make the presentation reasonably self-contained.  Regrettably we could not include all interesting recent developments, such as Mineyev's \cite{Mineyev12} use of left-orderable group theory to prove the Hanna Neumann conjecture.

This book may be used as a graduate-level text; there are quite a few problems assigned to the reader.   It may also be of interest to experts in topology, dynamics and/or group theory as a reference.  A modest familiarity with group theory and with basic topology is assumed of the reader.   

We gratefully acknowledge the help of the following people in the preparation of this book: Maxime Bergeron, Steve Boyer, 
Patrick Dehornoy, Colin Desmarais, Andrew Glass, Cameron Gordon, Herman Goulet-Ouellet, Tetsuya Ito, Darrick Lee, Andr\'es Navas, Akbar Rhemtulla, Crist\'obal Rivas, Daniel Sheinbaum, Bernardo Villareal-Herrera, Bert Wiest.

\aufm{Adam Clay and Dale Rolfsen}

\vfill

{\em In these days the angel of topology and the devil of abstract algebra fight for the soul of each individual mathematical domain. 
}
\bigskip

\aufm{Hermann Weyl, 1939}


\mainmatter


\chapter{Orderable groups and their algebraic properties}

In this chapter we will discuss some of the special algebraic properties enjoyed by orderable groups, which come in two basic flavors: left-orderable and the more special bi-orderable groups.
As we'll soon see, a group is right-orderable if and only if it is left-orderable.  The literature is more or less evenly divided between considering right- and left-invariant orderings.   Some authors (including those of this book) have flip-flopped on the issue of right vs. left.  Of course results from the ``left'' school have dual statements in the right-invariant world, but as with driving, one must be consistent.

There are several useful reference books on ordered groups, such as {\em Fully ordered groups} by Kokorin and Kopytov \cite{KK74}, {\em Orderable groups} by Mura and Rhemtulla \cite{MR77},  {\em Right-ordered groups} by Kopytov and Medvedev \cite{KM96} and A. M. W. Glass' {\em Partially ordered groups} \cite{Glass99}.  Many interesting results and examples on orderability of groups which won't be discussed here can be found in these books.  We will focus mostly on groups of special topological interest and results relevant to topological applications.  On the other hand, we try to include enough material to provide context and to make the core development of ideas in this book reasonably self-contained.  

 \section{Invariant orderings} 
 
 By a {\em strict ordering} \index{strict ordering} of a set $X$ we mean a binary relation $<$ which is transitive ($x < y$ and $y < z$ imply $x < z$) and such that $x < y$ and $y < x$ cannot both hold.  It is a strict {\em total} ordering  \index{strict total ordering} if for every $x, y \in X$ exactly one of $x < y$, $y < x$ or $x = y$ holds.

A group $G$ is called {\it left-orderable}  \index{left-orderable}  if its elements can
be given a strict total ordering $<$ which is left invariant, meaning that
$g < h$ implies $ fg < fh$ for all $f,g, h \in G$.  We will say that $G$ is
{\it bi-orderable}  \index{bi-orderable} if it admits a total ordering which is simultaneously
left and right invariant (historically, this has been called simply ``orderable'').
We refer to the pair $(G,<)$ as the ordered group. \index{ordered group}   We shall usually use the symbol $1$ to denote the identity element of a group $G$.  However, for abelian groups in which the group operation is denoted by addition, the identity element may be denoted by $0$.   In an ordered group the symbols $\le$ and $>$ have the obvious meaning: $g \le h$ means  
$g < h$ or $g = h$; $g > h$ means $h<g$.  Note that the opposite ordering can also be considered an ordering, also invariant. 

\begin{problem}
Show that
\begin{enumerate} 
\item In a left-ordered group one has $1 < g$ if and only if $g^{-1} < 1$.
\item In a left-ordered group, if $1<g$ and $1<h$, then $1<gh$. 
\item A left-ordering is a bi-ordering if and only if the ordering is invariant under conjugation.
\end{enumerate}
\end{problem}

As already mentioned, the class of right-orderable groups is the same as the class of left-orderable groups.  In fact, a concrete correspondence can be given as follows.

\begin{problem}
\label{right order}
If $<$ is a left-invariant ordering of the group $G$, show the recipe 
$$g \prec h \iff h^{-1} < g^{-1} $$ 
defines a right-invariant ordering $ \prec $ which has the same ``positive cone'' -- that is: 
$1 \prec g \iff 1 < g$. 
\end{problem}

The following shows that left-orderable groups are infinite, with the exception of the trivial group, consisting of the identity alone.

\begin{proposition} \label{torsionfree}
A left-orderable group has no elements of finite order. In other words, it is torsion-free. \index{torsion-free}
\end{proposition}

\pf   If $g$ is an element of the left-ordered group $G$ and $1 < g$, then $g < g^2$, $g^2 < g^3$ and so on, and by transitivity we conclude that $ 1 < g^n$ for all positive integers $n$.  The case $g<1$ is similar. \qed

\begin{problem}\label{problem between}  Show that if $f$ and $g$ are elements of a left-ordered group and $f \ne 1$ then $g$ is strictly between $fg$ and $f^{-1}g$ and also strictly between $gf$ and $gf^{-1}$.
\end{problem}

\section{Examples}
\begin{example}
The additive reals $(\R, +)$, rationals $(\Q, +)$ and integers $(\Z, +)$ are bi-ordered groups, under their usual ordering.  On the other hand, the multiplicative group of nonzero reals, $(\R \setminus \{ 0 \}, \cdot )$, cannot be bi-ordered.  The element $-1$ has order two; by Proposition \ref{torsionfree} this is impossible in a left-orderable group.\end{example}

\begin{example}
Both left- and bi-orderability are clearly preserved under taking subgroups.
If $G$ and $H$ are left- or bi-ordered groups, then so is their direct product 
$G \times H$ using lexicographic ordering  \index{lexicographic ordering}, which declares that 
$(g, h) < (g', h')$   if and only if  $g <_G g'$ or else $g=g'$ and  
$h <_H h'.$
\end{example}

\begin{example}\label{orderZ2}
Consider the additive group $\Z^2$.  It can be ordered lexicographically as just described, taking $G = H = \Z$.  Another way to order  $\Z^2$ is to think 
of it sitting in 
the plane $\R^2$ in the usual way, and then choose a vector $\vec{v} \in \R^2$ which has irrational slope.
We can order $\vec{m} = (m_1, m_2), \vec{n} = (n_1, n_2) \in \Z^2$ according to their dot product with $\vec{v}$, that is
$$\vec{m} < \vec{n} \iff  m_1v_1 + m_2v_2 < n_1v_1 + n_2v_2$$
We leave the reader to check that this is an invariant strict total ordering, and that one obtains uncountably many different orderings of  $\Z^2$ in this way.  If 
$\vec{v}$ has rational slope, then one may also compare as above, but using lexicographically the dot product with $v$ and then with some pre-chosen vector orthogonal to $\vec{v}$.  Higher dimensional spaces can be invariantly ordered in a similar manner.
\end{example}

\begin{problem}\label{extension}
Suppose $G$ is a group with normal subgroup $K$ and quotient group $H \cong G/K$.  In other words, suppose there is an exact sequence
$$1 \rightarrow K \hookrightarrow G \xrightarrow{p} H \rightarrow 1. $$
Further suppose $(H, <_H)$ and $(K, <_K)$ are left-ordered groups.  Verify that we can then give $G$ a left-ordering defined in a sort of lexicographic way: declare that $g < g'$ if and only if either $p(g) <_H p(g')$ or else $p(g) = p(g')$ 
(so $g^{-1}g' \in K$) and $1 <_K g^{-1}g'$. 
\end{problem}

\begin{figure}%
\label{Klein}
\setlength{\unitlength}{6cm}
\begin{picture}(1,0.94439153)%
    \put(0,0){\includegraphics[height=\unitlength]{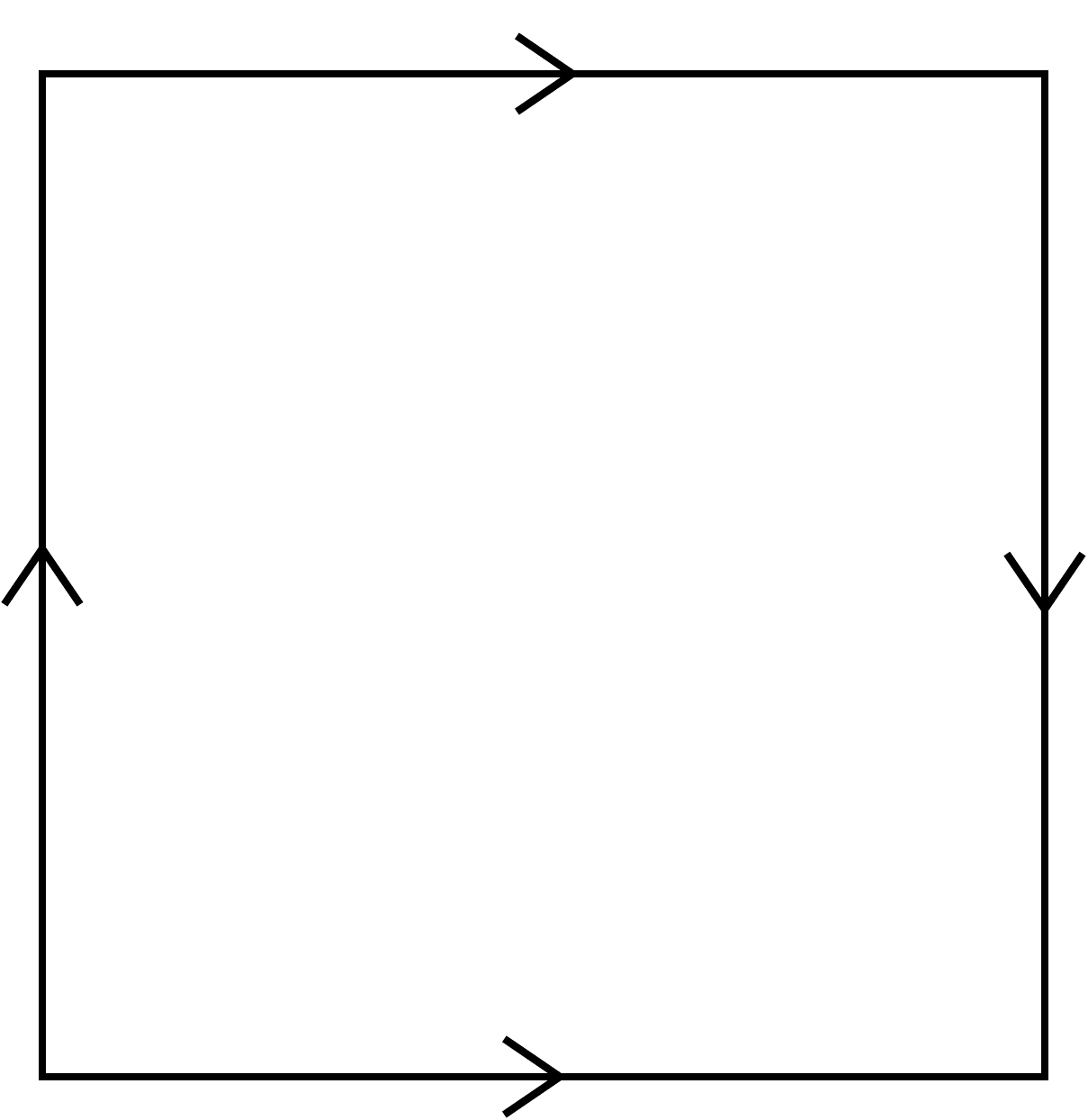}}%
    \put(0.46276399,0.85614612){$x$}%
    \put(0.45149788,0.10546392){$x$}%
    \put(-0.05290441,0.47864313){$y$}%
    \put(1.03,0.47864313){$y$}%
  \end{picture}%
\caption{The Klein bottle as a square with opposite sides identified as shown.}
\end{figure}
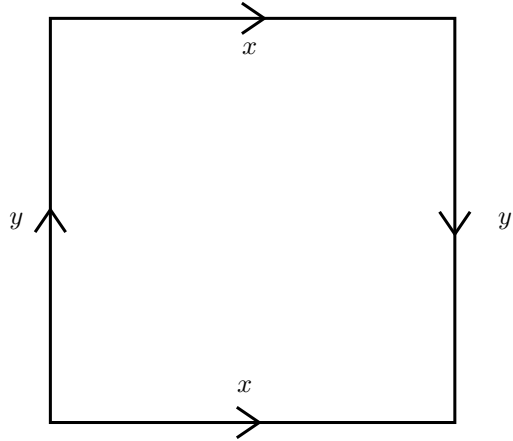

\begin{example}
The Klein bottle \index{Klein bottle} is a nonorientable surface, \index{nonorientable surface} which can be considered as a square with opposite sides identified with each other in the directions indicated in Figure \ref{Klein}.  We see that its fundamental group has the presentation with two generators $x$ and $y$ and the relation $yxyx^{-1} = 1$.  In other words,
$$K = \pi_1(Klein \; Bottle) \cong \langle x, y \mid xyx^{-1} = y^{-1} \rangle$$
\end{example}

\begin{problem}\label{klein}
Show that the subgroup $\langle y \rangle$ of the Klein bottle group $K$ which is generated by $y$ is a normal subgroup isomorphic to $\Z$ and that the quotient subgroup $K / \langle y \rangle$ is also isomorphic with $\Z$.  Use this to show that $K$ is left-orderable.  Finally, conclude that $K$ cannot be given a bi-invariant ordering, by using the defining relation to derive a contradiction.
\end{problem}

\begin{example}\label{homeo+}
Let $\mathrm{Homeo}_+(\R)$ \index{H@$\mathrm{Homeo}_+(\R)$ } denote the group of all order-preserving homeomorphisms of the real line 
\index{homeomorphisms of the real line} -- that is, continuous functions with continuous inverses and which preserve the usual order of the reals.  This is a group under composition.  It can be left-ordered in the following way.  Let 
$x_1, x_2, \dots$ be a countable dense set of real numbers.  
For two functions $f, g \in \mathrm{Homeo}_+(\R)$, compare them by choosing $m = m(f,g)$ to be the minimum $i$ for which $f(x_i) \ne g(x_i)$ and then declare that 
$f \prec g$ if and only if $ f(x_m) < g(x_m)$ (in the usual ordering of $\R$). 
\end{example}

\begin{problem}\label{homeo2}
Verify that $\prec$ is a left-ordering of $\mathrm{Homeo}_+(\R)$.  Hint: to show that 
$f \prec g, g \prec h \implies f \prec h$, consider the cases $m(f,g) = m(g,h)$ and 
$m(f,g) \ne m(g,h)$ separately.
\end{problem}

We will see later that $\mathrm{Homeo}_+(\R)$ is universal for countable left-orderable groups, in the sense that any countable left-orderable group embeds in $\mathrm{Homeo}_+(\R)$.

\begin{problem}
\label{covering group problem}
Suppose that $G$ is a path connected topological group, which as a space has universal cover $\widetilde{G}$.  Show that there is a multiplication on $\widetilde{G}$ that is compatible with the multiplication on $G$, meaning that the covering map $p : \widetilde{G} \rightarrow G$ becomes a group homomorphism.
\end{problem}

Recall that a (left) {\em action} of a group $G$  \index{group action} on a set $X$ is a binary operation $G \times X \to X$ which satisfies $1x = x$ and
$(gh)x = g(hx)$ for all $g, h \in G, x \in X$.  

\begin{problem}
\label{covering group actions}
Suppose that $G$ and $\widetilde{G}$ are as above and $G$ acts on a space $X$.  Show that if $\widetilde{X}$ is the universal cover of $X$, then $\widetilde{G}$ acts on $\widetilde{X}$.
\end{problem}

\begin{example} \label{special linear groups}\index{S@$ \mathrm{SL}(2, \R)$}
The group 
\[ \mathrm{SL}(2, \R) = \left\{ 
\left( \begin{array}{cc}
a & b  \\
c & d  \end{array} \right) : a,b,c,d \in \R, ad-bc=1 \right\}
\]
is naturally a subgroup of $\mathrm{SL}(2,\C)$, and it is conjugate to the subgroup
\[
\mathrm{SU}(1, 1) = \left\{ 
\left( \begin{array}{cc}
\alpha & \beta  \\
\bar{\beta} & \bar{\alpha}  \end{array} \right) : \alpha, \beta \in \C, |\alpha|^2 - |\beta|^2 =1 \right\}
\]
The conjugacy is given by sending each matrix $A \in \mathrm{SL}(2, \R)$ to the matrix $JAJ^{-1} \in \mathrm{SU}(1, 1) $, where $J = \left( \begin{array}{cc}
1 & -i  \\
1 & i  \end{array} \right)$.
Now thinking of the group in this way, we can observe a faithful action of $\mathrm{PSL}(2, \R) = \mathrm{SL}(2, \R)/ \{ \pm I \}$ on the unit circle $S^1 \subset \C$ by homeomorphisms.  An element of $\mathrm{PSL}(2, \R)$ acts on $z \in S^1$ by first choosing a representative $A \in \mathrm{SL}(2, \R)$, converting $A$ to an element of $\mathrm{SU}(1,1)$ and then applying the associated M{\"o}bius transformation.  In other words if $A = \left( \begin{array}{cc}
a & b  \\
c & d  \end{array} \right)$ then $JAJ^{-1} = \left( \begin{array}{cc}
\alpha & \beta  \\
\bar{\beta} & \bar{\alpha}  \end{array} \right)$ for some $\alpha, \beta \in \C$ with  $|\alpha|^2 - |\beta|^2 =1$, and then we can define
\[
A(z) =\cfrac{\alpha z + \beta}{\bar{\beta} z +\bar{\alpha}}
\]

By considering $\mathrm{SL}(2, \R) $ as a subspace of $\R^4$, we can think of it as a $3$-manifold and its quotient $\mathrm{PSL}(2, \R) $ is also a manifold.  Thus it admits a universal covering space $p: \widetilde{\mathrm{PSL}}(2,\R)\rightarrow \mathrm{PSL}(2,\R)$, and the universal covering space has a group structure that is lifted from the base space, as in Problem \ref{covering group problem}.  The action of $\mathrm{PSL}(2,\R)$ on the circle lifts to an action of $\widetilde{\mathrm{PSL}}(2,\R)$ on $\R$ by orientation-preserving homeomorphisms by Problem \ref{covering group actions}, so we can think of $\widetilde{\mathrm{PSL}}(2,\R)$ as a subgroup of $\mathrm{Homeo}_+(\R)$ (see \cite{khoi03} for details).  Since $\mathrm{Homeo}_+(\R)$ is left-orderable, so is $\widetilde{\mathrm{PSL}}(2,\R)$.

\begin{problem}
Check that the definition in the previous example yields an action of $\mathrm{PSL}(2, \R) $ on $S^1$, by checking that $A \mapsto JAJ^{-1}$ defines an isomorphism of $\mathrm{SL}(2, \R)$ with $\mathrm{SU}(1, 1)$, and that $\left| \cfrac{\alpha z + \beta}{\bar{\beta} z +\bar{\alpha}} \right| =1$ whenever $|z|=1$.
\end{problem}

\begin{problem}
Show that, as a subspace of $\R^4$, $\mathrm{SL}(2, \R)$  is homeomorphic with an open solid torus:
$\mathrm{SL}(2, \R) \cong S^1 \times \C$.  Moreover show that the action on $\mathrm{SL}(2, \R) $ given by $M \to -M$ is fixed-point free, and so $\mathrm{PSL}(2, \R)$ is a manifold, in fact also an open solid torus, and the projection map $\mathrm{SL}(2, \R) \to \mathrm{PSL}(2, \R) $ is a covering space.
\end{problem}

\begin{problem}
Conclude that $\widetilde{\mathrm{PSL}}(2,\R)$ is homeomorphic with $\R^3$.
\end{problem}

\end{example}

\section{Bi-orderable groups} \label{genord}

We summarize a few algebraic facts about bi-orderable groups, which do not hold in general for left-orderable groups, and leave their proofs to the reader.  
For example, inequalities multiply:

\begin{problem} \label{multineq}
In a bi-ordered group $g_1<h_1$ and $g_2<h_2$ imply $g_1g_2 < h_1h_2$.
\end{problem}

\begin{problem}\label{uniqueroots} 
Bi-orderable groups have unique roots, \index{unique roots} that is, if $g^n = h^n$ for some $n > 0$ then $g=h$.
\end{problem}

The following was observed by B. H. Neumann \cite{Neumann49b}.

\begin{problem} \label{central}
In a bi-orderable group $G$, $g^n$ commutes with $h$ if and only if $g$ commutes with $h$.  Hint: For the nontrivial direction, assume $g$ and $h$ do not commute, say $g < h^{-1}gh$, and multiply this inequality by itself several times to conclude $g^n$ cannot commute with $h$.  Show more generally that if $g^n$ and $h^m$ commute for some nonzero integers $m$ and $n$, then $g$ and $h$ must commute.
\end{problem}

\begin{problem}
Bi-orderable groups do not have generalized torsion: \index{generalized torsion} any product of conjugates of a nontrivial element must be nontrivial.  In particular, $x^{-1}yx = y^{-1}$ implies $y=1$.
\end{problem}

On the down side, bi-orderable groups do not behave as nicely under extension as left-orderable groups do.  As seen in Problem \ref{klein} we have a group $K$ which is flanked by bi-orderable groups in a short exact sequence (and is left-orderable for that reason) but it is not bi-orderable.

\begin{problem}\label{extendO}
Consider groups $K$, $G$ and $H = G/K$ as in Problem \ref{extension}, with $$1 \rightarrow K \hookrightarrow G \xrightarrow{p} H \rightarrow 1$$
exact.  Suppose $K$ and $H$ are bi-ordered.  Then the recipe of  Problem \ref{extension} defines a bi-ordering of $G$ if and only if the conjugation action of $G$ upon $K$ preserves the given ordering of $K$.  
\end{problem}

\section{Positive cone}

\begin{theorem}
A group $G$ is left-orderable if and only if there exists a subset $P \subset G$ such that 

(1) $P\cdot P \subset P$  and 

(2) for every $g \in G$, exactly one of $g = 1$,  $g \in P$ or $g^{-1} \in P$ holds.  
\end{theorem}

\pf
Given such a $P$, 
the recipe $g<h$ if and only $g^{-1}h \in P$ is easily seen to define a
left-invariant strict total order, and conversely such an ordering defines 
the set  $P = \{ g \in G : 1 < g \}$, called the {\it positive cone}. \index{positive cone}

\begin{problem}
Verify the details of this proof.  
\end{problem}

\begin{problem}
Show that $G$ is bi-orderable if and only if it admits a subset $P$ satisfying 
(1), (2) above, and in addition 

(3) $gPg^{-1} \subset P$ for all $g \in G$.
\end{problem}

\begin{example}
The positive cone for the ordering of $\Z^2$ described in Problem \ref{orderZ2} is the set of all points in the plane which lie to one side of the line through the origin which is orthogonal to $\vec{v}$, if $\vec{v}$ has irrational slope.  If the slope is rational, one must also include points of $\Z^2$ on one half of that orthogonal line to lie in the positive cone.
\end{example}

\begin{example}
In Problem \ref{extension}, the positive cone for the ordering described for $G$ is the union of the positive cone of (the ordering of) $K$ and the pullback $p^{-1}(P_H)$ of the positive cone of $H$.  That is: $P_G = P_K \cup p^{-1}(P_H)$.
\end{example}

\begin{problem}
Let $(G,<)$ be a left-ordered group.  Then the following are equivalent:

(1) The ordering $<$ is also right-invariant.

(2)  For every $g, h \in G$, if $g < h$ then $h^{-1} < g^{-1}$.

(3)   For every $g, h \in G$, if $g < gh$ then $g < hg$.

(4)  If $g_1 < h_1$ and $g_2 < h_2$ then $g_1g_2 < h_1h_2$.

\end{problem}

\begin{problem} \index{Klein bottle group}
Show that the Klein bottle group discussed above is isomorphic with the group  
$\langle a, b\ ; \  a^2 = b^2 \rangle$.  Define an explicit function
$h: \langle a, b\ ; \  a^2 = b^2 \rangle \to  \langle x, y : xyx^{-1} = y^{-1} \rangle$
by assigning $h(a)$ and $h(b)$ expressions as words in $x$ and $y$ and show that the relation $a^2 = b^2$ in the domain implies $xyx^{-1} = y^{-1}$ in the range, so that $h$ is a homomorphism.  Similarly define a homomorphism in the other direction and verify that it is inverse to $h$. 
\end{problem}

Another way of seeing this isomorphism is to observe that the Klein bottle is the union of two M\"obius bands, glued along their boundaries, and apply the theorem of Seifert and Van Kampen.

\begin{problem}
Show that the Klein bottle group does {\em not} have unique roots.  Indeed, we have $a \ne b$ (why?) but $a^2 = b^2$.  This gives another proof that it is not bi-orderable.
\end{problem}

\section{Topology and the spaces of orderings}

It is time for topology to enter the picture.  We recall that a topological space \index{topological space} is a set $X$ and a collection of subsets of $X$, called open sets, for which finite intersections and arbitrary unions of open sets are open.  The space $X$ itself and the empty set $\emptyset$ are always considered open.  A subset is closed if its complement is open.   Any subset $A$ of $X$ inherits a topology from a topology on $X$ by taking sets of the form 
$A \cap U$, where $U$ is an open subset of $X$, to be open in $A$.  The discrete topology on a set is the one in which {\em every} subset is open.

An open covering of a space is a collection of open sets whose union is the whole space.  A space is {\em compact} if every open covering has a finite subcollection whose union is the space.   A {\em basis} for a topology on $X$ is a collection $\mathcal{B}$ of subsets of $X$ such that the open sets are exactly all unions of sets in $\mathcal{B}$.  

%
%
%
 
\subsection{Topology on the power set}
\label{topology on power set}

For any set $X$, one may consider the collection of all its subsets---that is, its power set---often denoted $\mathcal{P}(X)$ or $2^X$.  This latter notation indicates that the power set may be identified with the set of all functions $X \to \{0, 1\}$ (using von Neumann's  definition $2 :=  \{0, 1\}$), via the characteristic function $\chi_A : X \to  \{0, 1\}$ associated to a subset $A \subset X$ defined by 
$$\chi_A(x) =
  \begin{cases}
     1 \text{ if } x \in A, \\
    0 \text{ if } x \notin A.
   \end{cases}
$$

The set $2^X$ is a special case of a product space: one gives $\{0, 1\}$ the discrete topology, and $2^X$ is considered the product of copies of $\{0, 1\}$ indexed by 
the set $X$.  The product topology is the the smallest topology on the set $2^X$ such that for each $x \in X$ the sets $ \{f \in 2^X : f(x) = 0\}$ and $ \{f \in 2^X : f(x) = 1\}$ are open.  In other notation, the subsets of  $\mathcal{P}(X)$ of the form
$$U_x = \{A \subset X : x \in A\} \quad \text{ and } \quad U_x^c = \{A \subset X : x \notin A\}$$
are open in the ``Tychonoff'' topology on the power set.  Note that the sets $U_x$ and $U_x^c$ are also closed, as they are each other's complement.   A basis for the topology can be gotten by taking finite intersections of various $U_x$ and $U_x^c$.
A famous theorem of Tychonoff asserts that an arbitrary product of compact spaces is again compact.  Since the space $\{ 0, 1 \}$ is compact, we conclude:

\begin{theorem}
The power set $\mathcal{P}(X)$ of any set $X$, with the Tychonoff topology, is compact. \end{theorem}

\begin{problem}  A space is said to be {\em totally disconnected} \index{totally disconnected} if for each pair of points, there is a set which is both closed and open and which contains one of the points and not the other.  Show that $\mathcal{P}(X)$, with the Tychonoff topology, is totally disconnected.
\end{problem}

If $X$ is finite, then so is $2^X$ and the Tychonoff topology is just the discrete topology.  If $X$ is countably infinite, then  $2^X$ is homeomorphic to the Cantor space obtained by deleting middle thirds successively of the interval $[0, 1]$.   In particular,
the Tychonoff topology on $\mathcal{P}(X)$ is metrizable when $X$ is countable.
A useful characterization of the Cantor space is that any nonempty compact metric space which is totally disconnected will be homeomorphic with the Cantor space if and only if it has no isolated points.  A point is {\em isolated} if it has an open neighborhood disjoint from the rest of the space.  See \cite[Corollary 2.98]{HY61} for details.

\begin{problem}  If $A \subset X$ is a fixed subset, there is a natural inclusion 
$\mathcal{P}(A) \subset \mathcal{P}(X)$.  Show that $\mathcal{P}(A)$ is a closed subset.
\end{problem}

\begin{problem}  
Consider the complementation function $C: \mathcal{P}(X) \to \mathcal{P}(X)$ on the power set of the set $X$ defined by  $C(Y) = X \setminus Y$.  Show that $C$ is a fixed-point free involution---that is, $C$ is a homeomorphism of $\mathcal{P}(X)$ with $C^2$ the identity map and $C(Y) \ne Y$ for all $Y \in \mathcal{P}(X)$.
\end{problem}

\begin{example}\label{semigroupsclosed}
Let $G$ be a group and define $\mathcal{S}(G)$ to be the collection of all sub-semigroups of $G$.  That is, $\mathcal{S}(G) = \{ S \subset G : g, h \in S \implies gh \in S\}$.  Note that 
$\mathcal{S}(G) \subset \mathcal{P}(G)$.  We will argue that 
$\mathcal{S}(G)$ is in fact a {\em closed} subset of $\mathcal{P}(G)$.  Consider the complement $\mathcal{P}(G) \setminus \mathcal{S}(G)$.  A subset $Y$ of $X$ belongs
to $\mathcal{P}(G) \setminus \mathcal{S}(G)$ if and only if there exist $g, h \in Y$ with 
$gh \notin Y$.  Therefore
$$\mathcal{P}(G) \setminus \mathcal{S}(G) = 
\bigcup_{g, h \in G} (U_g \cap U_h \cap U_{gh}^c ).$$
Each term in the parentheses is an open set, by definition, and therefore so is the intersection of the three, and so $\mathcal{P}(G) \setminus \mathcal{S}(G)$ is a union of open sets.  It follows that  $\mathcal{S}(G)$ is closed.
\end{example}

\subsection{The spaces of orderings}
\label{space of orderings section}
In this section we will show how to topologize the set of all orderings of a group, so as to make a compact space of orderings.
\begin{definition}  The space of left-orderings of a group $G$, denoted \index{LO@$LO(G)$} $LO(G)$,  is the collection of 
all subsets $P \subset G$ such that (1) $P$ is a sub-semigroup, (2) $P \cap P^{-1} = \emptyset$ 
and (3) $P \cup P^{-1} = G \setminus \{1\}.$ 
\end{definition}

\begin{problem}  Show that $LO(G)$ is a closed subset of $\mathcal{P}(G\setminus \{1\})$ 
and of $\mathcal{P}(G)$, and is therefore a compact and totally disconnected space (with the subspace topology).
\end{problem}

\begin{problem} Suppose $<$ is a left-invariant ordering of the group $G$, and suppose we have a finite string of inequalities $g_1 < g_2 < \cdots < g_n$ which hold.  Show that the set of all left-orderings in which all these inequalities hold forms an open neighborhood of 
$<$ in $LO(G)$.  The set of all such neighborhoods is a basis for the topology of $LO(G)$.
Equivalently, a basic open set in $LO(G)$ consists of all orderings in which some specified finite set of elements of $G$ are all positive.
\end{problem}

In particular, an ordering of $G$ is {\em isolated} \index{isolated ordering} in $LO(G)$ if it is the only ordering satisfying some finite set of inequalities.  This property is also known as ``finitely determined'' in the literature.   Some groups $G$ have isolated points in $LO(G)$, while others do not, as we will see in Chapter \ref{space of orderings chapter}.

Similarly, we can define the set $O(G)$ of bi-invariant orderings on the group $G$ to be the collection of subsets $P \subset G$ satisfying (1), (2) and (3) above; and also 
$g^{-1}Pg \subset P.$

\begin{problem}  Show that $O(G)$ is a closed subset of $LO(G)$, so it is also a compact totally disconnected space. 
\end{problem}


To our knowledge, this definition of $LO(G)$ first appeared in \cite{Sikora04}.  We will discuss the structure of $LO(G)$, some of Sikora's results and other applications in greater detail in Chapter \ref{space of orderings chapter}.

\begin{problem}\label{Tychonoff topology}
\index{Tychonoff topology}
Suppose a     countable left-orderable group $G$ has its non-identity elements enumerated, so $G \setminus \{ 1\} = \{ g_1, g_2, \dots \}$.  If $<$ and $<'$ are two left-orderings of $G$, define
$$d( <, <' ) = 2^{-n},$$
where $n$ is the first index at which $<$ and $<'$ differ on $g_n$ (i.e. either $1<g_n$ and 
$g_n<'1$ or else $1<'g_n$ and $g_n<1$)  In other words, $g_n$ is in the symmetric difference of their respective positive cones.  Show that this really is a metric (the triangle inequality is the only nontrivial part).  Moreover, verify that the topology generated by this metric is the Tychonoff topology.
\end{problem}

\section{Testing for orderability}

Suppose we wish to determine if a given group $G$ is left-orderable.  Consider a set $S$ of generators of $G$, which may be infinite. 
That is, each $g \in G$ may be written as a finite product of elements of $S$ and their inverses.  The length $l(g)$ of a group element (relative to the choice of generators) is the smallest integer $k$ such that
$$g = g_{1}^{\e_1}\cdots g_{k}^{\e_k}$$ 
where each $g_i \in S$ and $\e_i = \pm 1$.  Let $G_k$ denote the set of all elements of $G$ of length at most $k$.  If $S$ is finite, $G_k$ is also a finite set, which includes the identity (length zero) and also is invariant under taking inverses.
It can be regarded as the $k$-ball of the Cayley graph of $G$, relative to the given generators.

Now let us define a subset $Q$ of $G_k$ to be a {\em proper $k$-partition} if 
(1) whenever $g, h \in Q$ and $gh \in G_k$ then $gh \in Q$, (2) $Q \cap Q^{-1} = \emptyset$ 
and (3) $Q \cup Q^{-1} = G_k \setminus \{1\}.$   

Notice that if $P$ is a positive cone 
(of a left-ordering) of $G$, then $P \cap G_k$ is a proper $k$-partition.
So the following is clear:

\begin{proposition}
Suppose $G$ is a group with generating set $S$, with respect to which there is no proper $k$-partition of $G_k$ for some positive integer $k$.  Then $G$ is not left-orderable.
\end{proposition}

Perhaps surprisingly, there is a converse.

\begin{theorem}\label{partition}
Suppose $G$ is generated by $S \subset G$ with respect to which, for all $k \ge 1$, there is a proper $k$-partition of $G_k$.  Then $G$ is left-orderable.
\end{theorem}

\begin{proof}
We will prove this using compactness of $\mathcal{P}(G)$.  
Consider the set $\mathcal{P}_k$ of all subsets of $G$ whose intersection with $G_k$ is a proper $k$-partition.  One argues as usual that $\mathcal{P}_k$ is a closed subset of 
$\mathcal{P}(G)$, and by hypothesis $\mathcal{P}_k$ is nonempty.  Note also that 
for all $k$ we have $\mathcal{P}_{k+1} \subset \mathcal{P}_k$.  Thus the $\mathcal{P}_k$
form a nested descending sequence of nonempty compact subsets of $\mathcal{P}(X)$.
We conclude that 
$$\bigcap_{k=1}^\infty \mathcal{P}_k \ne \emptyset.$$

Also observing that if 
$g, h$ belong to $G_k$ then $gh$ is in $G_{2k}$, we see that if 
$P \in \cap_{k=1}^\infty \mathcal{P}_k$ then $P \in LO(G)$ and we conclude that in fact 
$$LO(G) = \bigcap_{k=1}^\infty \mathcal{P}_k \ne \emptyset,$$
completing the proof.
\end{proof}

In the case of a finitely generated group, it is a finite task to check whether or not there exists a proper $k$-partition of $G_k$ for a particular fixed $k$.  If one can decide the word problem algorithmically for $G$ (with given generators), then there is an algorithm to decide whether a proper $k$-partition exists.
This means that if a finitely-generated group is not left-orderable, then the algorithm will discover that fact in finite time (although one does not know when!)  Moreover, one can design the algorithm to supply a proof of non-left-orderability if it finds a $G_k$ having no proper partition.  On the other hand, if the group under scrutiny {\em is} left-orderable, the algorithm will never end.  An example of such an algorithm, due to Nathan Dunfield, is described in \cite{CD03} and is available from his website.  In \cite{CD03} this algorithm was used to discover Example \ref{Weeks}, showing a certain torsion-free group (the fundamental group of the Weeks manifold) is not left-orderable.

\begin{theorem}\label{fgLO}
A group is left-orderable if and only if each of its finitely-generated subgroups is left-orderable.
\end{theorem}

   The ``only if'' part is trivial.  The proof in the other direction will use the following version of compactness.  A collection of sets is said to have the {\em finite intersection property} \index{finite intersection property} if every finite subcollection of the sets has a nonempty intersection.

\begin{problem}
A topological space is compact if and only if every collection of closed subsets with the finite intersection property has a nonempty total intersection.
\end{problem}

To prove the nontrivial part of Theorem \ref{fgLO}, consider any finite subset $F$ of the given group $G$ and let
$\langle F \rangle$ denote the subgroup of $G$ generated by $F$.  Define
$$\mathcal{Q}(F) = \{ Q \subset G : Q \cap \langle F \rangle \text{ is a positive cone for }\langle F \rangle \}$$
For each finite $F \subset G$, $\mathcal{Q}(F)$ is a closed subset of $\mathcal{P}(G)$.  The family of all $\mathcal{Q}(F)$, for finite $F \subset G$, is a collection of closed sets which has the finite intersection property, because 
$$ \mathcal{Q}(F_1 \cup F_2 \cup \cdots \cup F_n) \subset \mathcal{Q}(F_1) \cap \mathcal{Q}(F_2) \cap \cdots \cap \mathcal{Q}(F_n).$$

By compactness, $\bigcap_{F \subset G \hspace{1mm} \rm{finite}} \mathcal{Q}(F) \ne \emptyset$.

\begin{problem}  Verify that any element of $\bigcap_{F \subset G \hspace{1mm} \rm{finite}} \mathcal{Q}(F)$
is a left-ordering of $G$, completing the proof.  In fact 
$$\bigcap_{F \subset G \hspace{1mm} \rm{finite}} \mathcal{Q}(F) = LO(G).$$
\end{problem}

\begin{theorem}
\label{torsion free abelian}
An abelian group $G$ is bi-orderable if and only if it is torsion-free.
\end{theorem}

\proof  We need only show that torsion-free abelian groups are left-orderable (which in this case is equivalent to bi-orderable).  But any finitely generated subgroup is isomorphic to 
$\Z^n$ for some $n$, which we have already seen to be bi-orderable (Example \ref{orderZ2}).  The result follows from Theorem \ref{fgLO}. \qed

\section{Characterization of left-orderable groups}  

Following \cite{Conrad59}, we have a number of characterizations of left-orderability of a group $G$.  If $X \subset G$, we let $S(X)$ denote the semigroup generated by $X$, that is all elements of $G$ expressible as (nonempty) products of elements of $X$ (no inverses allowed).

\begin{theorem}\label{loequiv}
A group $G$ can be left-ordered if and only if for every finite subset  $ \{x_1, \dots , x_n \}$ of $G$ which does not contain the identity, there exist $\e_i = \pm 1$ such that 
$1 \not\in S(\{x_1^{\e_1}, \dots , x_n^{\e_n}\})$.
\end{theorem}

One direction is clear, for if $<$ is a left-ordering of $G$, just choose $\e_i$ so that $x_i^{\e_1}$ is greater than the identity.  For the converse,
by Theorem \ref{fgLO} we may assume that $G$ is finitely generated, and by Theorem \ref{partition} we need only show that each $k$-ball $G_k$, with respect to a fixed finite generating set, has a proper $k$-partition.  To do this, let $ \{x_1, \dots , x_n \}$ denote the entire set $G_k \setminus \{1\}$, and choose $\e_i = \pm 1$ such that 
$1 \not\in S(\{x_1^{\e_1}, \dots , x_n^{\e_n}\})$.  

\begin{problem}  Show that the set $G_k \cap S(\{x_1^{\e_1}, \dots , x_n^{\e_n}\})$ is a proper $k$-partition of $G_k$, completing the proof of Theorem \ref{loequiv}.
\end{problem}

Another characterization of left-orderability is due to Burns and Hale \cite{BH72}.

\begin{theorem}[Burns-Hale]\label{burnshale}
\index{Burns-Hale theorem}
A group $G$ is left-orderable if and only if for every finitely-generated subgroup 
$H \ne \{ 1 \}$ of $G$, there exists a left-orderable group $L$ and a nontrivial homomorphism $H \to L$.
\end{theorem}

\begin{proof}  One direction is obvious.  To prove the other direction, assume the subgroup condition.  According to Theorem \ref{loequiv}, the result will follow if one can show:

Claim:  For every finite subset  $ \{x_1, \dots , x_n \}$ of $G \setminus \{1\}$ , there exist $\e_i = \pm 1$ such that  $1 \not\in S(x_1^{\e_1}, \dots , x_n^{\e_n})$.

We will establish this claim by induction on $n$.  It is certainly true for $n = 1$, for $S(x_1)$
cannot contain the identity unless $x_1$ has finite order, which is impossible since the cyclic subgroup $\langle x_1\rangle$ must map nontrivially to a left-orderable group.

Next assume the claim is true for all finite subsets of $G \setminus \{1\}$ having fewer than $n$ elements, and consider $ \{x_1, \dots , x_n \} \subset G \setminus \{1\}$.  By hypothesis, there is a nontrivial homomorphism 
$$h : \langle x_1, \dots , x_n   \rangle \rightarrow L$$ 
where $(L, \prec)$ is a left-ordered group.  Not all the $x_i$ are in the kernel since the homomorphism is nontrivial, so we may assume they are numbered so that 
$$h(x_i) 
  \begin{cases}
   \ne 1 \text{ if } i= 1, \dots, r, \\
   = 1  \text{ if } r < i \le n.
   \end{cases}
$$   
Now choose $\e_1, \dots, \e_r$ so that $ 1 \prec h(x_i^{\e_i})$ in $L$ for $i= 1, \dots, r$.
For $i > r$, the induction hypothesis allows us to choose $\e_i = \pm 1$ so that 
$1 \not\in S(x_{r+1}^{\e_{r+1}}, \dots, x_n^{\e_n})$.  We now check that $1 \not\in S(x_1^{\e_1}, \dots , x_n^{\e_n})$ by contradiction.  Suppose that $1$ {\em is} a product of some of the $x_i^{\e_i}$.  If all the $i$ are greater than $r$, this is impossible, as $1 \not\in S(x_{r+1}^{\e_{r+1}}, \dots, x_n^{\e_n})$.  On the other hand if some $i$ is less than or equal to $r$, we see that $h$ must send the product to an element strictly greater than the identity in $L$, again a contradiction.  \end{proof}

A group is said to be {\em indicable} if it has the group of integers $\Z$ as a quotient, and \index{locally indicable} {\it locally indicable} if each of its nontrivial finitely-generated subgroups is indicable.  This notion was introduced by Higman \cite{Higman40} to study zero divisors and units in group rings (see Section \ref{zero divisor section}).

\begin{corollary}
Locally indicable groups are left-orderable.
\end{corollary}

\begin{corollary}
Suppose $G$ is a group which has a (finite or infinite) family of normal subgroups 
$\{G_\alpha\}$ such that  $\cap_\alpha G_\alpha = \{ 1 \}$.  
If all the factor groups $G/G_\alpha$ are 
left-orderable, then $G$ is left-orderable.
\end{corollary}
\begin{proof}
If $H$ is a finitely generated nontrivial  subgroup of $G$, one can choose 
$\alpha$ for which $H \setminus G_\alpha$ is nonempty.  Then the composition of homomorphisms $H \to G \to G/G_\alpha$ is a nontrivial homomorphism of $H$ to a 
left-orderable group. \end{proof}

\begin{problem} 

Show that each of the following conditions on a group $G$ is equivalent to left-orderability: 

(1)  For each element $g \ne 1$ in $G$, there exists a subsemigroup $S_g$ of $G$ which contains $g$ but not $1$ and such that $G \setminus S_g$ is also a semigroup.

(2)  For each finite subset $x_1, \dots , x_n$ of $G$, the intersection of the $2^n$ subsemigroups
$S(1, x_1^{\e_1}, \dots, x_n^{\e_n})$ is equal to $\{ 1 \}$, where the $\e_i$ are $\pm 1$.

(3) There exists a set ${\bf S}$ of subsemigroups of $G$ whose intersection is $\{ 1\}$ and such that for every $g \in G$ and $S \in {\bf S}$, either $g \in S$ or $g^{-1} \in S$. 

See \cite{Conrad59} if you get stuck, but note that he uses the right-ordering convention.
\end{problem}

A subset $Q$ of a group $G$ is called a {\it partial left-order} if it is a subsemigroup 
($Q\cdot Q \subset Q$) such that $Q \cap Q^{-1} = \emptyset$.  $Q$ can be regarded as the positive cone of a left-invariant {\em partial}     order of the group.  In particular, $Q$ corresponds to a total left-order if and only if $G \setminus \{ 1\} = Q \cup Q^{-1}$.  If $Q$ and $Q'$ are partial 
left-orders such that $Q \subset Q'$, then $Q'$ is called an {\it extension} of $Q$.  The following is a useful criterion for a partial order to extend to a total one.

\begin{problem}
A partial left-order $Q$ on $G$ has an extension to a total left-order if and only if whenever
 $ \{x_1, \dots , x_n \}$ is a finite subset of $G$ which does not contain the identity $1$ of $G$, there exist $\e_i = \pm 1$ such that $1 \not\in S(Q \cup \{x_1^{\e_1}, \dots , x_n^{\e_n}\})$.
\end{problem}

\section{Group rings and zero divisors}
\label{zero divisor section}
We will now discuss one of the algebraic reasons it is worth knowing that a group is left-orderable.

If $R$ is a ring with identity and $G$ is a group (written multiplicatively), then the group ring $RG$ is defined to be the free left $R$-module generated by the elements of $G$, endowed with a natural multiplication analogous to products of polynomials.  That is, a typical element of $RG$ is a finite formal linear combination
$$
\sum_{i=1}^m r_ig_i
$$
with $r_i \in R$ and $g_i \in G$.  The product is defined by the formula
\begin{equation}
\left(\sum_{i=1}^m r_ig_i \right) \left(\sum_{j=1}^n s_jh_j \right) = 
\sum_{i=1}^m\sum_{j=1}^n r_is_j(g_ih_j)
\label{prod}
\end{equation}
Of course, on the right-hand side of Equation (\ref{prod}), cancellations may be possible, and this leads to some mischief, as the example below illustrates.  If $1$ is the identity of $G$, then the group ring element $r1$ is customarily denoted simply as $r$, and likewise for the ring identity, also denoted by $1$, $1g$ may be abbreviated as $g$.

Group rings (known as group algebras if $R$ is a field) arise naturally in representation theory, algebraic topology, Galois theory, etc.  An important problem is the so-called zero-divisor conjecture, which dates back at least to the 1940's, often attributed to Kaplansky.  It remains unsolved even for the case $R = \Z$.  Recall that an element $\alpha \ne 0$ of a ring is called a {\em zero divisor} if there exists another ring element $\beta \ne 0$ such that 
$\alpha\beta = 0$.

\begin{conjecture}[Zero divisor conjecture]  \label{zerodivisors}
\index{Zero divisor conjecture}
If $R$ is a ring without zero divisors and $G$ is a torsion-free group, then $RG$ has no zero divisors.
\end{conjecture}

One of the strongest reasons for knowing whether a group is orderable is that the zero divisor conjecture is true for left-orderable groups.  Before proving this, let us discuss by example how zero divisors, and nontrivial units (elements with inverses), can arise in group rings.  If $r$ is an invertible element of $R$ and $g$ an arbitrary element of $G$, then the ``monomial'' $rg$ is clearly a unit of $RG$: $(rg)(r^{-1}g^{-1}) = 1$.  Such a unit is called a {\it trivial} unit of $RG$.

\begin{example}
Consider the ring of integers $R = \Z$ and the cyclic group of order five, 
$G = \langle x \mid  x^5 = 1 \rangle$.  Define the following elements of $RG$:
$$
\alpha =  1 + x + x^2 + x^3 + x^4, \quad \beta = 1 - x, 
\quad \gamma =  1 - x^2 - x^3, \quad \delta = 1 - x - x^4
$$
\end{example}

\begin{problem}  Verify that $\alpha\beta = 0$ and $\gamma\delta = 1$.  Therefore, the group ring in this example has zero divisors and nontrivial units as well.
\end{problem}

The existence of nontrivial units in group rings, like the zero divisor problem, is a notoriously difficult problem in algebra.  However, for left-orderable groups the answer is straightforward.

\begin{theorem}
If $R$ is a ring without zero divisors and $G$ is a left-orderable group, then the group ring $RG$ does not have zero divisors or nontrivial units. 
\end{theorem} 

\begin{proof} Consider a product, as in Equation (\ref{prod}), where we assume that
the $r_i$ and $s_j$ are all nonzero, the $g_i$ are distinct and the $h_j$ are written in strictly ascending order, with respect to a given left-ordering of $G$.  At least one of the group elements $g_ih_j$ on the right-hand side of (\ref{prod}) is minimal in the left-ordering.  If $j > 1$ we have, by left-invariance, that
$g_ih_1 < g_ih_j$ and $g_ih_j$ is not minimal.  Therefore we must have $j=1$.  On the other hand, since we are in a group and the $g_i$ are distinct, we have that $g_ih_1 \ne g_kh_1$ for any $k\ne i$.  We have established that there is exactly one minimal term on the right-hand side of (\ref{prod}), and similarly there is exactly one maximal term.  It follows that they survive any cancellation, and so the right-hand side cannot be zero (because $r_is_1 \ne 0$).  Thus $RG$ has no zero divisors.  If one of $n$ or $m$ is greater than one, there are at least two terms on the right-hand side of 
(\ref{prod}) which do not cancel, so the product cannot equal $1$.  This implies that all units of $RG$ are trivial. \end{proof}

\section{Torsion-free groups which are not left-orderable}  
\label{torsion free nonlo section}

Left-orderable groups are torsion-free, but there are many examples to show the converse is far from true.  One of the simplest examples, which has appeared several times in the literature, is the following.

%

%

\begin{example}
\index{crystallographic group}
We will consider a crystallographic group $G$ which is torsion-free but not left-orderable.
Specifically consider the group $G$ with generators $a, b, c$ acting on $\R^3$ with coordinates $(x, y, z)$ by the rigid motions:
$$a(x, y, z) = (x+1, 1-y, -z)$$
$$b(x, y, z) = (-x, y+1, 1-z)$$
$$c(x, y, z) = (1-x, -y, z+1)$$   

One can easily check the relations $a^2 b a^2 = b, b^2 a b^2 = a$ and  $abc = id$.  By the last relation we see that one generator may be eliminated.  In fact $G$ has the presentation
$G = \langle a, b \mid a^2 b a^2 = b, b^2 a b^2 = a \rangle$.
\end{example}

\begin{problem}
Check the relations cited above.  Argue that the group $G$ is torsion-free.
\end{problem}

\begin{problem}
Argue that $G$ is not left-orderable as follows.  First show that for all choices of 
$m,n \in \{-1, +1\}$ one has $a^{2m}b^na^{2m} = b^n$ and $b^{2n}a^mb^{2n} = a^m$.  Then argue that
\begin{align*}
(a^m b^n)^2(b^n a^m)^2 &= 
a^m b^{-n} b^{2n} a^m b^{2n} a^m b^n a^m \\
&= a^m b^{-n} a^{2m} b^{n} a^{2m} a^{-m} \\
&= a^m b^{-n} b^{n} a^{-m} = 1
\end{align*}
Conclude that if $G$ were left-orderable, all choices of sign for $a$ and $b$ would lead to a contradiction.
\end{problem}

\begin{problem}
Show that the subgroup $A = \langle a^2, b^2, c^2 \rangle$ is generated by shifts (by even integral amounts) in the directions of the coordinate axes, and so is a free abelian group of rank 3.  Moreover $A$ is normal in $G$ and of finite index.  Therefore $G$ is virtually bi-orderable, in the sense that a finite index subgroup is bi-orderable.
\end{problem}

\index{Klein bottle group} 

Next we will construct an infinite family of examples.  Consider the Klein bottle \index{Klein bottle} group
$K = \langle a, b \mid a^2 = b^2\rangle$.   

\begin{problem} Verify that $a^2$ and $ab$ commute, that the subgroup $H = \langle a^2, ab \rangle$ is an index two subgroup of $K$ and that $H \cong \Z^2$.
\end{problem}

In fact, $H$ can be regarded as the fundamental group of the 2-dimensional torus which double-covers the Klein bottle as in Figure \ref{torus double cover}, the so-called oriented double cover. 

\begin{figure}%
\includegraphics[scale=0.8]{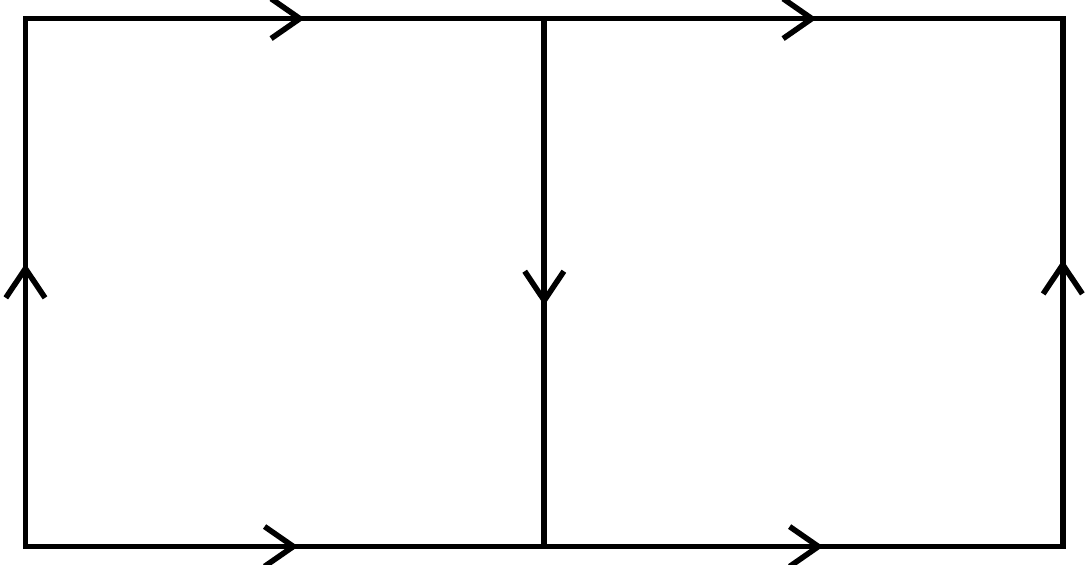}
\caption{The torus as a rectangle with opposite sides identified, which we can subdivide into two Klein bottles as shown.}
\label{torus double cover}
\end{figure}
 
   Alternatively, we can realize $K$ as the 2-dimensional crystallographic group generated by the glide reflections $$a(x, y) = (x+1, -y) \quad b(x,y) = (x+1, 1-y)$$ and $H$ as the subgroup of orientation-preserving motions.

Now take two copies $K_1$ and $K_2$ of the Klein bottle group, and amalgamate them along their corresponding subgroups $H_1$ and $H_2$.  An isomorphism 
$\phi: H_1 \to H_2$ is given by a $2 \times  2$ matrix (using the bases $\{a_i^2, a_ib_i\}$)
$$\phi \sim 
\left(
\begin{array}{cc}
p & q   \\
r  & s   \\ 
\end{array}
\right) $$
with determinant $\pm 1$.  We take this to mean, in multiplicative notation,
$$\phi(a_1^2)=(a_2^2)^p(a_2 b_2)^q; \quad \phi(a_1b_1) = (a_2^2)^r(a_2 b_2)^s$$

This identification defines an amalgamated free product
$$G_\phi := K_1 *_{\phi}K_2$$
which has the presentation
$$
G_\phi = \langle a_1, b_1, a_2, b_2 \mid a_1^2 = b_1^2, \; a_2^2 = b_2^2, \;
a_1^2 = (a_2^2)^p(a_2 b_2)^q,\; a_1b_1 = (a_2^2)^r(a_2 b_2)^s \rangle
$$

The groups $G_\phi$ are torsion-free, since they are amalgamated products of torsion-free (in fact left-orderable) groups.  This can be seen by considering the normal form for elements of an amalgamated free product, see for example \cite{Serre80}, Section 1.3, Corollary 2.

\begin{example}
\label{glued Klein bottles}
Suppose $p, q \ge 0$ and $r, s \le 0$ (or vice-versa).  Then $G_\phi$ is \textit{not} left-orderable.
\end{example}

To see this, suppose for contradiction that $G_\phi$ is left-orderable.  Then the first relation implies that $a_1$ and $b_1$ must have the same sign (either both are positive or both are negative) and the second implies  $a_2$ and $b_2$ also have the same sign.  The third relation 
implies that $a_1$ (and hence $b_1$) has the same sign as $a_2$ and $b_2$ (note that one of $p$ or $q$ must be strictly positive).  But then the last relation implies $a_1b_1$ has the opposite sign as $a_2$ and $b_2$, the desired contradiction. \qed

\begin{problem}
Calculate that the abelianization of $G_\phi$ is a finite group of order $16|p + q -r -s|$, and therefore this construction provides infinitely many non-isomorphic groups which are torsion-free but not left-orderable.
\end{problem}

It will be seen later that the $G_\phi$ are the fundamental groups of an interesting class of  3-manifolds: the union of two twisted $I$-bundles over the Klein bottle.  Further examples of torsion-free groups which are not left-orderable are discussed in Chapter \ref{three manifolds chapter}.

Finally, we mention a useful result, due independently to Brodskii \cite{Brodskii84} and Howie \cite{Howie82}.  See also \cite{Howie00} for a simpler proof.  The difficult direction is to show that torsion-free implies locally indicable.
\index{locally indicable}

\begin{theorem}  If $G$ is a group which has a presentation with a single relation, the following are equivalent:
\index{one-relator groups}
\begin{enumerate} 
\item $G$ is torsion-free 
\item $ G$ is locally indicable  
\item $G$ is left-orderable.
\end{enumerate}
\end{theorem}

Note that the examples of torsion-free non-left-orderable groups described above have two or more defining relations.

We end this chapter with an open question.  Chehata \cite{Chehata52} constructed a bi-orderable group which is simple.  But the example is uncountable, and therefore not finitely generated. In fact, every bi-orderable simple group must be infinitely generated, because finitely generated bi-orderable groups have infinite abelianization (for a proof of this fact, see Theorem \ref{biorderable implies LI}).

\begin{question}
Is there a finitely generated left-orderable simple group?
\end{question}

\chapter{H\"{o}lder's theorem, convex subgroups and dynamics}\label{Holder chapter}

In this chapter we introduce some of the essential dynamical properties of left-orderings of groups.

\section{H\"{o}lder's Theorem}

A left-ordering $<$ of a group $G$ is called \textit{Archimedean} \index{Archimedean ordering} if for every pair of positive elements $g,h \in G$ there exists $n > 0 $ such that $h <g^n$.   For example, the standard orderings of $(\Q, +)$ and $(\R, +)$ are Archimedean.

\begin{problem}
Verify that the orderings of $\mathbb{Z}^2$ constructed in Example \ref{orderZ2} are Archimedean, whenever the vector $\vec{v} \in R^2$ has irrational slope.  On the other hand, the lexicographic ordering is not Archimedian.
\end{problem}

There is a reason why these few examples of Archimedean ordered groups are rather simple.   It turns out that all Archimedian 
left-orderings must be bi-orderings, from which we can prove that Archimedean ordered groups are abelian.  We begin with a proof of these two facts.

\begin{lemma}\cite[Theorem 3.8]{Conrad59}
\label{lem:conrad}
Every Archimedean left-ordering is a bi-ordering.
\end{lemma}
\begin{proof}
Let $P$ denote the positive cone of an Archimedean left-ordering $<$ of a group $G$.  In order to show that $<$ is a bi-ordering, we must show that $g^{-1}Pg \subset P$ for all $g \in G$.

So, let $h \in P$ and $g \in G$, and first we will suppose that $g$ is positive.   Because the ordering is Archimedean there exists $n >0$ such that $ g<h^n$.  Therefore $1 < g^{-1} h^n$, and so $1 < g^{-1} h^n g$ since it is a product of the positive elements $g^{-1} h^n$ and $g$.  Now $1<g^{-1} h g$ since its $n$-th power is positive, and we conclude that $g^{-1} h g \in P$ for all $h \in P$.  In other words, $g^{-1} P g \subset P$.

In the second case where $g$ is negative and $h \in P$, suppose that $g^{-1} h g \notin P$ and we'll arrive at a contradiction.  Since  $g^{-1} h g \notin P$ we have 
\[ 1 < (g^{-1} h g)^{-1} = g^{-1} h^{-1} g,
\]
and by the previous paragraph, conjugation of this element by the positive element $g^{-1}$ will give a positive element.  In other words
\[ 1 < g(g^{-1} h^{-1} g)g^{-1} = h^{-1},
\]
a contradiction.  Thus $g^{-1} P g \subset P$ for negative $g \in G$ as well.
\end{proof}

\begin{problem}  Show that in an Archimedean ordered group, for every nonidentity element $g$ and every  $h \in G$ there exists $n \in \mathbb{Z}$ such that $g^n \leq h < g^{n+1}$.
\end{problem}

\begin{lemma}
\label{lem:archabel}
Every Archimedean left-ordered group is abelian.
\end{lemma}

\pf  By the above, the ordering is bi-invariant.  We consider two cases.  

 Case 1:  The positive cone $P$ has a least element $p$.  Then we claim that the infinite cyclic subgroup $\langle p \rangle$ is the whole of $G$.  For if $g \in G \setminus \langle p \rangle$, there exists $n$ such that $p^n < g < p^{n+1}$ and therefore $1 < p^{-n}g < p$, contradicting minimality of $p$.  So $G \cong \Z$ in this case, and the theorem follows.

 Case 2: $P$ does not have a least element.  By way of contradiction, suppose $g, h \in G$ do not commute.  Without loss of generality, we may assume $g$ and $h$ and their commutator $ghg^{-1}h^{-1}$ are all positive.  Lemma \ref{square} guarantees the existence of $x > 1$ in $G$ such that $1 < x^2 < ghg^{-1}h^{-1}$.  Using the Archimedean property, there exist integers $m, n$ such that $x^m \le g < x^{m+1}$ and $x^n \le h < x^{n+1}$.
Then $g^{-1} \le x^{-m}$ and $h^{-1} \le x^{-n}$.  Multiplying the appropriate inequalities implies  
$ghg^{-1}h^{-1} < x^{m+1+n+1-m-n} = x^2$, a contradiction. \qed

\begin{lemma}\label{square}
If $G$ is bi-ordered and does not have a least positive element, then given $p>1$ there exists $q>1$ in $G$ such that $1 < q^2 < p$.
\end{lemma}

\pf Let $p > r > s > 1$, and consider $rs^{-1} > 1$.  If 
$(rs^{-1})^2 \ge r$, then $s^{-1}rs^{-1} \ge 1$ and $r \ge s^2$, so we can choose $q = s$.  Otherwise, let $q = rs^{-1}$. \qed

In the same sense that $\mathrm{Homeo}_+(\mathbb{R})$ is universal for countable left-orderable groups (see Theorem \ref{LO_universal}), H\"{o}lder's theorem tells us that the group $(\mathbb{R}, +)$ is universal for Archimedean ordered groups.

\begin{theorem}[H\"{o}lder 1901 , \cite{Holder01}]
\index{H\"{o}lder's Theorem}
\label{theorem:Holder}
If $G$ is a group with an Archimedian left-ordering, then $G$ is isomorphic with a subgroup of the additive reals, by an isomorphism under which the ordering of $G$ corresponds to the usual order of $\R$.
\end{theorem}
\begin{proof}

We first fix a nonidentity element $f \in G$ and note that any homomorphism $\phi : G \rightarrow (\mathbb{R}, +)$ can be post-composed with multiplication by the real number $\frac{1}{\phi(f)}$ in order to produce a homomorphism with $\phi(f) =1$.  So if we wish to show that there is a homomorphism $\phi:G \rightarrow (\mathbb{R}, +)$, there is no harm in beginning with $\phi(f) =1$.

Now for each $g\in G$ and  $n \in \mathbb{Z}$, an application of the Archimedean property yields a corresponding integer $a_n \in \mathbb{Z}$ such that 
\[  f^{a_n} \leq g^n < f^{a_n+1}
 \]
Thus if we are to succeed in creating an order-preserving homomorphism $\phi : G \rightarrow (\mathbb{R}, +)$ with $\phi(f)=1$, these inequalities give
\[ a_n \leq n \phi(g) < a_n+1
\]
for all $n$.  In particular, it means that we are forced to set
\[ \phi(g) = \lim_{n \to \infty} \frac{a_n}{n}
\]
whenever the limit exists.  

It turns out that this limit exists for all $g \in G$, so no matter our approach this must be the value that we assign to $\phi(g)$ once we fix $\phi(f) =1$.   However proving convergence of the sequence $\{a_n/n \}$ is a bit tricker than if we pass to the subsequence
\[ \phi_k(g) = \frac{a_{2^k}}{2^k},
\]
and proceed as in the following exercise.

\begin{problem} Verify that 
$|\phi_{k+1}(g) - \phi_k(g)| \le 1/2^k$, and conclude that $\{\phi_k(g)\}$ is a Cauchy sequence.  Hence there is a limit.  

Define
$$
\phi (g) = \lim_{k \to \infty} \phi_k (g)
$$
and verify (here the commutativity of $G$ is needed) that for any
$g, h \in G$,
$$
|\phi_k (gh) - \phi_k(g) - \phi_k(h)| \le 1/2^k .
$$
Conclude that $\phi$ is a homomorphism.  

Finally, verify that if $g > 1$ in $G$, then $\phi(g) > 0$ in $\R$.  Conclude both that $\phi$ is injective and order-preserving.
\end{problem}
\end{proof}

\begin{problem} 
Recall the bi-ordering of $\Z^2$ introduced in Example \ref{orderZ2}.   For $\vec{v} \in \R^2$ with irrational slope, and any two vectors $\vec{m} = (m_1, m_2), \vec{n} = (n_1, n_2) \in \Z^2$, we have
$$\vec{m} < \vec{n} \iff  m_1v_1 + m_2v_2 < n_1v_1 + n_2v_2.$$ 
Define a map $\phi: \Z^2 \rightarrow (\R, +)$ by
\[ \phi( \vec{m}) = \frac{ \vec{m} \cdot \vec{v}}{||\vec{v}||},
\]
note that this is the formula for orthogonal projection onto $\vec{v}$.  Verify that $\phi$ is order preserving and injective.

\end{problem}

\section{Convex subgroups}

Suppose $G$ is a left-ordered group with ordering $<$.   A subset $C \subset G$ is \textit{convex} \index{convex subgroup} relative to $<$ if for all $g, h \in C$ and $f \in G$, the implication $g<f<h \Rightarrow f \in C$ holds.   A subset $C \subset G$ is \textit{relatively convex} \index{relatively convex} if there exists an ordering of $G$ relative to which $C$ is convex.  We will be primarily interested in the case where $C$ is a subgroup of $G$, in which case $C$ is called a convex subgroup or a relatively convex subgroup of $G$.

\begin{problem}
\label{prop:convex_intersection_and_union}
Let $G$ be a group with left-ordering $<$, and suppose that $C$ and $D$ are subgroups that are convex relative to $<$.  Show that either $C \subset D$, or $D \subset C$.   
\end{problem}

The conclusion of Problem \ref{prop:convex_intersection_and_union} is often stated simply as ``the convex subgroups of $G$ with ordering $<$ are linearly ordered by inclusion."  If $C \subset D$ are convex, then the pair $(C, D)$ is called a {\em convex jump} \index{convex jump} if there is no convex subgroup strictly between them.

\begin{problem}
\label{problem:convex_power}
 Let $G$ be a left-ordered group with convex subgroup $C$.  Given an element $g \in G$ and integer $n \ne 0$, show that $g^n \in C$ implies $g \in C$.
\end{problem}

\begin{problem}
Show that an Archimedean ordered group has no convex subgroups other than the trivial ones: the whole group and $\{1 \}$.
\end{problem}

\begin{problem}
Show that the orderings of $\Z^2$ defined in Example \ref{orderZ2} have a nontrivial convex subgroup whenever the vector $\vec{v}$ has rational slope.
\end{problem}

Convex subgroups of a left-orderable group are closely related to orderability of its quotients.  If $G \rightarrow H$ is a homomorphism from a left-orderable group $G$ onto a left-orderable group $H$, then its kernel is relatively convex.  Conversely if the kernel of some homomorphism $\phi: G \rightarrow H$ is relatively convex then the image $\phi(G)$ is left-orderable.   This is essentially the content of Problem \ref{extension}, reworded using our new definitions.  

By using the notion of convexity of a subgroup, we can consider subgroups $C \subset G$ that are not normal in $G$, but whose left cosets nonetheless admit an ordering that is invariant under the left action of $G$.

\begin{problem} 
\index{lexicographic ordering}
\label{extension_generalization}
Prove the following generalization of Problem \ref{extension}, which created lexicographic orderings via short exact sequences:
Suppose that $C$ is a subgroup of $G$, denote the set of left cosets $\{gC \}_{g \in G}$ by $G/C$.  The subgroup $C$ is relatively convex in $G$ if and only if there exists an ordering $\prec$ of the cosets $G/C$ that is invariant under left multiplication by $G$, i.e. $gC \prec hC$ implies $fgC \prec fhC$ for all $f, g, h \in G$.
\end{problem}

A consequence of the previous problem is that whenever a subgroup $C$ is convex in a left-ordering of $G$, we can choose a different left-ordering of $C$ while keeping the same ordering of the left cosets $G/C$.  This gives a way of making new left-orderings of $G$. 

\begin{problem}
\label{nested_subgroups_problem}
Suppose that $C$ is a normal, convex subgroup of a left-ordered group $G$ with ordering $<$.   Show that a subgroup $H$ of $G$ satisfying $C \subset H \subset G$ is convex relative to the ordering $<$ of $G$ if and only of $H/C$ is convex in $G/C$ relative to the natural quotient ordering.
\end{problem}



\begin{problem}
Fix an ordering of $\Z^n$.  Show that there can be at most $n -1$ proper, nontrivial subgroups of $\Z^n$ that are convex relative to this ordering.
\end{problem}

Recall that a (left) action of a group $G$ \index{group action} on a set $X$ is a binary operation $G \times X \to X$ which satisfies $1x = x$ and
$(gh)x = g(hx)$ for all $g, h \in G, x \in X$.  The {\em stabilizer} \index{stabilizer} of $x \in X$ under the action is the subgroup $\{g \in G \mid gx = x\}$. One says $G$ acts {\em effectively} \index{effective action} if whenever $gx = x$ for all $x \in X$, then $g = 1$.  If $X$ is linearly ordered by $<$, then the action is {\em order-preserving} if $x < y$ implies $gx < gy$. 

\begin{problem}
\label{relatively_convex_intersection_problem}
Suppose that $G$ acts effectively (on the left) on a linearly ordered set $X$ by order-preserving bijections.  For a given $x_0 \in X$, choose a well-ordering $\prec$ of $X$ for which $x_0$ is the smallest element.  Construct a left-ordering $<$ of $G$ as in the proof of Example \ref{homeo+}.  Show that the stabilizer of $x_0$ is convex relative to the ordering $<$ of $G$. 

Therefore the stabilizer of each $x \in X$ is a relatively convex subgroup of $G$.
\end{problem}

\begin{problem}
\label{intersection_and_union}
Fix a left-ordering $<$ of a group $G$.   Show that an arbitrary union of convex subgroups of $G$ is a convex subgroup, and an arbitrary intersection of convex subgroups of $G$ is a convex subgroup.
\end{problem}

 In fact when considering the intersection of convex subgroups, we do not need that they all be convex relative to the same left-ordering of the group.   It suffices that each subgroup be relatively convex in order to conclude that their intersection is also relatively convex, though the proof is trickier than the solution to Problem \ref{intersection_and_union}.
\begin{proposition}
An arbitrary intersection of relatively convex subgroups of a left-orderable group is a relatively convex subgroup. 
\end{proposition}
\begin{proof}
Let $\{C_i\}_{i \in I}$ be a family of relatively convex subgroups of a group $G$, assume that for each $i \in I$ the subgroup $C_i$ is convex relative to the left-ordering $<_i$.  By Problem \ref{extension_generalization}, each set of left cosets $G/C_i$ admits an ordering $\prec_i$ that is invariant under the left action of $G$.  Fix an arbitrary well-ordering of the index set $I$, and use the well-ordering of $I$ to define a lexicographic ordering of product  $\Pi_{i \in I} G/C_i$, which we will denote by $\prec$.  Since $\prec$ restricts to the ordering $\prec_i$ on each factor $G/C_i$, it is preserved by the left-action of $G$ on $\Pi_{i \in I} G/C_i$.

We would like to use the left action of $G$ on $\Pi_{i \in I} G/C_i$ to create a left-ordering of $G$, but the action may not be effective.  Therefore we correct this problem as follows.  Fix a left-ordering $<$ of $G$, and order the union 
\[(\Pi_{i \in I} G/C_i ) \cup G
\]
by ordering $\Pi_{i \in I} G/C_i$ using $\prec$, ordering $G$ using $<$, and declaring that every element of $G$ is greater than every element of  $\Pi_{i \in I} G/C_i$.  The result is an effective, order-preserving action of $G$ on the totally ordered set $(\Pi_{i \in I} G/C_i) \cup G$.  The stabilizer of the element $(C_i)_{i\in I} \in \Pi_{i \in I} G/C_i$ is the intersection $\bigcap_{i \in I} C_i$, so by Problem \ref{relatively_convex_intersection_problem}, the intersection is relatively convex.  That it is a subgroup is clear as any intersection of subgroups is a subgroup.
\end{proof}

\section{Bi-orderable groups are locally indicable}
\label{section: BO implies LI}
\index{locally indicable}
With these preparations, we can argue that all bi-orderable groups are locally indicable: Consider a finitely generated bi-ordered group $G$ with generators which we may take to be positive and ordered $1 < g_1 < g_2 < \cdots < g_n$.  Let $C$ be the maximal convex subgroup of $G$ which  does not contain the largest generator $g_n$, that is, $C$ is the union of all convex subgroups not containing $g_n$. Clearly the only convex subgroup which properly contains $C$ is $G$ itself.  Since bi-orderings are conjugation invariant, so is the convex subgroup $C$, meaning that $C$ is normal and $G/C$ inherits a bi-ordering.  We wish to argue that the ordering of $G/C$ is Archimedean, so for contradiction suppose it is not.  Then there is an element $g \in G/C$ so that the set 
\[ H = \{ h \in G/C \mid \mbox{there exists $k \in \mathbb{Z}$ such that } g^{-k}<h <g^k \}
\]
is a proper convex subgroup of $G/C$.  By Problem \ref{nested_subgroups_problem}, $H$ corresponds to a proper, convex subgroup $D$ of $G$ which contains $C$.  Since $D$ is proper it cannot contain $g_n$---but this contradicts maximality of $C$.  Thus we have proved that $(C, G)$ is a convex jump and $G/C$ is a nontrivial finitely-generated torsion-free abelian group.  The homomorphisms $G \to G/C \to \Z$ establish the following, which had been observed by Levi \cite{Levi42}.

\begin{theorem}  
\label{biorderable implies LI}
If $G$ is a finitely generated bi-orderable group, then there is a surjective homomorphism $G \rightarrow \mathbb{Z}$.
\end{theorem}

\begin{corollary}
Bi-orderable groups are locally indicable.
\end{corollary}

Note that the converse of this corollary is not true.  For example the Klein bottle group is locally indicable, but it is not bi-orderable.  Local indicability and its relationship with orderability is discussed in detail in Chapter \ref{Conrad chapter}.

\begin{problem}
Let $G$ be a finitely generated group with infinite cyclic abelianization generated by the image of $t \in G$.  Show that $G$ is bi-orderable if and only if the conjugation action of $t$ on $[G, G]$ preserves a bi-ordering of $[G, G]$.
\end{problem}

\section{The dynamic realization of a left-ordering}
\label{dynamicalsection}
If $G$ is an Archimedean ordered group, then H\"{o}lder's theorem gives an order-preserving injective homomorphism $\phi : G \rightarrow (\R, +)$. Regarding $(\R, +)$ as the subgroup of translations of $\mathrm{Homeo_+(\R)}$, the homomorphism arising from H\"{o}lder's theorem is a special case of a more general construction that yields a homomorphism $G \rightarrow \mathrm{Homeo}_+(\R)$.

In this section we prove that a countable group $G$ is left-orderable if and only if there is an embedding $G \rightarrow \mathrm{Homeo}_+(\R)$; we already saw one direction of this proof in Chapter 1.  We also introduce a standard way of constructing such an embedding, called the \textit{dynamic realization}.  To begin we recall a classical theorem due to Cantor.

\begin{theorem}[Cantor]
\label{Cantor_theorem}
If $S$ is a countable, totally densely ordered set without a maximum or minimum element, then there exists an order-preserving bijection $\phi:S \rightarrow \mathbb{Q}$.
\end{theorem}
\begin{proof}
The argument we will present is known as Cantor's ``back and forth argument" \cite[p. 35--36]{Huntington55}.  

Let $\{ s_0, s_1, \ldots \}$ and $\{ r_0, r_1, \ldots \}$ be enumerations of the set $S$ and the rational numbers $\mathbb{Q}$ respectively.  Set $S_0 = \{s_0\}$ and $R_0 = \{ r_0 \}$, and we begin our construction of the map $\phi: S \rightarrow \mathbb{Q}$ by declaring $\phi(s_0) = r_0$.   Now assuming we have defined an order-preserving bijection $\phi: S_k \rightarrow R_k$ between finite subsets $S_k $ and $R_k $ of $S$ and $\mathbb{Q}$ respectively, we extend $\phi$ to an order-preserving bijection between larger finite subsets according to the following steps.  
\begin{enumerate}
\item Choose the smallest $i$ such that $s_i \notin S_k$, and set $S_{k+1} = S_k \cup \{ s_i \}$. Choose $r_j \in \mathbb{Q} \setminus R_k$ such that setting $\phi(s_i) = r_j$ defines an order-preserving bijection $\phi : S_{k+1} \rightarrow R_k \cup \{r_j\}$ (such a choice of $r_j$ is possible by density of the ordering of $\mathbb{Q}$).  Set $R_{k+1} = R_k \cup \{ r_j \}$.

\item Choose the smallest $j$ such that $r_j \notin R_{k+1}$ and set $R_{k+2} = R_{k+1} \cup \{r_j\}$.  Choose $s_i \in S \setminus S_{k+1}$ such that setting $\phi^{-1}(r_j) = s_i$ defines an order-preserving bijection $\phi^{-1}:R_{k+2} \rightarrow S_{k+1} \cup \{s_i \}$ (similar to step 1, this is possible by density of the ordering of $S$).  Set $S_{k+2} = S_{k+1} \cup \{s_i \}$.
\item Return to step 1, and repeat the process.
\end{enumerate}

This procedure produces a map $\phi:S \rightarrow \mathbb{Q}$ which is order-preserving and injective by construction.  The map is also surjective, because after $n$ iterations of these three steps, $r_n$ is sure to be in the image of $\phi$.
\end{proof}

\begin{theorem}
\label{LO_universal}
Suppose that $G$ is a countable group.  Then $G$ is left-orderable if and only if $G$ is isomorphic to a subgroup of $\mathrm{Homeo}_+(\mathbb{R})$.
\end{theorem}
\begin{proof}
If $G$ is isomorphic to a subgroup of $\mathrm{Homeo}_+(\mathbb{R})$, then $G$ is left-orderable because $\mathrm{Homeo}_+(\mathbb{R})$ is a left-orderable group (by Example \ref{homeo+}).

On the other hand, suppose that $G$ is countable and left-orderable, we will build an injective homomorphism $\rho :G \rightarrow \mathrm{Homeo}_+(\mathbb{R})$.  Observe that $G$ can be embedded in the group $G \times \mathbb{Q}$, which is also countable, and which can be densely left-ordered using the standard lexicographic construction.  So, by embedding $G$ into $G \times \mathbb{Q}$ if necessary, we can assume that the left-ordering of $G$ is dense.

By Theorem \ref{Cantor_theorem}, there is an order-preserving injective map $t : G \rightarrow \mathbb{R}$ whose image is $\mathbb{Q}$.   For each $g \in G$, define a map $\rho(g) : \mathbb{R} \rightarrow \mathbb{R}$ by first defining its action on $\mathbb{Q}$ according to the rule $\rho(g)( t(h) ) = t(gh)$ for all $h \in G$, this action preserves order because left-multiplication preserves the ordering of $G$.  By Problem \ref{unique_extension}, this uniquely determines an order-preserving homeomorphism $\rho(g): \mathbb{R} \rightarrow \mathbb{R}$.  This defines the required homomorphism $\rho :G \rightarrow \mathrm{Homeo}_+(\mathbb{R})$.
\end{proof}

\begin{problem}
\label{unique_extension}
Suppose that $f : \mathbb{Q} \rightarrow \mathbb{Q}$ is an order-preserving bijection.  Show that $f$ can be uniquely extended to an order-preserving homeomorphism $\bar{f} : \mathbb{R} \rightarrow \mathbb{R}$.
\end{problem}

Given a group $G$ with a left-ordering $<$, there is a second construction of an embedding $t:G \rightarrow \mathbb{R}$ which has become standard in the literature.  The action induced on $\mathbb{R}$ by this construction is called the \textit{dynamic realization} of the left-ordering $<$, and it is described as follows.
\index{dynamic realization}

 Fix an enumeration $\{ g_0, g_1, g_2, \ldots \}$ of $G$ with $g_0 = 1$, and proceed as follows to inductively define an order-preserving embedding $t:G\rightarrow \mathbb{R}$.  Begin by setting $t(g_0) = 0$.   If $t(g_0), \ldots, t(g_i)$ have already been defined and $g_{i+1}$ is either larger or smaller than all previously embedded elements, then set:
\begin{displaymath}
   t(g_{i+1}) = \left\{
     \begin{array}{ll}
       \mathrm{max}\{t(g_0), \ldots , t(g_i) \} +1 & \mbox{ if } g_{i+1} > \mathrm{max}\{g_0, \ldots , g_i \} \\
       \mathrm{min}\{t(g_0), \ldots , t(g_i) \} -1 & \mbox{ if } g_{i+1} < \mathrm{min}\{g_0, \ldots , g_i \} \\

     \end{array}
   \right.
\end{displaymath} 
On the other hand, if there exist $j, k \in \{ 0, \ldots, i \}$ such that $g_j < g_{i+1} < g_k$ and there is no $n \in \{ 0, \ldots ,i \}$ such that $g_j < g_{n} < g_k$, then set
\[ t(g_{i+1}) = \cfrac{t(g_j)+t(g_k)}{2}.
\]
The group $G$ acts in an order-preserving way on the set $t(G)$ according to the rule $g(t(h)) = t(gh)$.  

This rule extends to an order-preserving action on the closure $\overline{t(G)}$.
The complement of the set $\overline{t(G)}$ is a union of open intervals, with the action of every $g \in G$ defined on their endpoints.  For every $g \in G$ we can extend this action to an order-preserving homeomorphism $\rho(g)$ by extending the action of $g$ on $t(G)$ affinely on the complement $\mathbb{R} \setminus \overline{t(G)}$.  With a fair amount of work, one can show that this defines a faithful representation $\rho : G \rightarrow \mathrm{Homeo}_+(\mathbb{R})$.  The representation constructed in this way is the dynamic realization of $<$.  One can recover the original ordering of $G$ from the dynamic realization by declaring $g >1$ if and only if $\rho(g)(0)>0$.

\begin{problem}
Let $\rho: G \rightarrow \mathrm{Homeo}_+(\R)$ denote the dynamical realization of the left-ordering $<$ of $G$.  Show that $\mathrm{Homeo}_+(\R)$ admits a left-ordering that extends the natural left-ordering of $\rho(G)$.  (Hint:  Well-order the reals so that $0$ is the smallest element, and use the construction of Example \ref{homeo+})
\end{problem}

Since the construction of the dynamic realization involves many choices, for example a choice of enumeration of $G$ and a choice of order preserving function $t:G \rightarrow \R$, it is not unique.  The degree to which this construction is unique is a rather subtle question--we refer the reader to \cite{Navas10a} for an investigation of this question.

%
%
%



\chapter{Free groups, surface groups and covering spaces}
\label{chapter free and surface}

The goal of this chapter is to show that free groups, as well as almost all
surface groups (the exceptions being the projective plane and Klein bottle) are bi-orderable.  
We conclude the chapter with an interesting connection between the orderability of the fundamental group of a space and topological properties of the universal cover.

\section{Surfaces}  

By a surface we shall mean a metric space for which each point has a neighbourhood homeomorphic with either $\R^2$ or the closed upper half-space $\R_+^2$.  Unless otherwise stated, we will assume surfaces to be connected.  We include non-compact surfaces and surfaces with boundary  --- in both cases the fundamental group is a free group --- and also include non-orientable surfaces.  By a {\em surface group} \index{surface group}we mean a group isomorphic with the fundamental group $\pi_1(\Sigma)$ of a surface $\Sigma$.

The basic building blocks for closed (i.e. compact without boundary) surfaces are the sphere $S^2$, the torus $T^2 \cong S^1 \times S^1$ and the (real) projective plane $P^2$. 
One can regard $P^2$ as the quotient of $S^2$ in which antipodal points are identified, or equivalently, the union of a disk and M\"obius band, sewn together along their boundaries. The connected sum of two closed surfaces is gotten by deleting an open disk from the interior of each surface, and then sewing the resulting punctured surfaces together along their boundaries.  The connect sum operation is denoted by `$\#$', for example, Figure \ref{connect_sum} is the sum $T^2 \# T^2$. 

\begin{figure}%
\includegraphics[scale=0.8]{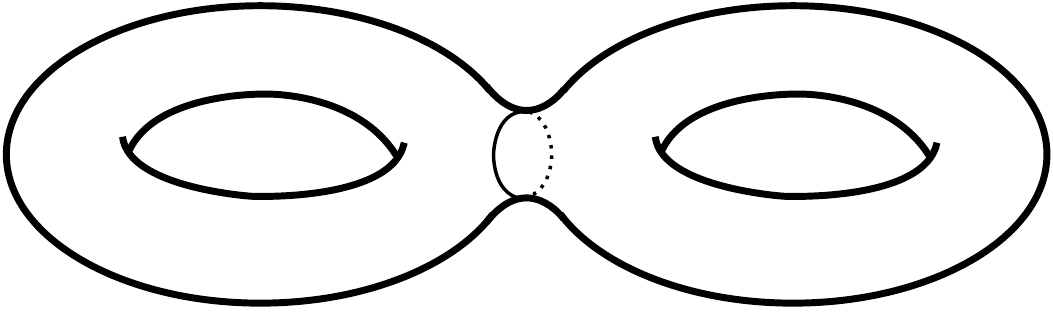}
\caption{The connect sum of two tori.}
\label{connect_sum}
\end{figure}

Closed surfaces are classified as follows:

\begin{itemize}
\item Orientable: $S^2, T^2, 2T^2 = T^2 \# T^2, 3T^2 = T^2 \# T^2 \# T^2, \dots $

\item Nonorientable: $P^2, 2P^2 = P^2 \# P^2, 3P^2 = P^2 \# P^2 \# P^2, \dots$

\end{itemize}
This allows us to define the \textit{genus} \index{genus} of a surface $\Sigma$ by identifying it with one of the surfaces in the list above:  if $\Sigma$ is either $kT^2$ or $kP^2$ then the genus of $\Sigma$ is $k$, if $\Sigma = S^2$ then its genus is $0$.

In fact, the set of surfaces can be regarded as a commutative monoid generated by $T^2$ and $P^2$ ($S^2$ is the identity element) and subject to the famous relation

$$T^2 \# P^2 \cong 3P^2.$$

They have geometric structures as follows:
\begin{itemize}
\item Spherical: $S^2, P^2$

\item Euclidean: $T^2$, $2P^2$ = Klein bottle 

\item Hyperbolic: all the rest.
\end{itemize}

\begin{problem}
Verify that $T^2 \# P^2 \cong 3P^2.$
\end{problem}

It has long been known that the fundamental groups of the closed {\em orientable} surfaces are bi-orderable.  One proof, due to Baumslag \cite{Baumslag2010}, depends on the fact that they are residually free,\index{residually free}\footnote{A group $G$ is residually free if for every nonidentity element $g \in G$ there is a homomorphism $\phi:G \rightarrow F$ onto a free group such that $\phi(g) \neq 1$.} and therefore embed in a direct product of free groups.  We will give another proof, which also applies to the nonorientable case.  Baumslag's argument does not apply to nonorientable surface groups (see \cite{BRW05} for a discussion).  Indeed, Levi \cite{Levi43} had remarked that nonorientable surface groups are NOT bi-orderable, though they were understood to be left-orderable (with the obvious exception of the projective plane $P^2$). 
The argument went that the embedded M\"obius bands introduced a relation saying that an element is conjugate to its inverse --- in fact this only happens for the Klein bottle (whose group has the defining relation $xyx^{-1} = y^{-1}$) and projective plane (where each element equals its inverse). This assumption apparently stood until 2001, when it was shown in \cite{RW01} that the fundamental groups of the hyperbolic nonorientable surfaces are bi-orderable after all.  The proof presented here is essentially the same as the one in that paper, where the interested reader may find further details.  Before considering surface groups, we turn to the simpler case of free groups.

\section{Ordering free groups}

The free group with generators $x_1, x_2, \dots$, possibly an infinite list, can be regarded as the set of equivalence classes of (finite) words in the letters $x_i$ and their formal inverses $x_i^{-1}$, where words are considered equivalent if one can pass from one to the other by removing (or inserting) consecutive letters of the form $x_i^{}x_i^{-1}$ or $x_i^{-1}x_i^{}$.  The group operation is concatenation and the empty word represents the identity element.

There are a number of ways to order free groups---in fact for free groups with more than one generator there are uncountably many.  The method we will use here, following Magnus, has the advantage that one can decide by straightforward calculation which of two given words is bigger in the ordering. 

Let $F = F(x_1, x_2, \dots)$ denote the free group on the generators
$x_1, x_2, \dots$, possibly an infinite list.  We define the ring
$$
\Lambda = \Z[[X_1, X_2, \dots]]
$$
to be the ring of formal power series in the non-commuting variables $X_i$, one for each generator of $F$.  If there are infinitely many variables, we only allow expressions involving a finite set of variables to belong to $\Lambda$, so that an element of $\Lambda$ has only a finite number of terms of a given degree.
The advantage of $\Lambda$ is that (unlike $F$) the variables have no negative exponents and so it is easier to define an ordering, without having to worry about cancellation problems.  Magnus used the following embedding (the Magnus expansion) \index{Magnus expansion} to argue, among other things, that the intersection of the lower central series of a free group is just the identity element.

Define the (multiplicative) homomorphism $\mu: F \to \Lambda$ on the generators of $F$ as follows:
\begin{eqnarray*}
\mu(x_i) & = & 1 + X_i \\
\mu(x_i^{-1}) & = & 1 - X_i + X_i^2 - X_i^3 + \cdots
\end{eqnarray*}

Thus, for example, 
\begin{eqnarray} \label{conjugate}
\mu(x_1x_2^2x_1^{-1}) & = & (1 + X_1)(1 + X_2)^2(1 - X_1 + X_1^2 - X_1^3 + \cdots)  \nonumber \\
  & = & 1 + 2X_2 + 2X_1X_2 - 2X_2X_1 + X_2^2 + O(3) 
\end{eqnarray}

The notation $O(n)$ stands for the sum of all terms of total degree $n$ or greater.  We now define an ordering on $\Lambda$.  First, write elements of $\Lambda$ with lower degree terms preceding terms of higher degree.  Within a fixed degree, they may be ordered arbitrarily, but to be definite we will order them lexicographically according to subscript (as in the example).
Now given two elements of $\Lambda$, write them in the standard form, as described above, and order them according to the coefficient of the first term at which they differ.

For example, $(1 + X_i)^p = 1 + pX_i + O(2)$, even for negative $p$, so that such expressions are ordered in the same way as their exponents, that is $(1 + X_i)^p < (1 + X_i)^q$ if and only if $p<q$.  

\begin{lemma} The homomorphism $\mu$ is injective, and embeds $F$ into the group of units of $\Lambda$ of the form $1 + O(1)$.
\end{lemma}

\begin{proof} We need only verify that the kernel of $\mu$ is trivial.  Write any non-identity element $w$ of $F$ in standard form 
$w = x_{i_1}^{p_1}x_{i_2}^{p_2}\cdots x_{i_k}^{p_k}$, where $i_j \ne i_{j+1}$ for each $j = 1, \dots, k-1$.  From the remark preceding this lemma, we see that $\mu(w)$ has a unique term $pX_{i_1}X_{i_2} \cdots X_{i_k},$ with $p = p_1 \cdots p_k$, which is the product of the degree one terms of the factors 
$(1 + X_{i_j})^{p_j}$.  It follows that $\mu(w) \ne 1$. \end{proof}

Our intention is to order $F$ by considering it as a subgroup (via the Magnus embedding $\mu$) of the (multiplicative) group
$\{ 1 + O(1)\}$ which lies inside $\Lambda$.  The ordering of $\Lambda$ that we have described is easily seen to be invariant under addition, but it is certainly not preserved by multiplication by certain terms, for example $-1$.  However, the following saves the day.

\begin{lemma} The multiplicative group $G = \{ 1 + O(1)\} \subset \Lambda$ is a bi-ordered group, under the ordering of $\Lambda$ described above.
\end{lemma}

\begin{proof} We check that the ordering is preserved by multiplication on the left, the verification for right-multiplication being similar.  Let $U, V, W \in O(1)$ and suppose $1 + V < 1 + W$, which is equivalent to $W - V > 0$.  We calculate 
\begin{eqnarray*}
(1 + U)(1 + V) & = & 1 + U + V + UV \\ 
(1 + U)(1 + W) & = & 1 + U + W + UW  
\end{eqnarray*}
In the difference $(1 + U)(1 + W) - (1 + U)(1 + V) = W - V + U(W - V)$
we note that all terms of $U(W - V)$ have degree greater than the first nonzero term of $W - V$.  So the difference is greater than zero, and we conclude that $(1 + U)(1 + V) < (1 + U)(1 + W)$. \end{proof}

Now we can formally define the ordering on the free group $F$ by declaring
$$
v < w \ {\rm in}\ F \quad \Leftrightarrow \quad \mu(v) < \mu(w)  \ {\rm in}\  \Lambda.
$$

By way of example, comparing $(1 + X_2)^2 = 1 + 2X_2 + X_2^2$ with the expression in Equation \ref{conjugate}, we conclude that 
$x_2^2 < x_1x_2^2x_1^{-1}$, because the first term at which their Magnus expansions differ, namely $(X_1X_2)$, has coefficient $0$ for $x_2^2$ and $2$ for $x_1x_2^2x_1^{-1}$.  We have established the following, moreover with an explicit, computable ordering.

\begin{theorem}\label{freeBO}
Every free group is bi-orderable.
\end{theorem}

\begin{problem}\label{lowercentral}  Suppose $G$ is a group and consider its lower central series $G_0 \supset G_1 \supset G_2 \supset \cdots$, defined by $G_0 = G$ and \index{lower central series}
$G_{n+1} = [G, G_n]$, the subgroup generated by all commutators $ghg^{-1}h^{-1}$ with $g \in G \ h \in G_n$.  Check that each $G_n$ is normal in $G$, and that $G_n/G_{n+1}$ is central in $G/G_{n+1}$ (so $G_n/G_{n+1}$ is abelian).  Assume that 
$
\displaystyle\cap_{i=0}^{\infty} G_i = \{1\}$  and that $ G_n/G_{n-1}$ is torsion-free for all $n$. 
Verify that any group $G$ satisfying these properties is bi-orderable.  (Hint: use Theorem \ref{torsion free abelian} to take an arbitrary ordering on each $G_n/G_{n+1}$, and define a non-identity
$g \in G$ to be positive if its projection in $G_n/G_{n+1}$ is positive, where $n$ is the largest integer such that $g \in G_n$.)
\end{problem}

This problem gives an alternate proof that free groups are orderable (in fact, it is essentially equivalent to the Magnus expansion argument).  Many other groups of topological interest satisfy the conditions of the problem, for example the fundamental group of any orientable surface and the pure braid groups, so they are also bi-orderable.  It first appeared as a theorem in \cite{neumann49a}, where there is a similar criterion for orderability involving the ascending central series. 

\begin{problem}  Determine the ordering of the following in a free group on generators $x, y$, with $x = x_1$ and $y = x_2$ in the ordering described above:
$1$, $ x $, $ y $, $x^2$,$ yx^2y^{-1}$, $ y^{-1}x^2y$, $xyx^{-1}y^{-1}.$
\end{problem} 

\begin{problem}   \index{dense ordering} Show that the ordering described on a free group of two (or more) generators is dense, meaning that given $u, v \in F(x,y)$ with $u < v$ there exists $w\in F(x,y)$ with $u < w < v$.  (Hint: argue that a left- or bi-ordering on a group is dense if and only if there does {\em not} exist a least element greater than the identity.)
\end{problem}  

\begin{problem}   Show that {\em every} bi-ordering a free group of two (or more) generators must be order-dense.  (Hint: If $u$ is the least element which is greater than the identity, argue that any conjugate of $u$ also has this property.  Conclude that $u$ is central --- but free groups have trivial centre.)
\end{problem}  

\begin{problem}
Suppose that $<$ is a bi-ordering of a free group $F$, and suppose that $g \in F$ is neither the identity nor a power of a nonidentity element.  Define a new ordering $\prec$ of $F$ as follows: if $h$ is not a power of $g$,  declare $1 \prec h$ if $g< hgh^{-1}$, if $h = g^k$ then declare $h$ to be positive if $k$ is positive.

Show that $\prec$ is a left-ordering of $F$, and that $g$ is the smallest positive element of $F$ with respect to the ordering $\prec$.  Compare your result with the solution of Problem \ref{relatively_convex_intersection_problem}.
\end{problem}

An element of a free group will be called {\it Garside positive} \index{Garside positive} if it can be expressed as a word in the generators without negative exponents.  The subset of Garside positive elements is a monoid.

\begin{problem}  Show that the set  $F^+$ of {\it Garside positive} words is positive in the ordering described above using the Magnus expansion, that is $1 < w$ whenever $1 \ne w \in F^+$.  Moreover, $F^+$ is {\it well-ordered} by this ordering, that is, every nonempty subset of $F^+$ has a smallest element.
\end{problem} 

\section{Ordering surface groups}

The goal of this section is to prove the following.

\begin{theorem}\label{BOsurface}
If $\Sigma$ is a surface other than the projective plane $P^2$ or the Klein bottle $P^2 \# P^2$, 
then $\pi_1(\Sigma)$ is bi-orderable.
\end{theorem}

Before embarking on the proof, note that we need only consider closed surfaces $\Sigma$, for otherwise the fundamental group is a free group, which we have just shown to be bi-orderable.  For $\Sigma$ the 2-sphere or torus $T^2$, the theorem is obvious, as their groups are trivial and $\Z \oplus \Z$ respectively.   The remaining cases will be settled by proving the theorem for the particular surface $\Sigma = 3P^2$, since its group contains all the remaining ones as subgroups, according to the following observations.  We note that for any integer $k > 2$ there is a covering map $kT^2 \to 2T^2$ with deck transformations the cyclic group of order $k - 1$.  It follows that $\pi_1(2T^2)$ contains  $\pi_1(kT^2)$ as a normal subgroup of index $k-1$. 

\begin{figure}[h!]
\includegraphics[scale=0.4]{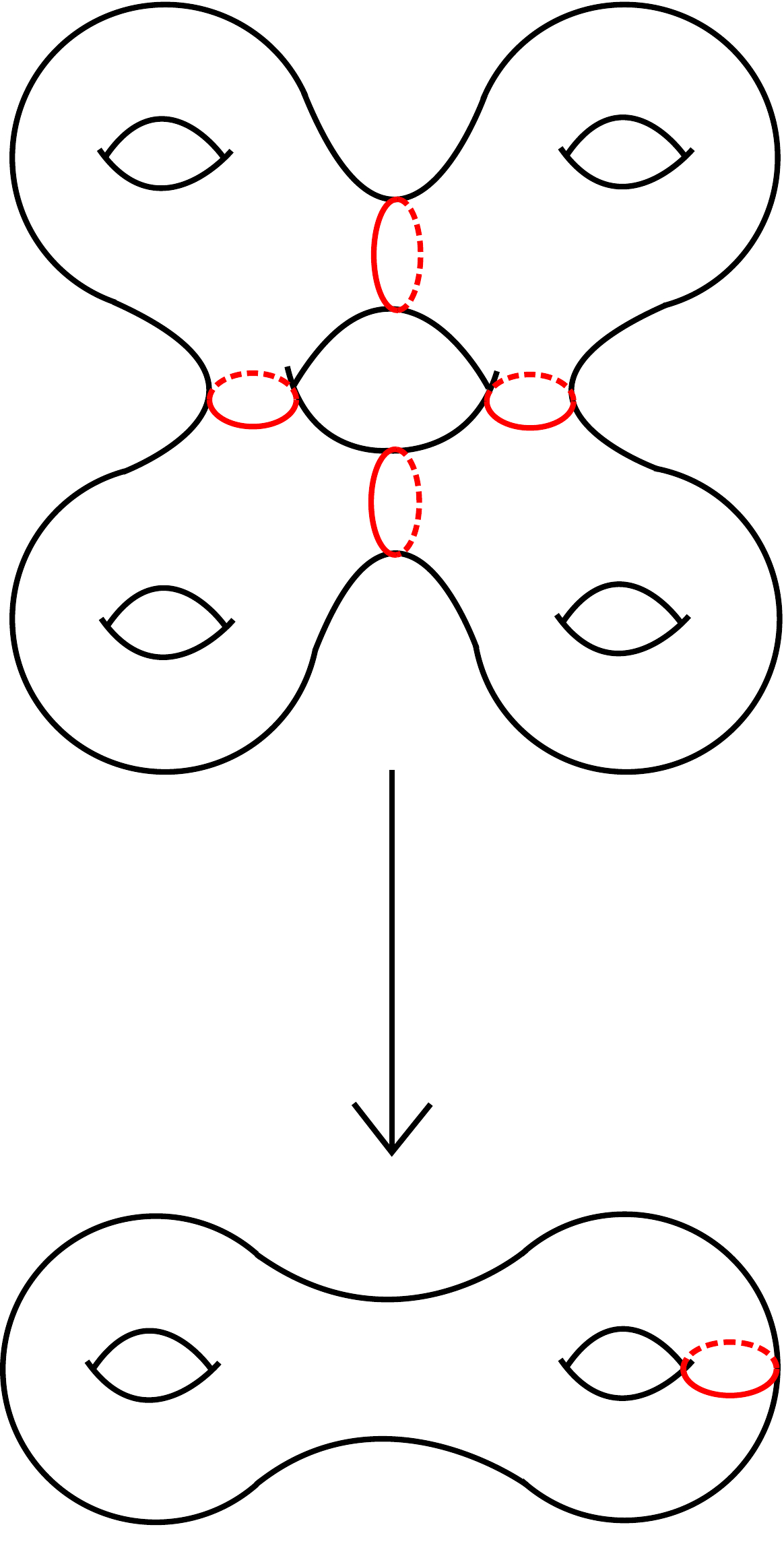}
\caption{A covering of $2T^2$ by $kT^2$ in the case $k=5$.  Deck transformations act as rotations by multiples of $\pi/2$.}
\label{covering of 2T2}
\end{figure}

\begin{problem}  Show that for $k \ge 1$ the nonorientable surface $kP^2$ has oriented double cover homeomorphic with $(k-1)T^2$ (where $0T^2 \cong S^2$).  Moreover for $k \ge 3$ there is a $(k-2)$-sheeted cyclic covering $kP^2 \to 3P^2$.  It follows that 
$\pi_1(3P^2)$ contains finite index subgroups isomorphic with $\pi_1(kP^2)$ for all $k \ge 3$ and also $\pi_1(kT^2)$ for $k \ge 2$.  (Hint: think of $3P^2 = T^2 \# P^2$ as a torus with a disk removed and replaced by a M\"obius band, and similarly think of $kP^2$.)
\end{problem}

Picture $\Sigma = 3P^2$ as a torus with a small disk removed and replaced by a M\"obius band (cross-cap).  The torus has universal cover $\R^2$, with covering group being integral translations 
$(x, y) \mapsto (x+m, y+n),\ m,n \in \Z$.  Removing small disks about each point $(m + 1/2, n + 1/2)$ and replacing them by cross-caps, we obtain a covering 
$\tilde \Sigma \rightarrow \Sigma$.  It is not the universal cover.  Indeed,
$\pi_1(\tilde \Sigma)$ is a free group on a countable number of generators, which we may picture as loops which are the central curves of the cross-caps, connected by ``tails'' to the origin of $\R^2$ in some canonical manner (see Figure \ref{cover_crosscap} and Problem \ref{crosscap cover problem}).  

\begin{figure}[h!]
\includegraphics[scale=0.6]{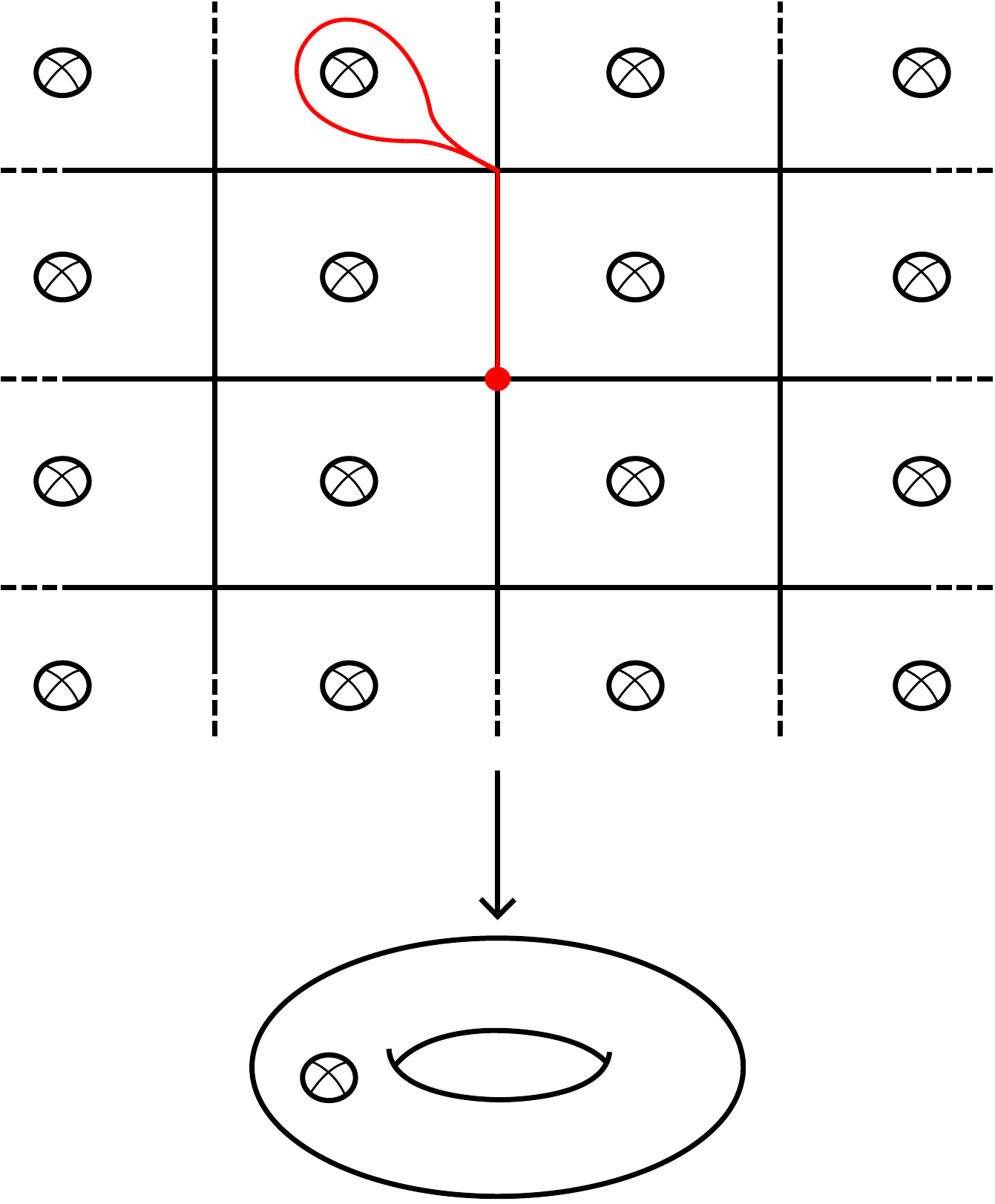}
\caption{A covering of $T^2$ with an added crosscap by a plane with crosscaps at each point $(m + 1/2, n + 1/2)$; the curve in red is the generator $a_{-1, 2}$ described in Problem \ref{crosscap cover problem}.}
\label{cover_crosscap}
\end{figure}

Let $x_{m,n}$ denote the generator corresponding to the cross-cap at $(m+ 1/2, n + 1/2)$.
We have an exact sequence
$$
1 \longrightarrow \pi_1(\tilde\Sigma) \longrightarrow \pi_1(\Sigma) \longrightarrow
\Z \oplus \Z \longrightarrow 1,
$$ 
in which $\pi_1(\Sigma)$ is flanked by bi-orderable groups.  

Using Problem \ref{extendO} the theorem will follow once we establish that the action of conjugation by $\pi_1(\Sigma)$ upon $\pi_1(\tilde \Sigma)$
preserves some ordering of $\pi_1(\tilde \Sigma)$.  To this end, we order
$\pi_1(\tilde \Sigma)$ using the Magnus expansion ordering, as described in the previous section.  The symbols $X_{m,n}$ corresponding to the generators of 
$\pi_1(\tilde \Sigma)$ are considered ordered lexicographically by their subscripts.  Now $\pi_1(\Sigma)$ acts on $\tilde \Sigma$ by covering translations, and therefore, each $x_{m,n}$ is taken to a conjugate of $x_{m+m_0,n+n_0}$.  It follows that the action preserves the Magnus ordering (see Problem \ref{MagnusConj}). \qed

\begin{problem}
\label{crosscap cover problem}
  Verify that $\pi_1(\tilde \Sigma)$ is a free group on a countable number of generators, as follows:  The fundamental group of $\R^2$ minus the disks is a free group with generators $a_{m,n}$ represented by a curve bounding the disk removed at 
$(m+ 1/2, n + 1/2)$, plus a tail connecting it to the basepoint.   If $x_{m,n}$ is represented by the central curve of the M\"obius band, we introduce, by Van Kampen's theorem, relations 
$a_{m,n} = x_{m,n}^2$, and we calculate:
$$\pi_1(\tilde \Sigma) \cong  \langle x_ {m,n}, a_{m,n} : a_{m,n} = x_{m,n}^2 \rangle 
\cong \langle x_{m,n} \rangle,$$
where $ \langle x_{m,n} \rangle$ is the free group generated by $\{ x_{m,n}\}_{m, n \in \mathbb{Z}}$.
\end{problem}

\begin{problem} \label{MagnusConj}
Define the "leading term" of a Magnus expansion (written with lower degree terms preceding higher) to be the first non-constant term with nonzero coefficient.  Verify that if two elements of the free group are conjugate, then the leading terms of their Magnus expansions are the same.   The positive cone of the ordering of $\pi_1(\tilde \Sigma)$ defined above consists of all group elements whose Magnus expansion has leading term with positive coefficient.  It follows that the covering translations of $\Sigma$ described above preserve the positive cone, and hence the bi-ordering defined in the proof of Theorem \ref{BOsurface}.
\end{problem}

\section{A theorem of Farrell}

In this section we will show how the topological properties of a universal covering space $p: \widetilde{B} \rightarrow B$ are related to orderability of $\pi_1(B)$ via the action by deck transformations  \cite{farrell76}.  

Also, for this section only we will consider \textit{right}-orderings of groups: we lose nothing by doing so, since by Problem \ref{right order} every left-ordering of a group uniquely determines a right-ordering and vice versa.   We adopt this convention so that our notation will agree with the standard convention in topology that concatenated paths are written in the order they appear: 
\begin{equation*}
(\alpha*\beta)(t)= 
\begin{cases} \alpha(2t) & \text{if $t \in [0, \frac{1}{2}]$,}
\\
\beta(2t-1) &\text{if $t \in [\frac{1}{2},1]$.}
\end{cases}
\end{equation*}

We also recall the standard correspondence between the fundamental group of $B$ and deck transformations $\widetilde{B} \rightarrow \widetilde{B}$.  To describe the correspondence, we fix a basepoint $b_0$ in $B$ and a preimage $\widetilde{b}_0$ in $\widetilde{B}$.  Then given an element $[\gamma] \in \pi_1(B, b_0)$, we send $[\gamma]$ to the deck transformation $d : \widetilde{B} \rightarrow \widetilde{B}$ satisfying $d(\widetilde{b_0}) = \widetilde{\gamma}(1)$, where $\widetilde{\gamma}$ is the lift of $\gamma$ which starts at $\widetilde{b_0}$.  We will write $[\gamma]$ in place of the associated deck transformation $d$, so that our equation becomes $[\gamma] (\widetilde{b_0}) =  \widetilde{\gamma}(1)$. For simplicity we'll write $[\gamma] \widetilde{b_0}$, thinking of the deck transformations as a group action.

\begin{proposition}
Suppose that $p: \widetilde{B} \rightarrow B$ is a universal covering space.  If there exists a continuous map $h: \widetilde{B} \rightarrow \mathbb{R}$ such that $f: \widetilde{B} \rightarrow B \times \mathbb{R}$ given by $f(x) = (p(x), h(x))$ is an injection, then $\pi_1(B)$ is right-orderable.
\end{proposition}

\begin{proof}
Suppose such a map $h$ exists, and define an ordering of $\pi_1(B, b_0)$ according to the rule $[\alpha] < [\beta]$ if and only if $h([\alpha]\widetilde{b_0}) < h([\beta]\widetilde{b_0})$.  

Now suppose that this ordering is not right invariant, so that there exists $[\alpha], [\beta], [\gamma] \in \pi_1(B, b_0)$ with $[\alpha] < [\beta]$ yet $[\alpha] [\gamma] > [\beta] [\gamma]$.  Let $\widetilde{\gamma}_1$ denote the lift of $\gamma$ with $\widetilde{\gamma}_1(0) = [\alpha]\widetilde{b_0}$, and $\widetilde{\gamma}_2$ denote the lift of $\gamma$ with $\widetilde{\gamma}_2(0) = [\beta]\widetilde{b_0}$.  Note that $\widetilde{\gamma_1}(1) = [\alpha * \gamma]\widetilde{b_0}=[\alpha ][ \gamma]\widetilde{b_0}$ and $\widetilde{\gamma_2}(1) = [\beta * \gamma]\widetilde{b_0}=[\beta ][ \gamma]\widetilde{b_0}$.

Then set $g(t) = h( \widetilde{\gamma}_1(t)) - h(\widetilde{\gamma}_2(t))$, so that $g(0) <0$ and $g(1)>0$.  By the intermediate value theorem, there exists $t_0$ with $g(t_0) =0$ and therefore $ \widetilde{\gamma}_1(t_0) = \widetilde{\gamma}_2(t_0)$. Consequently both of the lifts of $\gamma$ must be the same, since they overlap.  This implies $[\alpha]\widetilde{b_0}= [\beta]\widetilde{b_0}$, forcing $[\alpha] = [\beta]$, a contradiction.
\end{proof}
Note the necessity of right-orderability in the argument above.  For the converse, we first need to prepare some facts.
\begin{problem}
Show that there exists a discrete, densely ordered subset of the real line which has no maximum or minimum element (Hint: Consider the midpoints of the deleted intervals used to construct the middle thirds Cantor set).
\end{problem}

\begin{problem}
If $G$ is a countable right-ordered group, use the previous exercise to show that there exists an order-preserving map $t: G \rightarrow \mathbb{R}$ with discrete image. (Hint: mimic the arguments appearing in the proof of \ref{LO_universal}).
\end{problem}

\begin{proposition}
Suppose that $B$ is a space admitting a triangulation.  If $\pi_1(B)$ is right-orderable, then there exists a map $h: \widetilde{B} \rightarrow \mathbb{R}$ such that the map $f: \widetilde{B} \rightarrow B \times \mathbb{R}$ given by $f(x) =(p(x), h(x))$ is an embedding.
\end{proposition}
\begin{proof}
Begin by fixing a triangulation of $B$, and correspondingly a triangulation of $\widetilde{B}$ that we get by taking preimages of simplices under the covering map.  We also fix an order-preserving map $t:\pi_1(B) \rightarrow \mathbb{R}$ with discrete image. 

First, we define $h$ on the vertices of $\widetilde{B}$.  For each $v \in B$, pick a point $\widetilde{v} \in \widetilde{B}$ satsifying $p(\widetilde{v}) = v$.  Then every vertex in the triangulation of $\widetilde{B}$ can be written as $[\gamma] \widetilde{v}$ for some $v \in B$ and $[\gamma] \in \pi_1(B)$, and we define
\[ h( [\gamma] \widetilde{v}) = t([\gamma]).
\]

Now extend this definition linearly to the rest of $\widetilde{B}$ using barycentric coordinates.  Let $x \in \widetilde{B}$ be a point lying in a simplex with vertices $[\gamma_0]\widetilde{v}_0, [\gamma_1]\widetilde{v}_1, \ldots, [\gamma_n]\widetilde{v}_n$, and write $x$ as
\[ x = c_0[\gamma_0]\widetilde{v}_0 + c_1[\gamma_1]\widetilde{v}_1 + \cdots + c_n[\gamma_n]\widetilde{v}_n
\]
where $c_1 + c_2 + \ldots + c_n=1$.   Set 
\[ h(x) = c_0 t([\gamma_0]) + c_1t([\gamma_1]) + \cdots + c_n t([\gamma_n]).
\]

The next two exercises complete the proof by showing that $f(x) =(p(x), h(x))$ is an embedding. 

\begin{problem}
For injectivity, suppose that $f(x) = f(y)$, then $x$ and $y$ must lie inside simplices $\Delta_1$, $\Delta_2$ that are preimages of a common simplex $\Delta \subset B$.  Therefore there exists $[\gamma] \in \pi_1(B)$ such that the corresponding deck transformation sends $\Delta_1$ to $\Delta_2$.  Write $x$ and $y$ in barycentric coordinates, observe that the vertices of $\Delta_2$ are the image under the action of $[\gamma]$ of the vertices of $\Delta_1$.  Show that $[\gamma] >1$ and $[\gamma]<1$ result in $h(y)>h(x)$ and $h(y) < h(x)$ respectively, since $h$ is order-preserving.
\end{problem}

\begin{problem}
Show that $f$ is an embedding (it is here that we need the map $t:\pi_1(B) \rightarrow \mathbb{R}$ to have discrete image).
\end{problem}
\end{proof}

So in the case that $B$ is a triangulable space, we have the following equivalence.

\begin{theorem}[\cite{farrell76}]
If $B$ is a space admitting a triangulation, then $\pi_1(B)$ is right-orderable if and only if there is an embedding $f: \widetilde{B} \rightarrow B \times \mathbb{R}$ so that the following commutes:

\[\xymatrix{ 
\widetilde{B} 
\ar[rr]^{f} 
\ar[dr]_{p} 
&& B \times \mathbb{R}
\ar[dl]^{\pi_1 } \\ 
& B }
\]
Here, $\pi_1$ is projection onto the first factor.
\end{theorem}

It should be mentioned that Farrell actually proved a more general result.  One need only assume
that $B$ is a Hausdorff, paracompact space with a countable fundamental group.  Moreover there is a generalization to arbitrary regular covering spaces $\widetilde{B}$ of $B$ stating that there is an embedding $f: \widetilde{B} \rightarrow B \times \mathbb{R}$ making the above diagram commute if and only if the quotient group $\pi_1(B)/p_*\pi_1(\widetilde{B})$ is right-orderable.  The interested reader is referred to \cite{farrell76} for details.  

\chapter{Knots}
\label{knots chapter}

In this chapter we investigate left- and bi-orderability of knot groups.  It turns out that all knot groups are left-orderable (in fact, locally indicable), whereas some knot groups are bi-orderable while others are not.  We close the chapter with an application of left-orderability of surface groups to the theory of knots in thickened surfaces.

\section{Review of classical knot theory}

For the reader's convenience, we outline (mostly without proof) some of the basic ideas of classical knot theory.  By a {\em knot} $K$ we mean a smoothly embedded simple closed curve in the 3-dimensional sphere $S^3$, that is, $K$ is smooth submanifold of $S^3$ which is abstractly homeomorphic with $S^1$.  More generally a {\em link} is a disjoint finite collection of knots in $S^3$.  Other (essentially equivalent) versions of knot theory consider knots in $\R^3$ or require them to be piecewise linear.  Of course it is more convenient to visualize knots in 
$\R^3$ and consider $S^3$ to be $\R^3$ with a point at infinity adjoined.   We will not consider so-called {\em wild} knots.  

Two knots or links are considered {\em equivalent} (or, informally, equal) if there is an orientation-preserving homeomorphism of $S^3$ taking one to the other.  A well-known construction provides, for any knot $K$, a compact, connected, orientable surface  $\Sigma \subset S^3$ such that $\partial\Sigma = K$ \cite[Section 5.A.4]{Rolfsen90}.  The minimal genus $g(\Sigma)$ among all such surfaces bounded by a given $K$ is called the {\em genus} \index{knot genus} of the knot, and denoted $g(K)$.  In particular, the {\em trivial} knot (or {\em unknot}), which is equivalent to a round circle in $S^3$, is the unique knot of genus zero.

One may ``add'' two knots $K$ and $K'$ to form their connected sum $K \# K'$ as in Figure \ref{knotsum} \index{knot sum}.  This addition is associative and commutative and the unknot is a unit.  Moreover, genus is additive:
$$g(K \# K') = g(K) + g(K').$$

\begin{figure}[h!]
\includegraphics[scale=0.5]{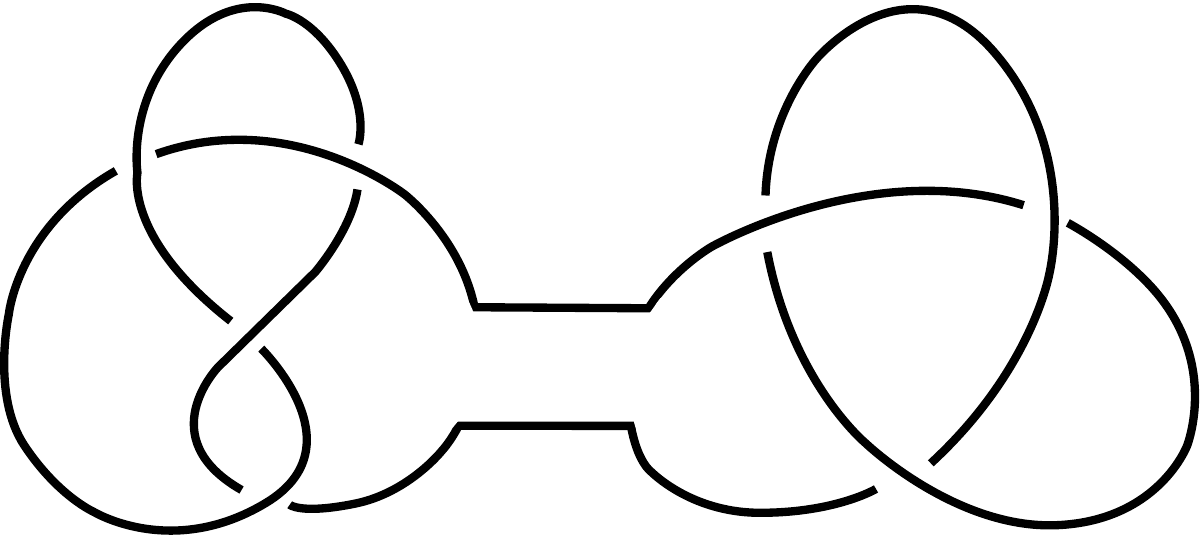}
\caption{The sum of the figure eight knot and trefoil.}
\label{knotsum}
\end{figure}

\begin{problem}  Use genus to argue that there are no inverses in knot addition: the connected sum of nontrivial knots cannot be trivial.
\end{problem}

A knot is said to be {\em prime} \index{prime knot} if it is not the connected sum of nontrivial knots.  Knots have been tabulated by crossing number, that is, the minimum number of simple crossings of one strand over another in a planar picture of the knot.  For example the first nontrivial knot, the trefoil, is denoted $3_1$ the first (and only) knot in the table with crossing number three.  Tabulations of prime knots up to 16 crossings have been made with the aid of computers; there are approximately 1.7 million \cite{HTW}.  Knots with more than ten crossings have names which include a letter `n' or `a' to indicate whether or not they are \textit{alternating}\index{alternating knot}, meaning they can be drawn in such a way that crossings are alternately over and under as one traces the curve.  Thus $11a_5$, pictured below, is the fifth eleven crossing alternating knot in the table.

\begin{figure}[h!]
\includegraphics[scale=0.8]{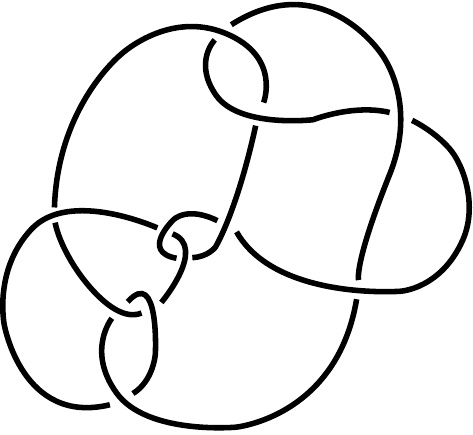}
\caption{The fifth eleven crossing alternating knot.}
\label{11a5}
\end{figure}

\begin{problem}
Knots of genus one are prime.
\end{problem}

If $K$ is a knot, then the fundamental group of its complement 
$\pi_1(S^3 \setminus K)$ is called the {\em knot group} \index{knot group} of $K$.  There are algorithms, for example the Wirtinger or Dehn methods, for explicitly calculating finite presentations of a knot group from a picture of the knot.  An important property of knot groups is that their abelianization, which may be identified with the integral homology group $H_1(S^3 \setminus K)$, is infinite cyclic.  This can be seen, for example, by Alexander duality or by taking the abelianization of the Wirtinger presentation (Problem \ref{abelianization problem}).  It is known that the unknot is the only knot whose group is abelian (and hence infinite cyclic).

If we are given two disjoint oriented knots $J$ and $K$ in $S^3$, since the fundamental group $\pi_1(S^3 \setminus K)$ abelianizes to $\mathbb{Z}$, the class $[J] \in \pi_1(S^3 \setminus K)$ determines an integer in the abelianization.  This integer is called the \textit{linking number} \index{linking number}of $J$ with $K$, denoted $\ell k(J, K)$.  It can be calculated from a diagram of the two knots as follows: for each crossing where $J$ passes under $K$, assign a value of $\pm1$ according to the convention in Figure \ref{linking convention}.  Summing these numbers over all crossings gives the quantity $\ell k(J, K)$.

\begin{figure}[h!]
\setlength{\unitlength}{5cm}
\begin{picture}(1,0.5)%
    \put(0,0){\includegraphics[width=\unitlength]{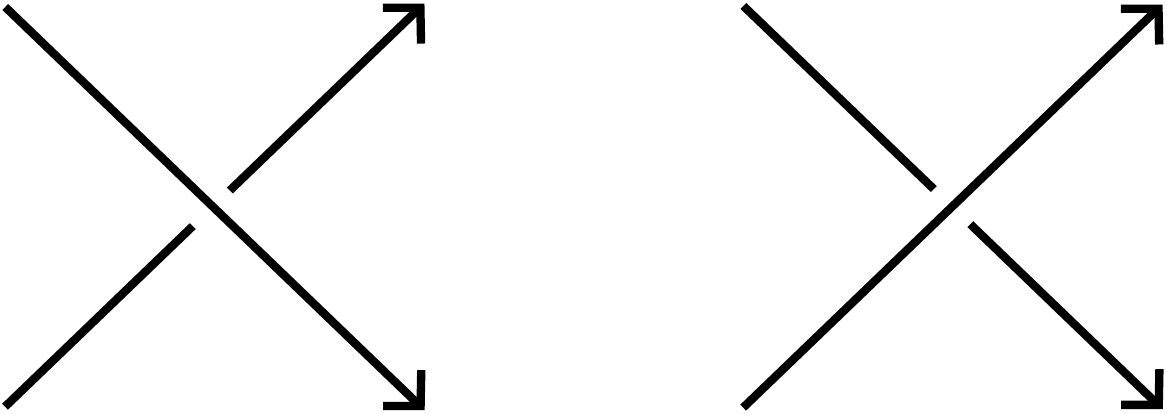}}%
    \put(0.36563984,0.36123092){$J$}%
    \put(0.36563984,-0.03){$K$}
    \put(0.99400173,0.36123092){$K$}
    \put(0.99400173,-0.03){$J$}
        \put(0,0.16){$+1$}%
    \put(0.65,0.16){$-1$}%
  \end{picture}%
\caption{The convention for calculation linking number.}
\label{linking convention}
\end{figure}

A family of knots whose groups are particularly simple are the {\em torus knots} \index{torus knots}.  Consider a torus $T \cong S^1 \times S^1$ which is the boundary of a regular neighborhood of an unknot $U$, as pictured in Figure \ref{torusknot}.  Note that $\pi_1(T) \cong H_1(T) \cong \Z \times \Z$.  We picture the generator of the first $\Z$ to be represented by an oriented curve $\mu$ that links $U$ and the generator $\lambda$ of the second factor represented by a curve running parallel to $U$, but on $T$ and homologically trivial in the complement of $U$.  If $p$ and $q$ are relatively prime integers, there is a knot $K_{p,q}$ on the surface $T$ which (when oriented) represents the class $p\mu + q\lambda \in H_1(T).$  The trefoil is $T_{3,2}$.  An application of the Seifert-van Kampen theorem gives the following presentation for the torus knot group:
$$\pi_1(S^3 \setminus K_{p,q}) \cong \langle a, b \mid a^p = b^q \rangle.$$

\begin{problem}
\label{torus knot groups problem}
Verify the presentation for the torus knot group given above, by proceeding as follows: The complement of $T_{p,q}$ consists of  a solid torus part, with a small trough removed from its surface following the path of the torus knot, and the part outside the torus, with a matching trough removed.  A Seifert--van Kampen argument gives the presentation $\pi_1(S^3 \setminus T_{p,q}) \cong  \langle a, b \mid a^p =b^q \rangle$.
\end{problem}

\begin{figure}[h!]
\includegraphics[scale=0.75]{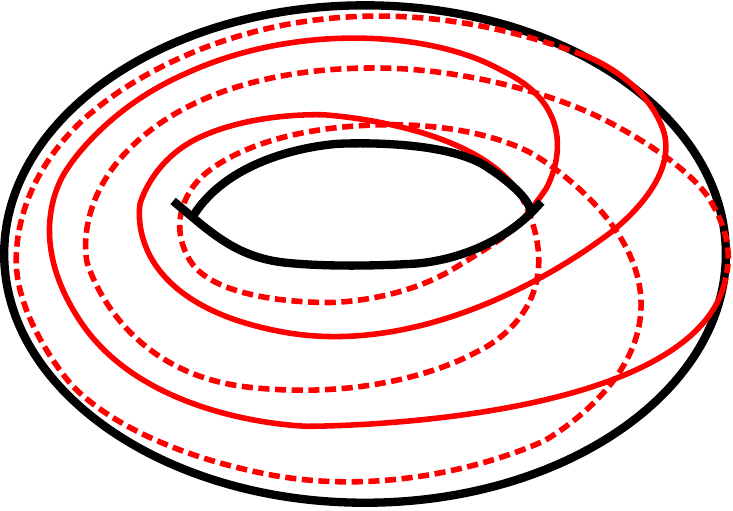}
\caption{The $(2,5)$-torus knot wrapping $2$ times meridionally and $5$ times longitudinally around a torus.}
\label{torusknot}
\end{figure}

A knot  $K \subset S^3$ is {\em fibred} \index{fibred knot} if there is a (locally trivial) fibre bundle map from its complement to the circle with fibre a surface.  All torus knots are fibred, but there are many other fibred knots, some of which are shown in the table later in this chapter.  From the long exact sequence associated with a fibration, we get the following short exact sequence associated to a fibred knot $K$, with fibre $\Sigma$:
$$1 \to \pi_1(\Sigma) \to \pi_1(S^3 \setminus K) \to \pi_1(S^1) \to 1.$$

Note that $ \pi_1(\Sigma)$ is a free group, since  $\Sigma$ is a surface with boundary, and of course
$\pi_1(S^1)$ is infinite cyclic.  Since both of these groups are locally indicable, we apply Problem \ref{LI extension} and we conclude the following:

\begin{theorem}
A fibred knot's group is locally indicable, hence left-orderable.
\end{theorem}

\begin{problem}
\label{LI extension}
Show that if $K$ and $H$ are locally indicable groups and 
\[ 1 \rightarrow K \rightarrow G \rightarrow H \rightarrow 1
\]
is a short exact sequence, then $G$ is locally indicable. 
\end{problem}
As we will soon see, this is true for all classical knot groups.

There are many polynomial invariants of knots.  The oldest of them is the Alexander polynomial, $\Delta_K(t)$ \index{Alexander polynomial}, which can be defined in several ways.  For example, it can be calculated from a presentation of the knot group or from a matrix determined by a surface bounded by the knot.  We refer the reader to \cite{CF77} or \cite{Rolfsen90} for details.  Important properties of the Alexander polynomial are that the coefficients are integers,  $\Delta_K(1) = \pm 1$ and $t^{2n}\Delta_K(t^{-1}) = \Delta_K(t)$, for some non-negative integer ~$n$.   The latter condition means it has even degree and the palindromic property that the coefficients read the same backwards as forwards.  The unknot has trivial polynomial $\Delta(t) = 1$, but so do many nontrivial knots.   It also behaves nicely under connected sum:
$$\Delta_{K \# K'}(t) = \Delta_{K}(t)\Delta_{K'}(t).$$

Alexander polynomials need not be monic, but for fibred knots they must be monic and of degree $2g$, where $g$ is the genus of the fibre surface.  This is because they may be considered as the characteristic polynomial of a linear map, as will be discussed later.

\section{The Wirtinger presentation}

Given a picture of a knot, there are various procedures for calculating the knot group.  One method is the Wirtinger presentation\index{Wirtinger presentation}, which we'll now describe.  We assume the planar knot diagram contains only simple crossings, and they are denoted by deleting a little interval of the lower strand near the crossing. We also assume the knot has been assigned an orientation, that is a preferred direction. What remains of the curve is now a disjoint collection of (oriented) arcs in the plane.  Give each arc a name, say $x, y, ... $.       The knot group will be generated by these symbols.  For each crossing, one introduces a relation in the following way.  Turn your head so that both strands at the crossing are oriented generally from left to right.  Two possibilities are pictured, corresponding to ``positive'' and ``negative'' crossings.  In each case we introduce a relation among the three generators which appear at the crossing, according to the rule given in Figure \ref{wirtinger}.  A presentation for the knot group then consists of the generators $x, y, ... $ and the relations corresponding to the crossings.  

\begin{figure}[h!]
\setlength{\unitlength}{8cm}
\begin{picture}(1,0.5)%
    \put(0,0){\includegraphics[width=\unitlength]{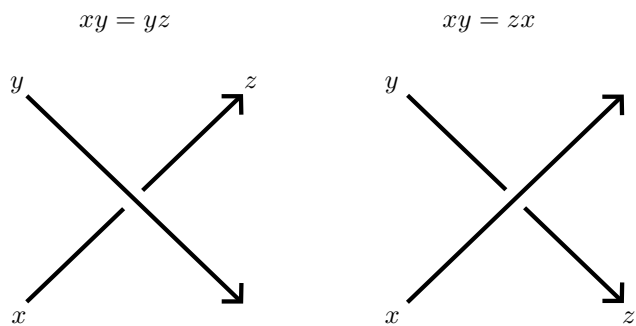}}%
    \put(-0.02333762,0.36123092){$y$}%
    \put(-0.02117985,-0.03){$x$}%
    \put(0.36563984,0.36123092){$z$}%
    \put(0.59954842,0.361230926){$y$}%
    \put(0.59954842,-0.03){$x$}
    \put(0.99400173,-0.03){$z$}
    \put(0.09160393,0.46264578){$xy=yz$}%
    \put(0.69556886,0.46264578){$xy=zx$}%
  \end{picture}%
\caption{Relations in the knot group determined at a crossing.}
\label{wirtinger}
\end{figure}

Here is an explanation of why this works.  Imagine the basepoint for $\pi_1(\R^3 \setminus K)$ to be your eye, situated above the plane of the projection.  For each oriented arc, draw a little arrow under the arc and going from right to left, if one views the arc oriented upward.  Then the loop corresponding to $x$ consists of a straight line running from your eye to the tail of the arrow, then along the arrow, and then returning to your eye again in a straight line, as in Figure \ref{wirtinger loop}.  With a little thought, the relations of Figure \ref{wirtinger} become clear.  We refer the reader to
\cite{Rolfsen90} for the proof that these relations are a complete set of relations (in fact discarding any one of the relations still leaves us with a complete set, but we will not need this).  The curves described above are called ``meridians'' of the knot.  

\begin{figure}[h!]
\includegraphics[scale=0.2]{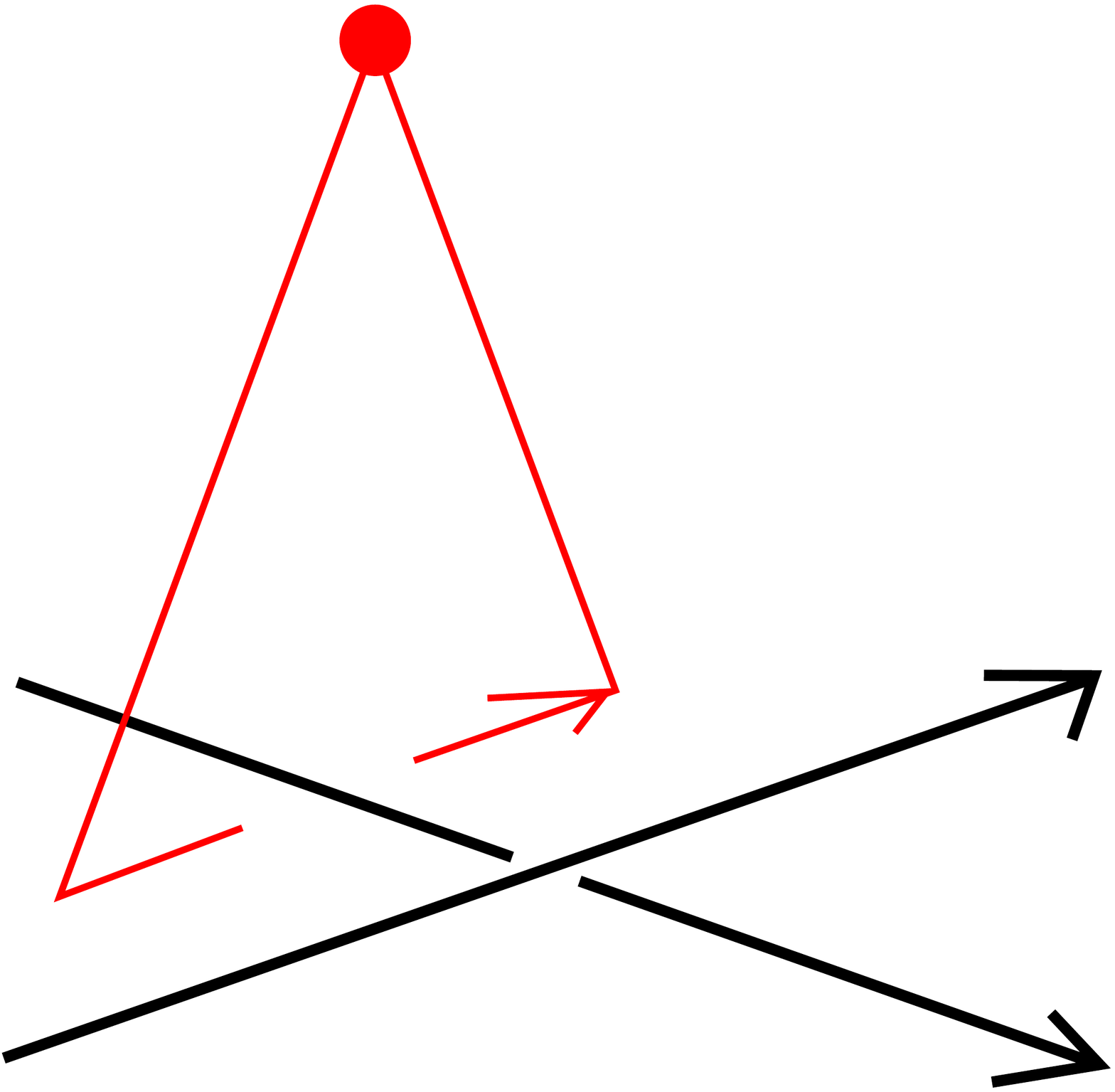}
\caption{A loop (in red) representing a generator of the Wirtinger presentation.}
\label{wirtinger loop}
\end{figure}

\begin{problem}
\label{abelianization problem}
Show that all the meridians in the Wirtinger presentation are conjugate to each other.  Conclude that the abelianization of every knot group is infinite cyclic.
\end{problem}

\begin{example}\label{trefoil group} The group of the  `right-handed'  trefoil $K$ pictured in Figure \ref{trefoil} has presentation with generators $x, y, z$.  The relations coming from the crossings are (1)  $zx = xy$, (2) $xy = yz$ and (3) $yz = zx$.  Clearly the third relation is redundant, so we have
$$\pi_1(S^3 \setminus K) \cong \langle x, y, z \mid   zx = xy = yz \rangle.$$
The second equation can be used to eliminate $z = y^{-1}xy$ and then we obtain a single relation $y^{-1}xyx = xy$, which yields the simpler presentation
$$\pi_1(S^3 \setminus K) \cong \langle x, y \mid   xyx = yxy \rangle.$$
\end{example}

\begin{figure}[h!]
\setlength{\unitlength}{6cm}
\begin{picture}(1,1)%
    \put(0,0){\includegraphics[width=\unitlength]{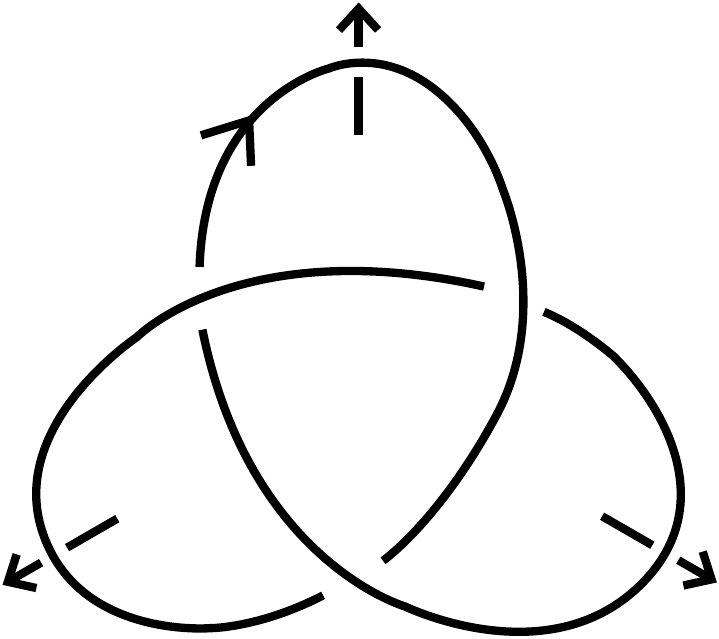}}%
    \put(0.51692502,0.71409156){$z$}%
    \put(0.97356178,0.15235031){$x$}%
    \put(-0.00501988,0.15388088){$y$}%
    \put(0.4514844,0){(1)}%
    \put(0.16947196,0.48398622){(2)}%
    \put(0.74447811,0.4824556){(3)}
  \end{picture}%
\caption{The right-handed trefoil.}
\label{trefoil}
\end{figure}

\begin{problem}
Another way to compute the trefoil's group is to consider it as the $(2, 3)$-torus knot group, and proceed as in Problem  \ref{torus knot groups problem}.  One finds $\pi_1(S^3 \setminus T_{2,3}) \cong \langle a, b \mid a^2 = b^3 \rangle.$
Verify algebraically that this presentation and the presentation  $\langle x, y \mid   xyx = yxy \rangle$ yield isomorphic groups.  
\end{problem}

\section{Knot groups are locally indicable}\label{section LI}

In this section, we begin our investigation into the orderability of knot groups--by showing that they are, in fact, locally indicable.

\index{locally indicable}
\begin{theorem}\label{LI}
Every knot group is locally indicable, and hence left-orderable.
\end{theorem}

\begin{proof}
Before beginning the proof, we first note that knot groups are indicable.  If $X = S^3 \setminus K$ then the Hurewicz homomorphism
$h: \pi_1(X) \to \ H_1(X)$ is surjective, and
$\ H_1(X)$ is infinite cyclic since it is equal to the abelianization of $\pi_1(X)$, which is infinite cyclic by Problem \ref{abelianization problem}.   To prove that $\pi_1(X)$ is 
{\em locally} indicable, we need to consider an arbitrary nontrivial finitely generated subgroup $G$ of $\pi(X)$ and argue that it admits a nontrivial homomorphism to $\Z$.  This argument is due essentially to Howie and Short \cite{HS85}.

Case 1: $G$ has finite index.  Then the restriction $h|_G$ of the Hurewicz homomorphism is nontrivial and we are done.

Case 2: $G$ has infinite index.  There is a covering space $p: \tilde{X} \to X $ such that, for suitably chosen basepoint, $p_*(\pi_1(\tilde{X})) = G$.  Although  $\tilde{X}$ must be noncompact, its fundamental group is finitely generated, by assumption.  By a theorem of P. Scott \cite{Scott73}, $\tilde{X}$ has a compact ``core'' --- that is a compact connected submanifold $C$ of $\tilde{X}$ such that inclusion induces an isomorphism  $i_* : \pi_1(C) \to \pi_1(\tilde{X})$.  See Figure ~\ref{core}.

\begin{figure}[h!]
\setlength{\unitlength}{7cm}
\begin{picture}(1,0.3694672)%
    \put(0,0){\includegraphics[width=\unitlength]{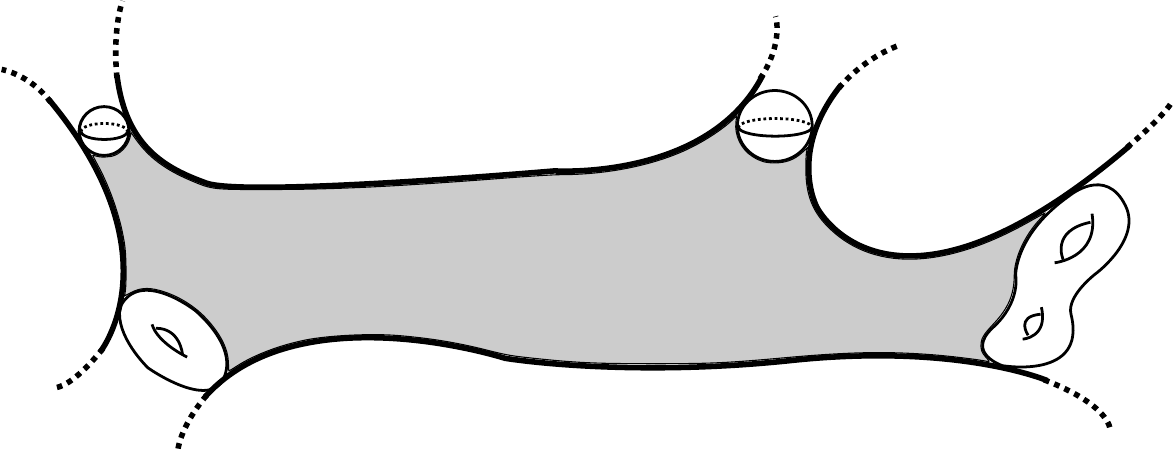}}%
    \put(0.43920994,0.13333918){$C$}%
    \put(-0.00355261,0.14992679){$\widetilde{X}$}%
  \end{picture}%

\caption{The compact core $C \subset \widetilde{X}$.}
\label{core}
\end{figure}

Noting that $C$ must have nonempty boundary, we first argue that  we can assume there are 
no 2-sphere components of $\partial C$.  For suppose $\Sigma \subset \partial C$ is a 2-sphere.
From Alexander's theorem one knows that knot complements are \textit{irreducible}\index{irreducible 3-manifold}, meaning that every tame $2$-sphere in the manifold bounds a ball.  Then, since irreducibility is inherited by coverings, we know $\tilde X$ is irreducible (see \cite{Hatchernotes}, for example, for proofs of these facts).  Therefore there is a 3-ball $B$ in $\tilde{X}$ with $\partial B = \Sigma$.  It is easy to see that $C$ is either a subset of 
$B$ or else disjoint from the interior of $B$.  But $C \subset B \subset \tilde{X}$ would contradict the fact that $i_*$ is a nontrivial homomorphism, so we conclude that $B$ is disjoint from the interior of $C$. 
 If we now define $C' := C \cup B$ we see that $C'$ will also serve as a compact core for 
$\tilde{X}$.  After repeating this a finite number of times we obtain a compact core, which we will again call $C$, such that $\partial C \ne \emptyset$ and every component of $\partial C$ is a surface of positive genus, as in Figure \ref{core_capped}.

\begin{figure}[h!]
\setlength{\unitlength}{7cm}
\begin{picture}(1,0.37675654)%
    \put(0,0){\includegraphics[width=\unitlength]{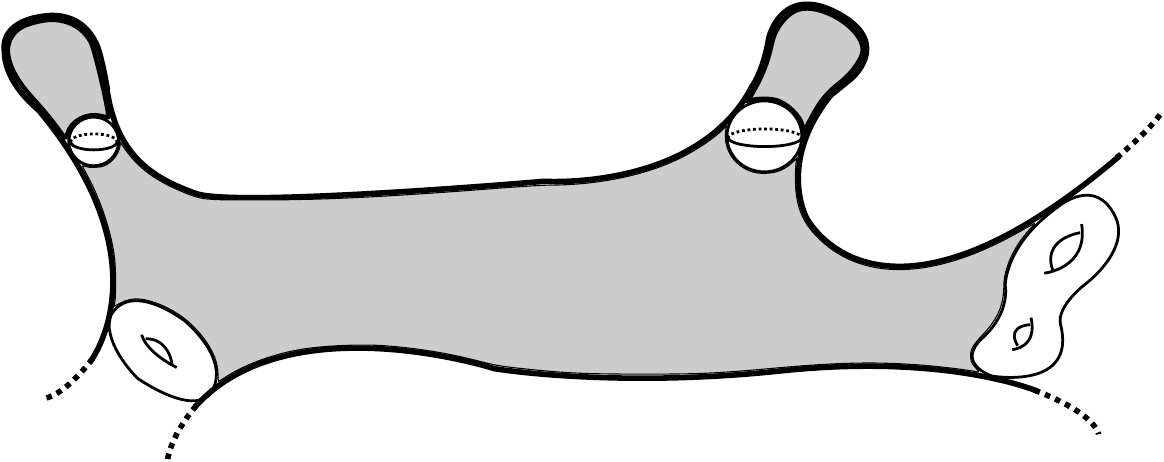}}%
    \put(0.43920994,0.13333918){$C$}%
    \put(-0.00355261,0.14992679){$\widetilde{X}$}%
    \put(0.09137035,0.35899734){$B$}%
    \put(0.7394862,0.31441083){$B$}%
  \end{picture}%
\caption{The compact core $C \subset \widetilde{X}$ with boundary spheres capped off.}
\label{core_capped}
\end{figure}

Lemma \ref{rank lemma} finishes Case 2, because then one easily constructs a surjection of the abelian group $H_1(C)$  onto $\Z$ and combines it with the Hurewicz map to get a surjection $ G \cong \pi_1(C) \to H_1(C) \to \Z$.  
\end{proof}

\begin{lemma}
\label{rank lemma}
The integral homology group $H_1(C)$ is infinite.
\end{lemma}

\begin{proof}  This is a standard argument, repeated here for the reader's convenience. We will show, equivalently, that the {\em rational} homology $H_1(C; \Q)$ has positive rank.  Recall that $C$ is a compact orientable 3-manifold with nonempty boundary containing no 2-spheres.  Consider the {\em closed} manifold $2C$ obtained from two copies of $C$, with their boundaries glued together by the identity map.  The Euler characteristic of a closed 3-manifold is always zero, so we have
$ 0 =  \chi (2C) = 2\chi (C) - \chi(\partial C) $. Our assumption on $\partial C$ implies that its Euler characteristic is less than or equal to zero.  Therefore $\chi(C) \le 0$.  But by definition of $\chi(C)$ as an alternating sum of ranks, we conclude
$$1 - \mathrm{rank} (H_1(C;\Q)) + \mathrm{rank}( H_2(C,\Q)) - 0 \le 0$$
which implies $\mathrm{rank}( H_1(C;\Q)) \ge 1$. \end{proof}

\section{Bi-ordering certain knot groups} 

In this section we'll investigate the bi-orderability of knot groups.  We organize our discussion by considering the cases of fibred and non-fibred knots separately.

\subsection{Fibered knots}
The focus of our discussion concerning fibred knots will be the following two results.  More details may be found in  \cite{PR03} and \cite{CR12}.

\begin{theorem}\label{AS}
If $K$ is a fibred knot whose Alexander polynomial $\Delta_K(t)$ has all roots real and positive, then its knot group is bi-orderable.
\end{theorem}

\begin{theorem} \label{AN}
If $K$ is a nontrivial fibred knot whose knot group is bi-orderable, then $\Delta_K(t)$ has at least two real positive roots.
\end{theorem}
 
Before discussing the proofs of these theorems, we'll consider some examples and then prepare some preliminary results in Subsections \ref{eigenvalues subsection} and \ref{linear algebra}. 

\begin{example}[Torus knots]
\index{torus knots}
Recall from Problem \ref{torus knot groups problem} that the $(p, q)$-torus knots have knot group
$\langle a, b \mid a^p = b^q \rangle.$
Note that $a$ commutes with $b^q$ but not with $b$ (unless the group is abelian, and the knot unknotted).  By Problem \ref{central}  we conclude:
 
\begin{proposition}
Nontrivial torus knot groups are not bi-orderable.
\end{proposition}

This could also be proved using Theorem \ref{AN} and the fact that torus knots are fibred.  
As a typical example, consider the knot $8_{19}$, which is the $(4,3)$-torus knot. 
It has Alexander polynomial 
$$\Delta_{8_{19}} = 1-t+t^3-t^5+t^6 = (t^2 +\sqrt{3} t + 1)(t^2 -\sqrt{3} t + 1)(t^2 - t + 1).$$ 
Its six roots are $(\sqrt{3} \pm i)/2, (-\sqrt{3} \pm i)/2$ and $(1 \pm i\sqrt{3} )/2$.  More generally, the torus knot $T(p, q)$ has Alexander polynomial
$$\Delta_{T_{p,q}} = \frac{(t^{pq} - 1)(t - 1)}{(t^p -1)(t^q - 1)}$$
whose roots are on the unit circle and not real. 
\end{example}

\begin{example}[The knot $4_1$]  This knot, sometimes called the figure-eight knot and pictured in the table below, is also a fibred knot.  It has Alexander polynomial $1 - 3t + t^ 2$, whose roots are $(3 \pm \sqrt{5})/2$, both real and positive.  Theorem \ref{AS} implies the following.

\begin{proposition}
The group of the knot  $4_1$ is bi-orderable.
\end{proposition}
 \end{example} 
  
\begin{example}[Other bi-orderable fibred knot groups.] The criterion of having all roots of $\Delta_K(t)$ real and positive is not really very common.  
The table below contains all nontrivial prime knots with $12$ or fewer crossings whose groups are known to be bi-orderable because the knots are fibred and all roots of the Alexander polynomial are in $\R^+$.  The diagrams were produced using Rob Scharein's program Knotplot \cite{Sch}.  Some of the data on the knots are from {\em KnotInfo} and its database \cite{CL}, kindly provided by Chuck Livingston.

{
\renewcommand{\arraystretch}{2.2}
\begin{longtable}{ccc}
	\hline
\nopagebreak Knot&  & Alexander polynomial \\
  \hline

  \multirow{3}{*}{\includegraphics[height=20mm]{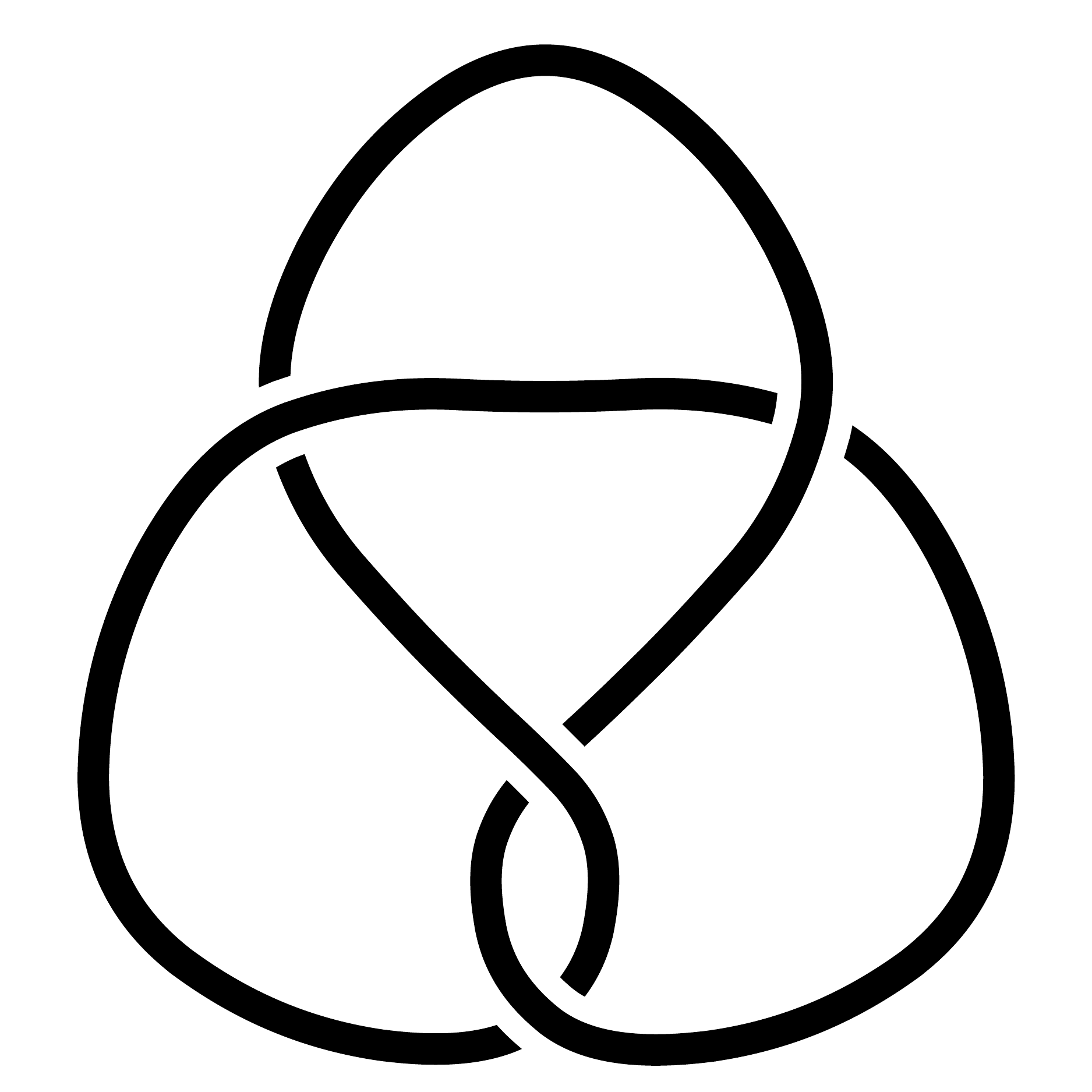}}&     \\
\nopagebreak   & $4_1$ & $1-3t+t^2$ \\
\nopagebreak  &  &  \\
  \multirow{3}{*}{\includegraphics[height=20mm]{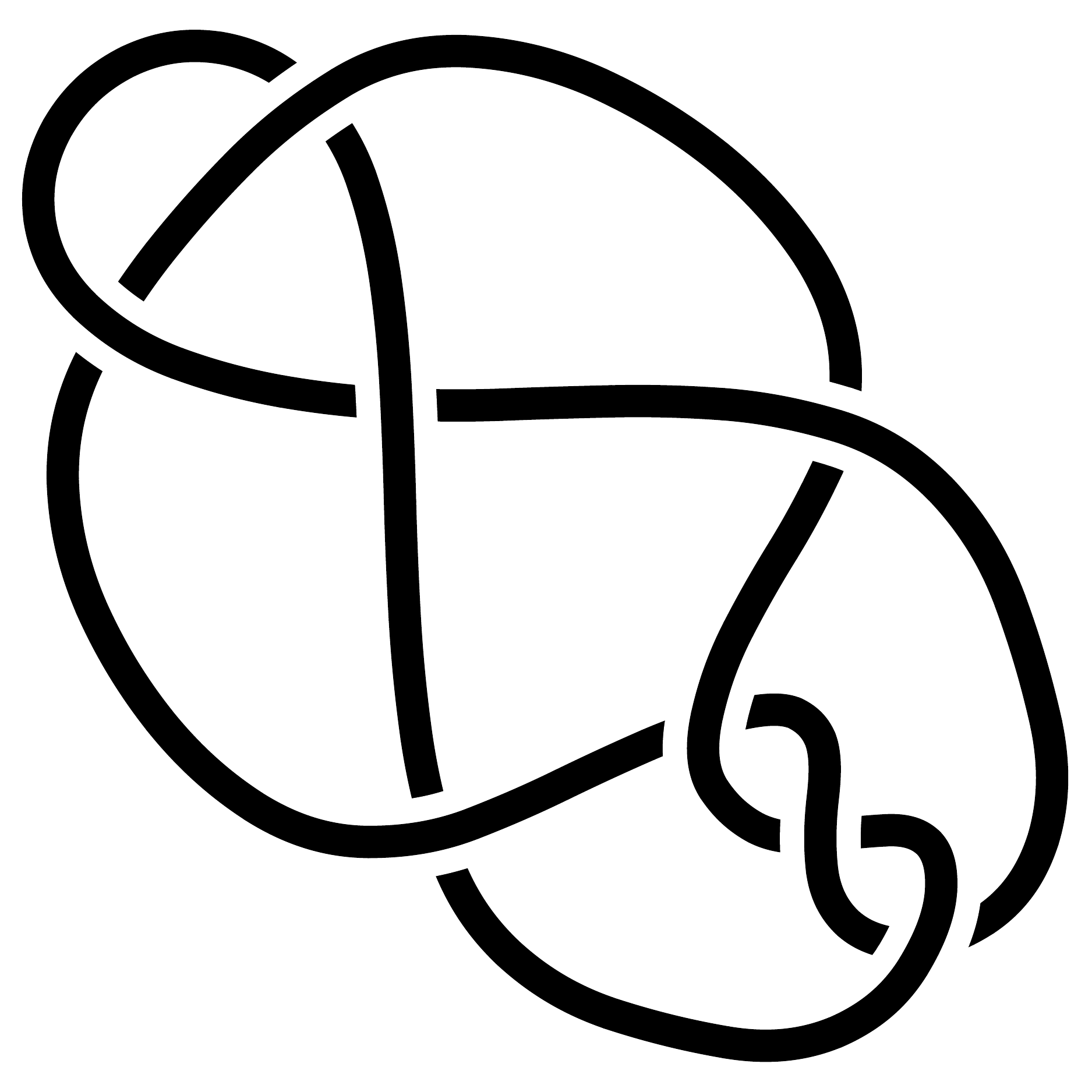}} &   \\
\nopagebreak   & $8_{12}$ & $1-7t+13t^2-7t^3+t^4$  \\
\nopagebreak  &  &  \\
  \multirow{3}{*}{\includegraphics[height=20mm]{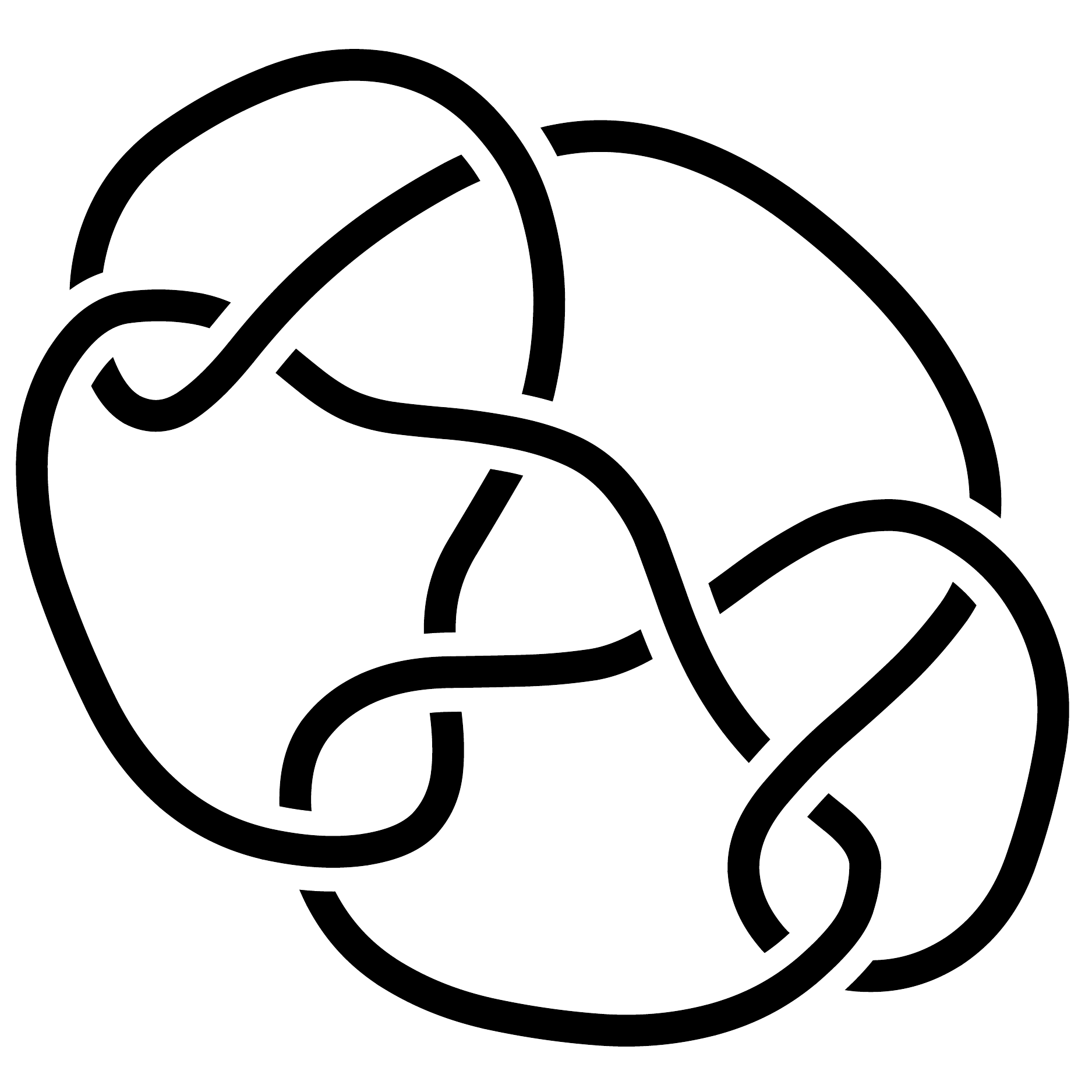}}&   \\
\nopagebreak   & $10_{137}$ & $1-6t+11t^2-6t^3+t^4$ \\
\nopagebreak  &  &  \\
  \multirow{3}{*}{\includegraphics[height=20mm]{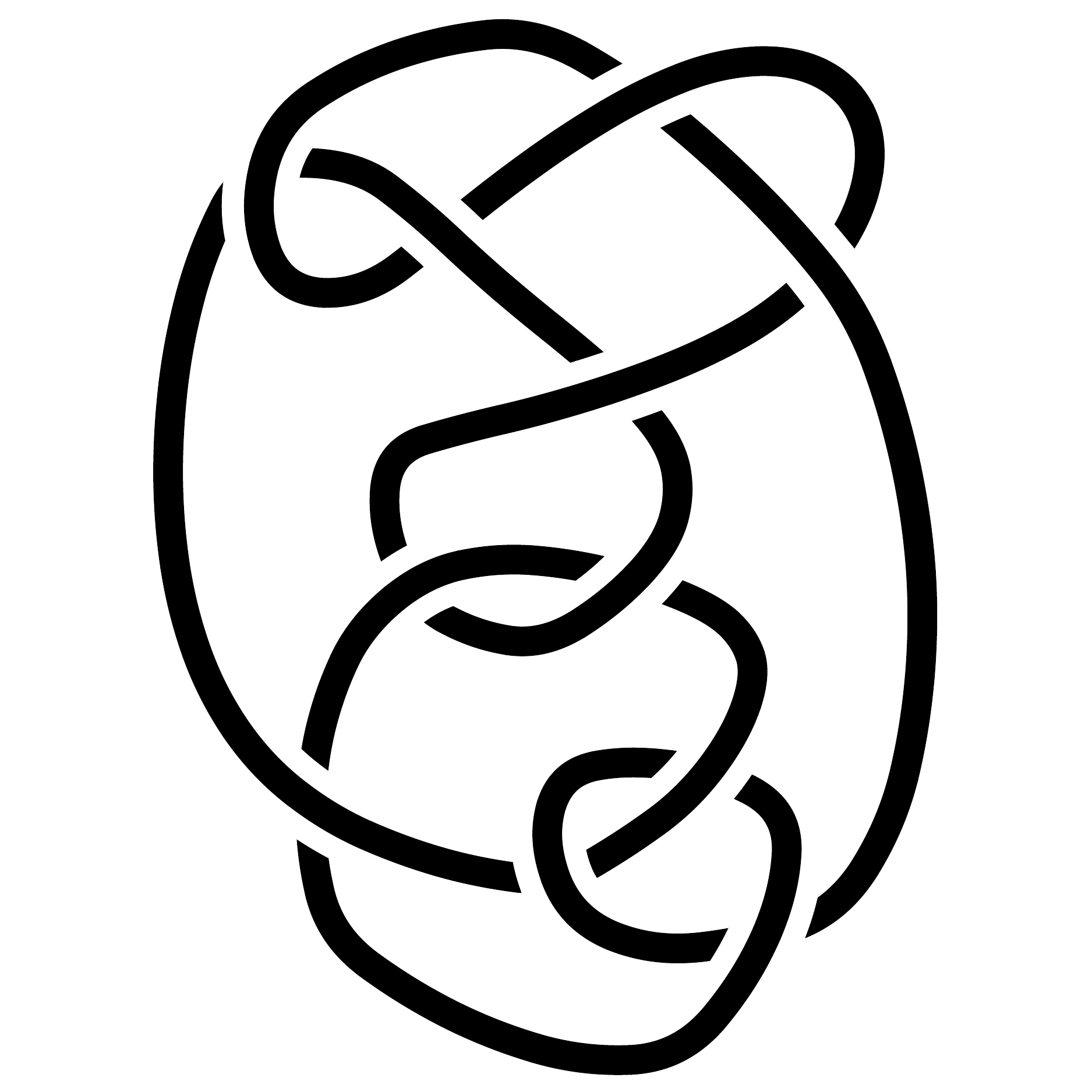}}&    \\
\nopagebreak   & $11a_{5}$ & $1-9t+30t^2-45t^3+30t^4-9t^5+t^6$  \\
\nopagebreak  &  & \\
\multirow{3}{*}{\includegraphics[height=20mm]{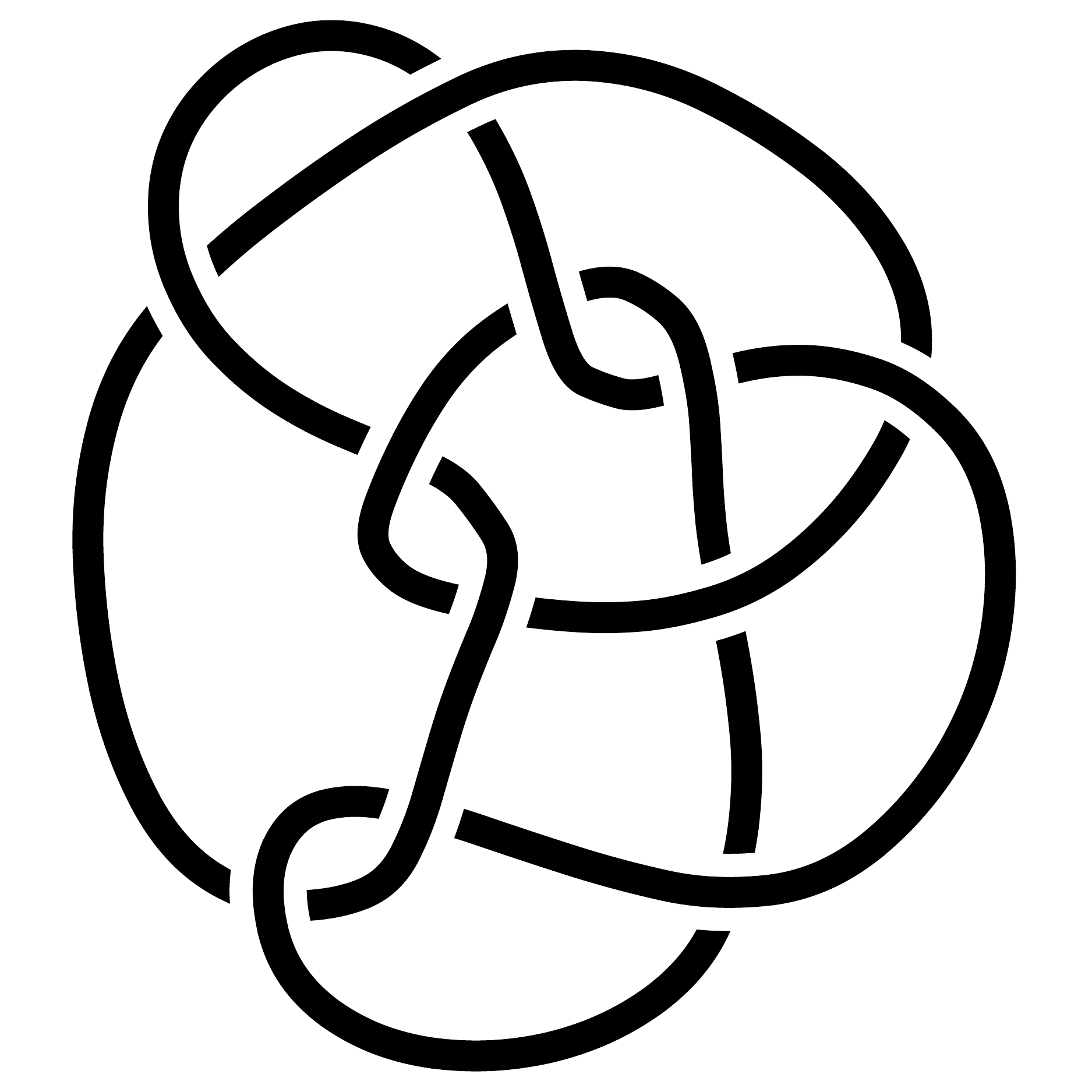}}&   \\
\nopagebreak   & $11n_{142}$ & $1-8t+15t^2-8t^3+t^4$ \\
\nopagebreak  &  &  \\
\multirow{3}{*}{\includegraphics[height=20mm]{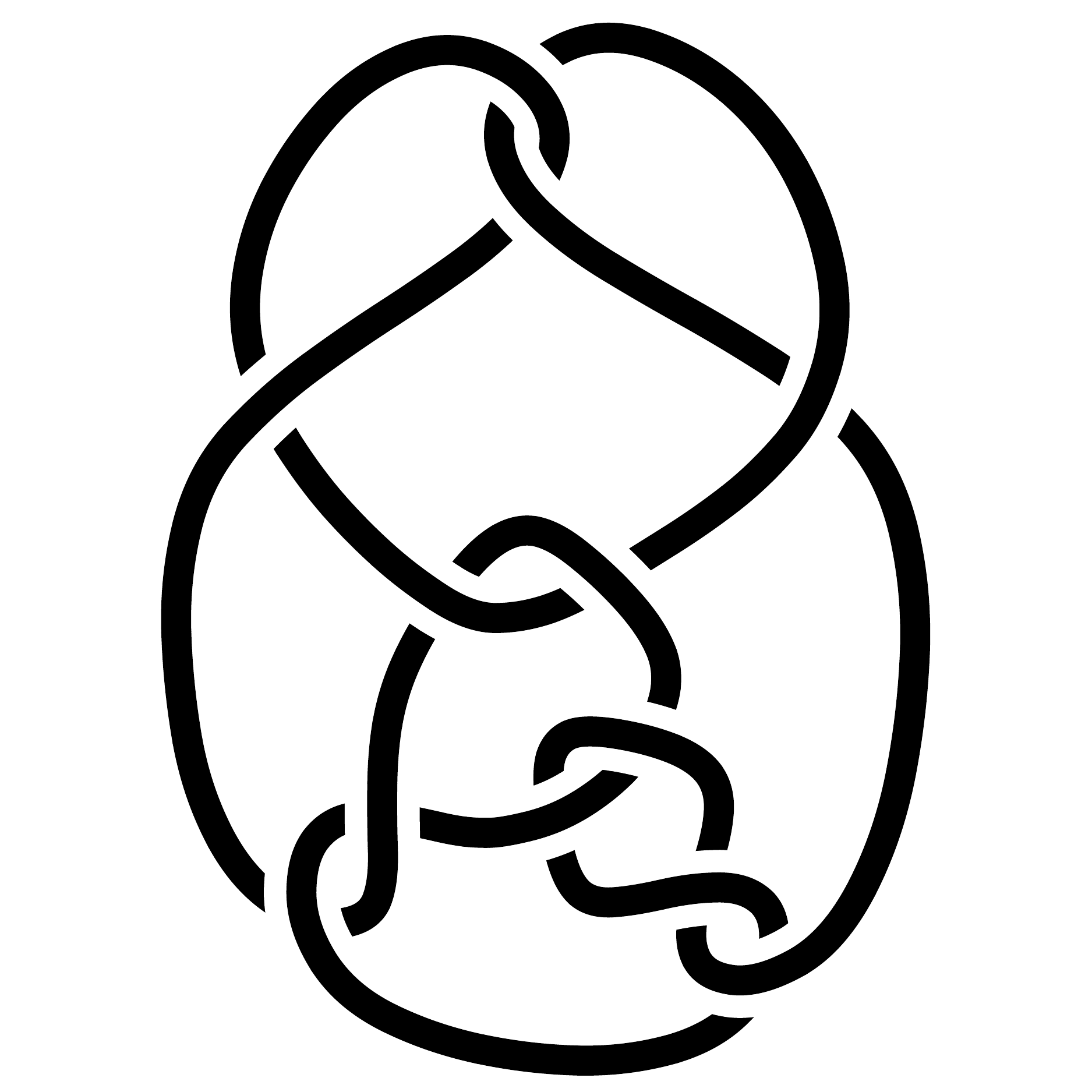}}&  \\
\nopagebreak   & $12a_{0125}$ & $1-12t+44t^2-67t^3+44t^4-12t^5+t^6$ \\
\nopagebreak  &  &  \\
\multirow{3}{*}{\includegraphics[height=20mm]{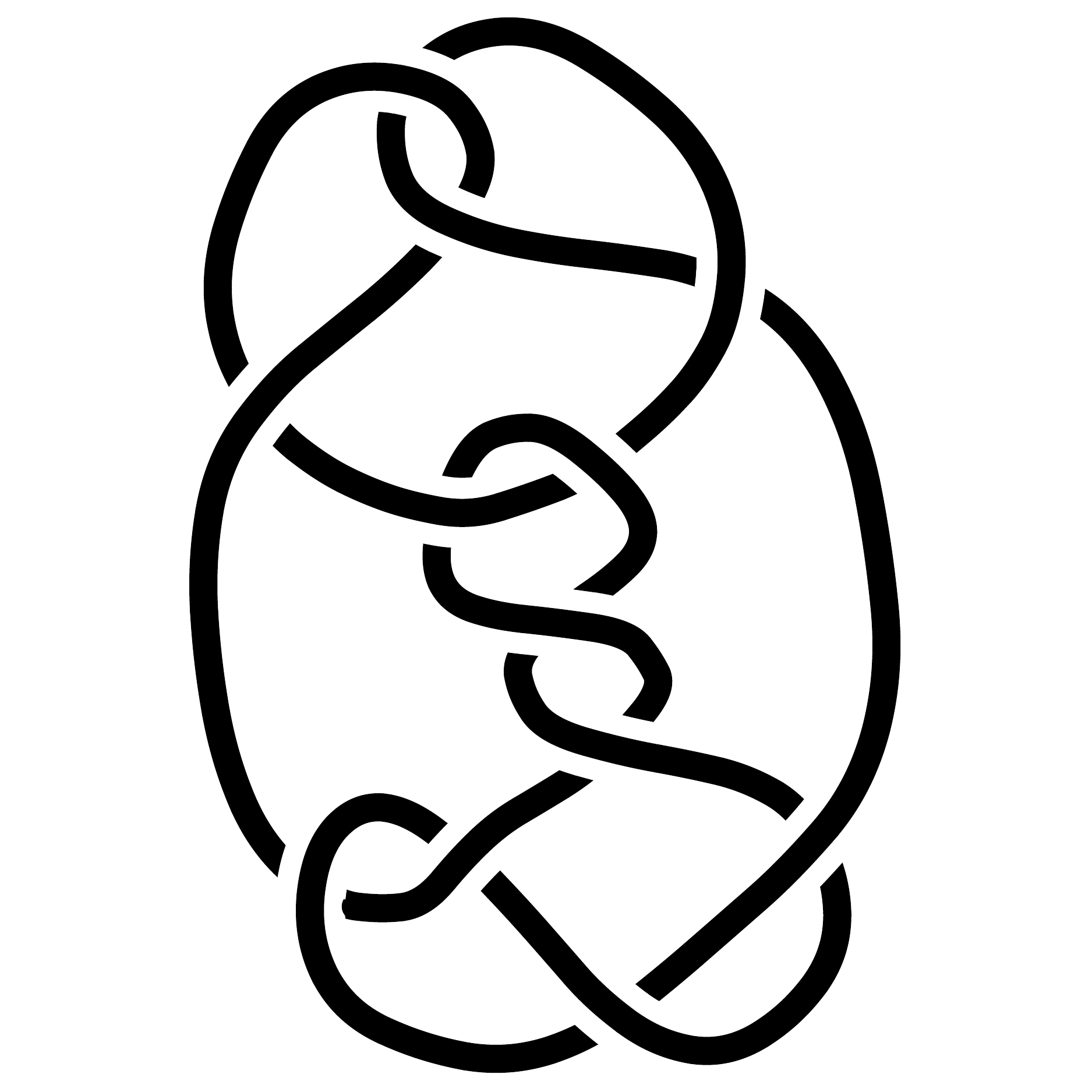}}&   \\
\nopagebreak   & $12a_{0181}$ & $1-11t+40t^2-61t^3+40t^4-11t^5+t^6$ \\
\nopagebreak  &  &  \\
\multirow{3}{*}{\includegraphics[height=20mm]{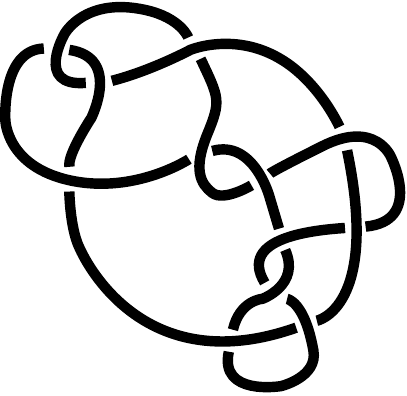}}&   \\
\nopagebreak   & $12a_{0477}$ & $1-11t+ 41t^2-63t^3+ 41t^4-11t^5+ t^6$ \\
\nopagebreak  &  &  \\
\multirow{3}{*}{\includegraphics[height=20mm]{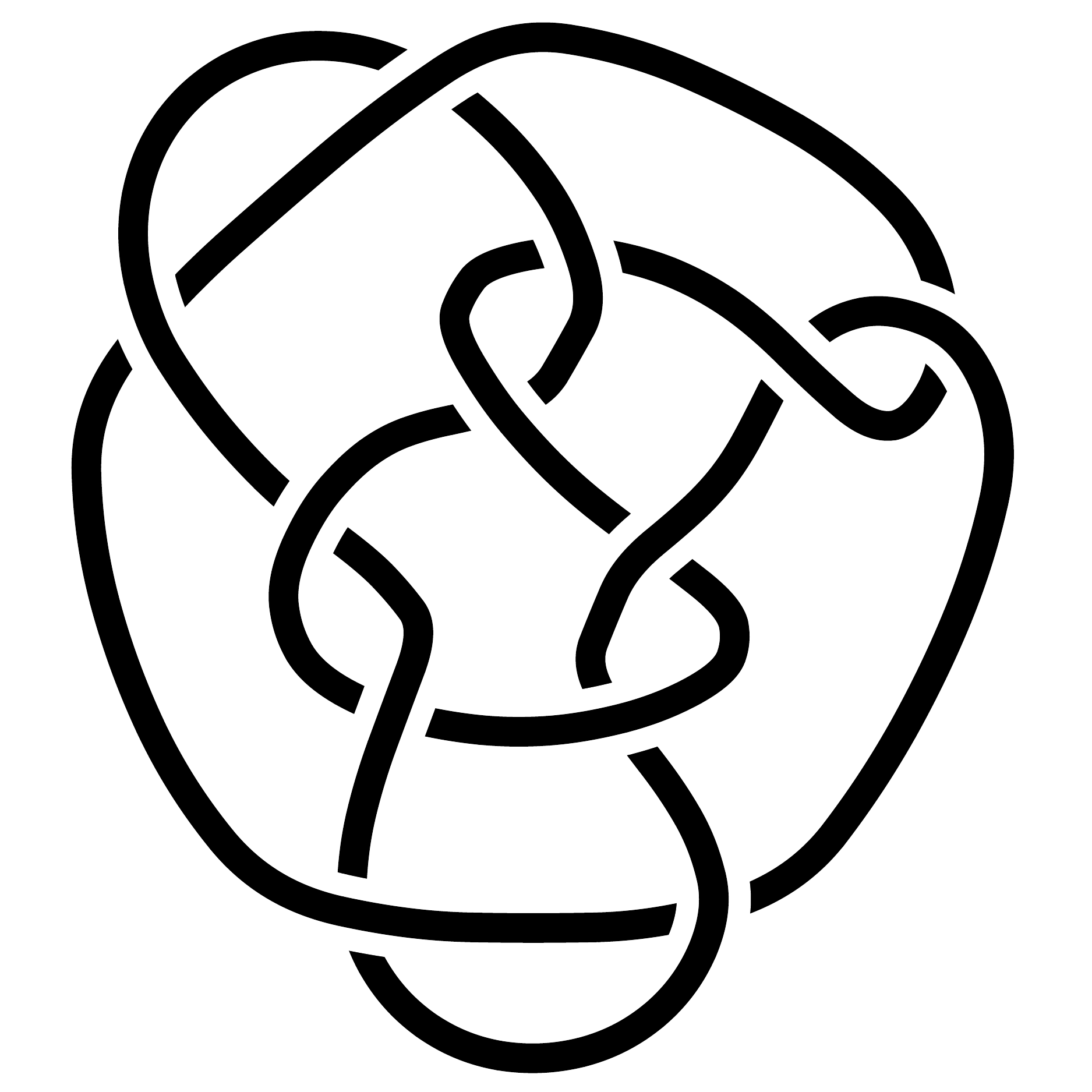}}&   \\
\nopagebreak   & $12a_{1124}$ & $1-13t+50t^2-77t^3+50t^4-13t^5+t^6$ \\
\nopagebreak  &  & \\
\multirow{3}{*}{\includegraphics[height=20mm]{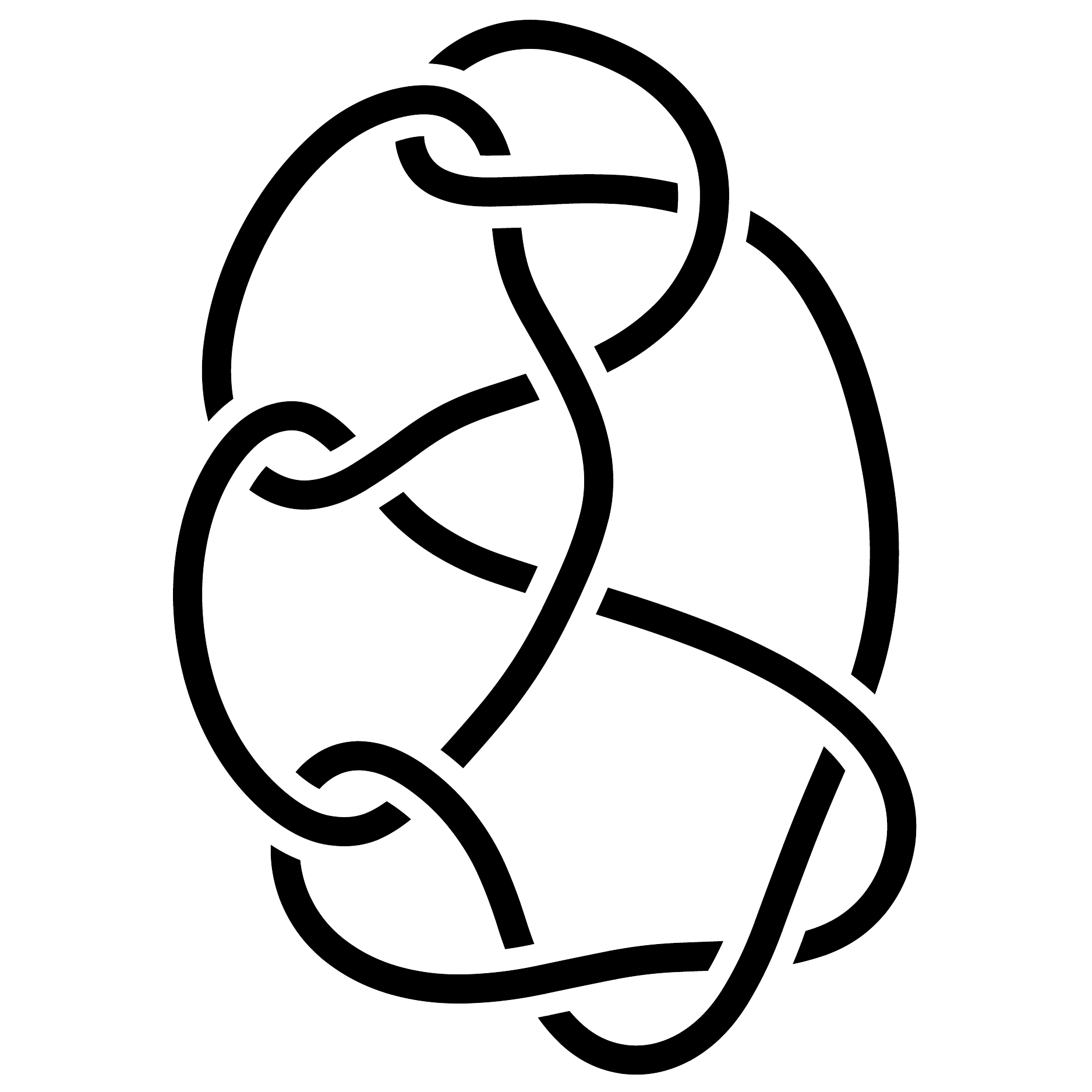}}&   \\
\nopagebreak   & $12n_{0013}$ &  $1-7t+13t^2-7t^3+t^4$\\
\nopagebreak  &  &  \\
\multirow{3}{*}{\includegraphics[height=20mm]{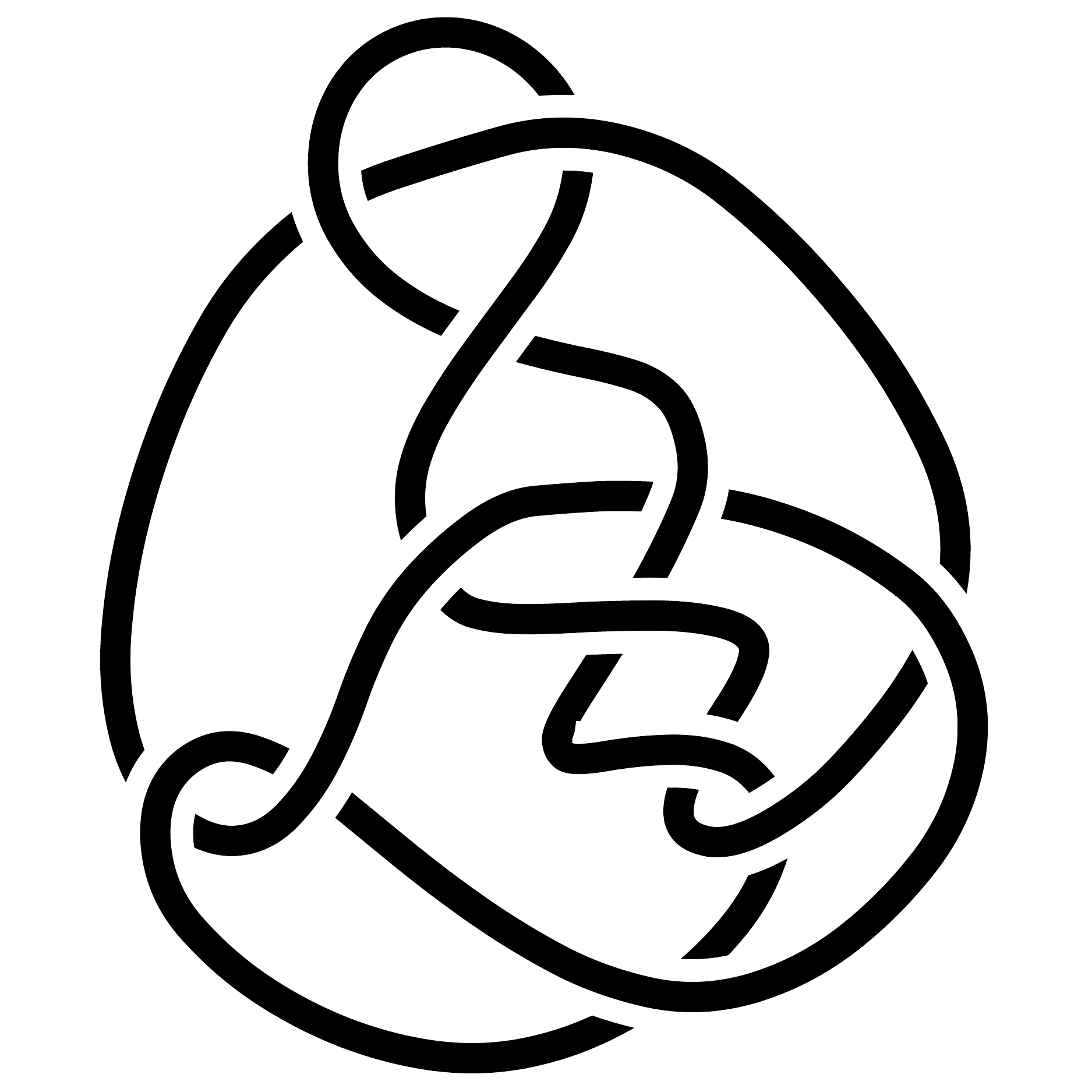}}&   \\
\nopagebreak   & $12n_{0145}$ & $1-6t+11t^2-6t^3+t^4$ \\
\nopagebreak  &  &  \\
\multirow{3}{*}{\includegraphics[height=20mm]{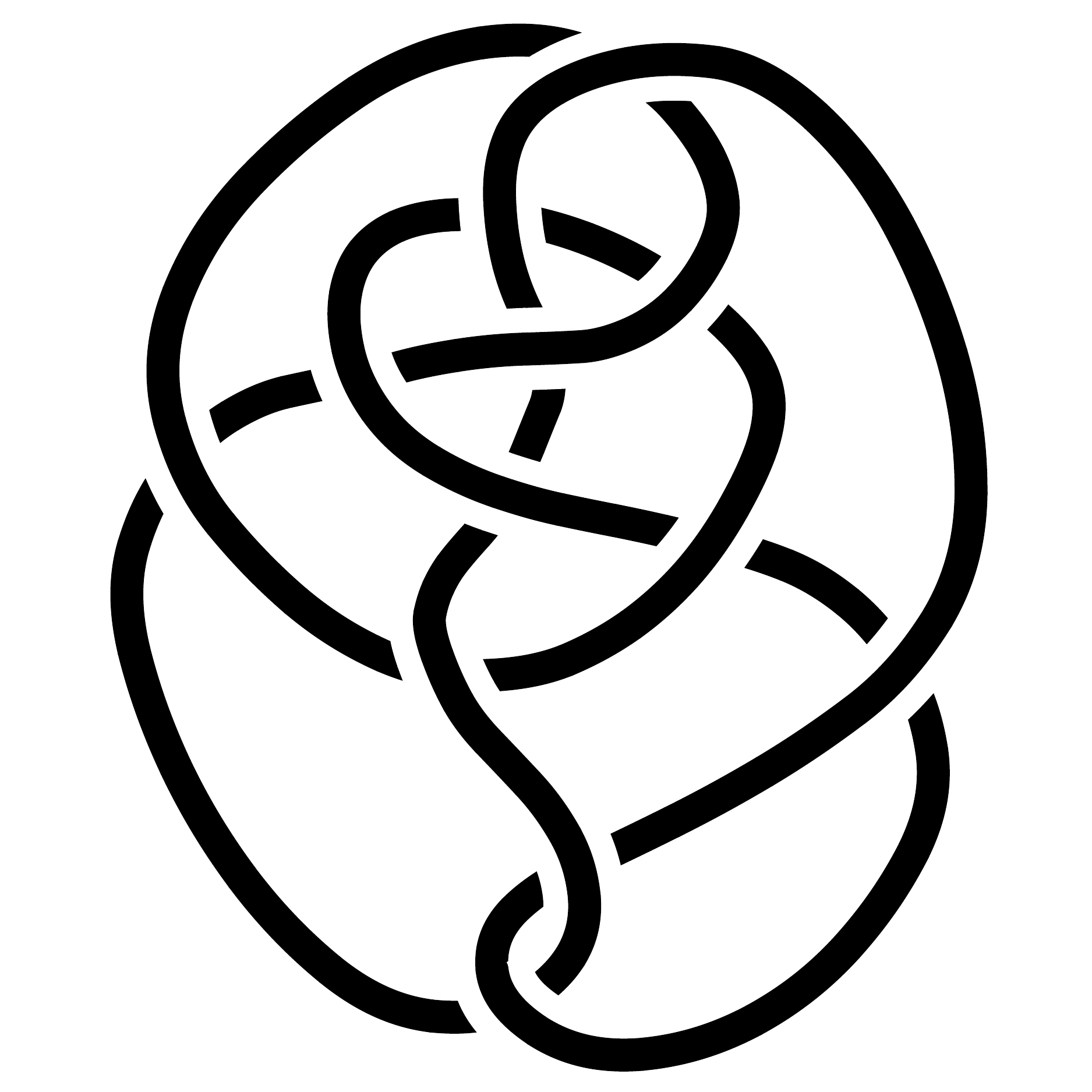}}&    \\
\nopagebreak   & $12n_{0462}$ & $1-6t+11t^2-6t^3+t^4$ \\
\nopagebreak  &  &  \\
\multirow{3}{*}{\includegraphics[height=20mm]{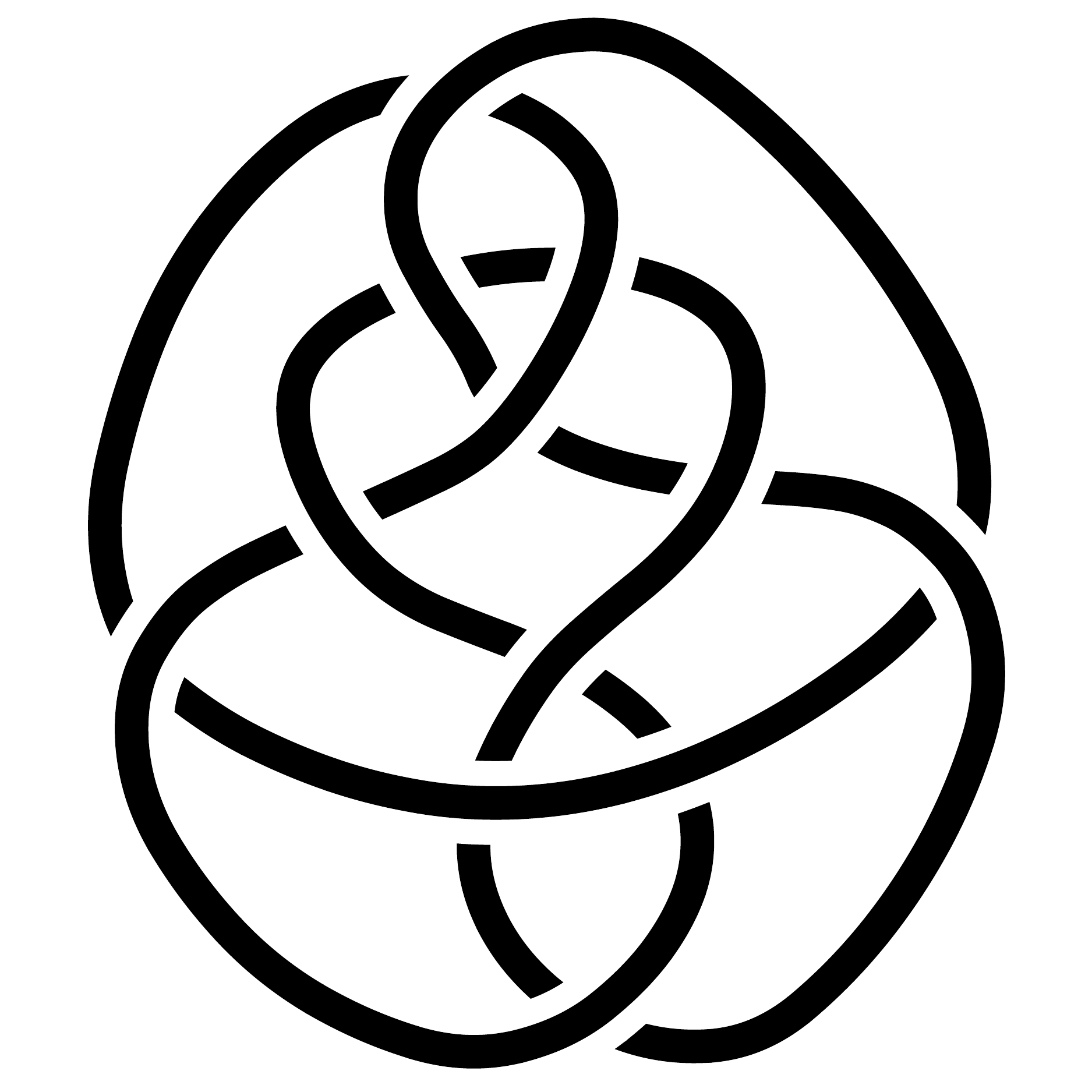}}&    \\
\nopagebreak   & $12n_{0838}$ & $1-6t+11t^2-6t^3+t^4$\\
\nopagebreak  &  &  \\
\end{longtable}
}

We think these are the only fibered knots of at most twelve crossings known to have bi-orderable group at the time of the writing. 
\end{example}

\begin{example}[Non-bi-orderable fibered knot groups]  There are many more knot groups which are known {\em not} to be bi-orderable by applying Theorem \ref{AN}.  According to \cite{CL}, among the knots of 12 or fewer crossings, 1246 of them are fibred and among those knots 485\footnote{While \cite{CL} says that 487 knots have Alexander polynomial with no roots in $\R^+$, there is actually a numerical error in two cases caused by rounding error.} have 
Alexander polynomials with {\em no} roots in ~$\R^+$, so they cannot be bi-orderable.  A complete list of them can be found in \cite{CR12}; the examples with up to ten crossings are: $3_{1}$, $5_{1}$, $6_{3}$, $7_{1}$, $7_{7}$, $8_{7}$, $8_{10}$, $8_{16}$, $8_{19}$, $8_{20}$, $9_{1}$, $9_{17}$, $9_{22}$, $9_{26}$, $9_{28}$, $9_{29}$, $9_{31}$, $9_{32}$, $9_{44}$, $9_{47}$, $10_{5}$, $10_{17}$, $10_{44}$, $10_{47}$, $10_{48}$, $10_{62}$, $10_{69}$, $10_{73}$, $10_{79}$, $10_{85}$, $10_{89}$, $10_{91}$, $10_{99}$, $10_{100}$, $10_{104}$, $10_{109}$, $10_{118}$, $10_{124}$, $10_{125}$, $10_{126}$, $10_{132}$, $10_{139}$, $10_{140}$, $10_{143}$, $10_{145}$, $10_{148}$, $10_{151}$, $10_{152}$, $10_{153}$, $10_{154}$, $10_{156}$, $10_{159}$, $10_{161}$, $10_{163}$.
\end{example}

\subsection{Fibred knots, bi-ordering and eigenvalues}
\label{eigenvalues subsection}

A fibration $X \to S^1$, with fibre $\Sigma$, can be regarded as a product  of  $\Sigma$ with an 
interval $I = [0, 1]$ with the ends identified via some homeomorphism $h$ of $\Sigma$ called the {\em monodromy} \index{monodromy} associated with the fibration: 
$$X = (\Sigma \times I )/ (x,1) \sim (h(x), 0).$$

The fundamental group of $X$ is an HNN extension of that of $\Sigma$ (see \cite{HNN49} for background on HNN extensions).   In the special case that $X$ is the complement of a fibred knot or link, $\Sigma$ is a connected surface with boundary and its fundamental group is a free group. We may write
$$\pi_1(X) = \langle x_1, \dots, x_{2g}, t  \; | \;  t^{-1}x_it = h_* (x_i), i = 1, \dots, 2g  \rangle$$
where $\{x_i\}$ is a set of free generators of $\pi_1(\Sigma)$, the map $h_* : \pi_1(\Sigma) \to \pi_1(\Sigma)$ is the homotopy monodromy map induced by $h$, and $g$ is the genus of the surface $\Sigma$.

We know that free groups are bi-orderable.  Although left-orderability is preserved under taking HNN extensions, bi-orderability may not be.  In fact it is not difficult to verify that an HNN extension of a bi-orderable group $G$ is again bi-orderable if and only if there is a bi-ordering of $G$ which is invariant under the automorphism $\varphi : G \to G$ associated with the extension.

Thus when $X$ is the complement of a fibred knot, its fundamental group is bi-orderable if and only if there exists a bi-ordering of the free group $\pi_1(\Sigma)$ which is invariant under the homotopy monodromy $h_* :\pi_1( \Sigma) \to \pi_1(\Sigma)$.  We also note that for a fibred knot $K$ the Alexander polynomial $\Delta_K(t)$ is precisely the characteristic polynomial for the {\em homology} monodromy  $h_* : H_1(\Sigma) \to H_1(\Sigma)$.  The homology groups may be considered with rational (or even real or complex) coefficients, so that $h_*$ is a linear map of vector spaces.  (Note that we are using the same symbol for the homotopy monodromy and the homology monodromy -- the induced map upon abelianization.  The context should make it clear which map is under discussion.)

\subsection{Digression on linear algebra}
\label{linear algebra}

Suppose we have an invertible linear map $L : \R^n \to \R^n$ and we wish to find a bi-ordering $<$ of $\R^n$ as an additive group such that $L$
preserves the order: ${v} < {w} \Leftrightarrow L(v) < L(w)$.  

It may be impossible --- for example if $L$ has finite order or even a finite orbit.  Indeed suppose that 
 $L(v) \ne v$ but $L^k(v) = v$ and $<$ is an invariant ordering.  If $L(v) < v$ we have 
 $L^2(v) < L(v) < v$ and inductively $L^k(v) < v$, a contradiction.  If $L(v) > v$ a similar contradiction ensues. 
 
On the other hand, suppose $L$ is represented by an upper triangular matrix, as in the following equation (take $n= 3$ for simpicity, here $a, b$ and $c$ are arbitrary):

$$\left( \begin{array}{ccc}
\lambda_1 & a & b \\
0 & \lambda_2 & c \\
0 & 0 & \lambda_3 \end{array} \right) 
\left(\begin{array}{c} x_1 \\ x_2 \\ x_3\end{array}\right) 
=\left(\begin{array}{c} \lambda_1 x_1 + ax_2  +bx_3\\ \lambda_2 x_2 +cx_3 \\ \lambda_3 x_3\end{array}\right).
$$

Further suppose the eigenvalues $\lambda_i$ are all positive.  Then we can order vectors by taking the positive cone to be all vectors whose last nonzero coordinate is greater than zero.  In other words, we use a reverse lexicographic order.  Then one sees from the above equation that $L$ preserves that positive cone and hence respects the ordering.  We have sketched a proof of the following.

\begin{proposition}
If all the eigenvalues of a linear map $L: \R^n \to \R^n$ are real and positive, then there is a bi-ordering of $\R^n$ which is preserved by $L$.  Similarly for $\Q^n$ in place of $\R^n$. 
\end{proposition}

\begin{problem}  Prove this proposition.
\end{problem}

There is a partial converse.
\begin{proposition}\label{posreal}
Suppose there is a bi-ordering of $\R^n$ which is preserved by the nonsingular linear map $L: \R^n \to \R^n$.  Then $L$ has at least one positive real eigenvalue.  Similarly for $\Q^n$ in place of $\R^n$. 
\end{proposition}

This may be proved topologically.  Consider the set $H$ of all points in $\R^n$ for which every neighbourhood contains points greater than zero and also points less than zero in the given ordering.  

\begin{problem}\label{hyperplane} Show that $H$ is a linear subspace of $\R^n$ by arguing that $0 \in H$ and $H$ is closed under addition and by multiplication by scalars.  Moreover argue that $H$ separates 
$\R^n$, so it is a codimension one subspace
with positive points on one side and negative points on the other.  Points on $H$ may be either positive or negative in the given ordering, with the exception of zero, of course.  Inductively $H$ may be separated by a subspace again in a similar manner.  
\end{problem}

To continue with the proof of Proposition \ref{posreal}, let $D \subset S^{n-1}$ be the (closed) half of the unit sphere on the positive side of $H$.  Then $D$ is homeomorphic with an $n-1$ dimensional ball.  Since $L$ preserves the ordering, the map $v \to L(v)/ |L(v)|$ takes $D$ to itself.  By the Brouwer fixed-point theorem, that map has a fixed point.  Finally, we observe that such a fixed point is an eigenvector of $L$ with positive eigenvalue.  To argue for $\Q^n$ just repeat this argument, considering $\Q^n$ inside $\R^n$ in the usual way. \qed

\subsection{Proof of Theorem \ref{AS}}

Since roots of the Alexander polynomial are exactly the eigenvalues of the homology monodromy associated with a fibred knot, our problem reduces to showing:

\begin{proposition}
Suppose $h : F \to F$ is an automorphism of a finitely-generated free group.  If all the eigenvalues of  the induced map  $h_* : H_1(F ; \Q) \to  H_1(F ; \Q)$
are real and positive, then there is a bi-ordering of $F$ preserved by $h$.
\end{proposition}

\begin{proof} One way to order a free group $F$ is to use the lower central series as discussed in Problem \ref{lowercentral}.  Recall the series is
$F_1 \supset F_2 \supset \cdots$ defined by

$$F_1 = F, \quad F_{i+1} = [F, F_i]$$

This has the properties that $\displaystyle \bigcap F_i = \{1\}$ and each $F_{i}/F_{i+1}$ is free abelian.

Choose an arbitrary bi-ordering of $F_{i}/F_{i+1}$, and define a positive cone of $F$ by declaring
$1 \ne x \in F$ positive if its class in $F_{i}/F_{i+1}$ is positive in the chosen ordering, where $i$ is the last subscript such that $x \in F_i$.  This is a bi-ordering of $F$.

If $h: F \to F$ is an automorphism it preserves the lower central series and induces maps of the lower central quotients:  $h_i: F_{i}/F_{i+1} \to F_{i}/F_{i+1}$.   With this notation, $h_1$ and $F/F_1$ are just the abelianization $h_{ab}$ and $F_{ab}$ respectively; and it turns out that, in a sense, all the $h_i$ are determined by $h_1$.  That is, there is an embedding of $F_{i}/F_{i+1}$ in the tensor power $F_{ab}^{\otimes k}$, and the map $h_i$ is just the restriction of $h_{ab}^{\otimes k}$.  The reader is referred to \cite{PR03} for details.

The assumption that all eigenvalues of $h_{ab}$ are real and positive implies that the same is true of all its tensor powers. 
This allows us to find bi-orderings of the free abelian groups $F_{i}/F_{i+1}$ which are invariant under $h_i$ for all $i$.  Using these to bi-order $F$, we get invariance under $h$, which proves the present proposition and Theorem \ref{AS} as well. \end{proof}

\begin{problem}  Verify the assertions of the preceding paragraph.
\end{problem}

\subsection{Proof of Theorem \ref{AN}}

Let's turn to the proof of the third main theorem: If $K \subset S^3$ is a nontrivial fibred knot whose knot group is bi-orderable, then $\Delta_K(t)$ has at least two real positive roots.  First of all, since the Alexander polynomial satisfies $\Delta_K(t) = t^{2g}\Delta_K(1/t)$ and $\Delta_K(1) = \pm 1$, any positive real root $r$ will produce another, namely $1/r$.  So we need only find one positive real root.

Our third theorem will follow from a more general result.   Suppose $G$ is an arbitrary finitely generated group. If $\phi : G \to G$ is an automorphism, we can define its {\it eigenvalues} to be the eigenvalues of its induced map on the rational vector space $H_1(G; \Q) \cong (G/G') \otimes \Q$. 

\begin{theorem}\label{poseval}
Suppose $G$ is a nontrivial finitely generated bi-orderable group and 
that  the automorphism $\phi : G \to G$ preserves a bi-ordering of $G$.  Then $\phi$ has a positive eigenvalue.
\end{theorem}

\begin{proof}
To prove this, assume $\phi : G \to G$ preserves a bi-ordering of $G$.  Then $\phi$
induces an automorphism $\phi_*: G/G' \to G/G'$, but (unless the commutator subgroup $G'$ is convex) we don't know that $G/G'$ inherits a $\phi_*$-invariant ordering.  However $G$, being finitely generated, does have a maximal proper convex subgroup $C$ as we saw in Section \ref{section: BO implies LI}.  Since $\phi$ respects the ordering, $C$ is $\phi$-invariant.  Moreover $G/C$ is abelian so $G' \subset C$ and we have the commutative diagram with exact rows:

\[
\begin{CD}0 @>>> C/G' @>>> G/G' @>>> G/C @>>> 0\mbox{ } \\
 & & @VVV @V{\phi_*}VV @V{\phi_C}VV \\
0 @>>> C/G' @>>> G/G' @>>> G/C @>>>0.
\end{CD}
\]
Also since $C$ is convex, $G/C$ inherits an order from $G$ which is invariant under 
$\phi_C$.

Writing $U = C/G' \otimes \Q$, $V= G/G' \otimes \Q$, and $W= G/C \otimes \Q$, tensoring with $\Q$  yields the commutative diagram of finite-dimensional vector spaces over $\Q$ with exact rows:
\[
\begin{CD}0 @>>>  U @>>> V @>>> W @>>> 0\mbox{ }  \\
& &  @VVV @V{\phi_V}VV @V{\phi_W}VV \\
0 @>>> U @>>> V @>>> W @>>>0 ,
\end{CD}
\]
where $\phi_W = \phi_C \otimes id$ and $\phi_V = \phi_* \otimes id$. 

Since $\phi_W$ preserves the induced ordering of $W$, it has a positive real eigenvalue.

Letting $\phi_U$ be 
$\phi_V$ restricted to $U$, we conclude that  $\phi_V = \phi_{U} \oplus \phi_W$. 
Therefore the characteristic polynomial of $\phi_V$ factors as
\[ \chi_{\phi_V}(\lambda) = \chi_{\phi_{U}}(\lambda) \cdot \chi_{\phi_W}(\lambda) . 
\]
The positive eigenvalue of $\phi_W$ is also an eigenvalue of $\phi_V$, concluding the proof.
\end{proof}

\subsection{Non-fibered knots} Next we turn our attention to non-fibred knots.
At the time of this writing, non-fibered knots are not as well understood as fibered knots.  Indeed, there are analogues of Theorems \ref{AS} and \ref{AN} in the non-fibered case, but they only apply to groups with two generators and a single relator \cite{CGW14}.  Moreover, the relator must satisfy some technical combinatorial conditions which we will not cover here.
  
Nevertheless, there are some classes of non-fibered knots whose groups have a presentation of the necessary form, and so bi-orderability of their groups can be determined by examining the roots of the Alexander polynomial.  We present three such examples below, see \cite{CDN14} for full details.  

\begin{example}[Two-bridge knots]
\index{two-bridge knots}
A \textit{two-bridge} knot is a knot which admits a diagram appearing as in Figure \ref{2bridge}.  In that diagram, each box represents some number $a_i$ of horizontal half-twists; the sign of $a_i$ indicates the direction of twisting.  For example, Figure \ref{twist example} shows the case of $a_i =3$ and $a_i =-1$ twists.

It is a result of Schubert that two-bridge knots are in one-to-one correspondence with coprime pairs of odd integers $p$ and $q$, with $0<p<q$ \cite{Schubert56}.  Thus every two-bridge knot may be written as $K_{\frac{p}{q}}$ where $\frac{p}{q}$ is a reduced fraction.  Their knot groups are given by the presentation
\[ \pi_1(S^3 \setminus K_{\frac{p}{q}}) = \langle a, b \mid aw=wb  \rangle  
\]
where $w= b^{\epsilon_1}a^{\epsilon_2} \cdots b^{\epsilon_{q-2}} a^{\epsilon_{q-1}}$ and $\epsilon_i = (-1)^{\lfloor \frac{ip}{q} \rfloor}$.  See  \cite{Murasugi61} for details of this presentation. 

\begin{figure}[h!]
\setlength{\unitlength}{8.5cm}
\begin{picture}(1,0.29618768)%
    \put(0,0){\includegraphics[width=\unitlength]{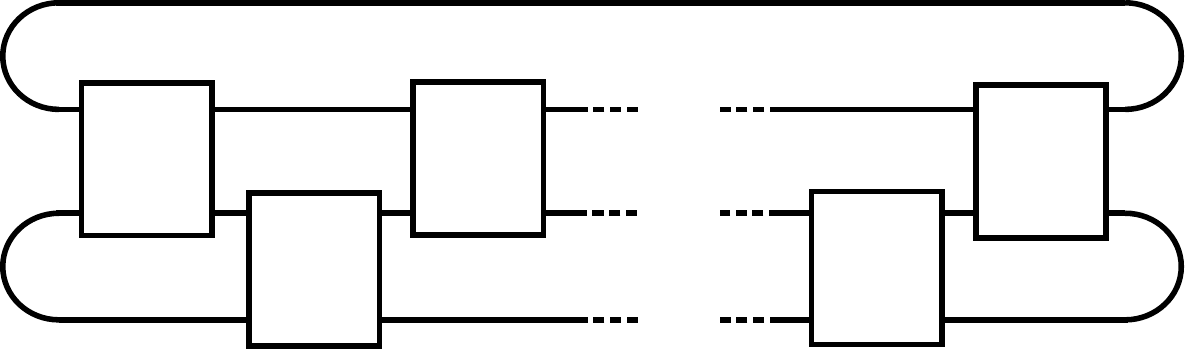}}%
    \put(0.10089022,0.15219595){$a_1$}%
    \put(0.24332844,0.05327264){$a_2$}%
    \put(0.37738792,0.15219595){$a_3$}%
    \put(0.70086088,0.05327264){$a_{n-1}$}%
    \put(0.85492036,0.15219595){$a_n$}%
  \end{picture}%
\caption{A two bridge knot.  In each box, $a_i$ is an integer which indicates the number of half twists.}
\label{2bridge}
\end{figure}

\begin{figure}[h!]
\includegraphics[width=3.5cm]{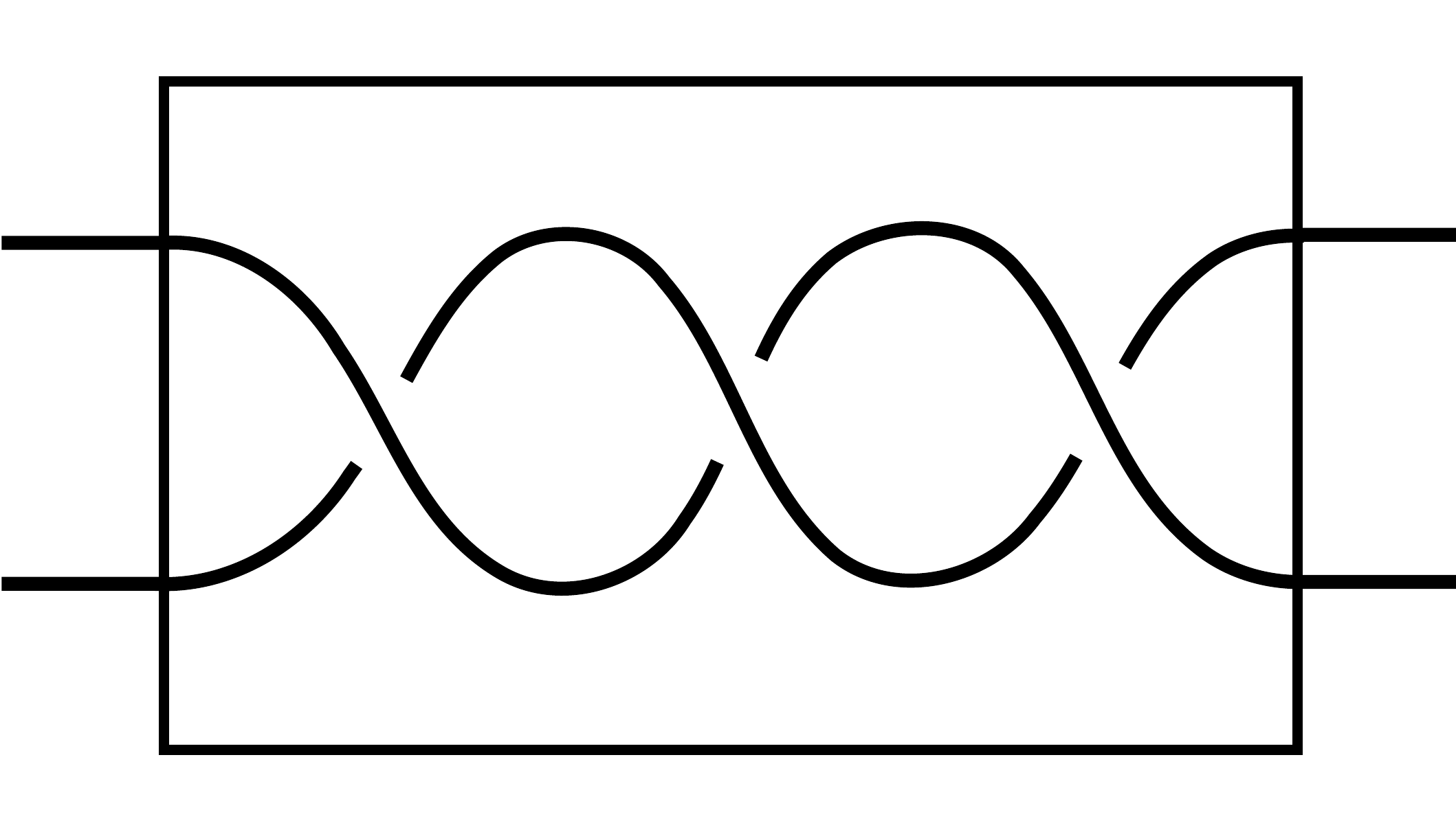}  \hspace{3em}  \includegraphics[width=3.5cm]{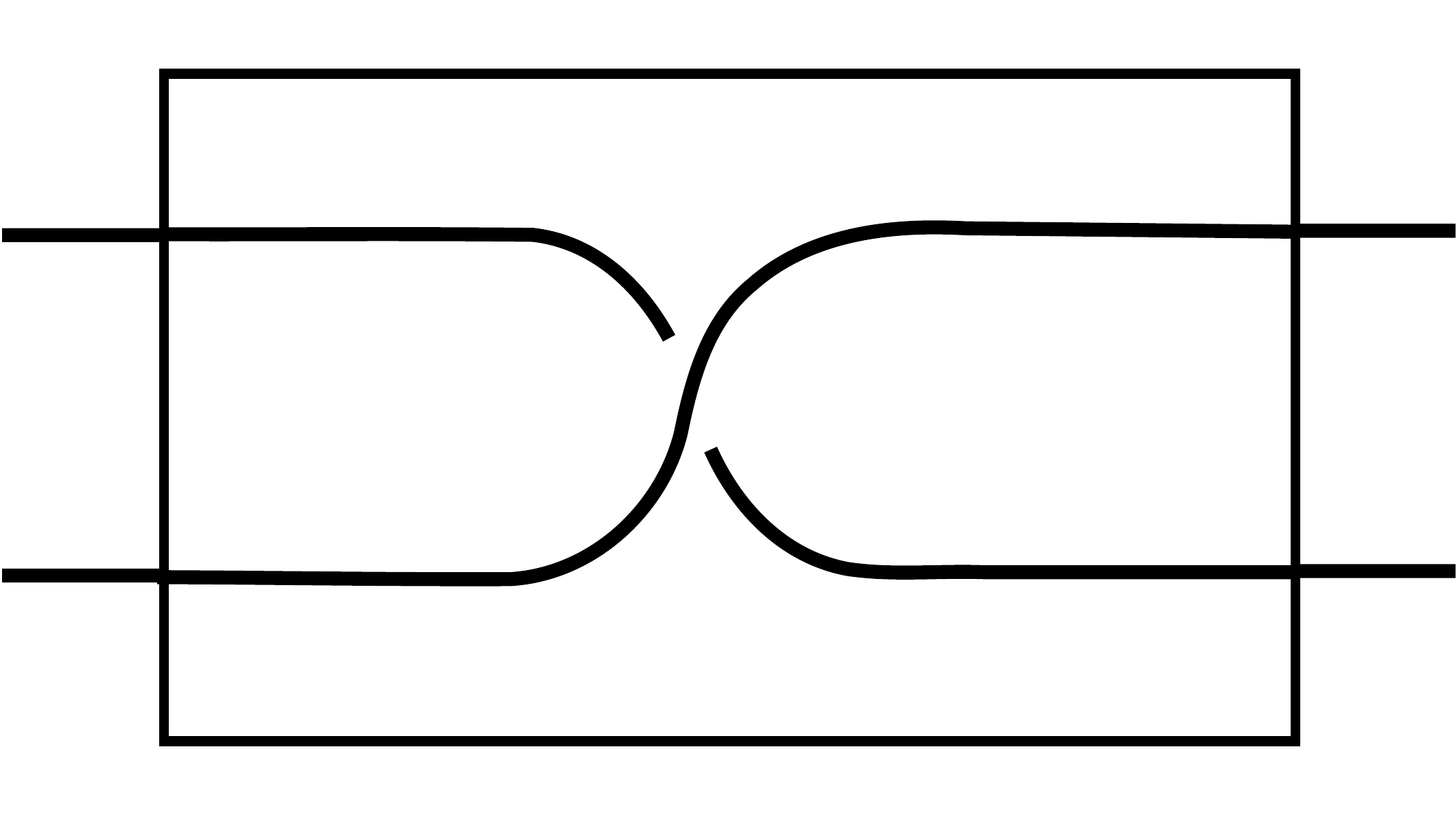}
\caption{Our twisting convention, with $a_i =3$ on the left and and $a_i =-1$ on the right.}
\label{twist example}
\end{figure}

Therefore two-bridge knot groups have two generators and one relation.  With some work, one can also show that the relator $awb^{-1}w^{-1}$ satisfies the combinatorial condition needed to apply \cite[Theorem A]{CGW14}, from which we conclude:

\begin{theorem} If $K$ is a two-bridge knot whose knot group is bi-orderable, then $\Delta_K(t)$ has at least two positive real roots.
\end{theorem}

This allows us to find many more knots whose group is not bi-orderable, a number of which are non-fibred.  For example, amongst knots with 10 or fewer crossings, the following knots are non-fibered, two-bridge, and their Alexander polynomials have no positive roots; therefore their groups are not bi-orderable: 
$5_{2}$, $7_{2}$, $7_{3}$, $7_{4}$, $7_{5}$, $8_{8}$, $8_{13}$, $9_{2}$, $9_{3}$, $9_{4}$, $9_{5}$, $9_{6}$, $9_{7}$, $9_{9}$, $9_{10}$, $9_{13}$, $9_{14}$, $9_{18}$, $9_{19}$, $9_{23}$, $10_{10}$, $10_{12}$, $10_{15}$, $10_{19}$, $10_{23}$, $10_{27}$, $10_{28}$, $10_{31}$, $10_{33}$, $10_{34}$, $10_{37}$, $10_{40}$.
\end{example}

\begin{example}[Twist knots] 
\index{twist knots}
\textit{Twist knots} are a special class of two-bridge knots that appear as in Figure \ref{twist}. Restricting to twist knots on can say more:

\begin{figure}[h!]
\setlength{\unitlength}{3cm}
 \begin{picture}(1.75,1)%
    \put(0,0){\includegraphics[height=\unitlength]{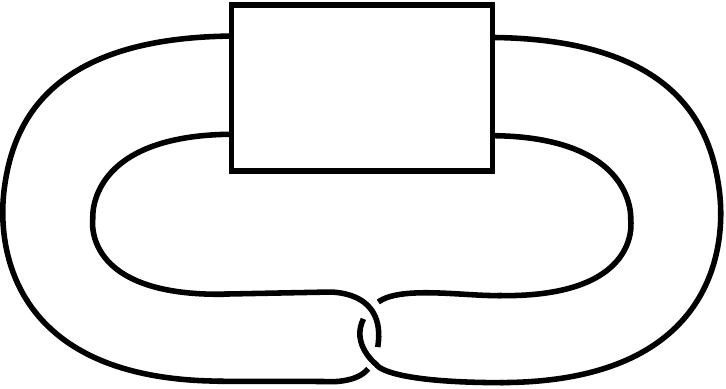}}%
    \put(0.69,0.75718257){$m$ twists}%
  \end{picture}%
\caption{The twist knot $K_m$ with $m$ positive half-twists.}
\label{twist}
\end{figure}

\begin{theorem}
\label{twist_theorem}
For each integer $m > 1$, let $K_m$ denote the twist knot with $m$ twists.  If $m$ is even, then $K_m$ has bi-orderable knot group, otherwise its group is not bi-orderable.
\end{theorem}

Using Theorem \ref{twist_theorem} one finds an additional 4 knots with 12 or fewer crossings which are non-fibered and have bi-orderable group: $6_1$, $8_1$, $10_1$, $12a_{0803}$.  
\end{example}

At present, very little is known about non-fibred knots that are not two-bridge.  We end this section with one such example.

\begin{example}[Other non-fibered knots with bi-orderable knot group]
The knot $10_{13}$ is also known to have a bi-orderable knot group, though it is not two-bridge.  Its group must be analyzed directly using the theorems of \cite{CGW14}, then one applies the fact that its Alexander polynomial $\Delta_{10_{13}} = 2-13t+ 23t^2-13t^3+ 2t^4$ has only positive real roots.

\begin{figure}[h!]
\includegraphics[scale=0.2]{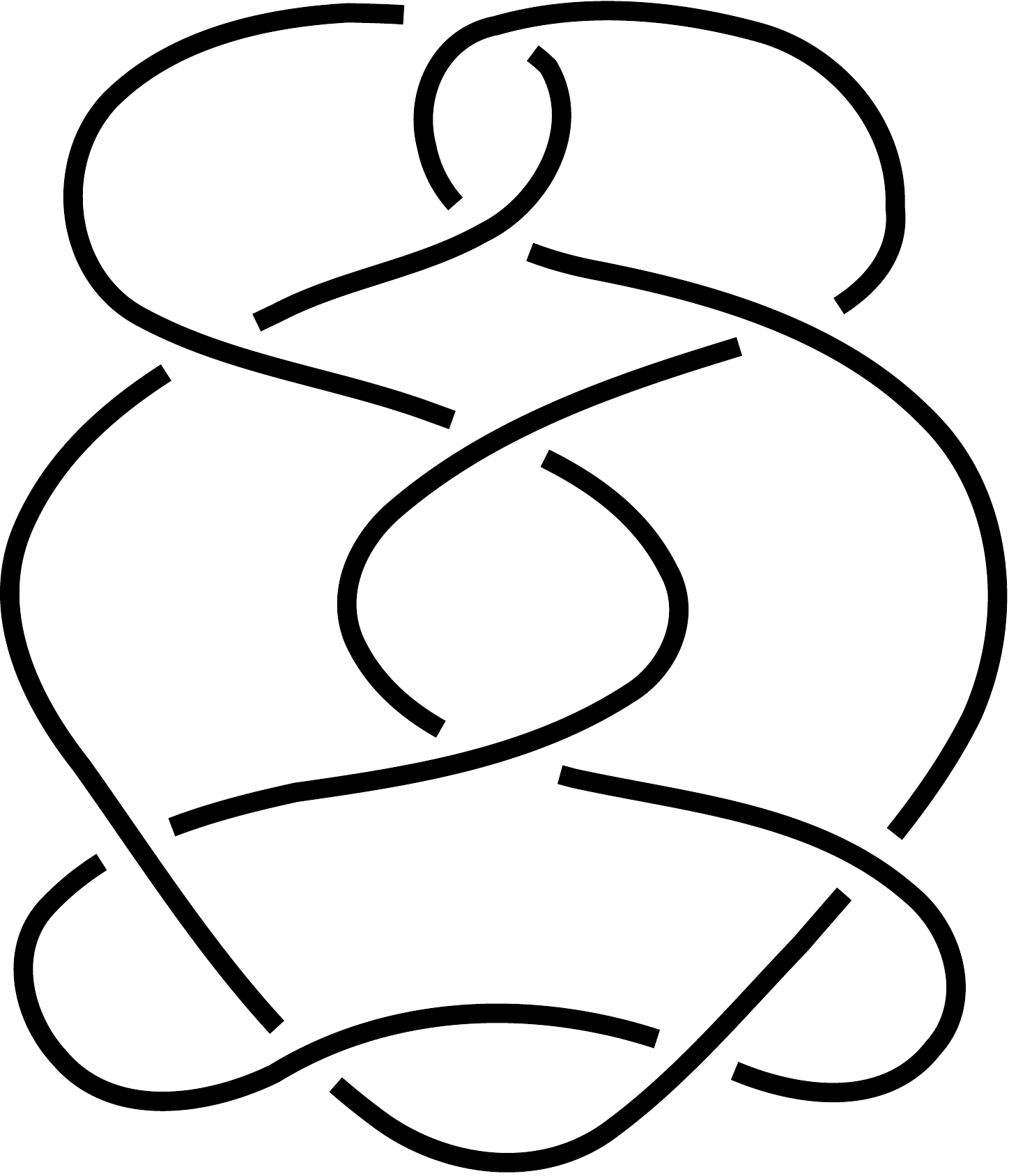}
\caption{The knot $10_{13}$.}
\end{figure}
\end{example}

\section{Crossing changes: a theorem of Smythe}

One of the early applications of orderable groups to knot theory appeared in a 1967 paper by N. Smythe \cite{smythe67}.  We have been depicting knots in $\R^3$ by drawing projections in the plane which have only a finite set of transverse double points, and are otherwise nonsingular.  At each double point, one strand is ``over" and the other ``under" referring to its co-ordinate in the third dimension.  This is traditionally depicted by putting a little gap in the understrand.

A {\em crossing change} consists of reversing the situation, so that the previously ``over" strand becomes the ``under'' and vice-versa.  It has long been known (for example, see \cite{Alexander28} p. 299) that any knot projection will become a projection of an unknot after making some crossing changes.
One method of doing this is the process of ``laying down a rope."  First choose a nonsingular point of the projection as basepoint and orient the curve.  Then, starting at the basepoint and travelling in the direction of the orientation, each time a crossing is encountered one makes the first visit to that crossing be ``under'' by changing the crossing if necessary.  The process is illustrated in Figure \ref{annulus} in which the knot on the left becomes the one on the right after changing crossings by laying the rope, starting at the base point indicated by the dot.   

\begin{figure}[h!]
\centering
\begin{minipage}[b]{0.45\linewidth}
  \includegraphics[scale=0.35]{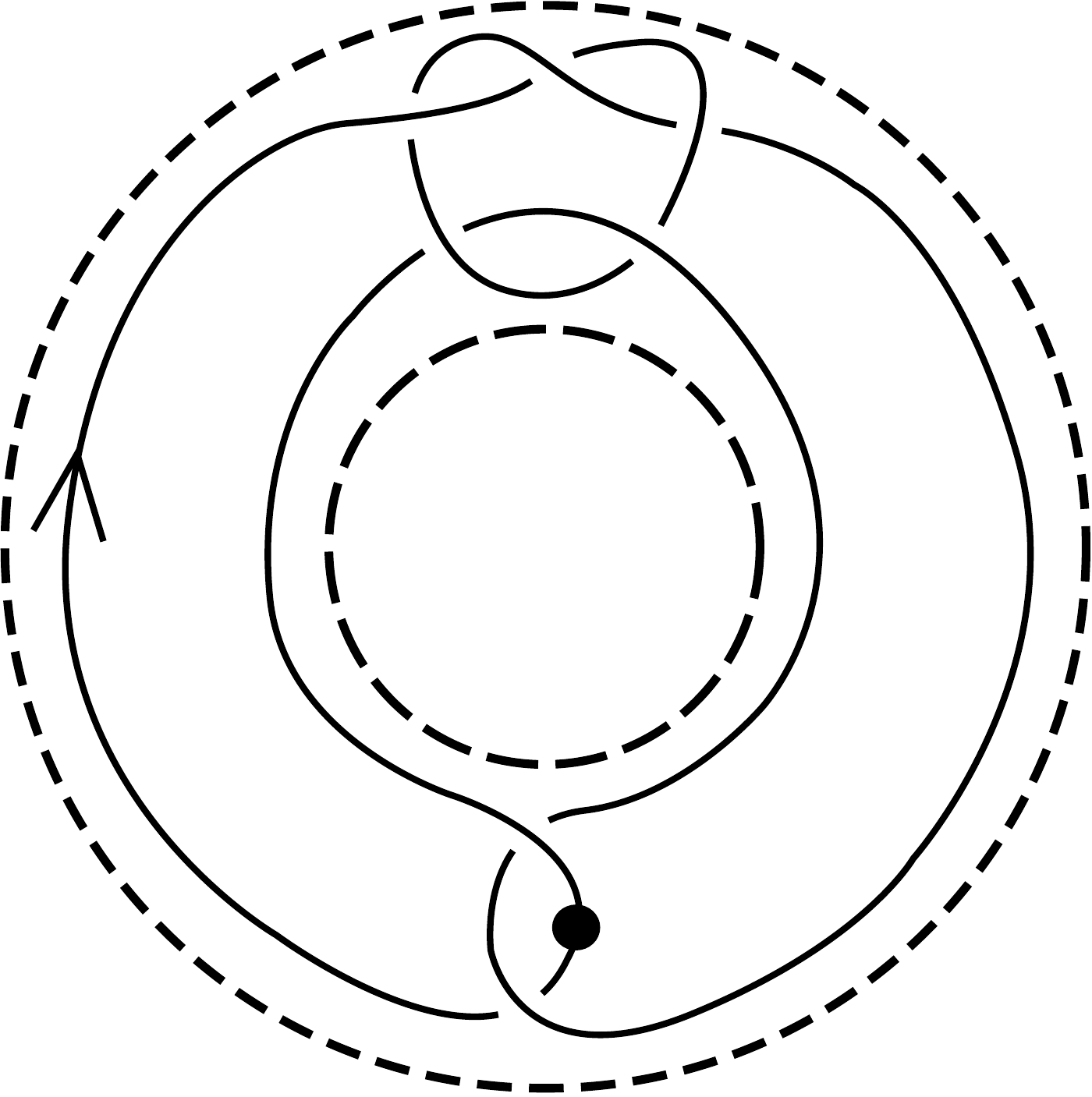}
\end{minipage}
\quad
\begin{minipage}[b]{0.45\linewidth}
  \includegraphics[scale=0.35]{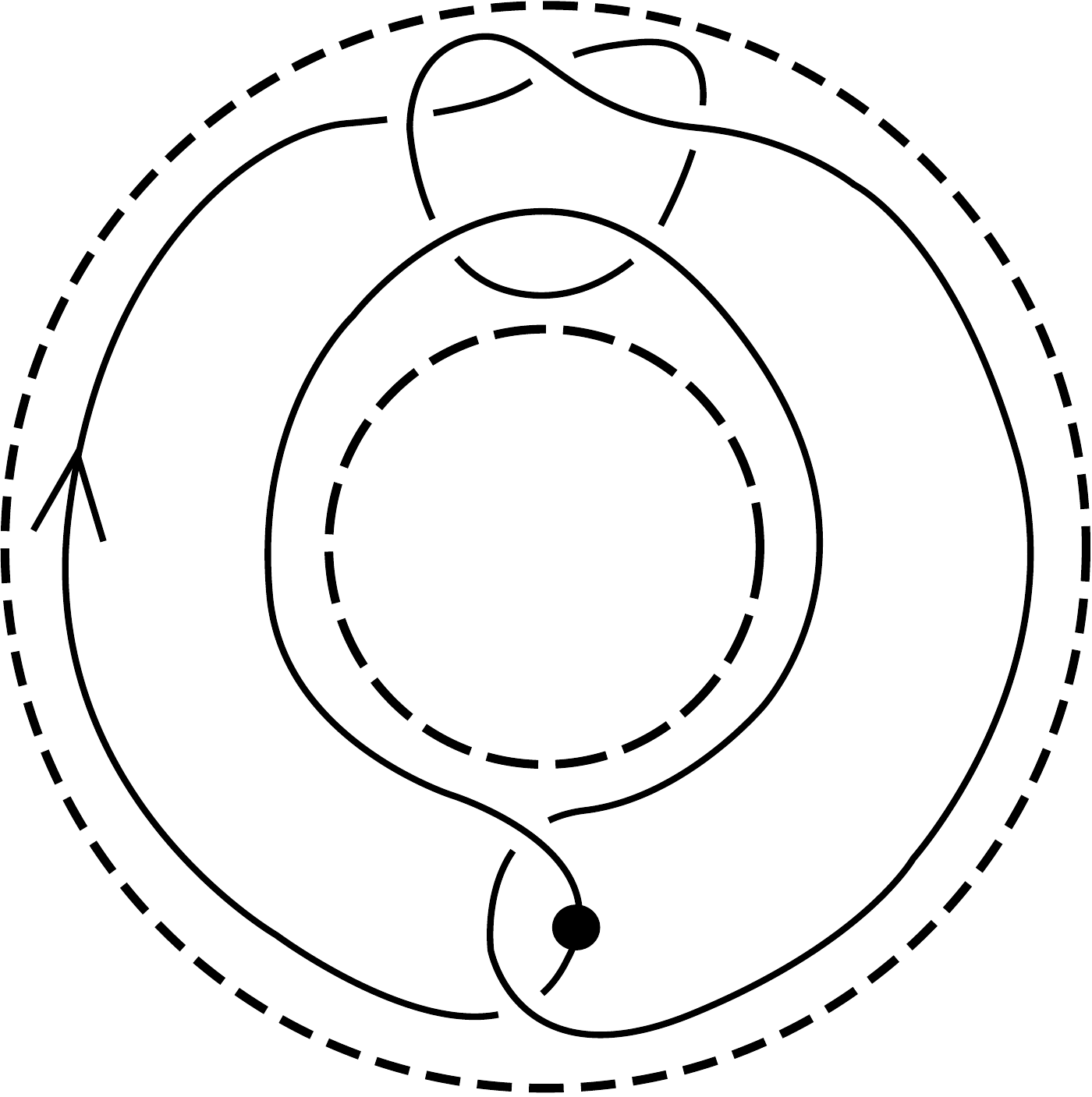}
\end{minipage}
\caption{A knot before and after crossing changes}
\label{annulus}
\end{figure}   
   
%
%
%

Note that the knot depicted on the right in Figure \ref{annulus} is indeed unknotted.  However, suppose that instead of the knots being in $\R^2 \times \R$ and projecting onto the first coordinate, we regard the knots as being in a thickened annulus $\Sigma \times \R$, where $\Sigma$ is the region between the dashed circles depicted.  Then the knot depicted on the right is {\em nontrivial} in $\Sigma \times \R$, in the sense that it does not bound a nonsingular disk in that space.  That is, the laying down the rope trick does {\em not} work in this setting.  Nevertheless, there is a method of changing crossings to trivialize knots in the more general setting of a ``thickened surface''  
$\Sigma \times \R$ and the obvious projection $p : \Sigma \times \R \to \Sigma \times \{0\}$.  If there is any hope of doing this, we must assume the knot is homotopically trivial, as changing the crossing can be realized by a homotopy and an unknot, which by definition bounds a disk, is contractible in  
$\Sigma \times \R$.

\begin{theorem}[Smythe \cite{smythe67}]  Suppose $\Sigma$ is a surface (orientable or not, and with or without boundary) and that $K$ is a knot in the interior of $\Sigma \times \R$ which projects to a curve in 
$\Sigma \times \{0\}$ with only transverse double points, otherwise nonsingular.  Further assume that $K$ is contractible in $\Sigma \times \R$.  Then one may change some of the crossings in the projection to obtain a knot $K'$ which is trivial, in the sense that it bounds a nonsingular disk in $\Sigma \times \R$. 
\end{theorem}

\begin{proof}  The case $\Sigma \cong S^2$ is clear: just find a point $x \in S^2$ so that $K$ is disjoint from the line $\{x\} \times \R$, and remove that line to bring the problem back to the standard 
$\R^2 \times \R$ case which is already known.  The case $\Sigma$ being the projective plane ${\R}P^2$ is dealt with separately in \cite{smythe67} and for simplicity we will ignore that case.  In case $\Sigma$ has boundary, remove the boundary and call the interior again $\Sigma.$

In all the remaining cases, the universal cover $\widetilde{\Sigma}$ is homeomorphic with the plane $\R^2$ and the fundamental group $\pi_1(\Sigma)$ is left-orderable, as we have seen in Chapter \ref{chapter free and surface} (in fact in most cases it is bi-orderable, but we will not need this).  Note that $\widetilde{\Sigma} \times \R$ is the universal cover of $\Sigma \times \R$.  Choosing a basepoint at a regular point of $K$ and orienting $K$, the knot $K$ lifts to a knot (rather than a path) in 
$\widetilde{\Sigma} \times \R$ because it represents a null-homotopic loop.  In fact, it lifts to infinitely many knots $\tilde{K}_u \subset \widetilde{\Sigma} \times \R$, parametrized by $u \in \pi_1(\Sigma)$, which we regard as the group of deck transformations of $\widetilde{\Sigma} \times \R$.  These lifts inherit basepoints and orientations from those of $K$.  See Figure \ref{covering-before} for an illustration.

\begin{figure}[h!]
 \begin{picture}(500,90)%
    \put(-50,0){\includegraphics[scale=.9]{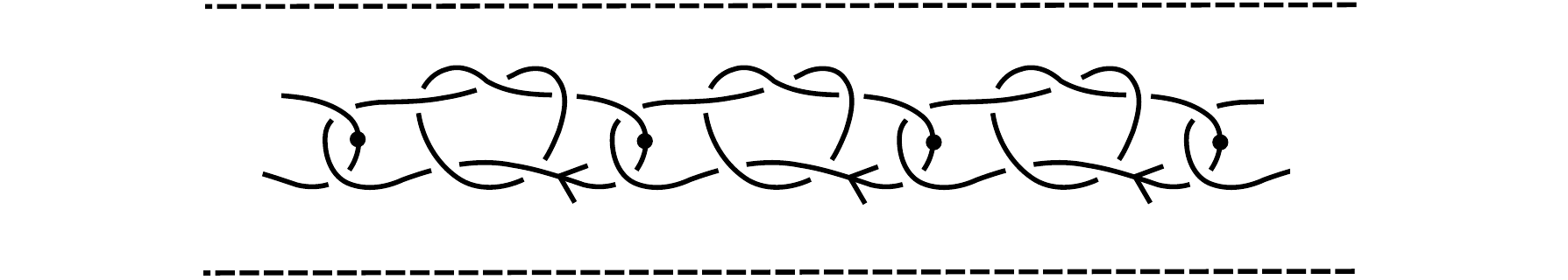}}%
    \put(90,10){$\widetilde{k}_{t^{-1}}$}%
 \put(175,10){$\widetilde{k}_{1}$}
 \put(260,10){$\widetilde{k}_{t}$}
  \end{picture}%
\caption{The lifts of $K$ in $\widehat{\Sigma} \times \R$, with $\pi_1(\Sigma) = \Z = \langle t \rangle$}
\label{covering-before}
\end{figure}

Note that under the projection $\widetilde{\Sigma} \times \R \to \widetilde{\Sigma} \times \{0\}$ the family 
$\tilde K_u, u \in \pi_1(\Sigma) $ projects to $\widetilde{\Sigma} \times \{0\}$ with simple double points, which are isolated, though infinite in number if there are any.   Now let $\prec$ be a left-invariant ordering of $\pi_1(\Sigma)$, and change the crossings of the projection of the $\tilde{K}_u$ according to the following algorithm:

Case 1.  If the crossing involves the projection of two distinct lifts $\tilde{K}_u$ and $\tilde{K}_v$, change it if necessary so that $\tilde{K}_v$ is above $\tilde{K}_u$ (\textit{i.e.}, has greater $\R$ coordinate) if and only if $u \prec v$.

Case 2.  If the crossing involves strands from the same lift, then make the strand lower if it is the first encounter of that crossing, when proceeding along the knot in the direction of the orientation, starting from the basepoint.
 
 This procedure is illustrated in Figure \ref{covering-after}.  It is easy to see, using the left-invariance of 
 $\prec$ that these crossing changes are equivariant with respect to the covering translations.  Now, letting $1$ denote, as usual, the identity element of $ \pi_1(\Sigma)$, we see that $\tilde{K}_1$ is above all the curves $\tilde{K}_u$ with 
 $u \prec 1$ and below those with $1 \prec u$.  The lifts of $K$ have become ``layered'' in 
 $\widetilde{\Sigma} \times \R$ because of the Case 1 moves.  In particular, by an isotopy of 
 $\widetilde{\Sigma} \times \R$ which preserves projection onto  $\widetilde{\Sigma}$ we can regard
 $\tilde{K}_1$ as lying in the slab
$\widetilde{\Sigma} \times (-1, 1)$ while all the other  $\tilde{K}_u$ are outside this slab, either above or below.  Moreover, $\tilde{K}_1$ has become unknotted in $\widetilde{\Sigma} \times (-1, 1)$, because Case 2 is just the ``laying of the rope algorithm'' for changing self-crossings of the projection of 
$\tilde{K}_1$ to $\widetilde{\Sigma} \times \{0\}$, recalling that $\widetilde{\Sigma} \cong \R^2$.  This is shown in Figure \ref{covering-after}.

\begin{figure}[h!]
\begin{picture}(500,90)%
    \put(-50,0){\includegraphics[scale=.9]{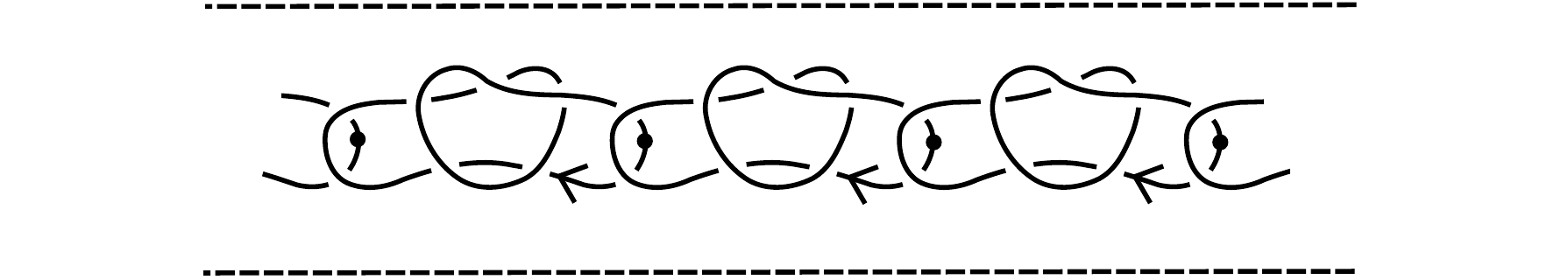}}%
  \end{picture}%
\caption{After applying the crossing-change algorithm.}
\label{covering-after}
\end{figure}

That is, $\tilde{K}_1$ bounds a nonsingular disk 
$D \subset \widetilde{\Sigma} \times (-1, 1)$.  By reversing the isotopy mentioned above, so that the lifts 
$\tilde{K}_u$ are again equivariant under covering translations, its image, which we'll again call $D$, remains a disk disjoint from all the $\tilde{K}_u, u \ne 1$.  Now consider the image $D'$ of $D$ under the covering projection $\widetilde{\Sigma} \times \R \to \Sigma \times \R$.  It may be that $D'$ is a singular disk, but since $D$ is disjoint from the lifts $\tilde{K}_u, u \ne 1$ all the singularities must be in the interior of $D'$ and $K$ is the nonsingular boundary of $D'$.  It follows from the theorem known as Dehn's lemma \cite{Papakyriakopoulos57} that $K$ bounds a nonsingular disk in $\Sigma \times \R$, so it is unknotted.
\end{proof}

\chapter{Three-dimensional manifolds}
\label{three manifolds chapter}

We saw in Chapter \ref{chapter free and surface} that surfaces, that is to say two-dimensional manifolds, have 
bi-orderable fundamental groups, with the exception of the Klein bottle (whose group is only left-orderable) and the real projective space.

Not surprisingly, in three dimensions there is a much greater variety of manifolds and also of groups which are their fundamental groups.  We recall that a 3-manifold \index{three-dimensional manifold} is a metric space each of whose points has a neighbourhood homeomorphic to $\R^3$, Euclidean 3-space, or else the closed half-space $\R^3_+ = \{(x, y, z) : z \ge 0\}$ for 3-manifolds with boundary.  We may assume our manifolds are smooth, meaning that there is a covering by charts homeomorphic with $\R^3$ or $\R^3_+$ such that the transition functions between charts are infinitely differentiable.  This enables us to talk about smooth surfaces in such a manifold and avoid pathologies such as the Alexander horned spheres and other ``wild'' examples.  We will also make the assumption that all the $3$-manifolds under consideration, other than knot and link complements, are compact.

A 3-manifold $M$ is {\em irreducible} \index{irreducible 3-manifold} if every smooth 2-sphere in $M$ bounds a 3-dimensional ball in $M$.  For example, we already saw in the proof of Theorem \ref{LI} that knot complements are irreducible.   There is a notion of connected sum of 3-manifolds, analogous to the case of surfaces, and a 3-manifold is called ``prime'' if it is not the connected sum of two manifolds, neither being $S^3$.  Irreducible 3-manifolds are prime, but there is one 3-manifold, namely $S^2 \times S^1$, which is prime but not irreducible.
It is a theorem of Milnor \cite{Milnor62} that each compact 3-manifold is uniquely (up to the ordering of the factors) expressible as a connected sum of prime manifolds.

\section{Ordering 3-manifold groups}

We will begin with a result of \cite{BRW05}, which might be considered a fundamental theorem of left-orderability of 3-manifold groups.

\begin{theorem} \label{fundamental}  Suppose $M$ is an orientable, irreducible 3-manifold, possibly noncompact and possibly with boundary.  Then $\pi_1(M)$ is left-orderable if and only if there exists a nontrivial left-orderable group $L$ and a surjective homomorphism $\phi: \pi_1(M) \to L$.
\end{theorem}

\begin{problem}  Consult the proof of Theorem \ref{LI} in Section \ref{section LI} and adapt it to prove 
Theorem \ref{fundamental}.  Show that a similar argument proves the next theorem as well.
\end{problem}

\begin{theorem}  If an irreducible 3-manifold $M$ has infinite first homology, then its fundamental group is locally indicable, hence left-orderable.
\end{theorem}

We can also show that some classes of reducible manifolds have locally indicable fundamental group, but for this we first need to to prepare a preliminary result.

\begin{problem}
\label{LI problem}
The goal of this exercise is to show that the free product of locally indicable groups is locally indicable, following the appendix of \cite{Higman40}.  First, use Problem \ref{LI extension} to show that the product $G \times H$ of two locally indicable groups is locally indicable.  Combine this results with the short exact sequence\footnote{That the kernel of this sequence is a free group follows from the Kurosh subgroup theorem.}
\[ 1 \rightarrow F \rightarrow G * H \rightarrow G \times H \rightarrow 1
\]
where $F$ is a free group, to show that if $G$ and $H$ are locally indicable, then so is $G * H$.
\end{problem}

\begin{corollary}  If $L$ is a link, that is a nonempty collection of disjoint knots in $\R^3$ or $S^3$, then the link group $\pi_1(S^3 \setminus L)$ is locally indicable. 
\end{corollary}

\begin{proof}
If $S^3 \setminus L$ is irreducible, the corollary follows immediately, since $H_1(S^3 \setminus L) \cong \Z^n$, where $n$ is the number of components of $L$.  If it is not irreducible, a classical argument provides 2-spheres separating $L$ into irreducible components, each of whose groups is therefore locally indicable.  But $\pi_1(S^3 \setminus L)$ is a free product of those groups, and so is also locally indicable by the previous problem.
\end{proof}

\section{Surgery}\label{section surgery}

Surgery is one of the key connections between knot theory and 3-manifolds.  It was pioneered by Max Dehn 
\cite{Dehn10} who used surgery to give an alternate construction of Poincar\'e's homology sphere.

\subsection{Surgery along a knot}
Consider a knot $K$ in the 3-sphere.  Surgery along $K$ consists of removing a tubular neighbourhood of $K$ and then sewing it back, but possibly by a different attaching map along the boundary.

Formally, we can think of a neighbourhood of $K$ as parametrized 
$$ N(K) \cong S^1 \times D^2,$$ the product of a unit circle and unit disk in $\C$.  Under this correspondence, the curve $S^1 \times \{0\}$ corresponds to $K$.  The curve $\lambda \cong S^1 \times \{1\}$ is called a parallel, or {\em longitude}.  It is disjoint from, but parallel to, $K$.  There is an ambiguity (choice of framing) in the choice of longitude, in the sense that it may wrap around the knot $K$ any  number of times.  We choose the parametrization so that 
$\lambda$ is homologically trivial in $S^3 \setminus K$ (this is called the `preferred' longitude).  It is characterized by specifying that the linking number \index{linking number}of $\lambda$ with $K$ be zero, where the linking number is calculated by orienting $\lambda$ and $K$ to be parallel.  A {\em meridian} is the curve 
$\mu \cong \{1\} \times S^1$.  Note that $\lambda$ and $\mu$ meet at a single point, corresponding to 
$\{1\} \times \{1\}$.  Since $\lambda$ and $\mu$ lie on the torus $\partial N(K)$, their classes in $\pi_1(S^3 \setminus K)$, which we will designate by the same symbols $\mu$ and $\lambda$, commute and serve as generators of the subgroup $\pi_1(\partial N(K))$ of 
$\pi_1(S^3 \setminus int(N(K))) \cong \pi_1(S^3 \setminus K)$.  It is a classical fact that if $K$ is a nontrivial knot, inclusion induces an {\em injective} homomorphism
$\pi_1(\partial N(K)) \to \pi_1(S^3 \setminus int(N(K))).$  It is convenient to orient 
$\mu$ so that its linking number with $K$ is $+1$.

If $J$ is any homotopically nontrivial simple closed curve on the torus $\partial N(K)$, we may form the space
$$M := (S^3 \setminus int(N(K))) \cup_\varphi (S^1 \times D^2),$$ where $\varphi: S^1 \times S^1 \to \partial N(K)$ is a homeomorphism that sends the meridian $\{1\} \times S^1$ of $S^1 \times D^2$ to $J$.  Such a map $\varphi$ exists for any such curve $J$.  The manifold $M$ depends, up to homeomorphism, only on the isotopy class of $J$ in $\partial N(K)$.

If $J$ is oriented, then there is an expression $[J] = \mu^p\lambda^q$ in $\pi_1(\partial N(K))$ where $p$ and $q$ are relatively prime integers (including the case $\{p, q\} = \{0, \pm 1\}$).  Reversing the signs of $p$ and $q$ simultaneously just changes the orientation of $J$, and does not affect the homeomorphism class of $M$.  So it is convenient to denote the choice of $J$ by the fraction $p/q \in \Q \cup \{ \infty \}$.  The surgery manifold $M$ is then specified as $p/q$-surgery along $K$, and denoted $M = S^3(K, p/q)$.  Note that in $M$, the curve $J$ bounds a disk, hence the class of $J$ in the knot group becomes trivial in $\pi_1(M)$.  In fact $\pi_1(M)$ can be computed from the knot group by killing $[J]$, that is, by adding the relation $[J] = 1$ to a presentation of the knot group.  The case $p/q = \pm1/0 = \infty$ corresponds to the `trivial' surgery, resulting in $S^3(K, \infty) \cong S^3$.

The homology of $M$ is also easily calculated.  Since $H_1(S^3 \setminus K)$ is infinite cyclic, generated by the class of a meridian, we can use a Mayer-Vietoris argument to conclude that for any knot $K$ we have 
$H_1(S^3(K, p/q)) \cong \Z/p\Z$.  Note that $S^3(K, p/q)$ is a {\em homology sphere} \index{homology sphere} (meaning that its first homology group is trivial) if and only if $p=1$.  Therefore, surgery on any knot will produce many examples of homology spheres.

\begin{figure}
\setlength{\unitlength}{6cm}
  \begin{picture}(1,1.06957193)%
    \put(0,0){\includegraphics[width=\unitlength]{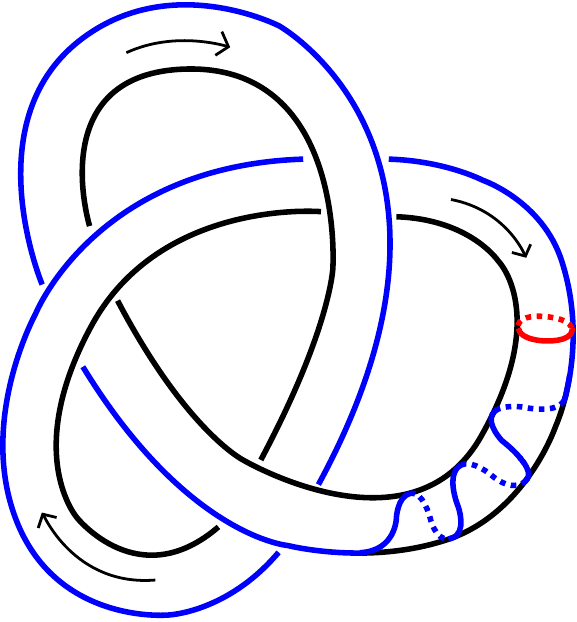}}%
\put(1.02,0.48){\huge{$\mu$}}
\put(0.97,0.25){\huge{$\lambda$}}
  \end{picture}%
\caption{A tubular neighbourhood of the right-handed trefoil, with meridian in red and preferred longitude in blue.}
\label{trefoil surgery}
\end{figure}

Dehn \cite{Dehn10} argued that $+1$ surgery on the right-handed trefoil yields the Poincar\'e dodecahedral space, \index{Poincar\'e dodecahedral space} which is a homology sphere with nontrivial, but finite, fundamental group (see Problem \ref{Dehn construction}).   

\begin{example} Let $M = S^3(K, -1)$ be constructed as $-1$ surgery on the right-handed trefoil $K$.  The framing curve $J$ looks just like the preferred longitude of Figure \ref{trefoil surgery}, but with one extra turn around $K$ so that the linking number of $J$ with $K$ is $-1$.  To work out $\pi_1(M)$ we take the knot group and kill the element corresponding to $J$.  Referring to Example \ref{trefoil group}, $J$ can be read off as  
$[J] = x^{-4}yzx = x^{-4}yxy^{-1}xy$.  So we have 
$$\pi_1(M) \cong \langle x, y \mid xyx = yxy,  x^4 = yxy^{-1}xy \rangle$$

We can simplify this presentation with the substitution $a = x$, $b = xy$, so $y = a^{-1}b$.  The first relation becomes
$ba = a^{-1}b^2$ or $aba = b^2$ or $(ab)^2 = b^3$.  The second becomes $a^4 = a^{-1}bab^{-1}ab$ or 
$a^6 = abab^{-1}ab$, which, in the presence of the relation $aba = b^2$ becomes $a^6 = bab$ or $a^7 = (ab)^2$.
So we have $$\pi_1(M) \cong \langle a, b \mid (ab)^2 = b^3 = a^7 \rangle.$$
\end{example}

\begin{problem}
\label{perfect check}
Check that this group is perfect, that is, it abelianizes to the trivial group.  Conclude that it is not indicable.
\end{problem}

By Problem \ref{perfect check} $M$ is a homology sphere.  In fact, $M$ is a \index{Brieskorn manifold} Brieskorn manifold, denoted $\Sigma(2,3,7)$.  In general, for positive integers $p, q, r$ one may define
$$\Sigma(p, q, r) = \{(u,v,w) \in \C^3 \mid  u^p + v^q + w^r = 0 \hspace{1em} \mbox{and} \hspace{1em} |u|^2 + |v|^2 + |w|^2 = 1\}.$$  In other words, it is the intersection of the unit 5-sphere in complex 3-space with the variety $u^p + v^q + w^r = 0.$ 

  Since 
$\pi_1(\Sigma(2, 3, 7))$ is not indicable (there is no map $\pi_1(\Sigma(2, 3, 7)) \rightarrow \mathbb{Z}$) it is certainly not locally indicable.  Nevertheless, Bergman \cite{Bergman91} gave an explicit embedding of this group in the group 
$\widetilde{{\rm PSL}(2, \R)}$, which is the universal cover of the group ${\rm PSL}(2, \R)$.  Now ${\rm PSL}(2, \R)$ acts on the circle $S^1$, for example by fractional linear transformations on $\R \cup \{ \infty \} \cong S^1$.  Moreover, as a space ${\rm PSL}(2, \R)$ has the homotopy type of $S^1$, so its universal cover 
$\widetilde{{\rm PSL}(2, \R)}$ is an infinite cyclic cover and is a group which acts effectively on the real line (see Example \ref{special linear groups}).  From Proposition \ref{LO_universal} we can conclude that 
$\widetilde{{\rm PSL}(2, \R)}$ is a left-orderable group and then conclude  

\begin{theorem} The group
 $\pi_1(\Sigma(2, 3, 7)) \cong  \langle a, b | (ab)^2 = b^3 = a^7 \rangle$ is a nontrivial left-orderable group which is not locally indicable.
\end{theorem}

A more general proof which includes this case appears in Theorem \ref{Seifert homology sphere}, noting that $\Sigma(2, 3, 7)$ is a Seifert fibred 3-manifold (see Section \ref{section sfs}).



\begin{problem}\label{Dehn construction}
Show that $+1$ surgery on the right-hand trefoil yields a manifold with fundamental group
$\langle a, b \mid (ab)^2 = b^3 = a^5 \rangle$.  
\end{problem}

This is a finite perfect group of order 120.  The resulting manifold is Dehn's construction of the Poincar\'{e} homology sphere, also known as $\Sigma(2, 3, 5)$.

\subsection{Surgery along a link}

Not all closed 3-manifolds arise as surgery on a knot.  In particular, the first homology group of such a manifold must be cyclic.  However, one may generalize the idea of surgery to a link $L = L_1 \cup \cdots \cup L_n$ of $n$ components.  One specificies coefficients $r_i = p_i/q_i$, one for each component $L_i$ and performs surgery on all the knots simultaneously, removing disjoint tubular neighbourhoods of the $L_i$ and sewing in solid tori $S^1 \times D^2$ according to the coefficients $r_i$ exactly as in the knot case.

\begin{theorem}[Lickorish - Wallace]
\label{lickorish wallace}
Every closed orientable 3-manifold arises as surgery along some link in $S^3$.
\end{theorem}

\begin{figure}
\setlength{\unitlength}{6cm}
 \begin{picture}(1,0.85)%
    \put(0,0){\includegraphics[width=\unitlength]{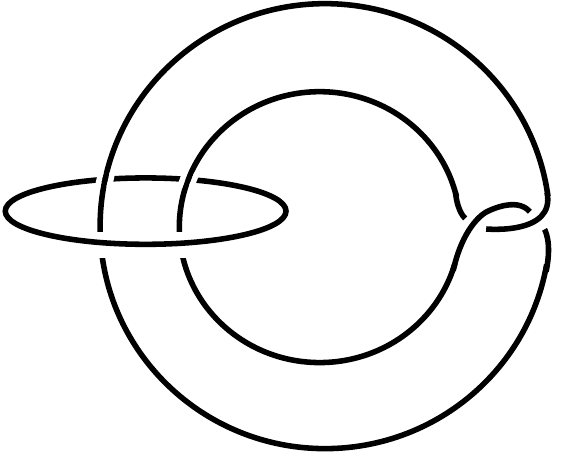}}%
    \put(1,0.3145598){\Large{$\cfrac{5}{2}$}}%
    \put(-0.00412026,0.29599183){\Large{$5$}}%
\end{picture}
\caption{Surgery description of the Weeks manifold.}
\label{weeks}
\end{figure}

The example of $\Sigma(2, 3, 7)$ in the previous section is a manifold whose geometry is modeled on $\widetilde{{\rm PSL}(2, \R)}$, one of the eight Thurston geometries in dimension three.  The most important of the geometries is hyperbolic.  A closed manifold is said to be {\em hyperbolic} \index{hyperbolic 3-manifold} if it has hyperbolic 3-space $\H^3$ as its universal cover, so that the deck transformations are hyperbolic isometries.  One can use this to define the {\em volume} of the manifold, and by Mostow rigidity the volume is actually a topological invariant.  According to Thurston, the set of volumes of hyperbolic 3-manifolds is a well-ordered set of real numbers, so in particular there is a smallest volume.

\begin{example}
\label{Weeks}  The Weeks manifold \index{Weeks manifold} $W$ described by surgery as in Figure \ref{weeks} is the closed hyperbolic 3-manifold of least volume \cite{GMM}.   Its fundamental group is
$$\pi_1(W) \cong \langle a, b \mid babab=ab^{-2}a, ababa=ba^{-2}b \rangle.$$
\end{example}

\begin{problem}
The abelianization of this group is isomorphic with $\Z / 5\Z \oplus \Z / 5\Z$.  
\end{problem}

So the group $\pi_1(W)$ is not perfect, but $W$ is a {\em rational} homology sphere, \index{rational homology sphere} in that its homology with rational coefficients agrees with that of $S^3$.  

\begin{theorem}[\cite{CD03}]
The fundamental group of the Weeks manifold $W$ is not left-orderable, though it is torsion-free.
\end{theorem}

It is torsion-free because the universal cover of $W$ is $\H^3$, which is contractible, so $\pi_1(W)$ has finite cohomological dimension.  This implies it cannot have torsion elements.  Here is the idea of the proof that $\pi_1(W)$ is not left-orderable.  The defining relations can be rewritten as $b^{-1}ab^{-2}a=(ab)^2=ba^{-2}ba^{-1}$ and $a^{-1}ba^{-2}b=(ba)^2=ab^{-2}ab^{-1}$.   If such a left-ordering were to exist, we may assume that $a>1$ and get a contradiction to each of the cases $b<1$, $a>b>1$ and $b>a>1$.  For example $a>b>1$ implies $a^{-1}b < 1$.  But then the relation $(ba)^2 = a^{-1}ba^{-2}b$ leads to a contradiction, as $(ba)^2$ is positive, whereas 
$a^{-1}ba^{-2}b = (a^{-1}b)a^{-1}(a^{-1}b)$ must be less than the identity, being a product of three negative elements.

\begin{problem}
Verify the remaining details of the argument that $\pi_1(W)$ cannot be left-ordered. 
\end{problem}


\section{Branched Coverings}

Another important application of knot theory to 3-dimensional manifolds is the construction of \index{branched coverings} branched coverings.  As with surgery, the branched cover construction over knots and links in $S^3$ gives rise to all closed orientable 3-manifolds.  Moreover, Hilden, Lozano and Montesinos \cite{MLM83} showed that there is a single knot in $S^3$ which produces all such 3-manifolds as branched covers of that single knot!  For this reason it is called a universal knot.  In fact the figure-eight knot $4_1$ is known to be universal \cite{HMM85}.

\subsection{Constructing branched covers}
The prototype of branched covers (in dimension 2) is the mapping of the plane $\R^2$, expressed in polar coordinates by $(r, \theta) \mapsto (r, n\theta)$, where $n$ is a fixed integer greater than 1.  If  the origin is removed, this mapping is a regular $n$-fold covering map 
$\R^2 \setminus \{0\} \to \R^2 \setminus \{0\}$ in the usual sense, but as a map
$\R^2 \to \R^2$ it has the origin as branch point, where it is one-to-one instead of $n$-to-one, as it is away from the origin.  The same formula provides a branched cover of the unit disk $D^2 \subset \mathbb{R}^2$ over itself.

In three dimensions we can define a branched cover of a solid torus $S^1 \times D^2$  over itself which will be useful in the following discussion.  For coordinates $\phi \in S^1$ and $(r, \theta) \in D^2$ define 
$$f(\phi,(r,\theta)) = (\phi,(r,n\theta)).$$  
Here the branch set is the central curve $S^1 \times \{0\}$ of the solid torus.  Again those branch points have a single preimage, whereas all other points have $n$ preimages.  

For simplicity, in this section we will only consider a special class of branched covers of a knot $K \subset S^3$, the $n$-fold \emph{cyclic} branched covers, which we may construct as follows.  Let $N \cong S^1 \times D^2$ be a tubular neighbourhood of $K$ in $S^3$, where $S^1 \times \{0\}$ corresponds to the knot $K$.  Let $X := S^3 \setminus int(M)$ denote the complement of the interior of $N$, sometimes called the ``exterior'' of the knot.  It differs from the knot complement in that it is compact and has a boundary, which is a torus.  However $\pi_1(X)$ is isomorphic with the knot group $\pi_1(S^3 \setminus K)$. Consider the Hurewicz homomorphism $\pi_1(X) \to H_1(X) \cong \Z$, which can also be thought of as the abelianization of the knot group.  Also, let $\Z \to \Z/{n\Z}$ be the mapping onto the finite cyclic group which is reduction modulo $n$.  Taking the composite of these homomorphisms defines the homomorphism $\pi_1(X) \to \Z/{n\Z}$.  The kernel of this map is a subgroup of 
$\pi_1(X)$ which is normal and has index $n$.  By standard covering space theory, there is a regular covering space $p: \tilde{X} \to X$ corresponding to this kernel.  Thus a loop in $X$ lifts to a loop (rather than just a path) in $\tilde{X}$ if and only if the homology class it represents is divisible by $n$; in other words, its linking number with $K$ is a multiple of $n$\index{linking number}.

Notice that the boundary $\partial X$ of $X$ is also the boundary $\partial N$ of $N$, and the lift 
$\widetilde{\partial N} = p^{-1}(\partial{N})$ is a torus $\widetilde{\partial N} \cong S^1 \times S^1$.  The preimage of a meridian on $\partial N$ (corresponding to $\{*\} \times S^1$) is a closed curve on $\widetilde{\partial N}$ which covers the meridian $n$ times.  We may identify the restriction $p|\widetilde{\partial N}$ with the map $f$ defined above,  restricted to the boundary $S^1 \times S^1$ of the solid torus $S^1 \times D^2$.  But now we can use $f$ to extend the covering map $p$ to a solid torus $\tilde{N} \cong S^1 \times D^2$ whose boundary is attached to $\tilde{X}$ along 
$\widetilde{\partial N}$ to define a branched covering $\tilde{X} \cup \tilde{N} \to X \cup N = S^3$.  Its branch set downstairs is exactly the knot $K$ and the manifold $$\Sigma_n(K) :=  \tilde{X} \cup \tilde{N}$$ is called the $n$-fold cyclic branched cover of $S^3$ branched along $K$.  

An interesting class of examples of cyclic branched covers are the Brieskorn manifolds \index{Brieskorn manifolds}   $\Sigma(p,q,r)$ which were discussed in Section \ref{section surgery}.  It was noted by Seifert
\cite[p. 412, Theorem 17]{ST80}
that $\Sigma(p,q,r)$ is the $r$ fold branched cover of $S^3$ branched over the $p,q$ torus knot $T_{p,q}$ that is, 
$\Sigma(p,q,r) \cong \Sigma_r(T_{p,q})$.  And by symmetry the same may be said for any permutation of $p,q,r$ in that formula!

Using more complicated, not necessarily regular, coverings of $X$ we can define other branched coverings of $S^3$ branched over $K$, whose branch sets upstairs may not be single curves, as is the case for the cyclic branched covers.  This is what makes some knots universal.  One can also define branched covers over links.  However, we will not consider these more sophisticated branched covers here.

\subsection{Double branched covers of alternating knots}  
A knot is called \emph{alternating} \index{alternating knot} if it has a planar diagram in which, as one moves along the knot in a fixed direction, the crossings encountered alternate as over- and under-crossings.  This is a large class of knots, but there are many non-alternating knots too.  A beautiful connection between branched coverings and orderability is the following theorem due to Boyer, Gordon and Watson \cite{BGW13}.

\begin{theorem}
\label{branched cover nonlo}
If $K$ is a nontrivial alternating knot in $S^3$, then the fundamental group $\pi_1(\Sigma_2(K))$ of the 2-fold branched cover of $S^3$ over $K$ is not left-orderable.
\end{theorem}

In the remainder of this section we'll discuss a particularly elegant proof of this due to Greene \cite{Greenepreprint}.  First some preliminaries.  Consider a planar diagram of a knot $K$ in which (as usual) we assume the only crossings are simple 2-fold crossings.  This diagram separates the plane into regions, which may be coloured black or white in a ``checkerboard'' manner.  Thus any arc of the diagram (away from the crossings) has a black region on one side and a white region on the other side, and at each crossing locally the regions are coloured as in Figure \ref{crossing_sign}, which illustrates the two possible configurations.  For example, we might take the black regions to be the ones consisting of points with the property that a curve transverse to the diagram, from that point to the unbounded region, meets the diagram in an odd number of points.

\begin{problem}
Verify that the construction outlined in the previous sentence does, in fact, produce a checkerboard colouring of the knot diagram.
\end{problem}

\begin{figure}
\setlength{\unitlength}{6cm}
 \begin{picture}(1,0.37628816)%
    \put(0,0){\includegraphics[width=\unitlength]{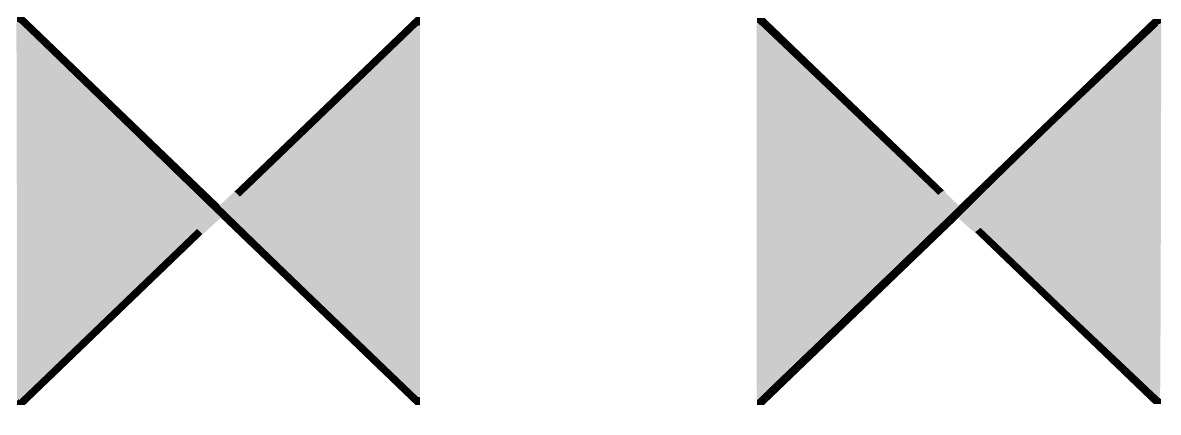}}%
    \put(0.10086978,0.0082987){$\epsilon=+1$}%
    \put(0.7294674,0.0082987){$\epsilon=-1$}%
  \end{picture}%
\caption{Signs assigned to crossings}
\label{crossing_sign}
\end{figure}

We now construct a graph $W = (V, E)$, the ``white graph'' as follows.  The vertices 
$V = \{v_0, \dots, v_p\}$ are the white regions, with a distinguished ``root'' vertex $v_0$ (say the unbounded region).  For each crossing of the diagram, there is an edge $(v,w) \in E$ connecting the two white regions that appear at that vertex. This edge is assigned a sign $\epsilon$ as in Figure \ref{crossing_sign}. Note that $W$ is a connected planar graph.
  
Form a group $\Gamma$ as follows.  It has one generator $x_v$ and one relation $r_v = 1$ for each $v \in V$. To describe the relation $r_v$, imagine a small loop going around the vertex $v$ in a counterclockwise direction, starting at an arbitrary point, and let $(v, w_1), \dots ,(v, w_k)$ be the edges encountered with respective signs $\epsilon_1, \dots, \epsilon_k$ in order.  Then 
$r_v = \prod_{i=1}^{k}(x_{w_i}^{-1}x_v)^{\epsilon_i}.$
We also introduce the relation $x_{v_0} = 1$.

\begin{proposition}[Greene \cite{Greenepreprint}]
The fundamental group of $\Sigma_2(K)$ is isomorphic to $\Gamma$.
\end{proposition}

\begin{problem}\label{equal signs}
Show that a knot diagram is alternating if and only if all the signs at the crossings are the same.
\end{problem}

We now assume the diagram of our knot $K$ is alternating and that $\Gamma \cong \pi_1(\Sigma_2(K))$ can be given a left-ordering $<$.  By Problem \ref{equal signs}, and taking the mirror image of an alternating diagram if necessary, which does not change 
$\pi_1(\Sigma_2(K))$, we may assume all the crossings have sign $\epsilon = +1$.  Choose a vertex $v$ for which $x_w \le x_v$ for all $w \in V$.  If $x_w = x_v$ for all $w \in V$, then from the relation $x_{v_0} = 1$ we conclude
$\pi_1(\Sigma_2(K))$ is the trivial group.  By a standard result of knot theory, this implies that $K$ is unknotted.

Thus if $K$ is not the trivial knot, there is some $w \in V$ for which  $x_w < x_v$.  Since $W$ is connected, we may assume $(v, w) \in E$.  In the relation corresponding to ~$v$ we have every term in the product
$r_v = \prod_{i=1}^{k}(x_{w_i}^{-1}x_v)$ is greater than or equal to the identity, and one is strictly greater, and we conclude that $r_v > 1$, contradicting that $r_v = 1$. \qed

\subsection{Heegaard-Floer homology and L-spaces} \index{Heegaard-Floer homology} \index{L-space}
One reason the above result is interesting has to do with a conjecture which is open at the time of this writing.
Ozsv\'ath and Szab\'o \cite{OS05} 
define an {\em L-space} to be a closed 3-manifold $M$ such that $H_1(M; \Q) = 0$ (that is, a rational homology sphere) and its Heegaard-Floer homology 
$\widehat{HF}(M)$ is a free abelian group of rank equal to $|H_1(M; \Z)|$, the simplest possible.  Lens spaces, and more generally 3-manifolds with finite fundamental group are examples of L-spaces, but there are also many L-spaces with infinite fundamental group.

\begin{proposition}[\cite{OS05double}, Proposition 3.3]
\label{branched cover ls}
Every 2-fold branched cover of an alternating knot in $S^3$ is an L-space.
\end{proposition}

Together, Theorem \ref{branched cover nonlo} and Proposition \ref{branched cover ls} provide evidence for the truth of the following conjecture.

\begin{conjecture}[Boyer-Gordon-Watson \cite{BGW13}]
An irreducible rational homology 3-sphere is an L-space if and only if its fundamental group is NOT left-orderable.
\end{conjecture}

\begin{problem}  Use the results of this section to show that the (3,7)-torus knot is not alternating.
\end{problem}

\subsection{Other branched covers and complete presentations}
The paper \cite{DPT05} has many examples of branched covers of $S^3$ whose fundamental groups are not left-orderable.  We will concentrate on just one family of examples to illustrate another technique for showing a group is not left-orderable.  It involves what the authors call a \emph{complete} \index{complete presentation}presentation. \index{complete presentation}

\begin{definition}  Consider the free group $F_n$ with generators $x_1, \dots, x_n$.

(1) Given a finite sequence of `signs' $\epsilon_1, \dots, \epsilon_n$, $\epsilon_i = \pm 1$ for all $i$, and a nonempty reduced word 
$w = x_{a_1}^{b_1}\cdots x_{a_m}^{b_m}$ in $F_n$, say that $w$ \emph{blocks} the sequence $\epsilon_1, \dots, \epsilon_n$ if either 
$\epsilon_{a_j}b_j > 0$ for all $j= 1, \dots, m$ or else $\epsilon_{a_j}b_j < 0$ for all $j= 1, \dots, m$.

(2)  A set $W$ of reduced words of $F_n$ is \emph{complete} if for each sequence $\epsilon_1, \dots, \epsilon_n$ of signs, there is a word in $W$ which blocks that sequence.

(3)  The presentation of a group with generators $x_1, \dots, x_n$ and relations $w = 1$ for $w \in W$ is \emph{complete} if the set $W$ is complete.
\end{definition}

\begin{problem} If a group has a complete presentation, then it is not left-orderable.
\end{problem}

A general method for computing the fundamental group of a branched cover is discussed in \cite{DPT05}.  We state one of the conclusions without proof.

\begin{proposition}[\cite{DPT05}] The fundamental group of the $n$-fold branched cover of $S^3$ over the figure-eight knot $4_1$ has the presentation  $$\pi_1(\Sigma_n(4_1)) \cong \langle x_1, x_2, \dots ,x_{2n} \mid  x_i = x_{i-1}^{-1}x_{i+1}, \; x_2x_4 \cdots x_{2n} = 1 \rangle $$
where $i = 1, 2, \dots ,2n$ and subscripts are taken modulo $2n$.
\end{proposition}

\begin{problem} Show that this is a complete presentation and conclude that $\pi_1(\Sigma_n(4_1))$ is not left-orderable for any $n$ greater than 1.
\end{problem}

\section{Bi-orderability and surgery}

\begin{theorem}
Suppose $K$ is a fibred knot in $S^3$ and nontrivial surgery on ~$K$ produces a 3-manifold $M$ whose fundamental group is bi-orderable.  Then the surgery must be longitudinal (that is, 0-framed) and  $\Delta_K(t)$ has a positive real root.  Moreover, $M$ fibres over $S^1$.
\end{theorem}

This is actually an easy application of Theorem \ref{AN}.  First, note that the surgery must be longitudinal for homological reasons.  If $X$ is the complement of the tubular neighbourhood of the knot $K$, the knot group $\pi_1(X)$ has preferred elements which generate the fundamental group of $\partial X$: meridian $\mu$, represented by a loop bounding a disk transverse to the knot, and longitude $\lambda$, which is parallel to the knot and homologically trivial in $X$.  A surgery framing curve $J$ is then represented by a pair of relatively prime integers $p$ and $q$, where
$$ [J] = \mu^p\lambda^q $$
and the surgery manifold $M$ has fundamental group obtained from $\pi_1(X)$ by killing 
$\mu^p\lambda^q $.  Similarly, a Mayer-Vietoris argument shows that $H_1(M)$ can be calculated 
from $H_1(X)$, which is infinite cyclic and generated by the meridian, by killing $\mu^p$.  So $H_1(M)$ is a {\em finite} cyclic group unless $p=0$.  

Now suppose $\pi_1(M)$ is bi-orderable.  It is a nontrivial group by the so-called property P theorem \index{property P}\cite{KM04}.  Recalling Theorem 
\ref {biorderable implies LI}
 and that $\pi_1(M)$ itself is finitely generated, there is a surjection $\pi_1(M) \to \Z$.   This cannot happen if $H_1(M)$ is finite.  It follows that $p$ must be zero, or in other words $[J] = \lambda$.  Since the preimages of points under the fibration map
$X \to S^1$ are longitudinal curves on $\partial X$, they bound meridian disks in $S^1 \times D^2$. The fibration map to $S^1$ thus extends to $S^1 \times D^2$ and we see that $M$ fibres over $S^1$.

The fibration of $M$ is essentially that of the knot exterior $X$: the fibres $\widehat{\Sigma}$ of the former are the 
fibres $\Sigma$ of the latter with disks sewn to the boundary.  The first homology of the fibres coincide in the two cases, and the same is true of the homology monodromy.  By Theorem \ref{poseval}, applied to the group $G = \pi_1(\widehat{\Sigma})$ we conclude that the homology monodromy has a positive eigenvalue.  But this is a root of the Alexander polynomial of the knot. \qed

\begin{theorem}
\label{L-surgery}
If surgery on a knot $K$ in $S^3$ results in an $L$-space, then the knot group 
$\pi_1(S^3 \setminus K)$ is not bi-orderable.
\end{theorem}

Of course, one may rephrase this by saying that the group of a knot is bi-orderable, then surgery on that knot never results in an L-space.   

Here is an outline the proof, one may consult \cite{CR12} for details.   
By Yi Ni \cite{Ni07} if surgery on $K$ yields an L-space, $K$ must be fibred.  Moreover, 
Ozsv\'ath and Szab\'o show that the Alexander polynomial of $K$ must have a special form if it admits and L-space surgery.  Then one argues that a polynomial of this form has no positive real roots, so the knot group cannot be bi-ordered.


\chapter{Foliations}

In this chapter we consider a natural way of constructing actions of a fundamental group on an ordered space, via a topological structure on the manifold called a \textit{foliation}\index{foliation}.   Loosely speaking, a foliation of a $n$-dimensional manifold $M$ is a decomposition of $M$ into lower dimensional manifolds, which we can strictly state as follows.

 We can partition the $n$-manifold $M$ into connected $k$-dimensional manifolds called \textit{leaves} for some positive $k<n$, and cover $M$ by charts $\phi: U \rightarrow \mathbb{R}^k \times \mathbb{R}^{n-k}$ such that each leaf $L \subset M$ satisfies 
\[ \phi(L \cap U) = \bigcup_{i \in I} \mathbb{R}^k \times \{ P_i \}
\]
where $\{P_i \}_{i \in I}$ is a collection of points in $\mathbb{R}^{n-k}$. In other words, each connected component of $L \cap U$ is a small bit of $\mathbb{R}^k$, called a plaque.  Overlapping plaques piece together to form maximally connected immersed submanifolds, these are the leaves.  The quantity $n-k$ is called the \textit{codimension} \index{codimension} of the foliation.


\section{Examples}
\begin{example}
\index{suspension foliation}
\label{suspension_foliation_example}
Given manifolds $B$ and $F$, let $\widetilde{B} \rightarrow B$ be a universal cover, and suppose we are given a representation $\rho: \pi_1(B) \rightarrow \mathrm{Homeo}(F)$.  Then one can construct a space $M$, which is a fibre bundle with base $B$ and fibre $F$, by using the action of $\pi_1(B)$ on $\widetilde{B} \times F$ given by 
\[ g \cdot (b, y) = (g (b), \rho(g)(y) ) .
\]
Here, the action of $g \in \pi_1(B)$ on $b \in \widetilde{B}$ is by deck transformations.  Then the quotient 
\[ M \cong ( \widetilde{B} \times F )/ \sim 
\]
inherits a foliation depending on the choice of $\rho$ whose leaves are the images of $\widetilde{B} \times \{y_0\}$.  Foliations constructed in this way are called \textit{suspension} foliations.

It is perhaps easiest to see this in a low-dimensional case, such as the torus, which is a (trivial) bundle $S^1 \hookrightarrow T^2 \rightarrow S^1$.  A representation $\rho : \pi_1(S^1) \rightarrow \mathrm{Homeo}(S^1)$ is determined by the image of a generator of $\pi_1(S^1)$, which is infinite cyclic.  Suppose that one of the generators of $\pi_1(S^1)$ is sent to $f : S^1 \rightarrow S^1$.

Then the suspension foliation is constructed as 
\[ (\mathbb{R} \times S^1) / \sim, \hspace{2em} (x, t) \sim (x+1, f(t))
\]
where we understand that the covering map $\mathbb{R} \rightarrow S^1$ is the quotient $\mathbb{R} \rightarrow \mathbb{R} / \mathbb{Z}$. The leaves are the images of $\mathbb{R} \times \{ t \}$ in the quotient.
 This can also be understood as the quotient
\[  ([0,1] \times S^1)/ \sim \hspace{2em} (0, t) \sim (1, f(t)).
\]
Written as above, the suspension of the map $f$ appears as in Figure \ref{torus_suspension}.

\begin{figure}
\setlength{\unitlength}{9cm}
 \begin{picture}(1,1.07475544)%
    \put(0,0){\includegraphics[width=\unitlength]{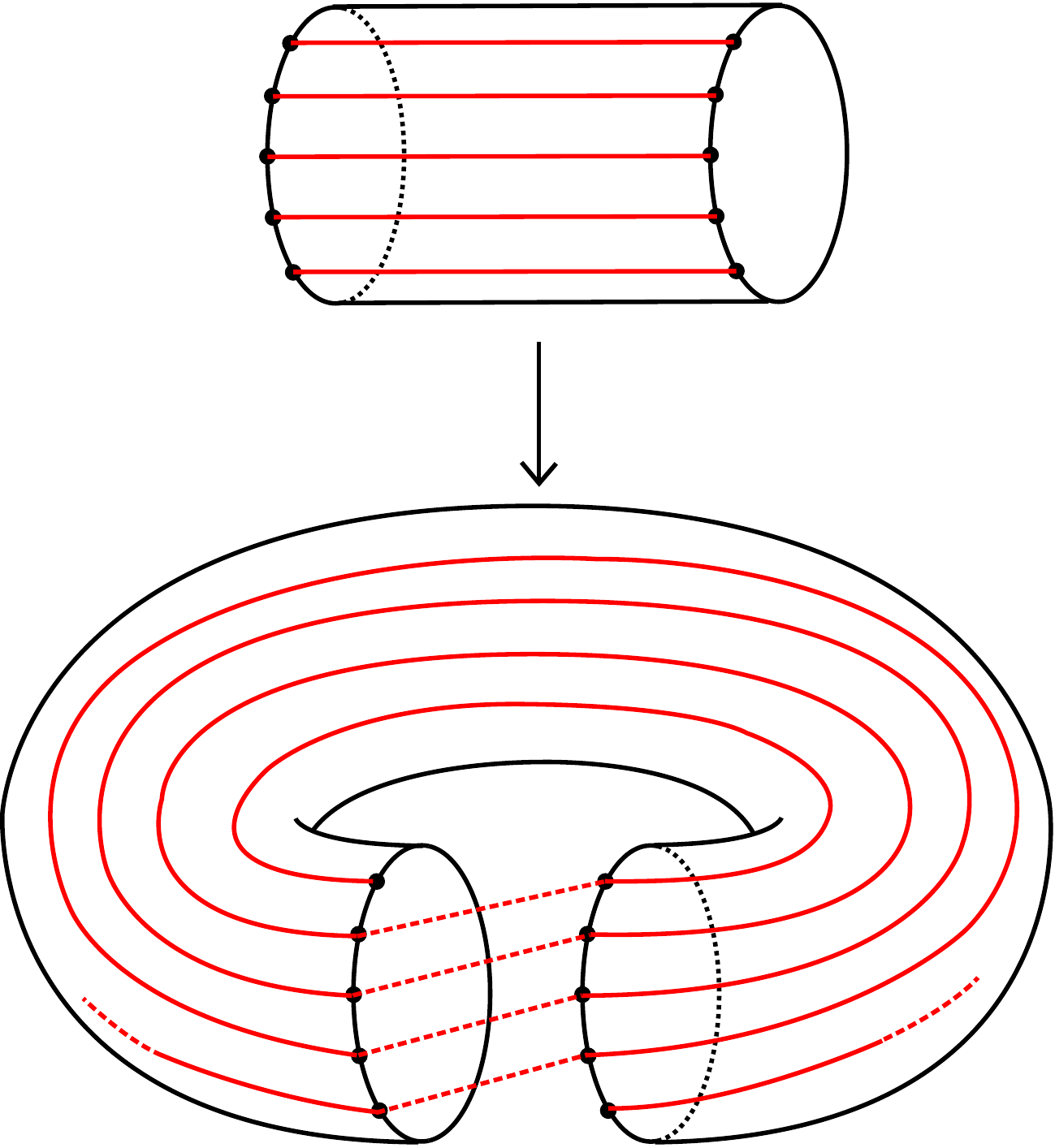}}%
    \put(0.17982127,0.802){$(0, t)$}%
    \put(0.11802056,0.857){$(0, f(t))$}%
    \put(0.09578306,0.912){$(0, f^2(t))$}%
    \put(0.0950021982,0.967){$(0, f^3(t))$}%
    \put(0.11038451,1.022){$(0, f^4(t))$}%
  \end{picture}%

\caption{A suspension foliation of the torus determined by a function $f:S^1 \rightarrow S^1$, with the gluing in the quotient indicated by dotted lines.  Highlighted in red is a portion of a leaf of the foliation, together with its preimage upstairs.}\label{torus_suspension}
\label{crossing sign}
\end{figure}


\end{example}

\begin{example}
\label{reeb annulus}
Consider the vertical strip $[-1, 1] \times \mathbb{R}$, with a foliation whose leaves are described as follows:  We first declare any set of the form 
\[ \left\{ (x, y) : y = \frac{1}{1-x^2}+ c, \mbox{ where } c \in \mathbb{R} \right\}
\]
to be a leaf of the foliation.  This covers the interior of $[-1, 1] \times \mathbb{R}$ by leaves, only two vertical lines remain, namely $x=-1$ and $x=1$.  Declare each vertical line to also be a leaf.  

From this template we can make a foliation of the plane by repeating the construction above on each vertical strip $[k, k+2] \times \mathbb{R}$ where $k \in \mathbb{Z}$ is odd.  The resulting foliation appears in Figure \ref{plane foliation}.

\begin{figure}[h!]
\includegraphics[scale=1.3]{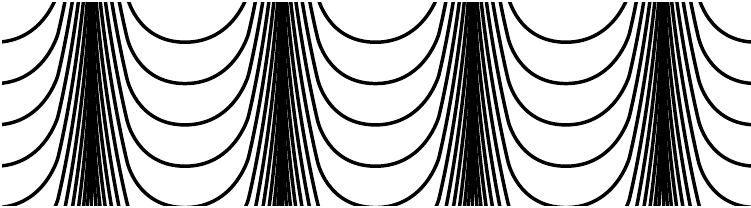}
\caption{A periodic foliation of the plane.}
\label{plane foliation}
\end{figure}

We also create a foliation of the annulus using this template.  Take the quotient map $[ -1, 1] \times \mathbb{R} \rightarrow [-1, 1] \times \mathbb{R} / \mathbb{Z}$ and the image of each leaf projects to a leaf in the quotient, providing us with a foliation of the annulus.
\end{example}

We cannot give many more interesting examples without increasing the dimension of our manifolds:  It is known that the only closed, compact 2-manifolds that admit a foliation (necessarily of codimension 1) are the torus and the Klein bottle \cite{Novikov64}. 

\begin{problem}
\index{Klein bottle} 
Describe the Klein bottle $K$ as a fibre bundle.  Show how your description can be used to construct many suspension foliations of $K$, specifically highlighting the difference between your construction and the construction indicated in Figure \ref{torus_suspension}.
\end{problem}

\begin{example}
\label{reeb_torus}
Consider the foliation of the vertical strip $[-1, 1] \times \mathbb{R}$ of Example \ref{reeb annulus}, and construct $D^2 \times \mathbb{R}$ as a solid of rotation by spinning the strip $[-1, 1] \times \mathbb{R}$ about the $y$-axis.  The resulting space $D^2 \times \mathbb{R}$ inherits a foliation, which descends to a foliation of the solid torus $D^2 \times S^1$ via the quotient map $D^2 \times \mathbb{R} \rightarrow D^2 \times \mathbb{R} / \mathbb{Z}$.  The resulting foliation is called a \textit{Reeb foliation} \index{Reeb foliation} of the solid torus, note that the boundary of the torus is a single leaf of the foliation.
\end{example}

\begin{example}
\label{reeb sphere}
The $3$-sphere can be decomposed as a union of two solid tori.  Explicitly we can take $S^3 \subset \mathbb{R}^4$ to be all the unit vectors, which we can write as \[  S^3 = \{ (x, y , z, w) : x^2 +y^2+z^2+w^2 =1 \}.
\]
Then the two solid tori are
\[ \{(x, y , z, w) \in S^3 :   w^2+x^2 \leq 1/2 \}  \mbox{ and } \{(x, y , z, w) \in S^3 :  y^2+z^2 \leq 1/2 \}
\]
with common boundary the so-called``Clifford torus" $$\{(x, y , z, w) : w^2 + x^2 = 1/2, y^2 +z^2 = 1/2 \}.$$
If each solid torus is given the Reeb foliation from Example \ref{reeb_torus}, then the two foliations piece together to give the Reeb foliation of $S^3$.  Note that we can glue the two foliations together and obtain a foliation of $S^3$, because in each Reeb foliated solid torus the boundary of the torus is a leaf.
\end{example}

\begin{example}
\label{seifert torus}
Here is our first example of a codimension two foliation.   A \textit{model Seifert fibring}  of the solid torus $D^2 \times S^1$ is a decomposition of $D^2 \times S^1$ into disjoint circles, called fibers.  The fibers are constructed by taking the solid torus $T^2$ and building it as
\[ T^2 = (D^2 \times [0,1]) / \sim
\]
where $\sim$ identifies $D^2 \times \{ 0 \}$ with $D^2 \times \{1 \}$ with a $\frac{2 \pi p}{q}$ twist for some $\frac{p}{q} \in \mathbb{Q}$ ($p,q,$ relatively prime positive integers).  Illustrated in Figure \ref{Seifert torus} is the case $q=6$. Then if $p=5$, for example, we would rotate the top by $5$ `clicks' and glue it to the bottom.
\begin{figure}[h!]
\setlength{\unitlength}{1.2cm}
\begin{picture}(3,4)
\put(0,0){
\includegraphics[scale=0.5]{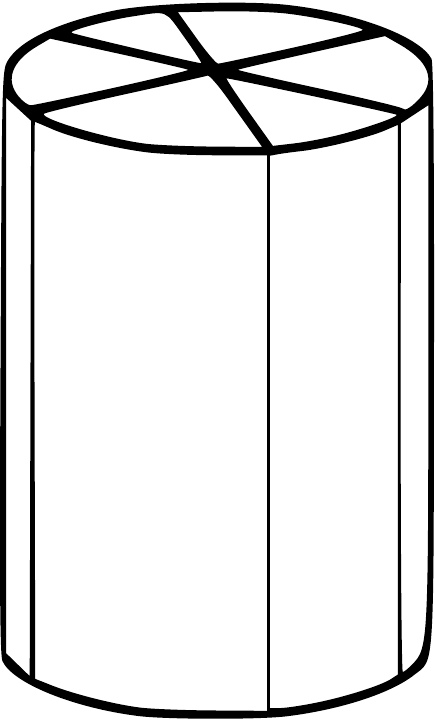}}
\put(0.15,3.25){\includegraphics[scale=0.5]{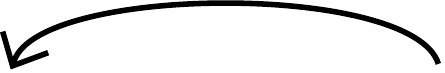}}
\put(2.25,0.15){\includegraphics[scale=0.5]{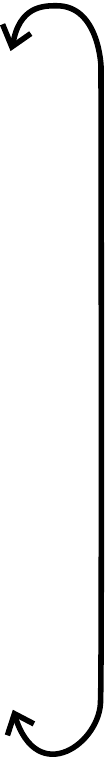}}
\put(0,3.8){Rotate $5$ clicks...}
\put(3,2){then glue top to bottom.}
\end{picture}
\caption{A model Seifert fibering}\label{Seifert torus}
\end{figure}
The leaves of $T^2$ are then built out of segments $\{x \} \times [0,1]$ and are of two kinds:
\begin{itemize}
\item The image of $\{0\} \times [0,1]$, called an exceptional fibre.
\item The image of $q$ equally spaced segments $\{x \} \times [0,1]$ which are glued together end to end, called a regular fibre.
\end{itemize}
This particular foliation of the solid torus will be important in the material to come.
\end{example}

\section{The leaf space}

Given a foliation $\mathcal{F}$ of a manifold $M$, one can construct a quotient space by declaring that two points in $M$ are equivalent if and only if they lie in the same leaf of $\mathcal{F}$.  In other words, you crush each leaf to a single point.  The resulting set is then equipped with the quotient topology and is called the \textit{leaf space} of the foliation $\mathcal{F}$, and it is denoted $M / \mathcal{F}$.

\begin{problem}
Describe the leaf space of a suspension foliation of the torus (Example \ref{suspension_foliation_example}), the Reeb foliation of the annulus (Example \ref{reeb annulus}), and the Reeb foliation of $S^3$ (Example \ref{reeb sphere}).
\end{problem}

\begin{problem}
\label{leaf space action}
Describe the leaf space of the foliation depicted in Figure \ref{plane foliation}.   Thinking of Figure \ref{plane foliation} as the universal cover of a torus $T = S^1 \times S^1$, where the covering map is given by  the quotient $(x, y) \sim (x+2, y+2)$, describe the action of $\pi_1(T)$ on Figure \ref{plane foliation} by deck transformations.  Show that the action of $\pi_1(T)$ by deck transformations on $\mathbb{R}^2$ descends to an action on the leaf space, and describe the action.
\end{problem}

Our interest in the leaf space stems from the situation we observed in Problem ~\ref{leaf space action}.   When $M$ admits a foliation $\mathcal{F}$, there is a corresponding pull-back foliation $\widetilde{\mathcal{F}}$ of the universal cover $\widetilde{M}$.  Then the image of any leaf of $\widetilde{\mathcal{F}}$ under a deck transformation of $\widetilde{M}$ will again be a leaf.  This means that the deck transformations, which are in 1-1 correspondence with elements of $\pi_1(M)$, will act on the leaf space $\widetilde{M} / \widetilde{\mathcal{F}}$.  This fact will be an essential ingredient in both the proof of Theorem \ref{SF theorem} and in Section ~\ref{R-covered section}.
%
%

\section{Seifert fibred spaces}\label{section sfs}

A \textit{Seifert fibring} \index{Seifert fibring} of a $3$-manifold $M$ is a decomposition of $M$ into circles such that the neighbourhood of each circle is fibre-preserving diffeomorphic to the neighbourhood of a fibre in some model Seifert fibring of the solid torus (Example \ref{seifert torus}).  A manifold is called Seifert fibred if it admits a Seifert fibring.  For ease of exposition we will stick to the case of orientable Seifert fibred manifolds. 

Since we have already introduced foliations, it is worth mentioning that Seifert fibred manifolds can also be defined as the $3$-manifolds admitting a codimension $2$ foliation whose leaves are circles.  The equivalence between this definition and our definition in terms of fibred solid torus neighbourhoods is due to a deep theorem of Epstein \cite{Epstein72}.

Here is a way to construct orientable Seifert fibered manifolds, following \cite{Hatchernotes}.  Let $\Sigma$ be a compact, connected, orientable surface with $m$ boundary components.  Choose disks $D_1, \ldots ,D_n \subset \mathrm{int}(\Sigma)$, and let 
\[ \Sigma' = \Sigma \setminus (\mathrm{int}(D_1) \cup \ldots \cup \mathrm{int}(D_n) ).
\]
Let $M'$ denote the manifold $S^1 \times \Sigma'$. 

For each $\partial D_i \subset \Sigma'$ there's a corresponding torus $T_i = S^1 \times \partial D_i \subset \partial M'$.  Then on each torus $T_i$ we fix curves that determine a basis of $\pi_1(T_i)$: our basis will be $[h_i^*] = [\{1\} \times \partial D_i]$ together with $[h]$, the class of a circle $S^1 \times \{ pt \}$ in the product $S^1 \times \Sigma'$.  With these bases, we have a correspondence between curves $\gamma: S^1 \hookrightarrow T_i$ and fractions $\frac{\beta_i}{\alpha_i} \in \mathbb{Q} \cup \{ \infty \}$ by representing the class each such curve relative to our basis:
\[ [\gamma] = \alpha_i [h_i^*] + \beta_i [h]
\]

Now we construct a Seifert fibered manifold $M$ over the surface $\Sigma$ by choosing $n$ reduced fractions $\frac{\beta_i}{\alpha_i} \subset \mathbb{Q}$, for $i =1, \ldots, n$.  Then to each $T_i \subset \partial M'$, attach $D^2 \times S^1$ by gluing $\partial D^2 \times \{ y \}$ to the curve $\alpha_i [h_i^*] + \beta_i [h]$ on $T_i$. Writing $g$ for the genus of $\Sigma$ and $m$ for the number of boundary components, we denote the resulting manifold by 
\[ M( + g, m ; \beta_1/ \alpha_1 , \ldots, \beta_n / \alpha_n)
\]
with $+$ in front of the $g$ to indicate $\Sigma$ is orientable (a minus sign is used for $\Sigma$ non-orientable). \footnote{There is another standard notation for Seifert fibred manifolds which includes an integer `b', which is an Euler class of a certain bundle.  We do not discuss the matter of Euler classes here.}

With some care, we can mimic the above construction when $\Sigma$ is non-orientable by replacing $M'$ with an orientable $S^1$-bundle over $\Sigma$, in which case we end up with a Seifert fibred manifold denoted by $ M( -g, m ; \beta_1/ \alpha_1 , \ldots, \beta_n / \alpha_n)$.  We refer the reader to \cite{Hatchernotes} for details of this case.

\begin{problem}
Thinking of the manifold $M = M( +g, m ; \beta_1/ \alpha_1 , \ldots, \beta_n / \alpha_n)$ as a $3$-manifold with a codimension two foliation $\mathcal{F}$ whose leaves are circles, describe the leaf space $M /\mathcal{F}$.
\end{problem}

This construction is a good way to understand orientable Seifert fibered manifolds, since they all arise in this way.
\begin{proposition} \cite[Proposition 2.1]{Hatchernotes}
\label{SF uniqueness}
Every orientable Seifert fibered manifold is diffeomorphic, via a map which preserves the fibres, to one of the manifolds $M(\pm g, m ; \beta_1/\alpha_1 , \ldots, \beta_n/\alpha_n)$ for some choice of $\Sigma$ and fractions $\beta_i/\alpha_i$.  Two models $M(\pm g, m ; \beta_1/\alpha_1 , \ldots, \beta_n/\alpha_n)$ and $M(\pm g, m ; \beta'_1/\alpha'_1, \ldots, \beta'_n/\alpha'_n)$ are orientation-preserving diffeomorphic if and only if $\beta_i/\alpha_i \cong \beta_i'/\alpha_i' \mbox{ mod }  1$ (up to permuting indices); and if $b=0$ then 
$$ \beta_1/\alpha_1 + \cdots + \beta_n /\alpha_n=  \beta_1'/\alpha_1' + \cdots + \beta_n'/\alpha_n' $$ is also required.
\end{proposition}

\begin{problem}
Use the Seifert-Van Kampen theorem to show that the fundamental group of a closed Seifert fibred manifold $M$ constructed from a closed, orientable surface $\Sigma$ (that is $g \geq 0$) as above is:
\[
\begin{split}
&\quad
\pi_1(M) = \langle a_1, b_1, \ldots, a_g, b_g, \gamma_1, \ldots, \gamma_n, h \mid \hspace{10em}  \\
&\quad \hspace{8em} h \mbox{ central }, \gamma_j^{\alpha_j} = h^{-\beta_j}, [a_1, b_1] \ldots [a_g , b_g] \gamma_1 \ldots \gamma_n = 1 \rangle.
\end{split}
\]
\end{problem}

\bigskip

When $g<0$ and $\Sigma$ is closed, the fundamental group turns out to be
\[
\begin{split}
&\quad
\pi_1(M) = \langle a_1, \ldots, a_{|g|}, \gamma_1, \ldots, \gamma_n, h \mid \hspace{10em}  \\
&\quad \hspace{4em} a_j ha_j^{-1} = h^{-1}, \gamma_j^{\alpha_j} = h^{-\beta_j}, \gamma_j h\gamma_j^{-1} = h,  a_1^2\ldots a_{|g|}^2\gamma_1 \ldots \gamma_n = 1 \rangle,
\end{split}
\]
and one can calculate that if $g \neq 0, -1$, then these groups have infinite abelianization.  The group $\pi_1(M)$ will also have infinite abelianization if the underlying surface $\Sigma$ has nonempty boundary, because then the boundary of $M$ consists of a union of tori (for a proof of this fact, just use the same argument as in the proof of Lemma \ref{rank lemma}).  Therefore  if either $g \neq 0, -1$ or $\partial \Sigma \neq \emptyset$, then $|H_1(M)| = \infty$ and in these cases left-orderability is dealt with by using Theorem  \ref{fundamental}.

\begin{theorem} \cite{BRW05}
If $M$ is an orientable closed Seifert fibered manifold and $|H_1(M)| = \infty$, then $\pi_1(M)$ is left-orderable.
\end{theorem}
\begin{proof}
Theorem \ref{fundamental} applies to all irreducible Seifert fibred manifolds.  There is only one reducible closed orientable case, $S^1 \times S^2$, and it has group $\Z$.
\end{proof}

Therefore when $M$ is an orientable Seifert fibred manifold, the question of whether or not $\pi_1(M)$ is left-orderable reduces to the case $|H_1(M)| < \infty$.  In this case $\Sigma$ must be either $S^2$ or $\mathbb{R}P^2$ and the groups in question become
\[
\pi_1(M) = \langle \gamma_1, \ldots, \gamma_n, h \mid  h \mbox{ central }, \gamma_j^{\alpha_j} = h^{-\beta_j}, \gamma_1 \ldots \gamma_n = 1 \rangle
\] 
for the $S^2$ case, and for the $\mathbb{R}P^2$ case:
\[
\pi_1(M) = \langle \gamma_1, \ldots, \gamma_n, y, h \mid 
y hy^{-1} = h^{-1}, \gamma_j^{\alpha_j} = h^{-\beta_j}, \gamma_j h\gamma_j^{-1} = h, y^2\gamma_1 \ldots \gamma_n = 1 \rangle.
\]

\begin{problem}
Show that when $\Sigma$ is $S^2$ or $\mathbb{R}P^2$, then the fundamental group $\pi_1(M)$ has finite abelianization. (Hint: If $\phi$ is the abelianization map, the relation $\phi(\gamma_j)^{\alpha_j} = \phi(h)^{\beta_j}$ means that $\phi(\gamma_j)$ and $\phi(h)$ can both be written as powers of the same element.)
\end{problem}

In fact, when $\Sigma = \mathbb{R}P^2$ this presentation allows us to conclude that $\pi_1(M)$ is NOT left-orderable via direction calculation, here is how:  Assume that $h>1$ in some left-ordering of $\pi_1(M)$.  Each of the relations $\gamma_j^{\alpha_j} = h^{-\beta_j}$ forces $h^{-\beta_i} < \gamma_i^{-1} < h^{\beta_i}$ if $\beta_i>0$ and $h^{\beta_i} < \gamma_i^{-1} < h^{-\beta_i}$ if $\beta_i<0$.  Since $\gamma_i$ and $h$ commute for all $i$, these inequalities multiply together as in Problem \ref{multineq} to give $h^{-k} < (\gamma_n^{-1} \ldots \gamma_1^{-1}) < h^k$ for some positive integer $k$.  Therefore $h^{-k} < y^2 < h^k$, and similarly $h^{-k} < y^{-2} < h^k$.

\begin{problem}
Show that $h^{-k} < y^2 < h^k$ and $h^{-k} < y^{-2} < h^k$ where $k>0$ and $h>1$ forces $yhy^{-1}$ to be positive, contradicting $yhy^{-1} = h^{-1} <1$.  Conclude that $\pi_1(M)$ is never left-orderable when $\Sigma = \mathbb{R}P^2$.
\end{problem}

What remains is the case $\Sigma = S^2$, and in this case we can characterize left-orderability of the fundamental group in entirely topological terms.

A foliation of a Seifert fibered manifold is called \textit{horizontal} \index{horizontal foliation} if regular fibers are transverse to the leaves, meaning that $[h]$ is the class of a curve that cuts transversely through the plaques in $M$. A foliation is \textit{co-orientable} \index{co-orientable foliation} if the leaves of the foliation admit a coherent choice of normal vector.  

\begin{theorem} \cite{BRW05}
\label{SF theorem}
Suppose that $M \ncong S^3$ is an orientable Seifert fibred manifold over the surface $\Sigma =S^2$, so that $|H_1(M) | < \infty$.  Then $\pi_1(M)$ is left-orderable if and only if $M$ admits a co-orientable horizontal foliation.
\end{theorem}
\begin{proof}
First we suppose that $\pi_1(M)$ is left-orderable, and we fix a corresponding dynamic realization $\hat{\rho}: \pi_1(M) \rightarrow \mathrm{Homeo}_+(\mathbb{R})$ as constructed in Section ~\ref{dynamicalsection}. 

\begin{problem}
 Using the relators $\gamma_j^{\alpha_j} = h^{-\beta_j}$, argue that $\hat{\rho}(h)(x_0) = x_0$ implies that $\hat{\rho}(g)(x_0) = x_0$ for all $g \in \pi_1(M)$.  Therefore $\hat{\rho}(h)$ must act without fixed points.   
 \end{problem}
 
 \begin{problem}
Show that any homeomorphism $f: \mathbb{R} \rightarrow \mathbb{R}$ that acts without fixed points must be conjugate to a translation by either $+1$ or $-1$ (Hint: If $f$ has no fixed points, then the intervals $[f^k(0), f^{k+1}(0)]$ partition $\mathbb{R}$.  Conjugate each interval $[f^k(0), f^{k+1}(0)]$ to the interval $[k, k+1]$ via an appropriate function $g_k$ and then piece together all of the $g_k$'s).
\end{problem}

By the previous problems, we assume without loss of generality that $\hat{\rho}(h)$ conjugates to $\rho(h)(x) =x+1$ and so $\hat{\rho}$ conjugates correspondingly to a representation $\rho: \pi_1(M) \rightarrow \mathrm{Homeo}_+(\mathbb{R})$.
Since $\Sigma = S^2$ we know that $h \in \pi_1(M)$ is central, therefore the image of $\rho$ is contained within a certain subgroup of $\mathrm{Homeo}_+(\mathbb{R})$, namely
\[ 
\rho: \pi_1(M) \widetilde{\rightarrow \mathrm{Homeo}_+(S^1)} = \{ f \in \mathrm{Homeo}_+(\mathbb{R}) \mid f(x+1) = f(x) +1 \}.
\]
There is a second homomorphism of interest, which we will call $\phi$, and it is the quotient map
\[ \phi: \pi_1(M) \rightarrow \pi_1(M) / \langle h \rangle.
\]
The group $\pi_1(M) / \langle h \rangle$ acts on a surface $X$, specifically $X$ is the ``universal orbifold cover'' of $S^2$ with cone points of order $\alpha_1, \ldots, \alpha_n$.  For details of this action and a construction of $X$, see \cite[pp. 423 and Lemma 3.2]{Scott83b}.

We then construct a manifold $\widehat{M}$ as a quotient
\[ \widehat{M} = (X \times \mathbb{R} )/ \sim
\] 
where $\sim$ is defined by $(x, t) \sim (\phi(g)(x), \rho(g)(t))$.  Then by construction, we get $\pi_1(M) \cong \pi_1(\widehat{M})$, and from this we can check that in fact $M$ and $\widehat{M}$ are homeomorphic (either by explicitly computing the fractions $\beta_i/\alpha_i$ corresponding to $\widehat{M}$ and then applying Proposition \ref{SF uniqueness}, or by using the fact that the fundamental group of a $3$-manifold determines the manifold, as long as it's not a lens space).  So, this is a way of constructing our original manifold $M$, and this construction makes it clear that $M$ admits a horizontal, co-orientable codimension one foliation:

The planes $X \times \{t\}$ descend to leaves as in the torus example, and the lines $\{ x \} \times \mathbb{R}$ descend to Seifert fibers.  The Seifert fibers are obviously transverse to the leaves, and provide a coherent choice of normal to the leaves as well.

Conversely, suppose that $M$ admits a horizontal co-orientable foliation $\mathcal{F}$.  Let $p : \widetilde{M} \rightarrow M$ be the universal cover, and $\widetilde{\mathcal{F}}$ the pullback foliation of $\widetilde{M}$.  The fiber $h$ in $M$ pulls back to $p^{-1}(h) \cong \mathbb{R}$, and every leaf $\widetilde{L} \subset \widetilde{M}$ intersects $p^{-1}(h)$ transversely exactly once.  Collapsing each leaf $\widetilde{L}$ to a point, we therefore get 
\[ \widetilde{M}/\widetilde{F} \cong \mathbb{R},
\]
a copy of the reals (see \cite{EHN81} for full details).  The action of $\pi_1(M)$ on $\widetilde{M}$ by deck transformations descends to an action on $\widetilde{M}/\widetilde{F}$, and ``co-orientable'' guarantees that the action will be order-preserving.  Thus we have a representation 
\[ \rho: \pi_1(M) \rightarrow \mathrm{Homeo}_+(\mathbb{R})
\]
and it follows that $\pi_1(M)$ is left-orderable by applying Theorem \ref{fundamental}.
\end{proof}

\begin{problem}
Verify the details of the proof above.  Consult \cite[Section 6]{BRW05} or \cite[Proposition 6.3]{BCpreprint} if you get stuck.
\end{problem}

The proof above sets up a `correspondence' between foliations and left-orderings in the case that $M$ is Seifert fibred.  It is worth noting that this correspondence also extends to the case of manifolds $M$ which are constructed by gluing together Seifert fibred manifolds along torus boundary components, known as graph manifolds \cite{BCpreprint}.  We conclude this section by considering the case of Seifert fibred homology spheres; we are grateful to Cameron Gordon for pointing this out.

\begin{theorem}\label{Seifert homology sphere} If $M$ is a Seifert fibred homology sphere other than the Poincar\'e homology sphere $\Sigma(2,3,5)$, then $\pi_1(M)$ is left-orderable.
\end{theorem}

\begin{proof}
The sphere $S^3$ is simply-connected, so we assume $M \ncong S^3$.
First we note that we must have $\Sigma = S^2$ and in the presentation
\[
\pi_1(M) = \langle \gamma_1, \ldots, \gamma_n, h \mid  h \mbox{ central }, \gamma_j^{\alpha_j} = h^{-\beta_j}, \gamma_1 \ldots \gamma_n = 1 \rangle
\] 
we have $n \ge 3$ and the integers $\alpha_i$ are pairwise relatively prime (see for example \cite[p. 404, Theorem 12]{ST80}).  Now consider the group 
\[\Delta(\a_1, \a_2, \a_3) = \langle x, y, z \mid x^{\a_1} = y^{\a_2} = z^{\a_3} = xyz = 1 \rangle, \] which is the so-called triangle group.  If the $\a_i \ge 2$ are pairwise relatively prime, this is an infinite group, with the exception of $\{ \a_1, \a_2, \a_3\} = \{2, 3, 5\}$, which corresponds to $M \cong \Sigma(2,3,5)$.  For all other cases, $\Delta(\a_1, \a_2, \a_3)$ is isomorphic to a group of orientation-preserving isometries of the hyperbolic plane $\H^2$, a subgroup of index 2 of the group of isometries generated by reflection in the sides of a hyperbolic triangle having angles $\pi/\a_1, \pi/\a_2, \pi/\a_3$.  The case $\Delta(2,3,7)$ is illustrated as the subgroup of $\mathrm{Isom}_+(H^2)$ preserving the tesselation illustrated for the Poincar\'{e} disk model of $\H^2$ in Figure \ref{237}. 

Now, considering the $\a_i$ in increasing order, there is a surjective homomorphism $\pi_1(M) \to \Delta(\a_1, \a_2, \a_3)$ obtained by mapping 
$$\gamma_1 \to x, \gamma_2 \to y, \gamma_3 \to z, \gamma_i \to 1 ( i>3),  h \to 1.$$
Recalling that $\mathrm{Homeo}_+(\H^2) \cong \mathrm{PSL}(2, \R)$, this defines a nontrivial representation
$\varphi : \pi_1(M) \to \mathrm{PSL}(2, \R)$.  Since the universal cover of $M$ is contractible, $M$ is a 
$K(G, 1)$ (Eilenberg-MacLane space) and so 
$H^2(\pi_1(M);\Z)) \cong H^2(M; \Z) \cong H_1(M; \Z) = 0.$  Therefore there is no obstruction to lifting $\varphi$ to
$\tilde{\varphi}: \pi_1(M) \to \widetilde{ \mathrm{PSL}(2, \R)}$.  As noted in Example \ref{special linear groups},
$\widetilde{ \mathrm{PSL}(2, \R)}$ can be considered a subgroup of $\mathrm{Homeo}_+(\R)$, and so $\pi_1(M)$ has a nontrivial left orderable quotient.  It follows from Theorem \ref{fundamental} that $\pi_1(M)$ is left-orderable.
\end{proof}

\begin{figure}
\includegraphics[scale=0.6]{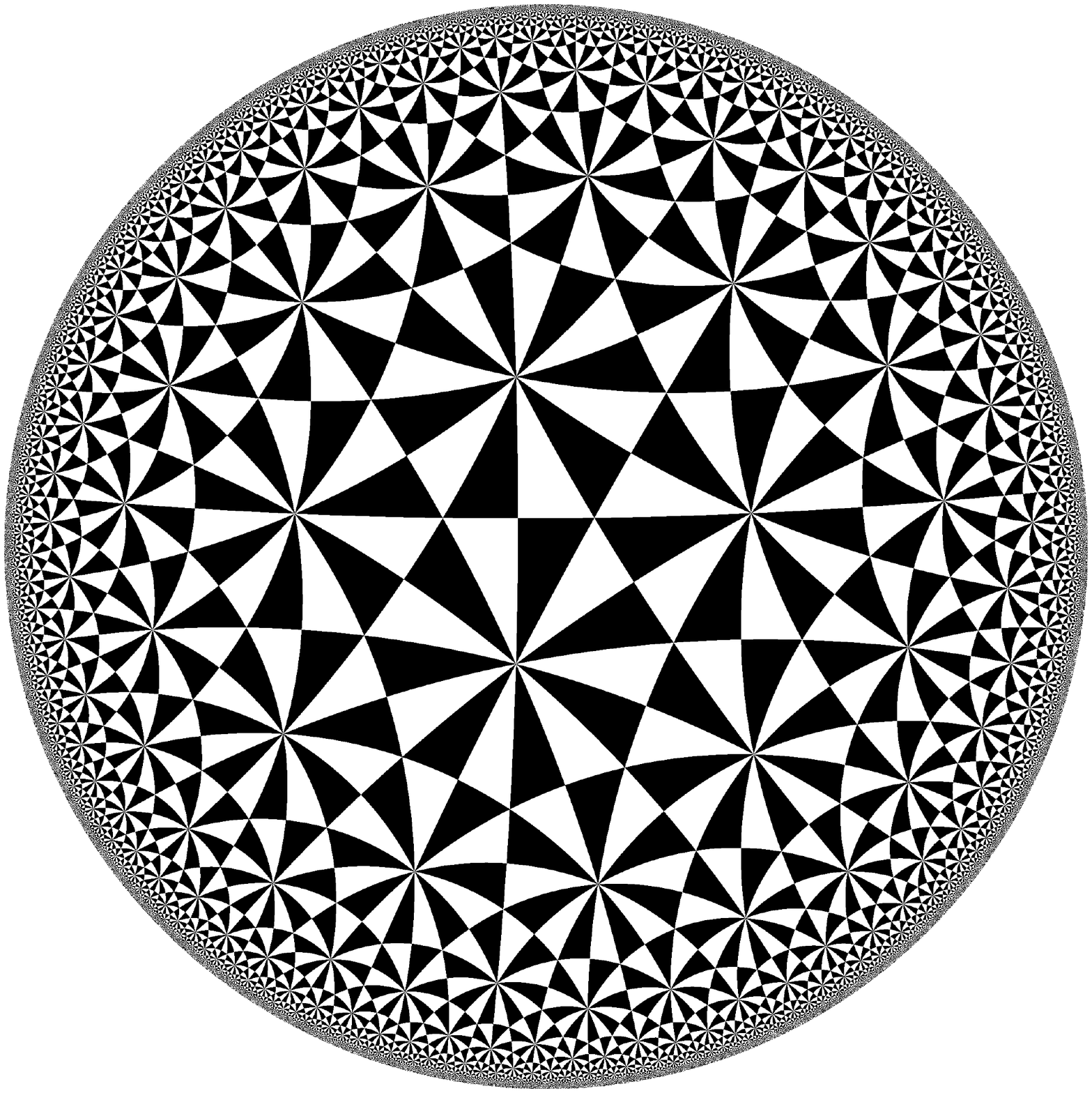}
\label{237}
\vskip-2cm
\caption{A tesselation of $\H^2$ corresponding to $\Delta(2,3,7)$.}
\end{figure}

\begin{corollary} If $M$ is a Seifert fibred homology sphere with infinite fundamental group, then $M$ has a co-orientable horizontal foliation. 
\end{corollary}

\section{$\mathbb{R}$-covered foliations}
\label{R-covered section}

As we saw in Theorem \ref{lickorish wallace}, every $3$-manifold can be constructed from the $3$-sphere by removing a family of tori and gluing them back differently \cite{Lickorish62, Lickorish63, Wallace60}.  By appropriately modifying the Reeb foliation of $S^3$ (Example \ref{reeb sphere}), and by equipping each re-glued torus with a Reeb foliation (Example \ref{reeb_torus}), it is possible to use the Lickorish-Wallace theorem to equip \textit{any} $3$-manifold with a codimension one foliation \cite{Lickorish65, Wood69}. Therefore we need our foliations to have a bit of added structure in order to make the question of their existence an interesting one (just as we asked for co-orientable and horizontal in the Seifert fibred case).  The property that one usually requires is that the codimension one foliation be \textit{taut}, meaning there is a transverse circle embedded in the $3$-manifold which intersects every leaf of the foliation.  Note that every horizontal foliation of a Seifert fibred 3-manifold is taut.

In this section, we will present a few examples of taut foliations that satisfy the stronger condition of being $\mathbb{R}$-covered: this means that the leaf space $\widetilde{M} /\widetilde{\mathcal{F}}$ is homeomorphic to $\mathbb{R}$ (here, $\widetilde{M}$ is the universal cover of $M$ and $\widetilde{\mathcal{F}}$ is the pullback foliation of $\widetilde{M}$).  This stronger condition is sufficient to guarantee left-orderability of the fundamental group when the foliation is also co-orientable.

\begin{example}
\index{Klein bottle} 
Consider the plane $\mathbb{R}^2$ and the isometries
\[ f(x, y) = (1+x, -y) \mbox{ and } g(x,y) = (x, 1+y).
\]
These maps generate a subgroup $H$ of the group of isometries of $\mathbb{R}^2$, and the quotient space $\mathbb{R}/H$ is homeomorphic to the Klein bottle $K$.  

There are two natural codimension one foliations $\widetilde{\mathcal{F}}_1$ and  $\widetilde{\mathcal{F}}_2$ of the plane $\mathbb{R}^2$, namely by horizontal and vertical lines.  Since the action of $H$ on $\mathbb{R}^2$ preserves both of these foliations, each descends to a foliation $\mathcal{F}_i$ of $K$ by circles.  Moreover the leaf space $\mathbb{R}^2/\widetilde{\mathcal{F}}_i$ is homeomorphic to $\mathbb{R}$, so each foliation $\mathcal{F}_i$ of $K$ is $\mathbb{R}$-covered.
\end{example}

\begin{problem}
Verify the details of the preceding example.  By identifying $\pi_1(K)$ with the group $H$, describe the action of $\pi_1(K)$ on $\mathbb{R}^2/\widetilde{\mathcal{F}}_i \cong \mathbb{R}$ that arises from each foliation $\mathcal{F}_i$.
\end{problem}

\begin{problem}
Set $S=  \{ (x, y, z) \mid -1 \leq z \leq 1 \}$ and define isometries $f,g : S \rightarrow S$ according to the formulas 
\[ f(x, y, z) = (1+x, -y, -z) \mbox{ and } g(x,y,z) = (x, 1+y, z).
\]
As in Section \ref{torsion free nonlo section}, let $H$ denote the subgroup of isometries of $S$ generated by $f$ and $g$.  Then set $K \tilde{\times} I := S /H$, which is the twisted $I$-bundle over the Klein bottle $K$ and has the same fundamental group as $K$, namely $\pi_1(K \tilde{\times} I) = \langle x, y, \mid xyx^{-1} = y^{-1} \rangle $.   The boundary of $K \tilde{\times} I$ is a torus $T$, and as generators of $\pi_1(T)$ we take $m= x^2$ and $l =y$.  Then take two copies of $K \tilde{\times} I$ and glue them together to create a new manifold $M$, using a homeomorphism $T \rightarrow T$ which induces the map
\[ \phi_* =  \left( \begin{array}{cc}
p & q  \\
r & s
 \end{array} \right)
\]
relative to the bases $\{ m, l \}$ on each copy of $T$.
The fundamental group of $M$ is the same group considered in Example \ref{glued Klein bottles}, it is:
\[ \pi_1(M) = \langle x_1, y_1, x_2, y_2 | x_1y_1x_1^{-1} = y_1^{-1}, x_2y_2x_2^{-1} = y_2^{-1}, x_1^{2p}y_1^q = x_2^2, x_1^{2r}y_1^s=y_2  \rangle 
\]

Show that the manifold $M$ admits a foliation $\mathcal{F}$ whose leaves are Klein bottles, except for one leaf which is a torus.  Show that $\mathcal{F}$ is $\mathbb{R}$-covered, but that the foliation is not co-orientable.
\end{problem}

The foliations constructed in the proof of Theorem \ref{SF theorem} were also $\mathbb{R}$-covered, and they are co-orientable as long as the base surface $\Sigma$ used in the construction of the Seifert fibred manifold $M$ is an orientable surface.

\section{The universal circle}

If a $3$-manifold $M$ has nonempty boundary and its boundary is a union of tori, then each boundary torus $T$ has an associated inclusion map $i: T \hookrightarrow M$.  If the boundary is \textit{incompressible}, then the inclusion of each boundary torus induces an injective homomorphism of fundamental groups $i_* : \pi_1(T) \rightarrow \pi_1(M)$, so that $\pi_1(M)$ contains a subgroup that is isomorphic to $\mathbb{Z} \oplus \mathbb{Z}$.  Subgroups that arise in this way are called \textit{peripheral} subgroups.  A $3$-manifold is called \textit{atoroidal} if its fundamental group $\pi_1(M)$ doesn't contain any $\mathbb{Z} \oplus \mathbb{Z}$ subgroups other than the peripheral ones.

It is a deep theorem of Thurston that co-orientable, taut foliations of atoroidal $3$-manifolds are connected to actions of their fundamental group on the circle $S^1$.

\begin{theorem}[Thurston's universal circle construction, \cite{CD03}]
\index{universal circle}
Let $M$ be a closed, irreducible, atoroidal rational homology $3$-sphere.  If $M$ admits a co-orientable taut foliation, then there is a nontrivial homomorphism $\rho :\pi_1(M) \rightarrow \mathrm{Homeo_+}(S^1)$.
\end{theorem}

While such homomorphisms can be used to introduce a \textit{circular ordering} of $\pi_1(M)$, a concept not studied in this book, in general they do not give left-orderings of $\pi_1(M)$ unless some additional conditions are met. 

Consider the covering group $\widetilde{\mathrm{Homeo_+}(S^1)}$, which can be identified as
\[\widetilde{\mathrm{Homeo_+}(S^1)} \cong \{ f \in \mathrm{Homeo}_+(\mathbb{R}) \mid f(x+1) = f(x) +1 \}   
\]
 There is a short exact sequence
\[ 1 \rightarrow \mathbb{Z} \rightarrow \widetilde{\mathrm{Homeo_+}(S^1)} \rightarrow \mathrm{Homeo}_+(S^1) \rightarrow 1
\]
where the copy of $\Z$ in $\widetilde{\mathrm{Homeo_+}(S^1)}$ is generated by the function $g(x) = x+1$.  So $\Z$ is central.

If $\rho :\pi_1(M) \rightarrow \mathrm{Homeo_+}(S^1)$ lifts to $\widetilde{\mathrm{Homeo_+}(S^1)}$ then, since 
$\widetilde{\mathrm{Homeo_+}(S^1)}$ is a subgroup of $\mathrm{Homeo}_+(\mathbb{R})$, every lift yields many left-orderings of $\pi_1(M)$ by applying Theorem  ~\ref{fundamental}.

Since $\widetilde{\mathrm{Homeo_+}(S^1)}$ is a central extension, the obstruction to that lifting is the Euler class of 
$\rho$, which is an element of the cohomology group $H^2(\pi_1(M); \Z)$.  One argues that $H^2(\pi_1(M); \Z)$ is isomorphic with $H_1(M;\Z)$, a finite group but nontrivial unless $M$ is a $\Z$-homology sphere.  Thus the Euler class may be nonzero, though it will have finite order.  Nevertheless, Calegari and Dunfield \cite{CD03} argue by direct construction that if one restricts $\rho$ to the commutator subgroup $[\pi_1(M), \pi_1(M)]$ then it lifts to 
$\widetilde{\mathrm{Homeo_+}(S^1)}$.  We refer the reader to \cite{CD03} for details.  This shows the following.

%
%

\begin{theorem} \cite{CD03}
\label{universal circle LO}
Let $M$ be a closed, irreducible, atoroidal rational homology $3$-sphere that admits a co-orientable, taut foliation.  Then the commutator subgroup $[\pi_1(M), \pi_1(M)]$ is left-orderable---in particular, $\pi_1(M)$ is virtually left-orderable.
\end{theorem}
\def\s{\sigma}
\chapter{Left-orderings of the braid groups}
\label{braids chapter}

No study of left-orderings and topology is complete without considering the braid groups, as they exhibit many deep connections with both subjects.  In the first two sections we define the standard left-ordering of $B_n$, known as the Dehornoy ordering, and consider whether or not the braid groups are locally indicable or bi-orderable as well.  In the third section we prove that the braid groups are left-orderable for $n \geq 3$ using hyperbolic geometry, and then finish the chapter by showing how left-orderings of the braid groups can be connected with knot theory.

First we recall one of the many definitions of the braid group $B_n$.  Consider $n$ points $\{ p_1, \ldots, p_n\}$
in the plane, which for concreteness we take to be evenly spaced along the $x$-axis inside the unit disk $\mathbb{D}$.  Let $\beta_i :[0,1] \rightarrow  \mathbb{D} \times [0,1]$ be a path that is the identity in the second coordinate,  satisfying $\beta_i(0) = (p_i, 0)$ and $\beta_i(1) = (p_j, 1)$ for some $j \in \{ 1, \ldots, n\}$ and intersecting each slice  $\mathbb{D} \times \{t \}$ exactly once.  An $n$-braid is a tuple of $n$ such paths
\[ \beta = ( \beta_1 (t) , \beta_2(t), \ldots, \beta_n(t))
\]
that do not cross one another, in the sense that $\beta_i(t) \neq \beta_j(t)$ whenever $i \neq j$.  Two $n$-braids are equivalent if one can be continuously deformed into the other through $n$-braids.  The braid group $B_n$ is then the collection of all equivalence classes, the group operation is concatenation of representative $n$-braids.   That is:

\[
(\beta\beta')_i(t) = \begin{cases} \beta_i(2t) &\mbox{if } 0 \le t \le 1/2 \\
\beta'_j(2t-1) & \mbox{if } 1/2 \le t \le 1. \end{cases} 
\]
In the formula above, the subscript $j$ of $\beta_j'$ is chosen so that $\beta_i(1) = (p_j, 1)$.  The identity is the braid consisting of $n$ straight lines $\beta_i(t) = (p_i, t)$ for all $t$ and the inverse of a braid $\beta$ is given by $\beta(1-t)$, its reflection.

Let $\sigma_i$ denote the equivalence class of an $n$-braid that consists of $n-2$ straight lines, except for the paths $\beta_i(t)$ and $\beta_{i+1}(t)$, which cross each other exactly once as in Figure \ref{braid_generator}. 
\begin{figure}[h!]
\includegraphics[scale=0.4]{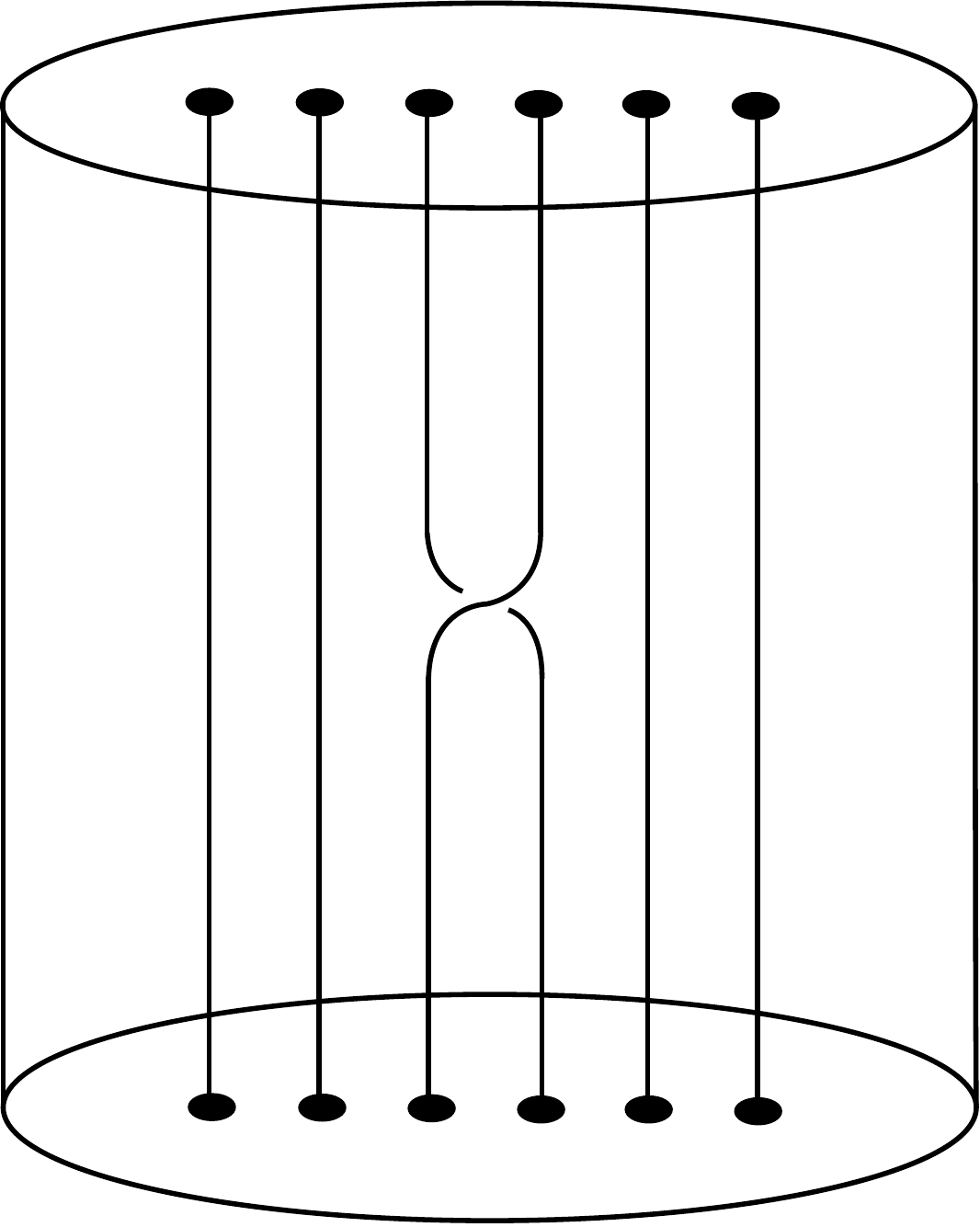}
\caption{The generator $\sigma_3$ in the braid group $B_6$.}
\label{braid_generator}
\end{figure}
As shown by E. Artin \cite{Artin25, Artin47} the braid group $B_n$ is generated by $\{ \sigma_1, \ldots, \sigma_{n-1} \}$, and it has presentation \index{braid groups}
\[ B_n = \left< \sigma_1 , \ldots, \sigma_{n-1}  \hspace{1em}
\begin{array}{|c}
\sigma_i \sigma_j = \sigma_j \sigma_i \mbox{ if $|i-j|>1$}\\
\sigma_i \sigma_j \sigma_i = \sigma_j \sigma_i \sigma_j \mbox{ if $|i-j|=1$} \end{array} 
 \right>
\]

Understood this way, every braid $\beta$ is an equivalence class of words in the generators $\sigma_i$.  

\begin{problem}\label{braid identities}
Verify the following identities in $B_n$, if $|i-j| = 1$:
$$\sigma_i \sigma_j \sigma_i^{-1} = \sigma_j^{-1} \sigma_i \sigma_j \hskip1cm
\sigma_i \sigma_j^{-1} \sigma_i^{-1} = \sigma_j^{-1} \sigma_i^{-1}\sigma_j$$
Show that the words $\sigma_1^2 \sigma_2^{-1} \sigma_1^{-1}$ and 
$\sigma_2^{-1}\sigma_1^{-1}\sigma_2^2$ represent the same braid in $B_3$.
\end{problem}

\begin{problem} Show that in $B_n$, all the generators $\sigma_i$ are conjugate to each other.
Show that the abelianization of $B_n, n \ge 2$ is isomorphic with $\Z$, and that the commutator subgroup $[B_n, B_n]$ is exactly the set of braids represented by words with total exponent sum zero in the generators $\sigma_i$.
\end{problem}

In addition to the abelianization homomorphism $B_n \rightarrow \mathbb{Z}$, there is also a homomorphism onto the symmetric group $S_n$.  Thinking of $S_n$ as a group of permutations of the points $\{ p_1, \ldots, p_n \}$, the map $B_n \rightarrow S_n$ associates a permutation to each braid by ``following the strands.''\footnote{This gives a homomorphism as long as we adopt the left-to-right convention for composing permutations, so that $(1, 2)(1, 3) = (1, 2, 3)$ and not $(1, 2)(1, 3)= (1, 3, 2)$, for example.}  So, for example, the image of the braid $\sigma_i$ is the transposition $(p_i, p_{i+1}) \in S_n$.  ,

\begin{problem}
Verify that map $B_n \rightarrow S_n$ described above is well-defined, and is indeed a homomorphism.
\end{problem}

\section{Orderability properties of the braid groups}

Before turning to left-orderability of the braid groups, we first determine precisely which braid groups are bi-orderable and which are locally indicable.

Our first observation is that the braid group $B_n$ is bi-orderable for  $n=2$, and not bi-orderable when $n>2$.  An easy way to prove this is by checking that the braid $\sigma_1 \sigma_2^{-1}$ is conjugate to its inverse. 

\begin{problem} Show that in $B_n, n > 2$ the following equation holds:
\[ (\sigma_1 \sigma_2 \sigma_1) (\sigma_1 \sigma_2^{-1})( \sigma_1 \sigma_2 \sigma_1)^{-1} = \sigma_2 \sigma_1^{-1}
\]
Verify that $\sigma_1 \sigma_2^{-1}$ and $\sigma_2 \sigma_1^{-1}$ are distinct braids (Hint:  Use the homomorphism $B_n \rightarrow S_n$).
\end{problem}

\begin{problem}
Verify that $B_n$ for $n>2$ does not have unique roots. (Hint: Consider $\sigma_1 \sigma_2$ and $\sigma_2 \sigma_1$, and verify that they are distinct by using the map $B_n \rightarrow S_n$).
\end{problem}

With $B_2$ being bi-orderable, we turn our attention to $B_n$ for $n>2$ and determine which of these groups are locally indicable.  We begin with a case we've  encountered before:

\begin{theorem}
The group $B_3$ is locally indicable.
\label{B_3 LI}
\end{theorem}

This is because $B_3$ is isomorphic with the trefoil knot group, according to the computation of Example \ref{trefoil group}.  So $B_3$ is locally indicable since, by Theorem \ref{LI}, knot groups are locally indicable.

%
%

For the 4-strand case
%
our arguments rely on \cite{GL69}, which provides a description of the commutator subgroup of $B_n$ for all $n$.  When $n=4$ the structure of the commutator subgroup will allow us to conclude that $B_4$ is locally indicable.

\begin{theorem}
The braid group $B_4$ is locally indicable.
\label{B_4 LI}
\end{theorem}
\begin{proof}

According to \cite{GL69}, the commutator subgroup of $B_4$ is isomorphic to a certain semidirect product $F_2 \ltimes F_2$.  In particular there's a short exact sequence
\[ 1 \rightarrow F_2 \rightarrow [B_4, B_4] \rightarrow F_2 \rightarrow 1.
\]
Since $F_2$ is bi-orderable, it is locally indicable by Theorem \ref{biorderable implies LI}, then by Problem \ref{LI extension} the group $[B_4, B_4]$ is locally indicable.  Thus from the short exact sequence
\[ 1 \rightarrow [B_4, B_4] \rightarrow B_4  \rightarrow \mathbb{Z} \rightarrow 1
\] and by another application of Problem \ref{LI extension}, $B_4$ is locally indicable as well.
\end{proof}

On the other hand, for $n \geq 5$ the structure of the commutator subgroup shows that $B_n$ is NOT locally indicable.

\begin{theorem}\label{braids not LI}
The braid group $B_n$ is not locally indicable for $n \geq 5$.
\end{theorem}
\begin{proof}
To prove that $B_5$ is not locally indicable, we will exhibit a finitely generated subgroup $H \subset B_5$ whose abelianization is trivial, so that there is no surjection $H \rightarrow \mathbb{Z}$.  Then since $B_5$ is a subgroup of $B_n$ for all $n >5$, it follows that $B_n$ is not locally indicable for $n \geq 5$.

Let $H$ be the subgroup of $B_5$ generated by the braids $\beta_1, \ldots, \beta_5$ where
\[ \beta_1 = \sigma_1^{-1} \sigma_2, \beta_2 = \sigma_2 \sigma_1^{-1}, \beta_3 = \sigma_1 \sigma_2 \sigma_1^{-2}, \beta_4 = \sigma_3 \sigma_1^{-1}, \beta_5 = \sigma_4 \sigma_1^{-1}.
\]
According to  \cite{GL69}, the subgroup $H$ is actually the commutator subgroup of $B_5$. 

One can verify that the following relations hold among the $\beta_i$'s:
\begin{enumerate}
\item $\beta_1 \beta_5 = \beta_5 \beta_2$
\item $\beta_2 \beta_5 = \beta_5 \beta_3$
\item $\beta_1 \beta_3 = \beta_2$
\item $\beta_1 \beta_4 \beta_3 = \beta_4 \beta_2 \beta_4$
\item $\beta_4 \beta_5 \beta_4 = \beta_5 \beta_4 \beta_5$
\end{enumerate}
Now we consider what happens upon abelianizing the subgroup $H$.  Relations (1) and (2) and (3) give $\beta_1 = \beta_2 = \beta_3 =1$.  Using this equality reduces  relation (4) to $\beta_4 = \beta_4^2$, so $\beta_4 = 1$.  Lastly relation (5) now easily gives $\beta_5 =1$, so $H / H' = \{ 1 \}$.
\end{proof}
\begin{problem} Verify the relations (1) - (5) above.
\end{problem} 
In light of the fact that $B_n$ is not locally indicable when $n \geq 5$, it might come as a surprise that $B_n$ for $n\geq2$ has a finite index subgroup that is bi-orderable.  This subgroup is the subgroup of \textit{pure braids}\index{pure braids}, denoted $P_n$.   Recalling our definition of a braid as an $n$-tuple of paths
\[ \beta = ( \beta_1 (t) , \beta_2(t), \ldots, \beta_n(t)),
\]
a pure braid is a braid $\beta$ for which $\beta_i(1)=(p_i, 1)$ for all $i$.  In other words, every strand in the braid starts and ends at the same point $p_i$.  The subgroup $P_n$ also arises as the kernel of the short exact sequence
\[ 1 \rightarrow P_n \rightarrow B_n \rightarrow S_n \rightarrow 1,
\]
and therefore $P_n$ is a subgroup of index $|S_n| = n!$ in $B_n$.  Here, $S_n$ is the group of permutations of $n$ elements.

\begin{theorem} \cite{RZ98, KR03}
The group $P_n$ is bi-orderable for $n \geq 2$.
\end{theorem}

The idea of the proof is a technique called Artin combing \cite{Artin47}, which is used to show that $P_n$ is an iterated semidirect product of free groups, which are bi-orderable, as shown in Theorem \ref{freeBO}.  We refer the reader to \cite{KR03} for details and a construction of an explicit bi-ordering.

\section{The Dehornoy ordering of $B_n$}
 \label{dehornoy introduction}
Next we'll see that $B_n$ is left-orderable for all $n$ --- in fact in uncountably many ways --- though there is a preferred `standard' ordering of $B_n$ which has a number of pleasing algebraic and combinatorial properties.  Moreover, like the Magnus ordering of the free groups the standard ordering of $B_n$ is computable, in the sense that one can decide via calculation which of two braids is bigger than the other.

The \textit{Dehornoy ordering} \index{Dehornoy ordering} (also known as the standard ordering, or the $\sigma$-ordering) of the braid group $B_n$ is a left-ordering that is defined in terms of representative words of braids as follows: A word $w$ in the generators $\sigma_1 ,\ldots, \sigma_{n-1}$ is called $i$-positive (respectively $i$-negative) if $w$ contains at least one occurence of $\sigma_i$, no occurences of $\sigma_1, \ldots, \sigma_{i-1}$, and every occurence of $\sigma_i$ has positive (respectively negative) exponent.  A braid $\beta \in B_n$ is called $i$-positive (respectively $i$-negative) if it admits a representative word $w$ in the generators $\sigma_1, \ldots, \sigma_{n-1}$ that is $i$-positive (respectively $i$-negative).
The Dehornoy ordering of the braid group $B_n$ is the ordering whose positive cone $P_D$ is the set of all braids $\beta \in B_n$ that are $i$-positive for some $i$.   Braids $\beta, \beta'$ satisfy $\beta <_D \beta'$ if and only if $\beta^{-1} \beta' \in P_D$.

\begin{example}
The word $\sigma_1^2 \sigma_2^{-1} \sigma_1^{-1}$ is neither $i$-positive nor  $i$-negative for any $i$, so the braid that it represents may be either positive or negative in the Dehornoy ordering.  But by Problem \ref{braid identities} the braid is also represented by the word
$\sigma_2^{-1}\sigma_1^{-1}\sigma_2^2$, which is $1$-negative.  Therefore it is negative in the Dehornoy ordering, and we write $\sigma_1^2 \sigma_2^{-1} \sigma_1^{-1}<_D 1$.
\end{example}

\begin{problem}
Show that the smallest positive element of the Dehornoy ordering of $B_{n}$ is $\sigma_{n-1}$.
\end{problem}

To show that the positive cone of the Dehornoy ordering is well-defined is a nontrivial task, so we will only cover the case of $B_3$ here; moreover there is already an entire book on the topic \cite{DDRW08}.  However we will provide a proof in the coming sections that $B_n$ is left-orderable for $n \geq 3$ by using hyperbolic geometry.  

That $P_D$ is well-defined and is the positive cone of a left-ordering in the case $n=3$ will follow from the next theorem and problem.

\begin{theorem}
\label{dehornoy in b3}
Every braid $\beta \in B_3$ is either a power of $\sigma_2$, or admits a $1$-positive representative word, or admits a $1$-negative representative word.
\end{theorem}
\begin{proof}  This argument is due to A. Navas.
We can rewrite the braid group $B_3 = \langle \sigma_1, \sigma_2 \mid \sigma_1 \sigma_2 \sigma_1 = \sigma_2 \sigma_1 \sigma_2 \rangle$ as $\langle a, b \mid ba^2b = a \rangle $ by making the substitution $a = \sigma_1 \sigma_2$ and $b = \sigma_2^{-1}$.  In this new notation, we can prove the theorem by showing that every word $w$ in $a$ and $b$ which is not a power of $b$ is equivalent to a word containing only positive (or only negative) powers of $a$.  

First note that $a^3$ is central, since 
\[ ba^3 = (ba^2)a = (ab^{-1})a = a(b^{-1}a) = a(a^2b) = a^3b
\]
Therefore, by using the identity $b^{-1}a^3 = a^2 b a^2$ and inserting powers of $a^3$ as necessary, we can assume that $w = w' a^{3 \ell}$, where $w'$ consists of only positive powers of $a$ and $b$.  Moreover, we can arrange for all powers of $a$ appearing in $w'$ to have exponent less than or equal to two.  

Next we can use $ba^2b = a$ to replace the leftmost occurence of $ba^2b$ in $w'$, and then shuffle any resulting powers of $a^3$ to the right.  By iterating this process, we can assume that
\[ w' = a^{r_1} b^{k_1} ab^{k_2} \ldots a b^{k_n} a^{r_2} \]
where $0 \leq r_j \leq 2$ for $j=1,2$ and $k_i \geq 1$ for all $i$.

With these assumptions we consider $w = w' a^{3 \ell}$.  If $\ell \geq 0$ we are done since every generator occurs with only positive exponents, so assume $\ell <0$.  In this case $w' a^{-3}$ in fact admits a representative word containing only negative powers of $a$, which shows that $w=w' a^{3 \ell}$ does as well. This follows from applying the identity $ab^ka^{-1} = (a^{-1}b^{-1})^k$ as follows
\[ w' a^{-3} = a^{r_1} b^{k_1} a \ldots b^{k_{n-1}}(a b^{k_n} a^{-1}) a^{r_2-2} = a^{r_1} b^{k_1} a \ldots a b^{k_{n-1}}(a^{-1}b^{-1})^{k_n} a^{r_2-2} 
\]
and observing that the resulting word has $ab^{k_{n-1}}a^{-1}$ as a subword. So, we are in a position to apply the same identity again.  Since all the $k_i$'s are positive, iterating the application of this identity yields a word with only negative occurences of $a$.
\end{proof}

This theorem shows that the subset $P_D$ of all $i$-positive braids of $B_3$ satisfies $B_3 = P_D \cup P_D^{-1}$.  The next problem completes the proof that $P_D$ is the positive cone of a left-ordering, by showing that $1 \notin P_D$, and so $P_D \cap P_D^{-1} = \emptyset$.

\begin{problem} 
Let $F_3$ denote the free group with generators $\{ x_1, x_2, x_3 \}$, and let $\hat \sigma_i : F_3 \rightarrow F_3$ for $i=1,2$ denote the automorphism 
\[\hat \sigma_i(x_j) =
\left\{
	\begin{array}{ll}
		x_i^2 x_{i+1}  & \mbox{if } j=i \\
		x_{i+1}^{-1}x_i^{-1}x_{i+1} & \mbox{if } j=i+1 \\
		x_j & \mbox{if } j \neq i, i+1
	\end{array}
\right.
\]
There is an injective homomorphism $B_3 \rightarrow \hbox{Aut}(F_3)$ defined by $\sigma_i \mapsto \hat\sigma_i$ \cite{wada92, Shpilrain01}, and it gives an action of $B_3$ on $F_3$.  Show that any $1$-positive braid $\beta$ yields an automorphism $\hat \beta$ such that $\hat \beta(x_1)$ begins with $x_1^2$.  Conclude that $P_D \cap P_D^{-1} = \emptyset$.
\end{problem}
Note that the previous exercise can be easily generalized to arbitrary $n$, and so provides one way of proving a key step in the proof that $<_D$ is well-defined for all ~$n$.


\section{Thurston's orderings of $B_n$}
First we must cover a bit of hyperbolic geometry in order to describe  Thurston's orderings.  For a more complete treatment of this material, see \cite{CB88}.  The methods of this section, due to W. Thurston, first appeared in \cite{SW00}.

\subsection{Braids as mapping class groups.}
The mapping class group of a punctured, compact, connected, oriented surface $\Sigma$ is defined as follows.  Let $P = \{ p_1, \ldots, p_n\}$ denote a finite subset of distinct points on $\Sigma$ that we will call punctures (one often removes the points $p_1, \ldots, p_n$ and speaks of the holes that they leave behind).  Let $\mathrm{Homeo}_+(\Sigma, P)$ denote the group of orientation-preserving homeomorphisms $h : \Sigma \rightarrow \Sigma$ satisfying $h(P) = P$ that restrict to the identity on the boundary of $\Sigma$.  The \textit{mapping class group} \index{mapping class group} of $\Sigma$ with punctures $p_1, \ldots, p_n$ is the group $\mathrm{Mod}(\Sigma, P)$ of isotopy classes of these homeomorphisms, in other words 
\[ \mathrm{Mod}(\Sigma, P) = \pi_0(\mathrm{Homeo}_+(\Sigma, P))
\]
The group operation is composition of representative homeomorphisms.

Our interest will be focused on the mapping class group of the closed unit disk in $\mathbb{C}$, which we'll denote as $\mathbb{D}$.  For each $n \geq 0$ we specify the points $p_1, \ldots, p_n$ to be $n$ evenly spaced points along the real axis, write $\mathbb{D}_n$ for the disk with these points removed (see Figure \ref{figure_Dn}), and write $\mathrm{Mod}(\mathbb{D}_n)$ instead of  $\mathrm{Mod}(\mathbb{D}, P)$.    

\begin{figure}[h!]
\setlength{\unitlength}{6cm}
\begin{picture}(1,1)%
    \put(0,0){\includegraphics[width=\unitlength]{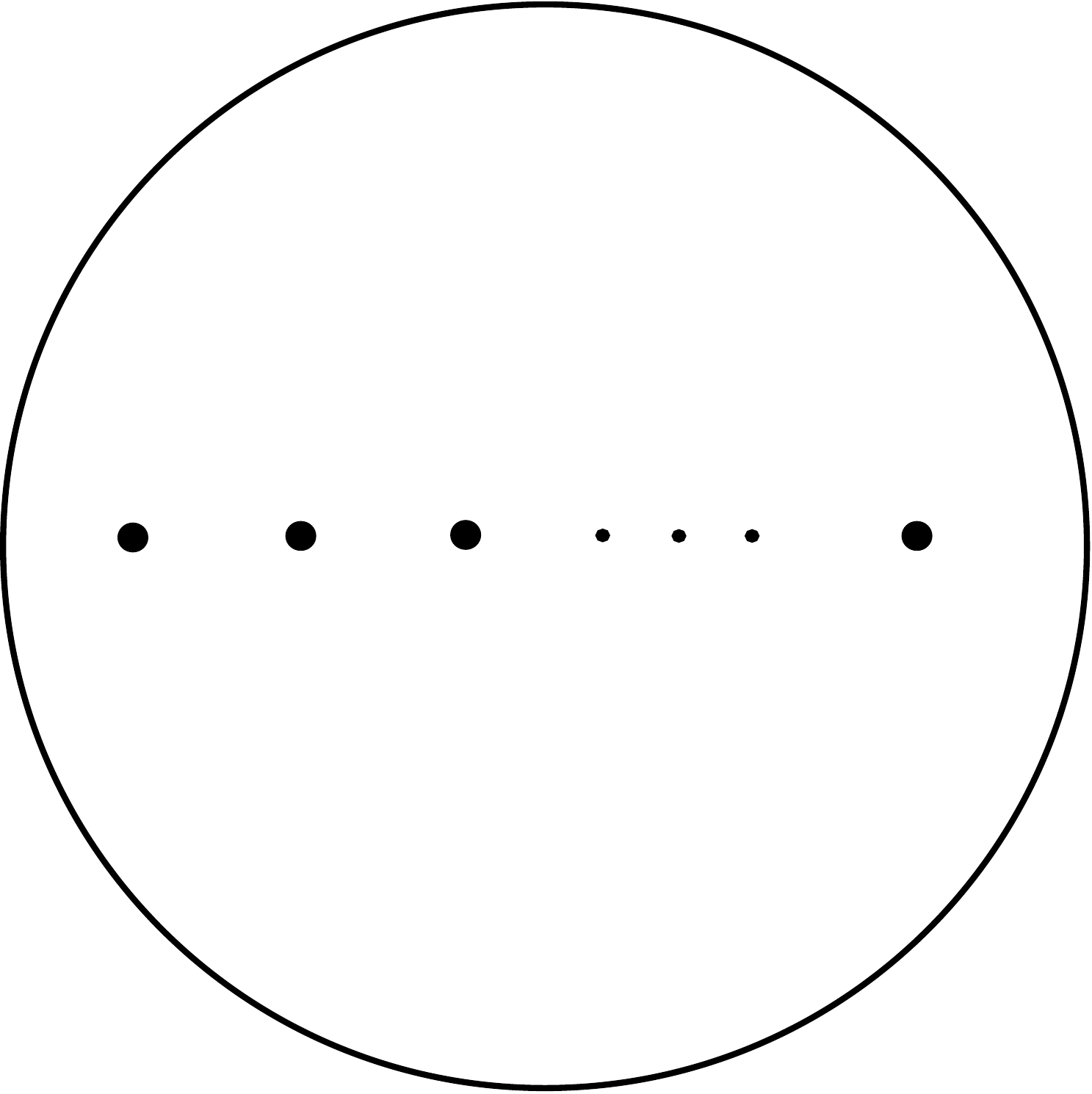}}%
    \put(0.09828143,0.44612921){$p_1$}%
    \put(0.2495178,0.44612932){$p_2$}%
    \put(0.39791846,0.44801962){$p_3$}%
    \put(0.79774962,0.44707441){$p_n$}%
  \end{picture}%
\caption{The disk $\mathbb{D}_n$  with $n$ evenly spaced punctures on the real axis.}
\label{figure_Dn}
\end{figure}
Let's calculate our first mapping class group.

\begin{lemma} \emph{(Alexander's trick)}
\index{Alexander's trick}
The mapping class group of $\mathbb{D}$ is trivial.
\end{lemma}
\begin{proof}
Let $h$ be any homeomorphism of $\mathbb{D}$ with itself that restricts to the identity on the boundary.  Define
\[
F(z,t) =
\left\{
	\begin{array}{ll}
		(1-t)h(\frac{z}{1-t})  & \mbox{if } 0 \leq |z| < 1-t \\
		z & \mbox{if } 1-t \leq |z| \leq 1
	\end{array}
\right.
\]
and define $F(z, 1) : \mathbb{D} \rightarrow \mathbb{D}$ to be the identity.  Then $F$ provides an isotopy of $h$ with the identity, by doing the original map $h$ on a small disk of radius $1-t$ at time $t$ and doing the identity outside.  At time $t=1$ the map is the identity on the whole disk.  Since every homeomorphism $h$ of the disk is isotopic to the identity in this way, $\mathrm{Mod}(\mathbb{D})$ is trivial.
\end{proof}

Now let us consider the group $\mathrm{Mod}(\mathbb{D}_n)$, and how to identify this group with the braid group.  Given a homeomorphism $h: \mathbb{D} \rightarrow \mathbb{D}$ representing an element $[h]$ of $\mathrm{Mod}(\mathbb{D}_n)$, $h$ is not isotopic to the identity via an isotopy that fixes punctures unless $[h]$ is the identity.  However, if we disregard the punctures and consider $[h]$ as an element of $\mathrm{Mod}(\mathbb{D})$ then there \textit{is} an isotopy taking $h$ to the identity since $\mathrm{Mod}(\mathbb{D})$ is trivial.  So, let $F(z,t)$ denote an isotopy which carries $h$ to the identity.  As $t$ ranges from $0$ to $1$, the images of the punctures $p_1, \ldots , p_n$ under the function $F(-, t)$ trace out non-intersecting paths in the cylinder $\mathbb{D} \times [0,1]$, as in Figure \ref{braidtube}.

\begin{figure}[h!]
\setlength{\unitlength}{5cm}
\begin{picture}(1,0.89604448)%
    \put(0,0){\includegraphics[scale=0.25]{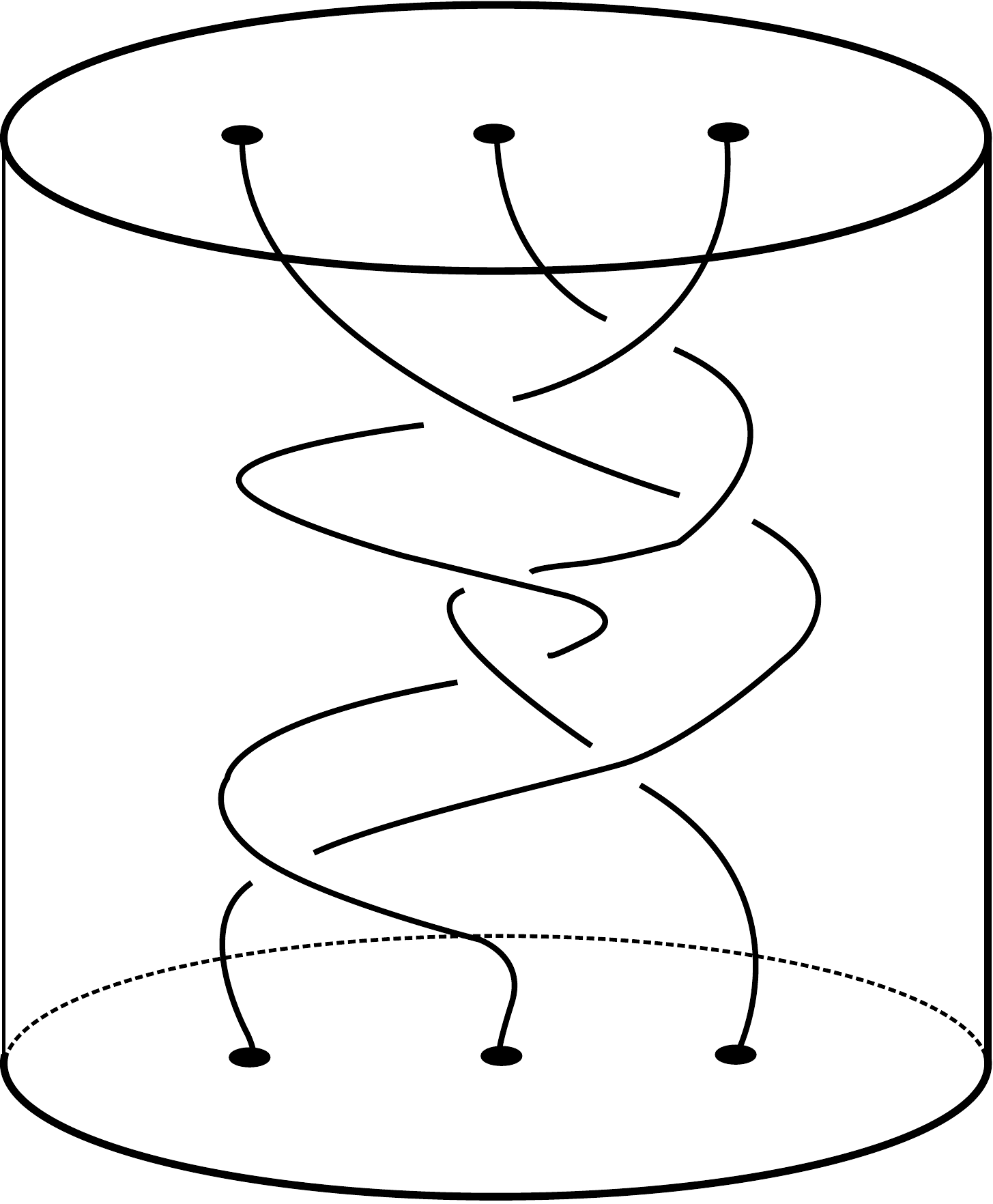}}%
    \put(0.74131285,0.83148657){$F(z,0)$}%
    \put(0.72517206,0.00830293){$F(z,1)$}%
  \end{picture}
\caption{The braid associated to an element of $\mathrm{Mod}(\mathbb{D}_3)$.}
\label{braidtube}
\end{figure}

The images of these paths determine a braid, and so determine an element of the group $B_n$.   This correspondence is actually an isomorphism, though the details are somewhat involved and will not be covered here (see \cite{birman74} for full details). However we have described the key idea behind the theorem:
\begin{theorem} 
The braid group $B_n$ is isomorphic to the mapping class group $\mathrm{Mod}(\mathbb{D}_n)$.
\end{theorem}

\begin{problem}
Describe a homeomorphism of the disc $\mathbb{D}_n$ that represents an equivalence class whose corresponding braid is the generator $\sigma_i$.  Verify that your homeomorphisms of $\mathbb{D}_n$ satisfy, up to isotopy, the braid relations $\sigma_i \sigma_j = \sigma_j \sigma_i$  if $|i-j|>1$ and 
$\sigma_i \sigma_j \sigma_i = \sigma_j \sigma_i \sigma_j$  if $|i-j|=1$.
\end{problem}

\subsection{A hyperbolic metric on the $n$-punctured disk.}
\label{hyperbolic_metric_on_disk}
For a fixed point $ z $ in the interior of $\mathbb{D}$, define an inner product on the tangent vectors (thought of as complex numbers $v, w$) at $z$ by the formula 
\[ g_z(v, w) = 4 \frac{\mathrm{Re}(v \bar{w})}{ (1 - |z|^2)^2}.
\]
This allows us to calculate the length of a tangent vector $w$ at $z$ as  
\[ g_z(w, w)^{1/2} = \left( 4 \frac{\mathrm{Re}(w \bar{w})}{ (1 - |z|^2)^2} \right)^{1/2}= 2 \frac{|w|}{ 1 - |z|^2}.
\]

Therefore, the length of a curve $\gamma:[0,1] \rightarrow \mathrm{int}(\mathbb{D})$ can be calculated by integrating the length of the tangent vector along the curve, yielding
\[ \ell_{\mathbb{H}} (\gamma) = \int_0^1 \frac{2|\gamma'(t)|}{1-|\gamma(t)|^2}dt
\]
  This gives the interior of the unit disk a hyperbolic metric, and equipped with this idea of length of curves we'll denote the open disk by $\mathbb{D}_{\mathbb{H}}$, it is called the \textit{Poincar\'{e} disk model}\index{Poincar\'{e} disk model}.  

The geodesics in $\mathbb{D}_{\mathbb{H}}$ are Euclidean circles meeting the boundary at right angles, and straight lines passing through $0$, as in Figure \ref{figure_geodesics}.

\begin{figure}[h!]
\includegraphics[scale=0.5]{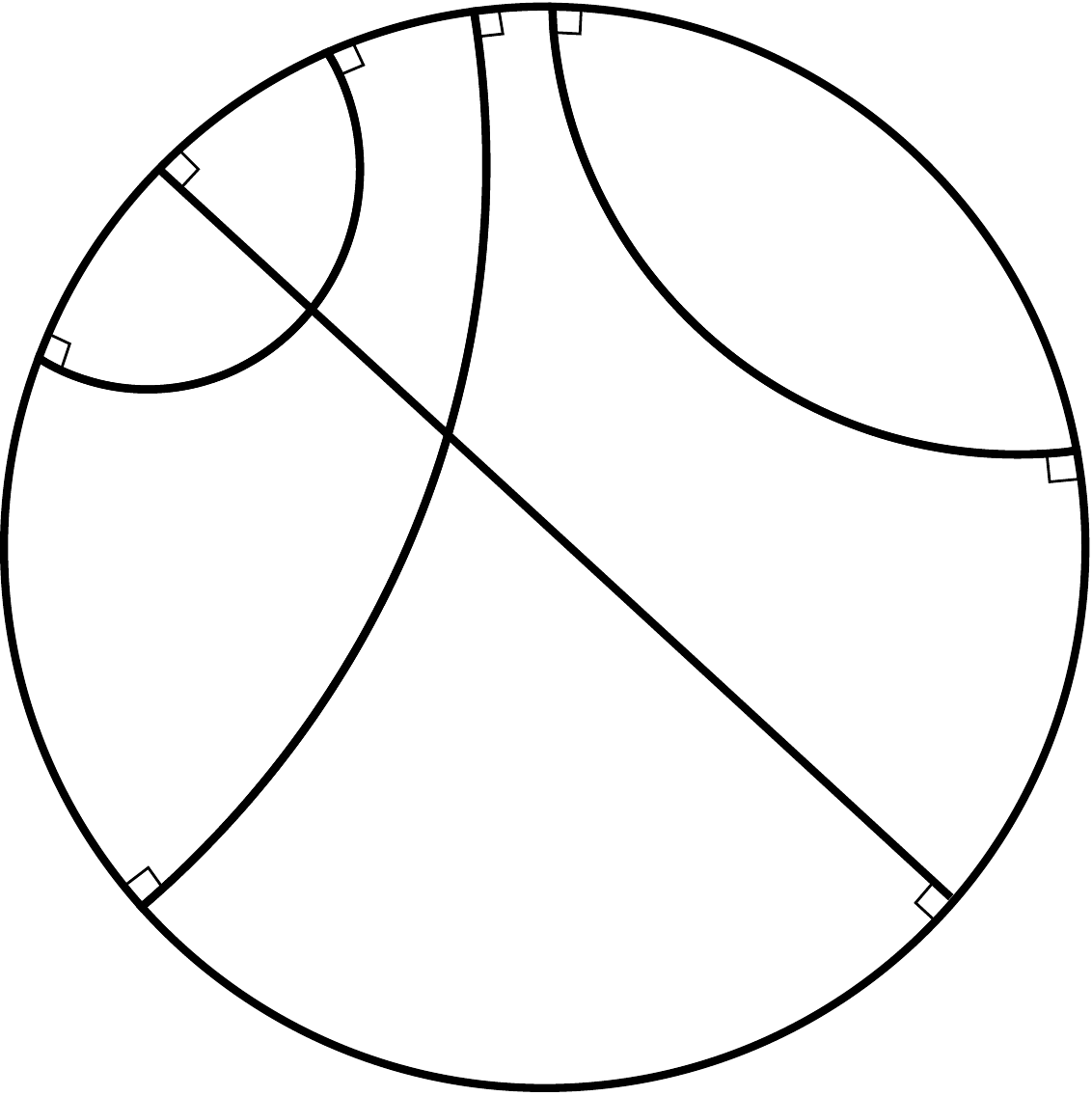}
\caption{Examples of geodesics in $\mathbb{D}_{\mathbb{H}}$.}
\label{figure_geodesics}
\end{figure}

The distance between two points $v, w \in \mathbb{D}_{\mathbb{H}}$ is denoted $d(v, w)$ and is calculated as the length of the shortest curve (i.e. a geodesic) between them.  The isometries of $\mathbb{D}_{\mathbb{H}}$ are the M\"{o}bius transformations of the form
\[ f_{\theta, a} (z)= e^{i \theta}  \left( \frac{ z-a}{1 - \bar{a} z} \right)
\]
where $|a| <1$.

\begin{problem} 
\label{dist_to_zero}
Using the fact that the straight line joining any point $z$ to $0$ is a geodesic, show that the distance between $z \in \mathbb{D}_{\mathbb{H}}$ and $0$ is $ 2 \mathrm{tanh}^{-1}(|z|) = \mathrm{log} \left(   \frac{|z|+1}{1-|z|} \right)$.
\end{problem}

\begin{problem} Show that the group of isometries of $\mathbb{D}_{\mathbb{H}}$ acts transitively on $\mathbb{D}_{\mathbb{H}}$, by exhibiting a M\"{o}bius transformation that will map an arbitrary $w \in \mathbb{D}_{\mathbb{H}}$ to $0$.   Use isometries and Problem \ref{dist_to_zero} to show that for any two points $v, w \in \mathbb{D}_{\mathbb{H}}$ the distance between them is
\[d(v, w) = 2 \mathrm{tanh}^{-1} \left| \frac{w-z}{\bar{z}w -1} \right|
\]
\end{problem}

Now we put a hyperbolic metric on the disk $\mathbb{D}_n$, $n\geq 2$, such that the boundary is a geodesic of any length we please and the area of $\mathbb{D}_n$ is finite.  To do this, we start with the Poincar\'{e} disk as described above and take two copies of the region indicated in Figure \ref{Dn_piece}. 

\begin{figure}
\setlength{\unitlength}{6cm}
 \begin{picture}(1,0.99434127)%
    \put(0,0){\includegraphics[width=\unitlength]{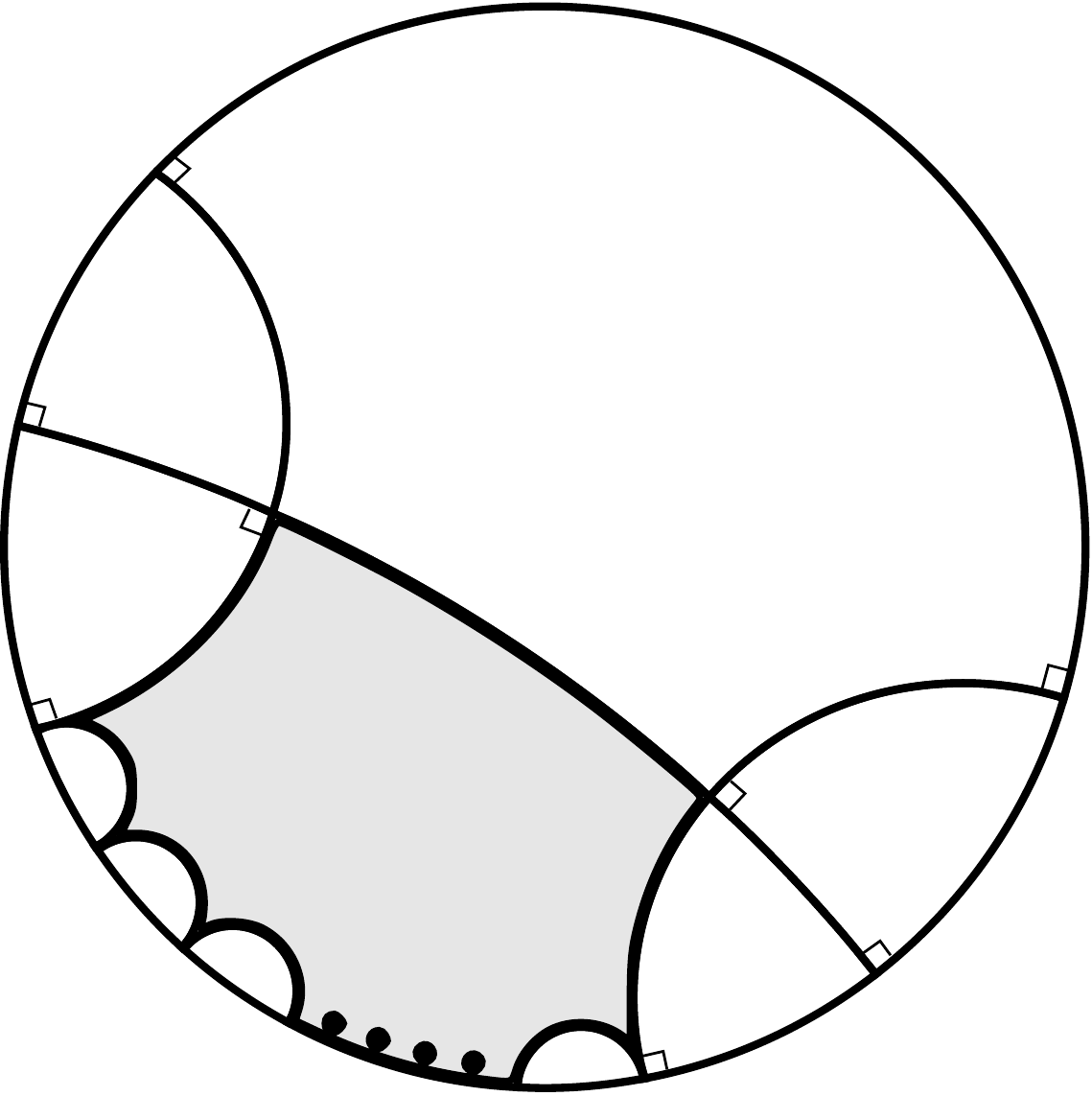}}%
    \put(0.11747152,0.43993631){$\gamma_1$}%
    \put(0.13309678,0.29930901){$\gamma_2$}%
    \put(0.19559778,0.20382136){$\gamma_3$}%
    \put(0.28414093,0.11180589){$\gamma_4$}%
    \put(0.48206079,0.08055545){$\gamma_n$}%
    \put(0.61053511,0.15000098){$\gamma_{n+1}$}%
  \end{picture}%
\caption{A piece of $\mathbb{D}_{\mathbb{H}}$ that we'll use to build $\mathbb{D}_n$.}
\label{Dn_piece}
\end{figure}

Isometrically identify the two copies of each boundary curve $\gamma_i$.  In this way the two pieces of $\mathbb{D}_{\mathbb{H}}$ that we cut out as in Figure  \ref{Dn_piece} will glue together to give a disk with $n$ punctures, see Figure \ref{glued_Dn}.   Each half of $\mathbb{D}_n$ comes equipped with a hyperbolic metric $g_z(v,w)$ inherited from $\mathbb{D}_{\mathbb{H}}$.  We use the inherited hyperbolic metrics on each half of $\mathbb{D}_n$ to define a continuously varying hyperbolic metric on $\mathbb{D}_n$ for which each of the punctures is `pushed off to infinity' to form what is called a \textit{cusp}.   The area of a cusp is finite, and consequently $\mathbb{D}_n$ has finite area when equipped with the metric we have just constructed.   
%

\begin{figure}
\setlength{\unitlength}{6cm}
\begin{picture}(1,1)%
    \put(0,0){\includegraphics[width=\unitlength]{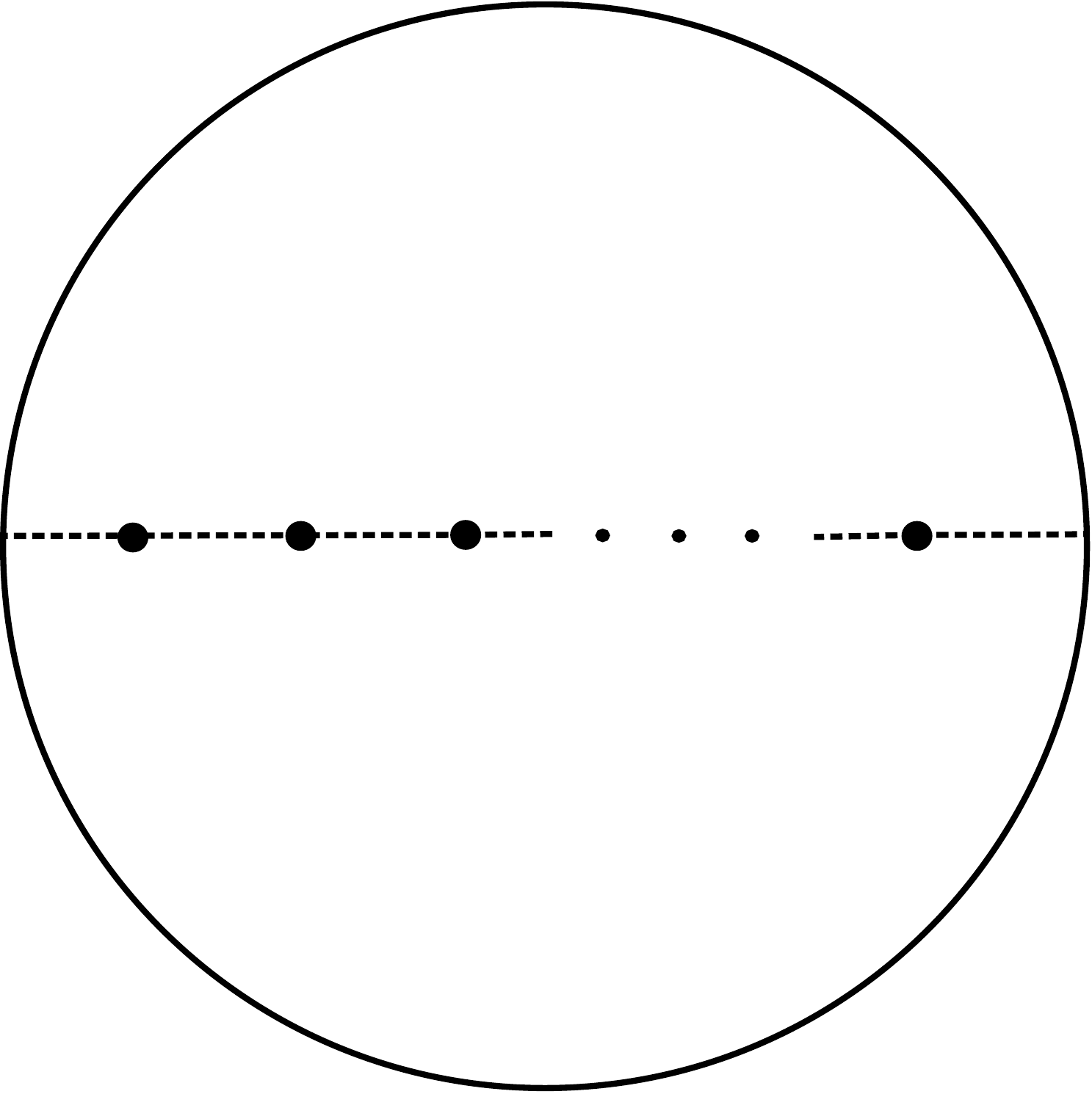}}%
    \put(0.09828143,0.44612921){$p_1$}%
    \put(0.2495178,0.44612932){$p_2$}%
    \put(0.39791846,0.44801962){$p_3$}%
    \put(0.79774962,0.44707441){$p_n$}%
    \put(0.03022506,0.55483039){$\gamma_1$}%
    \put(0.15594031,0.55199466){$\gamma_2$}%
    \put(0.32135503,0.55388513){$\gamma_3$}%
 \put(0.85,0.55388513){$\gamma_{n+1}$}%
  \end{picture}%

\caption{The disk $\mathbb{D}_n$ with the images of the curves $\gamma_i$ indicated by dotted lines.}
\label{glued_Dn}
\end{figure}

\subsection{A copy of the real line with a $B_n$-action.} 

In this section, we will prove that the braid groups are left-orderable by constructing an order-preserving, effective action of $B_n$ on the real line by using hyperbolic geometry.  Since the group $\mathrm{Homeo_+}(\mathbb{R})$ is left-orderable, it follows that $B_n$ is left-orderable.  The copy of the $\mathbb{R}$ that we will equip with a $B_n$-action arises as the boundary of a certain hyperbolic space, so first a word about such boundaries.

Given a hyperbolic space $X$ with distance $d(x,y)$, parameterize the geodesic rays $\gamma:[0,\infty) \rightarrow X$ by arc length.  Then two rays $\gamma, \gamma'$ are said to be \textit{a bounded distance apart} if there exists a distance $D >0$ such that  $\hbox{sup}_t \{ d(\gamma(t), \gamma'(t)) \} \leq D$.  The property of being a bounded distance apart is an equivalence relation on the set of geodesic rays, and we define the boundary $\partial X$ to be the set of equivalence classes.  For a mathematically rigorous treatment of what follows, this is the required notion of boundary.  However for the space $\mathbb{D}_{\mathbb{H}}$ your intuition likely tells you that $\partial \mathbb{D}_{\mathbb{H}}$ should look like the circle $\partial \mathbb{D}$, and this is correct.  The boundary of $\mathbb{D}_{\mathbb{H}}$ is homeomorphic to a circle $S^1_{\infty}$, called the circle at infinity, and in fact each point on the boundary admits a representative geodesic ray $\gamma: [0, \infty) \rightarrow \mathbb{D}_{\mathbb{H}}$ with $\gamma(0) = 0$, which is a (Euclidean) straight line.  We will rely upon this in the arguments which follow.

Denote the universal cover of $\mathbb{D}_n$ by $\widetilde{\mathbb{D}}_n$, fix a point $x_0 \in \partial \mathbb{D}_n$ and a point $\widetilde{x}_0 \in \partial \widetilde{\mathbb{D}}_n$ with $p(\widetilde{x}_0) = x_0$. 

\begin{proposition}
\label{lifting_the_action}
For $n \ge 2$, there exists an action of $\mathrm{Mod}(\mathbb{D}_n)$ on $\widetilde{\mathbb{D}}_n$ such that every element of  $\mathrm{Mod}(\mathbb{D}_n)$ fixes the point $\widetilde{x}_0$.
\end{proposition}

\begin{problem}
Prove the previous proposition. (Hint:  Consider cutting open $\mathbb{D}_n$ along the curves $\gamma_1, \gamma_2, \ldots, \gamma_n$, but not $\gamma_{n+1}$, to produce a simply connected space $X$.  There is an inclusion $i :X \rightarrow \mathbb{D}_n$ which can be lifted to $\tilde{i}_{\gamma} : X \rightarrow \widetilde{\mathbb{D}}_n$ and used to build an action of $\mathrm{Mod}(\mathbb{D}_n)$ on $\widetilde{\mathbb{D}}_n$).
\end{problem}

The action of $\mathrm{Mod}(\mathbb{D}_n)$ on $\widetilde{\mathbb{D}}_n$ can be extended to an action on $\partial \widetilde{\mathbb{D}}_n$, though we will not prove this here.  The difficulty in proving this claim arises from the fact that for a lifted homeomorphism $\tilde{h}$ and a geodesic ray $\gamma$ representing $[\gamma] \in \partial \widetilde{\mathbb{D}}_n$, one would like to define $\tilde{h} \cdot [\gamma] = [\tilde{h}(\gamma)]$.  However, $\tilde{h}$ is not an isometry and so $\tilde{h}(\gamma)$ is not a geodesic, and thus there is work to be done in correcting this problem.  For full details, see \cite[Chapter 1]{FM12}.

Next we consider how the action on $\partial \widetilde{\mathbb{D}}_n$ gives us an action on $\mathbb{R}$.  First, a surface with a hyperbolic distance as above is \textit{complete} if it is complete as a metric space, the disk $\mathbb{D}_n$ is complete with respect to the hyperbolic metric we've introduced.   The hyperbolic metric on $\mathbb{D}_n$ pulls back to define a hyperbolic metric on the universal cover $\widetilde{\mathbb{D}}_n$, and as the universal cover of a complete space is complete, $\widetilde{\mathbb{D}}_n$ is also a complete metric space.  Every complete, connected, simply connected hyperbolic surface is isometric to $\mathbb{D}_{\mathbb{H}}$ or a subset of $\mathbb{D}_{\mathbb{H}}$ \cite[Theorem 2.2]{CB88}.  Thus there is an isometric embedding $i : \widetilde{\mathbb{D}}_n \rightarrow \mathbb{D}_{\mathbb{H}}$, but we will write $\widetilde{\mathbb{D}}_n$ in place of $i(\widetilde{\mathbb{D}}_n)$.     Moreover, since isometries act transitively on $\mathbb{D}_{\mathbb{H}}$ we can assume that $\tilde{x}_0 = 0$ and picture $\widetilde{\mathbb{D}}_n$ as a subset of $\mathbb{D}_{\mathbb{H}}$ as in Figure \ref{universal_cover}.

\begin{figure}
\setlength{\unitlength}{6cm}
\begin{picture}(1,1.00000038)%
    \put(0,-0.155){\includegraphics[width=\unitlength]{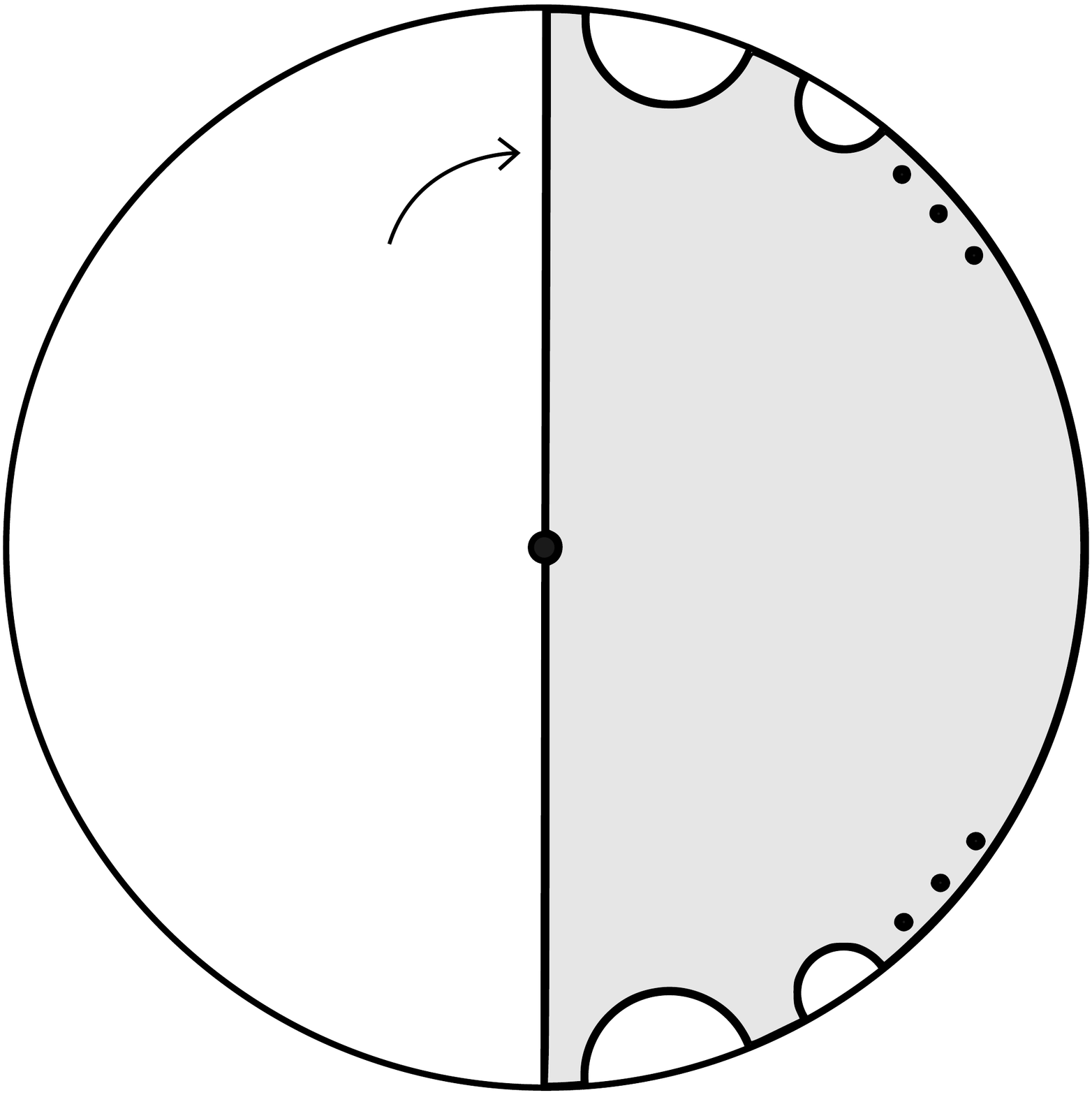}}%
    \put(0.32461226,0.71350257){$C$}%
    \put(0.28694796,0.49558791){$x_0=0$}%
    \put(0.67393283,0.57554325){$\widetilde{\mathbb{D}}_n$}%
  \end{picture}%
\caption{The embedded copy of $\widetilde{\mathbb{D}}_n$ in $\mathbb{D}_{\mathbb{H}}$ .}
\label{universal_cover}
\end{figure}

Writing $p:\widetilde{\mathbb{D}}_n \rightarrow \mathbb{D}_n$ for the covering map, the boundary $\partial \widetilde{\mathbb{D}}_n$ consists of points in $p^{-1}(\partial \mathbb{D}_n)$, and a set of points $X$ in $S^1_{\infty}$.  The points of $X$ are represented by geodesics in $ \widetilde{\mathbb{D}}_n$, which we can think of as geodesics in $\mathbb{D}_{\mathbb{H}}$ emanating from $0$ by employing the embedding $\widetilde{\mathbb{D}}_n \subset \mathbb{D}_{\mathbb{H}}$.   
The action of $\mathrm{Mod}(\mathbb{D}_n)$ is already defined on $p^{-1}(\partial \mathbb{D}_n)$, and the action uniquely extends to the set $X$.

Last, let $C$ denote the component of $p^{-1}(\partial \mathbb{D}_n)$ containing $0$. Since $C$ is fixed by the action of $\mathrm{Mod}(\mathbb{D}_n)$, there is an action of $\mathrm{Mod}(\mathbb{D}_n)$ on $\partial \widetilde{\mathbb{D}}_n \setminus C$ by homeomorphisms.  The set $\partial \widetilde{\mathbb{D}}_n \setminus C$ can be naturally identified with $(0, \pi)$, since each point corresponds to a unique geodesic ray $\gamma: [0, \infty) \rightarrow \mathbb{D}_{\mathbb{H}}$ with $\gamma(0) = 0$ and whose angle with $C$ lies in $(0, \pi)$.  Moreover, since the endpoints of $C$ are fixed by the action of $\mathrm{Mod}(\mathbb{D}_n)$, the action is order-preserving.  

\begin{problem}
Show that the action of $\mathrm{Mod}(\mathbb{D}_n)$ on $(0, \pi)$ is effective, since any element $[h] \in \mathrm{Mod}(\mathbb{D}_n)$ that fixes every point in $(0, \pi)$ would, in particular, fix every lift of $x_0 \in \mathbb{D}_n$ and so be the identity.
\end{problem}

Thus we have constructed an effective, order-preserving action of $B_n$ on $\mathbb{R}$, which proves the following:
\begin{theorem}
\label{thurston_ordering_theorem}
The braid groups $B_n$ are left-orderable for $n \geq 2$.
\end{theorem}

\subsection{Different left-orderings from Thurston's construction.}

Given an action of a group $G$ on $\mathbb{R}$ by order-preserving homeomorphisms, recall that we can construct a left-ordering of $G$ using the procedure in Example \ref{homeo+}.  That procedure requires a choice of countable dense sequence $x_1, x_2, \ldots$ of real numbers, and then for any two distinct elements $g, h \in G$, if $i$ is the lowest subscript for which $g(x_i)$ and $h(x_i)$ are different  we declare $g<h$ if $g(x_i) < h(x_i)$ and $h<g$ otherwise.  While not emphasized in that example, it is clear that different choices of sequences can give rise to different left-orderings of $G$ (for example, by simply re-indexing the same sequence one can potentially get a new left-ordering). 

Here, we'll investigate how different choices of sequences can give rise to different orderings of $B_n$ by applying Thurston's construction.  According to the setup of Theorem \ref{thurston_ordering_theorem}, $B_n$ acts on a copy of $\mathbb{R}$ that is identified with $\partial \widetilde{\mathbb{D}}_n$, each point of which corresponds to a geodesic ray $\widetilde{\gamma}: [0, \infty) \rightarrow \mathbb{D}_{\mathbb{H}}$.  Thus every left-ordering arising from Theorem \ref{thurston_ordering_theorem} depends on a choice of geodesic rays $ \widetilde{\alpha}_1, \widetilde{\alpha}_2, \ldots $ emanating from $0$ and corresponding to a dense sequence of points in $\partial \widetilde{\mathbb{D}}_n$.

Let's consider the particular case of $B_4$ to illustrate what information can be gleaned from a chosen sequence of geodesic rays.  We begin with geodesic rays $\{ \widetilde{\alpha}_1, \widetilde{\alpha}_2, \ldots \}$ starting at $0$, where 
the image of $\widetilde{\alpha}_1$ under the projection map $p: \widetilde{\mathbb{D}}_4 \rightarrow \mathbb{D}_4$ appears as on the left of Figure \ref{order_curves}, and $\widetilde{\alpha}_i$ are arbitrary for $i \geq 2$.  Note that this image in $\mathbb{D}_4$ uniquely determines the lift $\widetilde{\alpha}_1$, since we have specified that $\widetilde{\alpha}_1(0) =0$.  Let $<_{\alpha}$ denote the corresponding left-ordering of $B_4$ that is defined as in the first paragraph of this section.

\begin{figure}[h!]
\setlength{\unitlength}{4.5cm}
\begin{picture}(2.5,1.00000038)
\put(0, 0){\includegraphics[width=\unitlength]{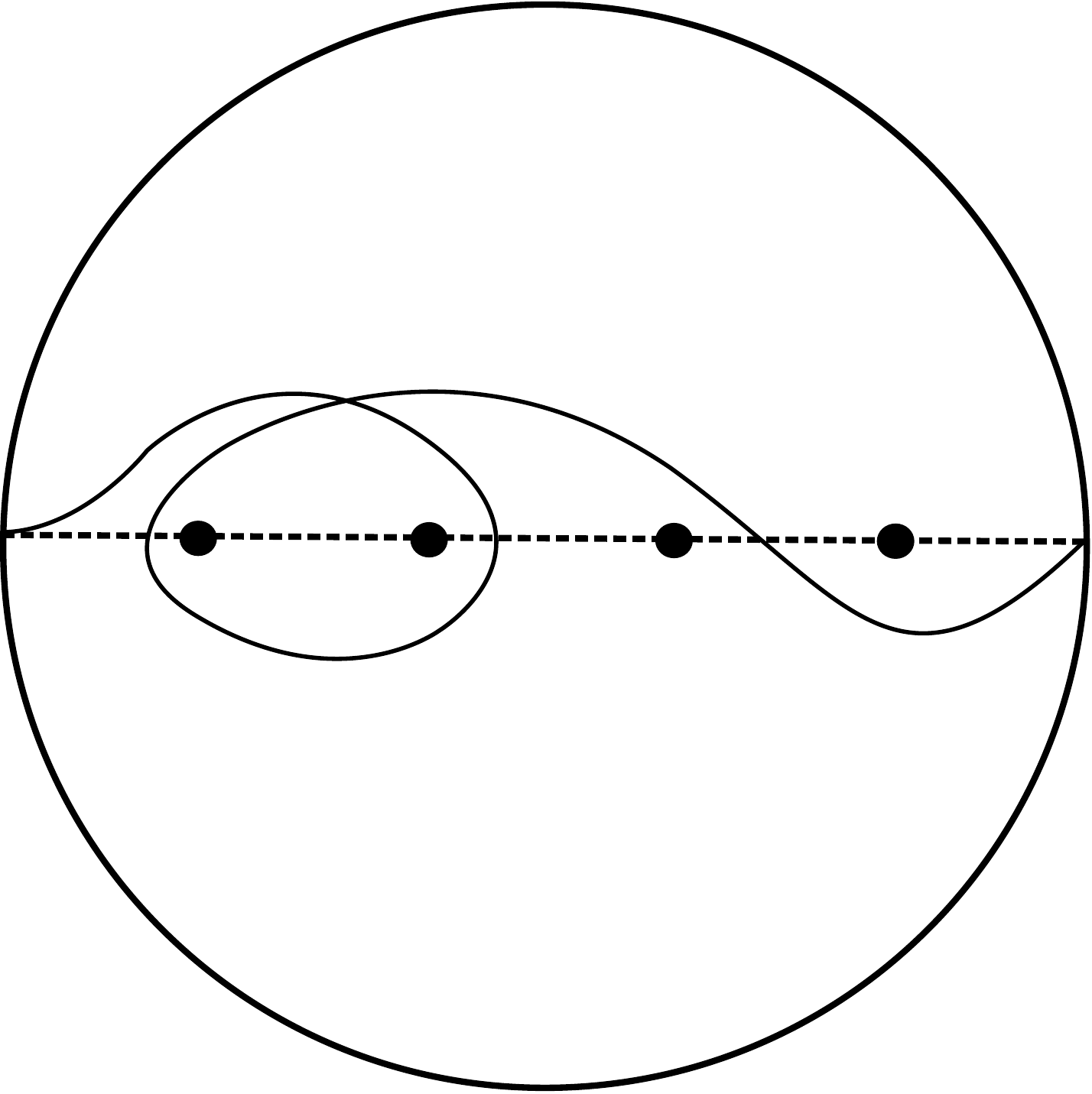}}
 \put(-0.04,0.48){\scalebox{1.5}{$\times$}}
\put(1.5, 0){\includegraphics[width=\unitlength]{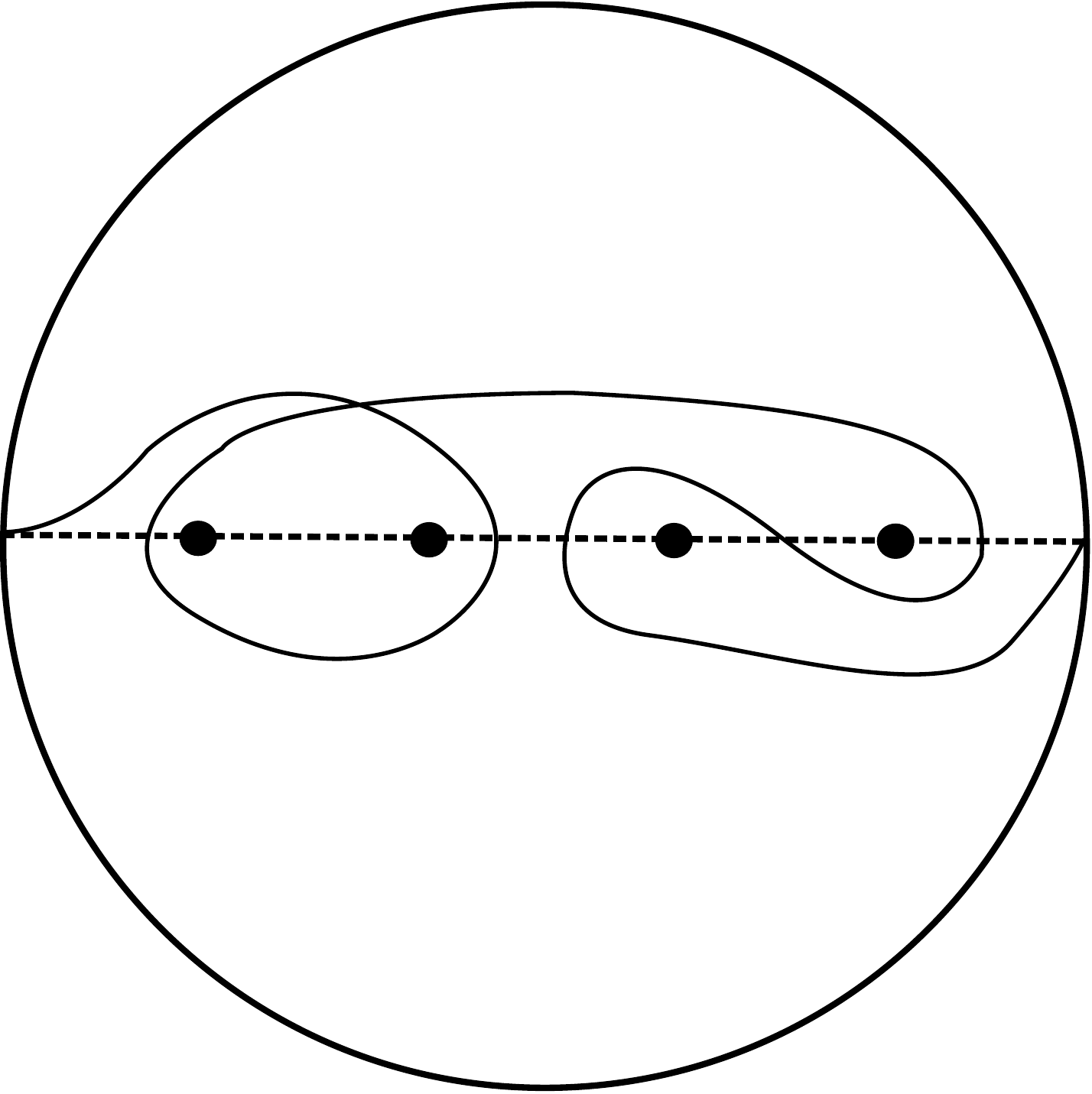}}
 \put(1.46,0.48){\scalebox{1.5}{$\times$}}
\end{picture}
\caption{The curve $\alpha_1$ on the left, and $\alpha_1$ after applying the homeomorphism $h_3$ on the right.}
\label{order_curves}
\end{figure}

Fix representative homeomorphisms $h_i$ of elements $[h_i] \in \mathrm{Mod}(\mathbb{D}_4)$, where $h_i$ swaps the punctures $p_i$ and $p_{i+1}$ and $[h_i]$ corresponds to the generator $\sigma_i$ of $B_4$.  Note that we can choose a representative homeomorphism $h_1$ whose support is disjoint from $\alpha_1$, and thus in the cover $\widetilde{\mathbb{D}}_4 $ the corresponding action of $[h_1]$ fixes the lift $\widetilde{\alpha}_1$.

By Problem \ref{relatively_convex_intersection_problem}, the stabilizer of $\widetilde{\alpha}_1$ is a convex subgroup $C$ in the left-ordering $<_{\alpha}$.  On the other hand, $h_3$ is not in $C$, because $[h_3]$ acts nontrivially on $\widetilde{\alpha}_1$.  We can see this by examining Figure \ref{order_curves}, and considering the intersection of $\alpha_1$ and $h_3(\alpha_1)$ with the curves $\gamma_1, \gamma_2, \ldots, \gamma_5$ which connect the punctures to one another, and to the boundary.   Starting from the ``X'', the curve $\alpha_1$ intersects the $\gamma_i$ in the order $\gamma_3, \gamma_1, \gamma_4$, and so the lift $\widetilde{\alpha}_1$ intersects lifts of the $\gamma_i$'s in that same order.  Considering $h_3(\alpha_1)$, we see a different order of intersection with the $\gamma_i$'s, namely $\gamma_3, \gamma_1, \gamma_5, \gamma_4, \gamma_3$, and so the lift $\widetilde{h_3(\alpha_1)}$ intersects the lifts of the $\gamma_i$'s in this order.

\begin{problem}
Argue that $h_3(\alpha_1)$ and $\alpha_1$ determine distinct points in $\partial \widetilde{\mathbb{D}}_n$, since each intersects the $\gamma_i$'s in a different order.
\end{problem}

\begin{problem}
Argue that when $\alpha_1$ is replaced by the curve $\alpha_1'$ shown in Figure ~\ref{order_curves_2}, the resulting left-ordering $<_{\alpha}$ is different.
\begin{figure}[h!]
\setlength{\unitlength}{4.5cm}
\begin{picture}(1,1.00000038)
\put(0, 0){\includegraphics[width=\unitlength]{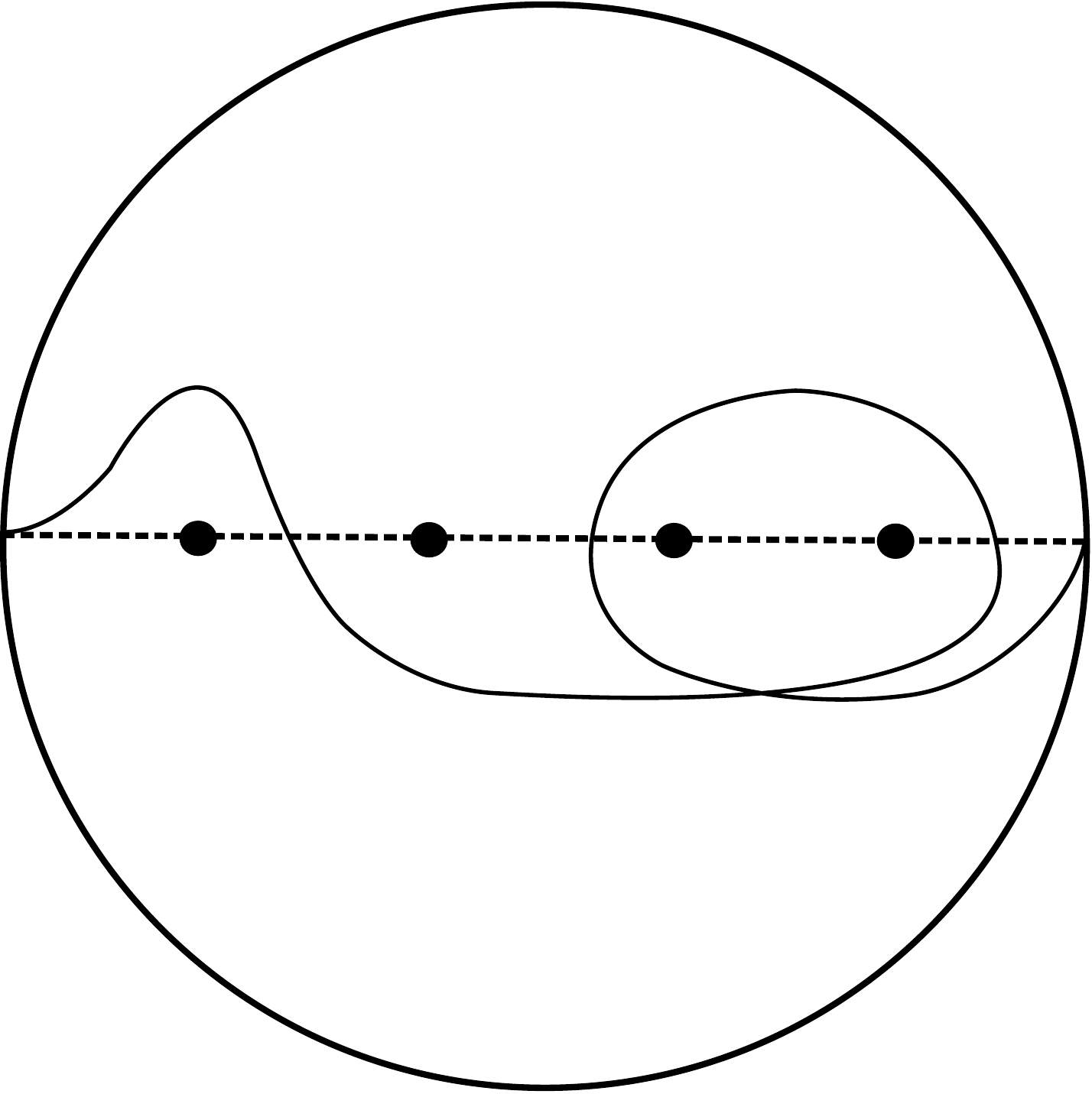}}
 \put(-0.04,0.48){\scalebox{1.5}{$\times$}}
\end{picture}
\caption{The curve $\alpha_1'$ .}
\label{order_curves_2}
\end{figure}
\end{problem}

Thus different choices of geodesics $\alpha_1, \alpha_2, \ldots$ in $\mathbb{D}_n$ correspond to different choices of geodesic rays $\widetilde{\alpha}_1, \widetilde{\alpha}_2, \ldots $ in $\widetilde{\mathbb{D}}_n $, each giving rise to a potentially distinct left-ordering of $B_n$.  In fact, the Dehornoy ordering also arises from Thurston's construction, by choosing $\alpha_1$ to be the curve in Figure \ref{dehornoy_geodesic}.  In this case it happens that the action of $B_n$ is free on the orbit of $\widetilde{\alpha}_1$, so the choice of $\alpha_1$ determines $<_D$ and the rest of the sequence $\{ \widetilde{\alpha}_2, \widetilde{\alpha}_3, \ldots \}$ is immaterial.  See \cite[Chapter XIII]{DDRW08} for full details of this construction of the Dehornoy ordering.

\begin{figure}[h!]
\setlength{\unitlength}{6cm}
\begin{picture}(1,1)%
    \put(0,0){\includegraphics[width=\unitlength]{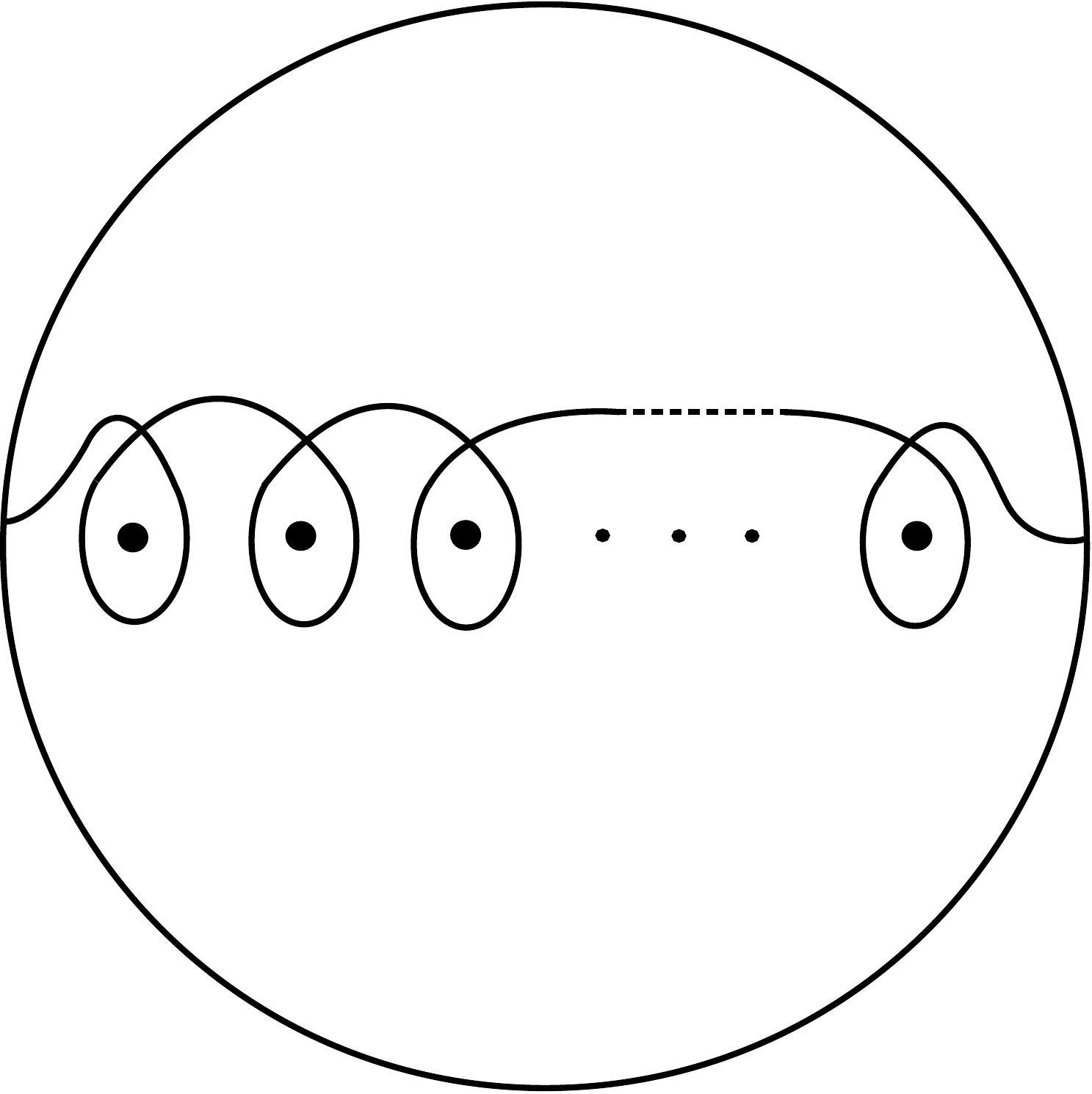}}%
    \put(0.09828143,0.34612921){$p_1$}%
    \put(0.2495178,0.34612932){$p_2$}%
    \put(0.39791846,0.34801962){$p_3$}%
    \put(0.79774962,0.34707441){$p_n$}%
  \end{picture}%
\caption{A geodesic that determines the Dehornoy ordering of $B_n$.}
\label{dehornoy_geodesic}
\end{figure}

\subsection{Generalizing to other mapping class groups.}
The remarkable feature of Thurston's approach is that it readily generalizes to other mapping class groups, the details of this generalization are outlined in the problems below. 

Recall from Chapter \ref{chapter free and surface} that closed, connected, orientable surfaces other than the sphere are all connect sums of tori.  For orientable surfaces with boundary (other than the sphere with $k>0$ disks removed), every surface is again a connect sum of tori, but where one of the tori in the decomposition has $k>0$ open disks removed.  The genus of such a surface is, as before, the number of tori in its connect sum decomposition.  A sphere, with or without disks removed, has genus $0$.

Let $\Sigma_{g,n}^b$ be an orientable, connected surface with genus $g$, $b$ boundary components, and $n$ punctures.  The Euler characteristic of $\Sigma_{g,n}^b$ is 
\[ \chi(\Sigma_{g,n}^b)  =2-2g-n-b
\]

\begin{problem}
\label{pants_decomposition}
A ``pair of pants'' is a surface that has genus zero and three boundary components as in Figure \ref{pants_figure}.  If $2-2g-b<0$, show that after possibly isotoping $\Sigma_{g,n}^b$ to reposition the punctures, the surface $\Sigma_{g,n}^b$ can be cut open along simple closed curves so that the resulting pieces are pairs of pants and a single annulus with $n$ punctures.
\end{problem}

\begin{figure}[h!]
\includegraphics[scale=0.4]{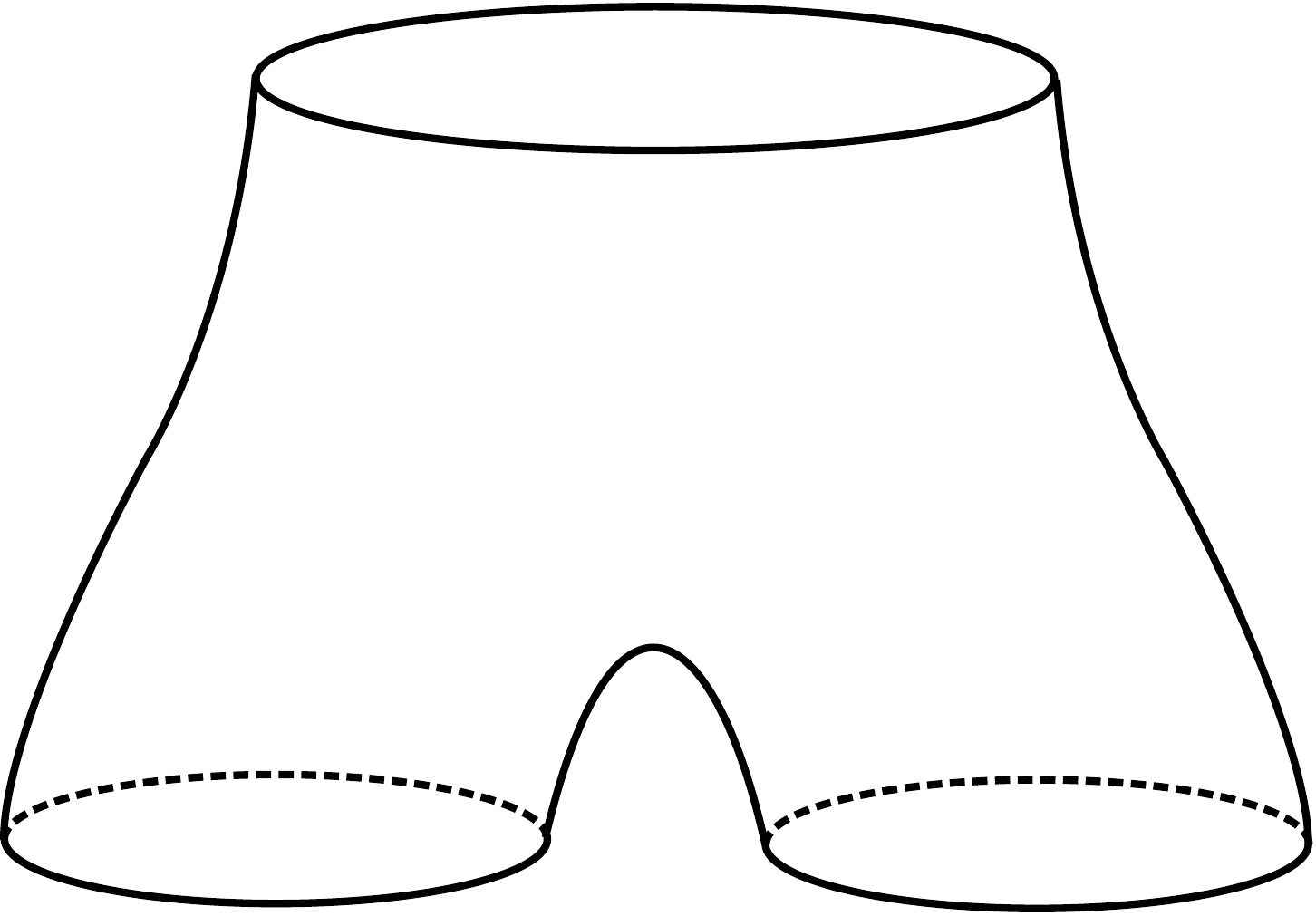}
\caption{A ``pair of pants'' surface.}
\label{pants_figure}
\end{figure}

\begin{problem}
Show that every surface with boundary satisfying $\chi(\Sigma_{g,n}^b)<0$ admits a hyperbolic metric, by breaking the argument into cases as follows.  If $2-2g-b \geq 0$ then we are in one of the following cases (recall that we are assuming $b>0$, which forces $g=0$ if $2-2g-b \geq 0$):
\begin{enumerate}
\item If $g=0$, $b=1$ and $n \geq 2$ then the surface is a punctured disk.  We dealt with this case in Section \ref{hyperbolic_metric_on_disk}.
\item If $g=0$, $b=2$ and $n \geq 1$ then the surface is an annulus with punctures.  Put a hyperbolic metric on this surface using the techniques of Section \ref{hyperbolic_metric_on_disk}.
\end{enumerate}
On the other hand, if $2-2g-b<0$ and $n \geq 0$, show that one can put a hyperbolic metric on each the pieces in the decomposition of Problem \ref{pants_decomposition} and glue the pieces together.
\end{problem}

Recall that the mapping class group of a surface with punctures $P=\{ p_1, \ldots, p_n\}$ is written $\mathrm{Mod}(\Sigma_{g,n}^b, P)$.  However when our surface is connected, we need not choose a specific set of punctures $\{ p_1, \ldots, p_n\}$ (since all $n$-element sets of points are isotopic to one another in $\Sigma_{g,n}^b$), and so we simplify the notation to $\mathrm{Mod}(\Sigma_{g,n}^b)$.  

\begin{problem}
Assume that $b>0$.  By making appropriate changes to the proof of Proposition \ref{lifting_the_action}, show that $\mathrm{Mod}(\Sigma_{g,n}^b)$ acts on $\widetilde{\Sigma_{g,n}^b}$ and that the action fixes a chosen point $x_0 \in \partial \widetilde{\Sigma_{g,n}^b}$.
\end{problem}  

Assume that $\chi(\Sigma_{g,n}^b)<0$ and $b>0$.  
The action constructed in the previous exercise extends to the boundary $\partial \widetilde{\Sigma_{g,n}^b}$, and as in the case of the disk with punctures, $\widetilde{\Sigma_{g,n}^b}$ is a complete, connected, simply connected hyperbolic surface and so can be isometrically embedded in $\mathbb{D}_{\mathbb{H}}$ \cite[Theorem 2.2]{CB88}.  By arguments more or less identical to the case of $\mathbb{D}_n$, one can use this embedding to show that the action of $\mathrm{Mod}(\Sigma_{g,n}^b)$ on $\partial \widetilde{\Sigma_{g,n}^b}$ is actually a faithful, order-preserving action on a space that is homeomorphic to $\mathbb{R}$.  From this, we have the following theorem, due to Rourke and Wiest \cite{RW00}, who also proved it for the nonorientable case.  See also \cite{SW00}.

\begin{theorem}
If $b>0$ and $\chi(\Sigma_{g,n}^b)<0$, the mapping class group  $\mathrm{Mod}(\Sigma_{g,n}^b)$ is left-orderable.
\end{theorem}
 
Note that the restriction $b>0$ is necessary:  if $b=0$ then $\mathrm{Mod}(\Sigma_{g,n}^0)$ has elements of finite order.  One kind of finite-order homeomorphism is a \textit{hyperelliptic involution}, depicted in Figure \ref{hyperelliptic}.  Because of this $\mathrm{Mod}(\Sigma_{g,n}^0)$ cannot be left-orderable.
 
 \begin{figure}[h!]
\includegraphics[scale=0.5]{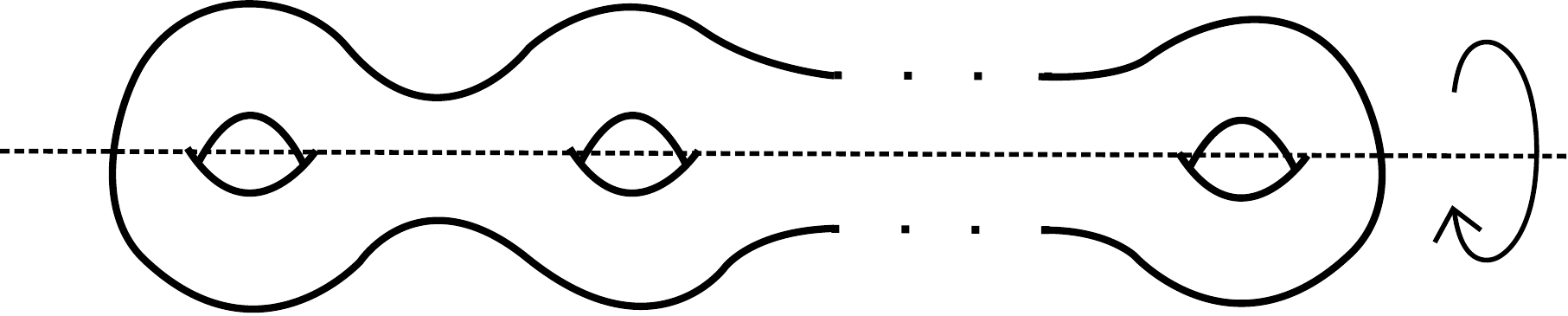}
\caption{Rotation by an angle of $\pi$ radians about the dotted line is an example of a hyperelliptic involution, and has order two.}
\label{hyperelliptic}
\end{figure}
 
 Despite having elements of finite order, the mapping class group of a surface without boundary is known to have a torsion free finite-index subgroup \cite[Theorem 6.9]{FM12}.  
 \begin{question}
Does the mapping class group of a surface without boundary have a left-orderable finite-index subgroup?
 \end{question}
 
\section{Applications of the Dehornoy ordering to knot theory}

To connect left-orderings of the braid groups with knot theory, we first have to connect braids with knots and links.  Every braid $\beta$ gives rise to a knot or link $\hat{\beta}$, called the \textit{closure} \index{braid closure} of $\beta$, via the following construction.  Consider $\mathbb{D} \times [0,1] $ embedded in $\mathbb{R}^3$ in the natural way, with points $p_i$ spaced evenly along the $x$-axis inside $\mathbb{D}^2$.

Given a braid 
\[ \beta = ( \beta_1 (t) , \beta_2(t), \ldots, \beta_n(t))
\]
in $\mathbb{D} \times [0,1]$, create the closure \index{braid closure} $\hat{\beta}$ by connecting $(p_i, 0)$ to $(p_i, 1)$ using $n$ non-intersecting paths in the plane $y=0$, as in Figure \ref{braid_closure}.

 \begin{figure}[h!]
 \setlength{\unitlength}{6cm}
\begin{picture}(1,1.25)%
    \put(0,0){\includegraphics[width=\unitlength]{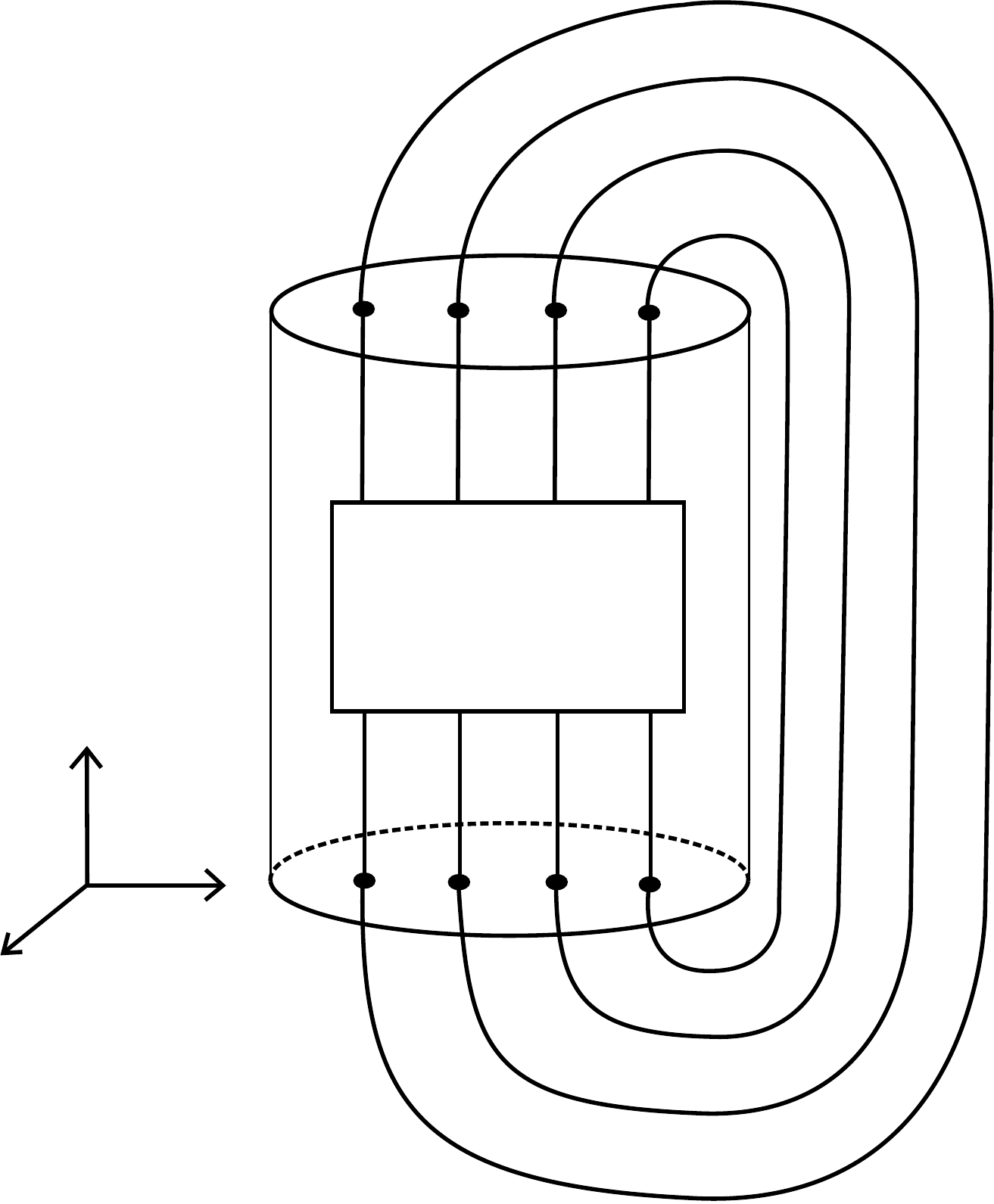}}%
    \put(0.47253147,0.57622535){$\beta$}%
    \put(0.06434703,0.48161983){$z$}%
    \put(0.20980913,0.25600519){$x$}%
    \put(-0.00689971,0.18772699){$y$}%
  \end{picture}
\caption{How to create the knot or link $\hat{\beta}$ from the braid $\beta$.}
\label{braid_closure}
\end{figure}

\begin{problem}
Determine necessary and sufficient conditions on the braid $\beta$ which ensure that $\hat{\beta}$ is a knot (as opposed to a link).
\end{problem}

That some knots can be expressed as the closure of a braid is no surprise--the trefoil, for example, is easily seen to be the closure of the braid $\sigma_1^3 \in B_2$, as in Figure \ref{trefoil closure}.  On the other hand, it may come as a surprise that \textit{every} knot can be written as the closure of a braid.

 \begin{figure}[h!]
\includegraphics[scale=0.5]{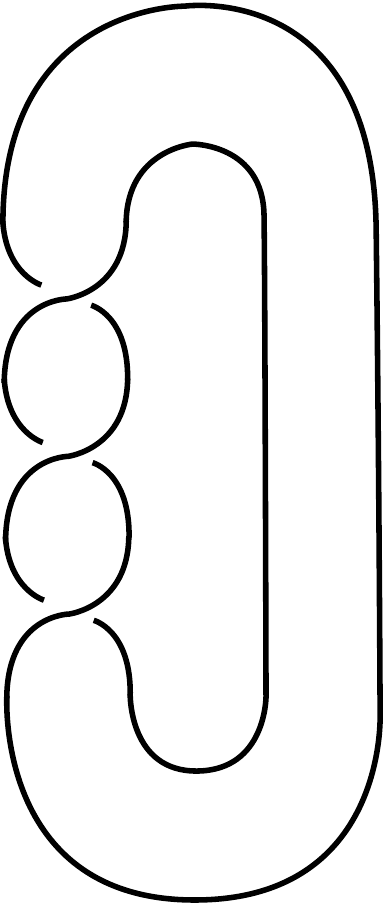}
\caption{The trefoil as the closure of $\sigma_1^3$ in $B_2$.}
\label{trefoil closure}
\end{figure}

\begin{theorem}[Alexander, 1923]
Given a knot $K$, there exists $n>1$ and $\beta \in B_n$ such that $K = \hat{\beta}$.
\end{theorem}
The idea behind the proof is illustrated in the next example.
\begin{example}
Given a diagram of a knot $K$ as in Figure \ref{knot to braid}, choose a point $p$ in one of the bounded regions.  Using $p$ as an axis of rotation, imagine an arrow rotating about $p$ and pointing towards your pen as you trace the knot.  As your pen moves, colour the arc normally whenever the arrow is rotating clockwise, and colour the arc in bold whenever the arrow is rotating anticlockwise.  See Figure \ref{knot to braid} (a), where the first anticlockwise segment encountered, starting from the little arrow, is coloured bold.  After colouring, move bold arcs to the other side of the axis as in Figure \ref{knot to braid} (b) and smooth out the remaining arcs.   Now when you trace along the knot a second time, the previously bold arc now corresponds to a segment of clockwise rotation.  Repeat this operation until there are no anticlockwise segments remaining; the resulting diagram is a closed braid.  Finally cut open along a dotted line which avoids all crossings, as in Figure \ref{knot to braid} (c).  Starting from the cut, we read off in the clockwise direction: $\sigma_2^2 \sigma_1 \sigma_2^{-1} \sigma_1^k$. Thus $K$ is closure of $\sigma_2^2 \sigma_1 \sigma_2^{-1} \sigma_1^k \in B_3$.

\begin{figure}[h!]
\setlength{\unitlength}{10cm}
\begin{picture}(1,0.7841735)%
    \put(0,0){\includegraphics[width=\unitlength]{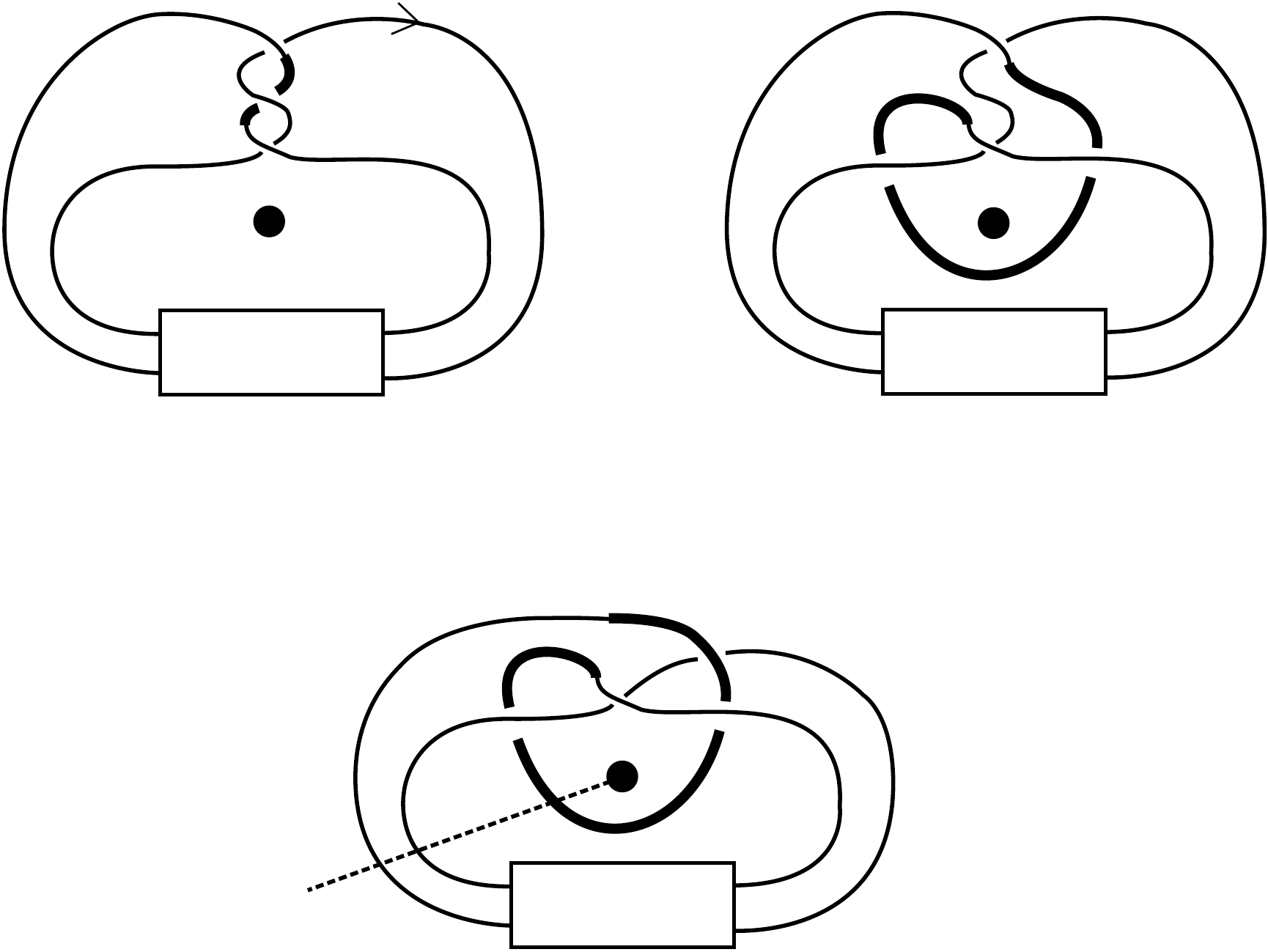}}%
    \put(0.15019163,0.46256718){$k$ twists}%
    \put(-0.00272443,0.76759876){(a)}%
    \put(0.7153676,0.46246095){$k$ twists}%
    \put(0.4262061,0.02161142){$k$ twists}%
    \put(0.56784894,0.76008707){(b)}%
    \put(0.2859236,0.31625275){(c)}%
    \put(0.19582629,0.02824518){cut}%
  \end{picture}%
%
%
\caption{How to write a knot $K$ as a closed braid.}
\label{knot to braid}
\end{figure}
\end{example}

\begin{problem}
Write each of the knots in Figure \ref{knots to braids} as the closure of a braid.
\begin{figure}[h!]
\includegraphics[scale=0.8]{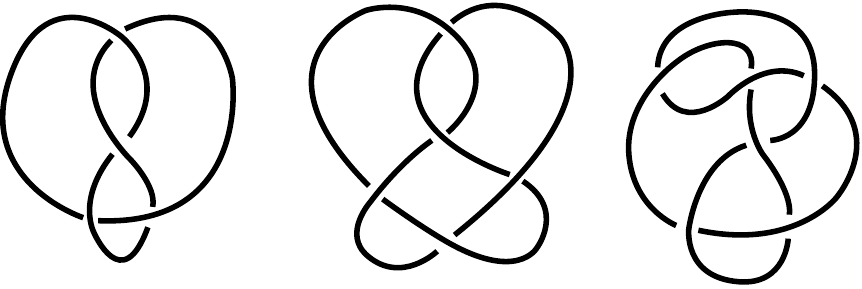}
\caption{From left to right: the figure eight, the knots $5_2$ and $6_3$.}
\label{knots to braids}
\end{figure}
\end{problem}

Different braids can give equivalent knots or links upon taking their closures.  For example, if $\alpha$ and $\beta$ are braids in $B_n$, then $\hat{\beta}$ and $\widehat{\alpha \beta \alpha^{-1}}$ are the same knot or link.  However there are many less obvious examples, indeed two braids which close to give the same knot or link need not even have the same number of strands.

This is explained by a result of Markov, which says that any two braids whose closures represent the same knot or link are related by a series of `Markov moves' \index{Markov moves} \index{Markov moves} (and their inverses).  The two types of Markov moves are:
\begin{itemize}
\item Replace a braid $\beta \in B_n$ with a conjugate $\alpha \beta \alpha^{-1}$, where $\alpha \in B_n$ (conjugation).
\item Replace a braid $\beta \in B_n$ with the braid $\beta \sigma_n^{\pm 1} \in B_{n+1}$ (stabilization, with inverse destabilization).
\end{itemize}
Note that stabilization and destabilization change the number of strands.

\begin{problem}
By manipulating the corresponding knot diagrams, show that the braids $\sigma_1^3 \in B_2$ and $(\sigma_1 \sigma_2)^2 \in B_3$ both close to give the same knot (the trefoil).  Give a sequence of Markov moves relating these two braids.
\end{problem}

 By using these braid moves, as well as more complex moves called \textit{flypes} and \textit{exchange moves}, one can show:
 
\begin{theorem}\cite{Malyutin04}
\label{nonprime theorem}
\index{prime knot}
Let $\beta \in B_n$. If $\hat{\beta}$ is not prime, then there exists $\alpha \in B_n$ and $\gamma, \gamma' \in B_{n-1}$ such that 
\[ \alpha \beta \alpha^{-1} = \gamma \sigma_{n-1} \gamma' \sigma_{n-1}^{-1}.
\]
In other words, $\beta$ is conjugate to a braid which appears as in Figure \ref{nonprime braid}.
\end{theorem}

\begin{figure}[h!]
\setlength{\unitlength}{2.5cm}
\begin{picture}(1,2)%
    \put(0,0){\includegraphics[width=\unitlength]{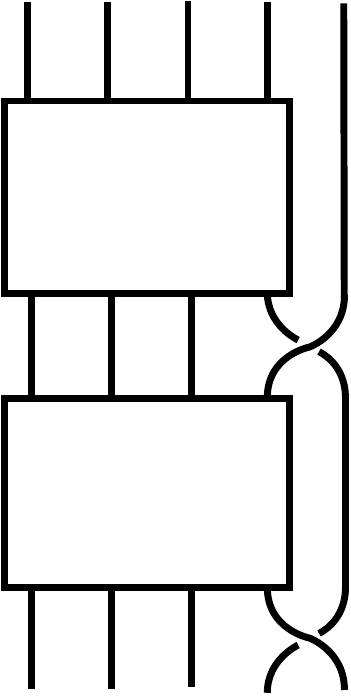}}%
    \put(0.4, 1.4){$\gamma$}%
    \put(0.4, 0.55){$\gamma'$}%
  \end{picture}
\caption{The form of the braid from Theorem \ref{nonprime theorem}, in the case $n=5$.}
\label{nonprime braid}
\end{figure}

The goal of the following exercises is to show how this theorem can be combined with the Dehornoy ordering to detect when certain links are prime (recall from Chapter \ref{knots chapter} that a knot or link is prime if it cannot be decomposed as a connect sum of two nontrivial knots or links).  To begin, we introduce Garside's ``half twist'' braid 
\[ \Delta_n = (\sigma_1 \sigma_2 \cdots \sigma_{n-1})(\sigma_1 \sigma_2 \cdots \sigma_{n-2}) \cdots (\sigma_1 \sigma_2) (\sigma_1).
\]
which appears as in Figure \ref{half_twist}.

The centre of $B_n$ is an infinite cyclic subgroup generated by the `` full twist'' on $n$ strands, and so its generator is the square of the Garside half twist, $\Delta_n^2$.

\begin{problem} 
\label{center_bounds}
\cite{Malyutin04}, \cite{DDRW08} Suppose that $\alpha, \beta \in B_n$.  Then in any left-ordering $\prec$ of $B_n$, we have:
\begin{enumerate}
\item If $\alpha \prec \Delta^{2k}_n$ and $\beta \prec \Delta^{2l}_n$, then $\alpha \beta \prec \Delta_n^{2k+2l}$,
\item If $\alpha \prec \Delta_n^{2k}$ then $\Delta_n^{-2k} \prec \alpha^{-1}$,
\item  If $\alpha \succ \Delta^{2k}_n$ and $\beta \succ \Delta^{2l}_n$, then $\alpha \beta \succ \Delta_n^{2k+2l}$.
\end{enumerate}
\end{problem}

\begin{figure}[h!]
\includegraphics[scale=0.45]{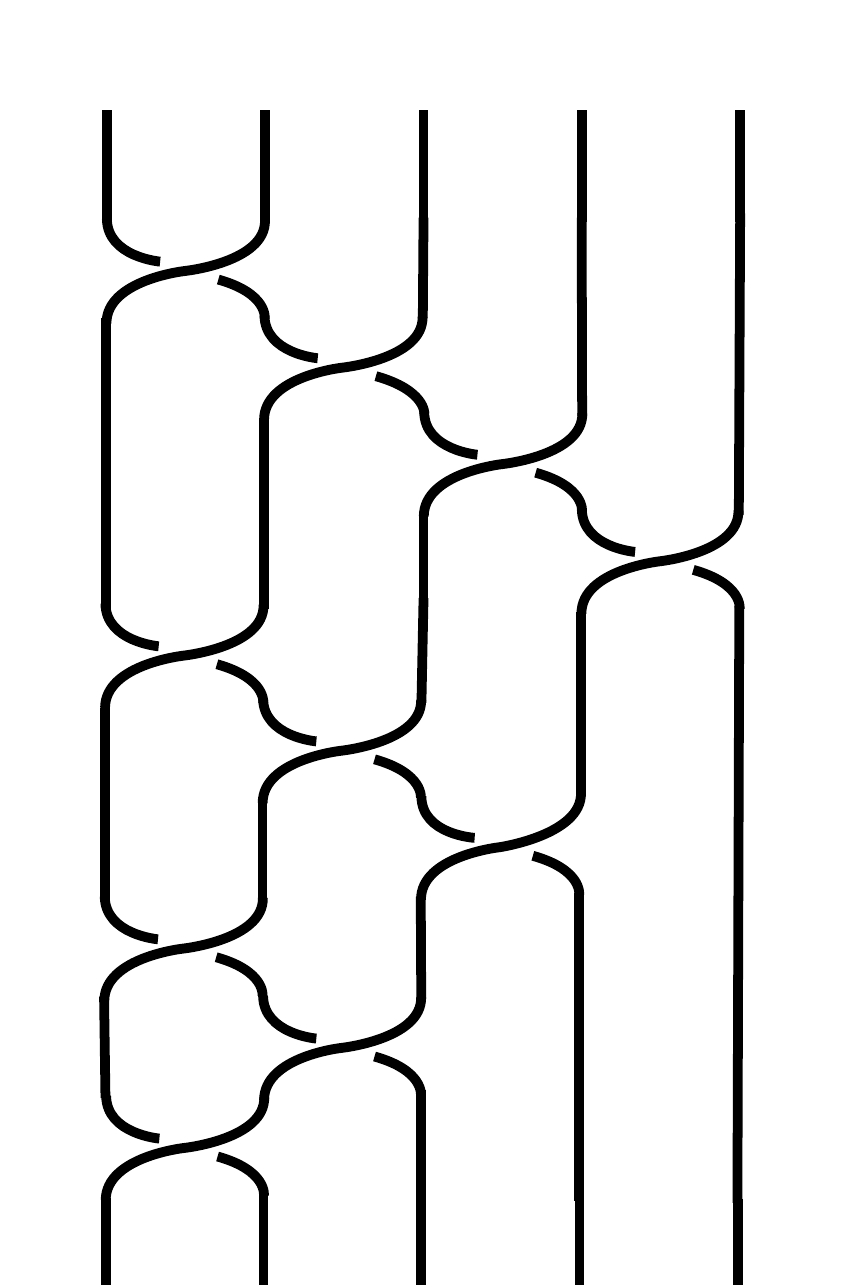}
\caption{The half-twist braid $\Delta_5$.}
\label{half_twist}
\end{figure}

\begin{problem}
Show that for all $\beta \in B_n$ there exists $k \in \mathbb{Z}$ such that $\Delta_n^{2k} \leq_D \beta <_D \Delta_n^{2k+2}$ (In fact, this holds for all left-orderings of $B_n$, not just for the Dehornoy ordering $<_D$).  Use this fact, together with the previous problem, to show that if $\alpha, \beta \in B_n$ and $\Delta_n^{2k} \leq_D \beta <_D \Delta_n^{2k+2}$ then $\Delta_n^{2k -2} \leq_D \alpha \beta \alpha^{-1} <_D \Delta_n^{2k+4}$.
\end{problem}

\begin{problem}
Verify that $\Delta \sigma_i^{\pm 1} \Delta^{-1} = \sigma_{n-i}^{\pm 1}$, and so by Theorem \ref{nonprime theorem} if $\hat{\beta}$ is a non-prime link then $\beta$ is conjugate to a braid of the form $ \gamma \sigma_{1} \gamma' \sigma_{1}^{-1} $, where $\gamma, \gamma'$ contain no occurences of $\sigma_1^{\pm 1}$.

Show that the inequalities prepared in the previous problems then imply that whenever $\hat{\beta}$ is a non-prime link, we have
\[  \Delta_n^{-4} <_D \beta <_D \Delta_n^4.
\]
\end{problem}

Since the Dehornoy ordering is computable, in the sense that there are algorithms for determining when a given braid is positive, the contrapositive of the last problem provides a test for primeness of a link.
\begin{corollary}
If either $\beta <_D \Delta_n^{-4}$ or $\beta >_D \Delta_n^4$, then $\hat{\beta}$ is a prime link.
\end{corollary}

We can also easily generate prime links this way---for example, if $\beta>_D1$ then $\widehat{\beta \Delta^4_n}$ is always a prime link.

The Dehornoy ordering can also be connected with other geometric and algebraic properties of knots, aside from primeness. For example, restricting our attention to braids $\beta$ for which $\hat{\beta}$ is a knot, the Dehornoy ordering is related to knot genus, which we denote by $g(K)$.

\begin{theorem}\cite{Ito11a}
\index{knot genus}
Suppose that $K = \hat{\beta}$ is a knot in $S^3$, $\beta \in B_n$, and that $k$ is the smallest integer such that $\Delta_n^{-2k-2} <_D \beta <_D \Delta_n^{2k+2}$.  Then $k < g(K) +1$.
\end{theorem}


\chapter{Groups of Homeomorphisms}
\label{groups of homeomorphisms chapter}

\def\PL{\textnormal{PL}}
\def\GL{\textnormal{GL}}
\def\Homeo{\textnormal{Homeo}}
\def\Diff{\textnormal{Diff}}
\def\fix{\textnormal{fix}}
\def\Id{\textnormal{Id}}


\section{Homeomorphisms of a space}

If $X$ is any topological space and $Y$ a closed subset of $X$, we consider the group $\Homeo(X, Y)$ of all homeomorphisms of $X$ onto itself which leave $Y$ pointwise fixed.  The group operation is composition of functions.  

In this chapter we will be considering manifolds of arbitrary dimension, especially concentrating on the $n$-dimensional cube.  It is based on the paper \cite{CR14}.

We'll first look at the case $n = 1$.  We saw in Example 
\ref{homeo+} that the group $\Homeo_+(\R)$ is left-orderable.   The space $\Homeo(I, \partial I)$ of homeomorphisms of the interval $I = [-1, 1]$ which are fixed on $\pm 1$ may be identified with $\Homeo_+(\R)$, so we conclude.

\begin{proposition}  
$\Homeo(I, \partial I)$ is left-orderable.  Moreover, every countable left-orderable group is isomorphic with some subgroup of $\Homeo(I, \partial I)$.
\end{proposition}

\begin{corollary}
$\Homeo(I, \partial I)$ is not bi-orderable or locally indicable.
\end{corollary}

This follows since there are plenty of countable left-orderable groups which are not locally indicable as observed by \cite{Bergman91}, but they all embed in $\Homeo(I, \partial I)$ by Theorem \ref{LO_universal}.  In fact, by Theorem \ref{braids not LI} the braid groups on 5 or more strands provide such examples.  

We may consider a natural subgroup of $\Homeo(I, \partial I)$, consisting of all functions which are {\em piecewise-linear}, \index{piecewise linear} meaning that there is a finite subdivision of $I$ for which the function is (affine) linear on each subinterval.  The subdivision may vary with the function, but the set 
$\PL(I, \partial I)$ of all piecewise-linear functions is easily seen to be a subgroup of 
$\Homeo(I, \partial I)$, that is, closed under composition and taking inverses.

\begin{proposition}
$\PL(I, \partial I)$ is bi-orderable.
\end{proposition}

In fact we can define the positive cone for a bi-ordering of by declaring that a function
$f \in \PL(I, \partial I)$ is positive if and only if its (right-sided) derivative satisfies  $f'(x_0) >1$ where $x_0$ is the maximum such that the restriction of $f$ to the interval $[-1, x_0]$ is the identity.  In other words, a PL function is positive if the first point at which its graph departs from the diagonal in $I \times I$ it goes above the diagonal.

\begin{problem}
Verify that the set defined above is actually a positive cone of a bi-ordering of 
$\Homeo(I, \partial I)$; that is, the composite of two positive functions is positive, any conjugate of a positive function is positive and for any nonidentity function, either it or its inverse is positive.
\end{problem}

It follows that $\PL(I, \partial I)$ is also locally indicable.  We will see that this generalizes to higher dimensions, but that bi-orderability does not.

\section{PL homeomorphisms of the cube}

If $n$ is a positive integer, the $n$-dimensional cube $I^n$ has boundary $\partial I^n$ which is homeomorphic with the $(n-1)$-dimensional sphere $S^n$.  A homeomorphism $f: I^n \to I^n$
will be called piecewise-linear (PL) if there is a finite triangulation of $I^n$ such that $f$ is (affine) linear on each simplex of the triangulation.

\begin{example}\label{example:Dehn_twist}
Figure~\ref{elementary_twist} depicts a PL homeomorphism of a square,
which preserves the foliation by concentric squares, and (for suitable choices
of edge lengths) whose 12th power is a \textit{Dehn twist}, \index{Dehn twist} that is a homeomorphism fixed on the inner square as well as the boundary of $I^2$ which ``twists'' the annulus in between.  Note that is is also area-preserving.
\begin{figure}[htpb]
\setlength{\unitlength}{4.5cm}
\begin{picture}(0,1)%
   \put(-1.2, 0){\includegraphics[height = \unitlength]{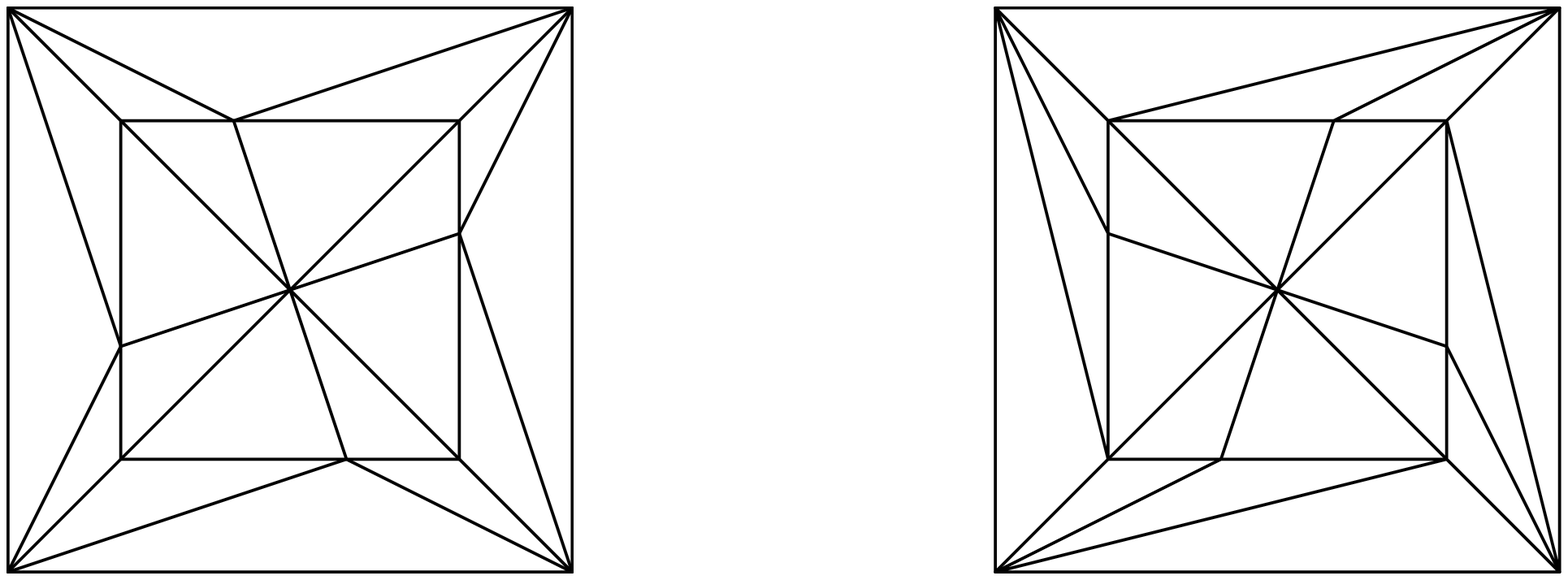}}%
    \put(-0.14, 0.5){$\xrightarrow{\hspace*{1cm}}$}%
  \end{picture}
\caption{A 12th root of a PL Dehn twist}\label{elementary_twist}
\end{figure}
\end{example}

\begin{theorem}\label{PLLI}
For any $n \ge 1$, the group $\PL(I^n, \partial I^n)$ is locally indicable, and therefore left-orderable.
\end{theorem}

Before proving this in the next section, we observe the following.

\begin{theorem}\label{theorem:PLnotBO}
For $n \ge 2$, the group $\PL(I^n, \partial I^n)$ is not bi-orderable.
\end{theorem}

We'll prove this for the case $n = 2$.  It will then follow for higher dimensions from Proposition \ref{proposition:suspension} below.
Consider two functions $f, g \in \PL(I^2, \partial I^2)$ defined 
as follows.  Let $h: I^2 \to I^2$ denote the function illustrated in Figure~\ref{elementary_twist}.  
Recall that $h^{12}$ is a Dehn twist, which is the identity on the inner square, as well 
as on $\partial I^2$.  Define $f$ to be the function $h^6$, so that $f$ rotates the inner square by 180 
degrees.  Referring to Figure~\ref{construction}, define $g$ to be the identity 
outside the little squares, which are strictly inside the inner square rotated by $f$.  
On the little square on the left, let $g$ act as $h$, suitably scaled, and on the 
little square to the right let $g$ act as $h^{-1}$.  Noting that $f$ interchanges 
the little squares, and that $h$ commutes with 180 degree rotation, one checks 
that $f^{-1}gf = g^{-1}$.

Such an equation cannot hold (for $g$ not the identity) in a bi-ordered group, as 
it would imply the contradiction that $g$ is greater than the identity if and only 
if $g^{-1}$ is greater than the identity (this is just as in the Klein bottle group).

\begin{figure}[htpb]
\setlength{\unitlength}{5.5cm}
\begin{picture}(0,1)%
   \put(-0.48, 0){\includegraphics[scale=0.6]{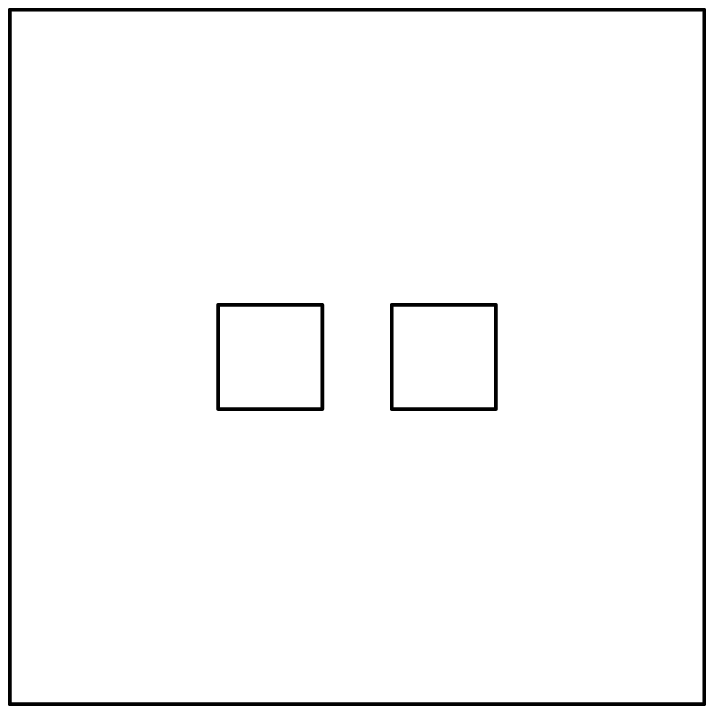}}%
    \put(-0.11, 0.46){$h$}%
    \put(0.05, 0.46){$h^{-1}$}
     \put(-0.03, 0.65){Id}
  \end{picture}
\caption{Building the function $g \in \PL(I^2, \partial I^2)$}\label{construction}
\end{figure}

\begin{proposition}\label{proposition:suspension}
For each $n \ge 1$ the group  $\Homeo(I^n, \partial I^n)$ is isomorphic with a subgroup of 
$\Homeo(I^{n+1}, \partial I^{n+1})$ and $\PL(I^n, \partial I^n)$ is isomorphic with a subgroup of 
$\PL(I^{n+1}, \partial I^{n+1})$.
\end{proposition}

The isomorphism may be constructed as follows.  Consider the cube $I^{n+1}$ as a subset of 
$\R^{n+1}$ with coordinates $(x_1, \dots ,x_{n+1})$ and let $I^n$ be the subset of $I^{n+1}$ with $x_{n+1} = 0$.  The suspension $\Sigma I^{n}$ can be embedded in  $I^{n+1}$ as the union of all straight line segments running from $I^n$ to the points $(0, \dots, 0, \pm 1)$.  Then a mapping 
$f: I^n \to I^n$ fixed on the boundary suspends naturally to a mapping 
$\Sigma f: \Sigma I^{n} \to \Sigma I^{n}$ fixed on the boundary.  Extend $\Sigma f$ by the identity to obtain the image of $f$ in $\Homeo(I^{n+1}, \partial I^{n+1})$.  This is easily checked to be an injective homomorphism, and it is clear that if $f$ is PL, then so is its image.
\qed

We note that all the self-homeomorphisms in this argument are actually area-preserving.  Moreover the suspension construction of Proposition \ref{proposition:suspension} takes area-preserving maps to area-preserving maps.  So we can refine Theorem \ref{PLLI} to apply to the subgroup of $\PL(I^n, \partial I^n)$ consisting of area-preserving PL homeomorphisms, which we denote by $\PL_\omega(I^n, \partial I^n)$.

\begin{theorem}\label{theorem:PLnotBO}
For $n \ge 2$, the group $\PL_\omega(I^n, \partial I^n)$ is not bi-orderable.
\end{theorem}

\section{Proof of Theorem \ref{PLLI}}

Consider an arbitrary {\em finitely generated} nontrivial subgroup $H = \langle h_1, \dots, h_k \rangle$ of $\PL(I^n, \partial I^n)$.  We will be done if we can find a nontrivial homomorphism from $H$ to some locally indicable group, for then its image will map nontrivially to $\Z$.  

Now the fixed point set $\fix(h_i)$ of each generator of $H$ is a simplicial complex in $I^n$ which contains $\partial I^n$.  The global fixed point set $\fix(H)$ is just the intersection of these sets, and so it is also a closed simplicial complex containing $\partial I^n$.  Since $H$ is nontrivial, there is a point in the complement of $\fix(H)$.  It is possible to connect this point to $\partial I^n$
by a simplicial curve that avoids the $(n-2)$-skeleton of the triangulation of $\fix(H)$.  Moving along this arc from $\partial I^n$, one can find a point $p$ in $\fix(H)$ and a neighborhood of 
$p$ for which every element of $H$ fixes an $(n-1)$-dimensional hyperplane $\pi$ in that neighborhood, every element of $H$ maps one side of $\pi$ linearly to the same side of $\pi$ and therefore preserves orientation.

Let $G$ be the group of germs of PL homeomorphisms in $H$ at $p$, so that maps are identified if they agree on some neighborhood of $p$.  There is an obvious homomorphism 
$H \to G$ which is nontrivial.  The group $G$ may be identified as a subgroup of $\GL(n,\R)$ consisting of linear maps which fix an $(n-1)$ dimensional subspace $\pi$ of $\R^n$ and preserve orientation.  Lemma \ref{lemma:linear_stabilizer} and Problem \ref{problem:linearLI} below imply that $G$ is locally indicable and the proof of Theorem \ref{PLLI} is complete.

\begin{lemma}\label{lemma:linear_stabilizer}
The subgroup of $\GL(n,\R)$ consisting of all orientation-preserving linear maps of $\R^n$ which pointwise fix an $(n-1)$ dimensional subspace $\pi$ of $\R^n$
is isomorphic to $\R^{n-1} \rtimes \R^*$.  The group operation on $\R^{n-1}$ is vector addition and $\R^*$ denotes the multiplicative group of positive real numbers.
\end{lemma}

To see this note that the subgroup of $\GL(n,\R)$ fixing $\pi$ can be conjugated into the set of 
matrices of the form $\left(\begin{smallmatrix} \Id & V \\ 0 & r \end{smallmatrix}\right)$
where $\Id$ is the identity in $\GL(n-1,\R)$, where $V$ is an arbitrary column vector
and where $r$ is a nonzero real number, which is positive in the orientation-preserving case. 

\begin{problem}\label{problem:linearLI}  Conclude the proof of the lemma by showing that matrix multiplication defines  a semidirect product structure on the set of pairs (V, r).  Use Problem \ref{LI extension} to argue that $\R^{n-1} \rtimes \R^*$ is locally indicable because each of its factors is locally indicable.
\end{problem}

\section{Generalizations}

Theorem \ref{PLLI} may be generalized in several directions, which we will discuss without proof, referring the reader to \cite{CR14} for details.   Our proof in the last section did not actually use the fact that the space was a cube, and applies more generally to piecewise-linear manifolds.  A PL $n$-manifold is defined as a space covered with charts homeomorphic with
$\R^n$ in such a way that the transition functions between charts may be chosen to be piecewise-linear.  One can then define PL homeomorphisms in a natural way.  A smooth $n$-manifold is defined similarly, with transition functions that are infinitely differentiable.  This setting enables the definition of differentiable maps on smooth manifolds.

 \begin{theorem}
Let $M$ be an $n$ dimensional connected $\PL$ manifold, and let $K$ be a
nonempty closed $\PL$ $(n-1)$-dimensional submanifold.
Then the group $\PL_+(M,K)$ of orientation-preserving PL homeomorphisms of $M$, fixed on $K$, is locally indicable and therefore left-orderable.
\end{theorem}

If $M$ is a smooth manifold and $K$ a smooth submanifold, $\Diff_+^1(M,K)$ denotes the group of {\em orientation-preserving} self-homeomorphisms of $M$ which are fixed on $K$ and which are continuously (once) differentiable.

\begin{theorem}\label{theorem:cilo}
Let $M$ be an $n$ dimensional connected smooth manifold, and let $K$ be a
nonempty $n-1$ dimensional closed submanifold. Then the group $\Diff_+^1(M,K)$ is
locally indicable and therefore left-orderable.
\end{theorem}

Since each group $\Diff_+^p(M,K)$ of $p$-times continuously differentiable homeomorphisms 
($p = 2, 3, \dots, \infty$) is a subgroup of $\Diff_+^1(M,K)$, these groups are also locally indicable.

\section{Homeomorphisms of the cube}

It is natural to ask if the results on PL and differentiable homeomorphisms of the cube (and other manifolds) apply to the more general setting of homeomorphisms which are just continuous (with continuous inverses).  Even for the 2-dimensional cube, the question is open at the time of this writing.

\begin{question} Is $\Homeo(I^2, \partial I^2)$ left-orderable?
\end{question}

However, it is a classical result that this group is torsion-free, and that also generalizes to higher dimensions.

\begin{theorem}
The group $\Homeo(I^n, \partial I^n)$ is torsion-free for all $n$.
\end{theorem}

We'll just sketch the proof, again referring to \cite{CR14} for details.  Suppose 
$f \in \Homeo(I^n, \partial I^n)$ has finite order, that is $f \ne \Id$ but $f^p = \Id$ for some $p$.  By possibly passing to a power of $f$ we can assume $p$ is prime.  Then P. A. Smith theory implies that the fixed point set $\fix\langle f \rangle$ of  the cyclic subgroup generated by $f$ has the mod $p$ homology of a point.  However $\partial I^n$ is a subset of $\fix\langle f \rangle$ and represents a nontrivial $(n-1)$-dimensional homology cycle.  Since it must bound, mod $p$, we conclude that $\fix\langle f \rangle$ is all of $I^n$, contradicting $f \ne \Id$. \qed


%
%
%
\chapter{Conradian left-orderings and local indicability}
\label{Conrad chapter}

In this chapter we will introduce a special type of left-ordering, called a Conradian left-ordering.  Conradian left-orderability is a more restrictive condition than left-orderability, in the sense that there are left-orderable groups that are not Conradian left-orderable, but weaker than bi-orderability since there are Conradian left-orderable groups that are not bi-orderable.  Our goal is to show that Conradian left-orderability is equivalent 
to local indicability.  Although this has appeared in the literature, our approach seems new.

To begin, we recall the notion of a convex jump, which was introduced in Chapter \ref{Holder chapter}.  Having fixed a left-ordering of a group, a \textit{convex jump} \index{convex jump} is a pair of distinct convex subgroups $(C, D)$ such that $C \subset D$ and there are no convex subgroups strictly between them.

\begin{problem}
\index{Klein bottle group} 
\label{klein_convex}
Recall the Klein bottle group $K= \langle x, y : xyx^{-1} = y^{-1} \rangle$ from Problem \ref{klein}.  Show that the subgroup $\langle y \rangle$ is convex in every left-ordering of $K$.
(Hint: Every element of $K$ can be written in the form $x^m y^n$ for some $m, n \in \Z$, by using the relations $y^{-1}x = xy$ and $yx = xy^{-1}$ to shuffle all the occurences of $x$ to one side).
\end{problem}

\begin{example}
\label{Klein_bottle_orderings}
By Problems \ref{klein_convex} and \ref{extension_generalization}, every left-ordering of $K$ is lexicographically defined from the short exact sequence
\[ 1 \rightarrow \langle y \rangle  \rightarrow K \rightarrow K/\langle y \rangle \rightarrow 1.
\]
Recall from  Problem \ref{klein} that $K/\langle y \rangle$ is infinite cyclic, so $K$ admits only four left-orderings since the kernel and the quotient each admit only two. In particular, in every left-ordering of $K$ the convex subgroups are $\{1\} \subset \langle y \rangle \subset K$.
Therefore  $(\{1 \}, \langle y \rangle )$ and $(\langle y \rangle, K)$ are the only convex jumps in a left-ordering of $K$. 
\end{example}

A Conradian left-ordering will be a kind of ordering built from Archimedean orderings of abelian quotients, similar to the way that the orderings of the Klein bottle group are built from the abelian subgroup $\langle y \rangle $ and the abelian quotient $K /\langle y \rangle$ in the previous example. To formalize this idea, we need the notion of a \index{Conradian jump}  \textit{Conradian jump}.  Suppose that $G$ is a left-ordered group with ordering $<$, that $(C,  D)$ is a convex jump such that $C$ is normal in $D$, and the natural ordering of $D/C$ is Archimedean.  In this case the convex jump $(C,D)$ is called a Conradian jump.  By Theorem \ref{theorem:Holder} there exists an order-preserving embedding 
\[ \tau: D/C \rightarrow (\mathbb{R}, +),
\]
we will refer to $\tau$ as the \textit{Conrad homomorphism} \index{conrad homomorphism} associated to the Conradian jump $(C, D)$.  First introduced by Conrad in \cite{Conrad59}, a \textit{Conradian left-ordering} \index{Conradian ordering} is a left-ordering all of whose convex jumps are Conradian.

\begin{example}
\index{Klein bottle group} 
By Example \ref{Klein_bottle_orderings}, every left-ordering of the Klein bottle group $K= \langle x, y : xyx^{-1} = y^{-1} \rangle$ has convex subgroups $\{1 \}, \langle y \rangle $, and $G$.  Since $\langle y \rangle$ and $G/\langle y \rangle$ are infinite cyclic groups, their orderings are always Archimedean, and so every convex jump of a left-ordering of the Klein bottle group is a Conradian jump. Therefore  every left-ordering of the Klein bottle group is Conradian.
\end{example}

\begin{example}
If $G$ is a finitely generated bi-ordered group, then the largest proper convex subgroup $C \subset G$ gives rise to an abelian quotient $G \rightarrow G/C \subset (\R, +)$.  By the arguments in Section \ref{section: BO implies LI}, the pair $(C, G)$ is a Conradian jump and the quotient map is the Conrad homomorphism.
\end{example}

\section{The defining property of a Conradian ordering}

The Conradian left-orderings of a group $G$ can also be characterized by asking that a certain inequality holds for all pairs of positive elements in $G$.   This recharacterization allows for a sort of ``local definition" of Conradian left-orderings, so that one can easily deduce that subgroups of Conradian left-ordered groups are Conradian left-ordered, etc.

\begin{theorem}\cite{Conrad59}
\label{theorem:conradequiv}
For a left-ordering $<$ of a group $G$, the following are equivalent:
\begin{enumerate}
\item $<$ is a Conradian left-ordering.
\item For every pair of positive elements $g, h \in G$ there exists $n>0$ such that $g<hg^n$.
\end{enumerate}
\end{theorem}
\begin{proof}
Suppose that $<$ is a Conradian left-ordering of $G$, and let $g, h >1$ be given.  Define two subgroups 
\[ C_g = \bigcup_{\substack{\text{$C$ convex}\\\text{$g \notin C$}}} C \mbox{ and } D_g =  \bigcap_{\substack{\text{$D$ convex}\\\text{$g \in D$}}} D,
\]
both $C_g$ and $D_g$ are convex by Problem \ref{intersection_and_union}.
It follows from their definitions that $(C_g,  D_g)$ is a convex jump.  Similarly, we can define the convex jump $(C_h, D_h)$ corresponding to $h$.  Since $<$ is a Conradian left-ordering of $G$, for each of the convex jumps $(C_g, D_g)$ and $(C_h, D_h)$, there are corresponding Conrad homomorphisms
\[ \tau_g : D_g /C_g \rightarrow (\R, +) \mbox{ and } \tau_h : D_h /C_h \rightarrow (\R, +).
\]
\begin{problem}
\label{problem:conrad_homom_pos}
Show that at least one of the homomorphisms $\tau_g, \tau_h$ satisfies 
\[ \tau(g^{-1} h g^n) >0
\]
for $n >0$ sufficiently large.  Since Conrad homomorphisms are order-preserving, conclude (1) implies (2).
(Hint: Recall from \ref{prop:convex_intersection_and_union} that convex subgroups are linearly ordered by inclusion.  Therefore either $C_g \subset D_g \subseteq C_h \subset D_h$, or $C_h \subset D_h \subseteq C_g \subset D_g$, or $C_g = C_h$ and $D_g = D_h$.  In each of the first two cases argue that the Conrad homomorphism of the larger convex jump satisfies $ \tau(g^{-1} h g^n) >0$, in the third case either Conrad homomorphism will do.)
\end{problem}

Conversely, suppose that $<$ is an ordering of $G$ that satisfies (2).  To begin, we have a lemma:

\begin{lemma}
\label{lemma:conrad_archimedean}
Suppose that a nontrivial group $G$ admits a left-ordering $<$ that satisfies property (2), and that $(C, D)$ is a convex jump.  Then for every pair of positive elements $g, h$  in $D \setminus C$, there exists $n>0$ such that $g^n>h$.  In particular, if $(G,<)$ satisfies (2) and contains no proper, nontrivial convex subgroups, then $<$ is Archimedean.
\end{lemma}
\begin{proof}
The following proof is due to Herman Goulet-Ouellet.
Fix a positive element $g \in D \setminus C$ and set $S_0 = \{ x \in G \mid \mbox{there exists } n>0 \mbox{ such that } x < g^n \}$.  By property (2), if $y \in G$ is positive then for every $k>0$ there exists $n_k>0$ such that $g^k < y(g^k)^{n_k}$.  Now suppose $x \in S_0$, and choose $k>0$ such that $x<g^k$.  Left-multiplying by $y^{-1}$ gives $y^{-1}x<y^{-1}g^k < (g^k)^{n_k}$, which shows that $y^{-1}x \in S_0$.  We conclude that $S_0 \subset y S_0$, and this holds for every positive $y \in G$. We use this fact repeatedly in what follows.

Define $ \mathcal{X} = \{ S \subset G \mid x \in S \mbox{ and } y<x \Rightarrow y \in S \}$, the set $\mathcal{X}$ is linearly ordered by inclusion. Moreover for $f \in G$ and $S \subset \mathcal{X}$ the rule $f S = \{ fx \mid x \in S \}$ defines an effective, order-preserving action of $G$ on $\mathcal{X}$.  Denote the stabilizer of $S_0 \in \mathcal{X}$ by $G_{S_0}$.

First, we show that the stabilizer $G_{S_0}$ is convex in our given left-ordering.  To this end, suppose $x<y<z$, with $x, z \in G_0$.  Then $x<y$ implies $1<x^{-1}y$, and so $S_0 \subset (x^{-1}y)S_0$, from the first paragraph.  But then $xS_0 \subset yS_0$, and since $xS_0 = S_0$ this gives $S_0 \subset yS_0$.  A similar argument shows that $yS_0 \subset zS_0 = S_0$, so that $yS_0 = S_0$.

Now note that $g \in G_{S_0}$, which we will use to show that $C \subset G_{S_0}$.  Given $c \in C$, since $C$ is convex and $g>1$ with $g \notin C$, we have $c^{\pm 1} <g$.  Therefore $1<c^{\pm 1} g$, which gives $S_0 \subset c^{\pm 1} g S_0$, or equivalently, $c^{\pm 1}S_0 \subset g S_0 = S_0$.  On the other hand, if $c>1$ then $S_0 \subset cS_0$ and we conclude $cS_0 = S_0$; if $c<1$ then $c^{-1} >1$ and $S_0 \subset c^{-1}S_0$ gives $c^{-1}S_0 = S_0$.  In either case, $c \in G_{S_0}$ and we conclude $C \subset G_{S_0}$.

Finally we note that $G_{S_0} \subset D$, and since $(C, D)$ is a convex jump this forces $D = G_{S_0}$.  But then every positive $h \in D$ is in $G_{S_0}$, in particular since $hS_0 = S_0$ and $1 \in S_0$ we have $h \in S_0$.  This means there exists $n>0$ such that $h<g^n$, as required.
\end{proof}

The next lemma will allow us to finish the proof of Theorem \ref{theorem:conradequiv}.

\begin{lemma}
\label{lemma:conrad_normal}
Suppose that $G$ is a left-orderable group with ordering $<$ that satisfies (2).  If $(C,D)$ is a convex jump, then $C$ is normal in $D$.
\end{lemma}
\begin{proof}

Choose $h \in D \setminus C$ with $h>1$.  For an arbitrary $g \in C$, we wish to show that $hgh^{-1} \in C$.   Supposing that $hgh^{-1} \notin C$ and $hgh^{-1} >1$, by Lemma \ref{lemma:conrad_archimedean} there exists $n>0$ such that $hg^nh^{-1} >h$.  But then left multiplying by $g^{-n}h^{-1}$, the inequality  $hg^nh^{-1} >h$ gives $h^{-1} >g^{-n}$.  Recalling that $h>1$, we have $1>h^{-1} >g^{-n}$, and hence $h^{-1} \in C$ by convexity of $C$.  This is a contradiction.

On the other hand if $hgh^{-1} \notin C$ and $hgh^{-1} <1$, by a similar calculation we arrive at $1>h^{-1}>g^n$ so that again $h^{-1} \in C$, a contradiction.  We conclude that $hCh^{-1} \subset C$ for all $h>1$.

Now we show $C \subset hCh^{-1}$, by showing $h^{-1} C h \subset C$.   First we consider an arbitrary element $c \in C$ and prove a short lemma.  Note that $1<c^{-1} h$, so by Lemma  \ref{lemma:conrad_archimedean}, there exists $n$ such that $h < (c^{-1}h)^n$.  Therefore $1< h^{-1}(c^{-1}h)^n$, and $1< h^{-1}(c^{-1}h)^n h$ since it is a product of the positive elements $h^{-1}(c^{-1}h)^n$ and $h$.   Now $h^{-1} (c^{-1} h )h$ is positive since its $n$-th power is positive, so we take $h^{-1} (c^{-1} h )h>1$ and left-multiply by $h^{-1}ch$ to get $h^{-1} c h < h$.  Since $c \in C$ was arbitrary, we have as a lemma that $h^{-1} c h < h$ for all $c \in C$.

Now let $g \in C$ and suppose $h^{-1} g h>1$, the case of $h^{-1} g h<1$ is similar.  Since $g^{ k} \in C$ for all $k>0$, by the previous paragraph $h^{-1} g^{k} h < h$ for all $k>0$.   But $1<h^{-1} g^{k} h =(h^{-1} g h)^k < h$ for all $k>0$ forces $h^{-1} g h \in C$, by Lemma \ref{lemma:conrad_archimedean}.  We conclude that $h^{-1} C h \subset C$.

Thus $h C h^{-1} =C$  for all positive $h \in D\setminus C$, it follows that $C$ is normal.
\end{proof}

Now let $(C,D)$ be a convex jump relative to our ordering $<$ of $G$ that satisfies (2), we check that the jump is Conradian.  By Lemma \ref{lemma:conrad_normal} $C$ is normal in $D$, and since the ordering of $G$ satisfies (2),  the natural ordering of the quotient $D/C$ also satisfies (2) because the quotient map $D \rightarrow D/C$ is order preserving.  Moreover, since $(C,D)$ is a convex jump the quotient $D/C$ (under the natural ordering) admits no proper convex subgroups.   Thus we are in a position to apply Lemma \ref{lemma:conrad_archimedean} to the quotient ordering of $D/C$, which tells us that the natural ordering of the quotient is Archimedean.  Thus, $(C,D)$ is a Conradian jump, and $<$ is a Conradian left-ordering of $G$.
\end{proof}

Theorem \ref{theorem:conradequiv} also shows that every bi-ordering is Conradian, because in a bi-ordering $g<hg^n$ is true for every positive pair of elements $g, h \in G$ by  taking $n=1$.  We can also conclude that any subgroup $H$ of a Conradian ordered group $G$ admits a Conradian ordering, since property (2) is obviously inherited by subgroups, whereas this is not so obviously the case with property (1).

\begin{problem}
\label{finite index Conradian}
Let $G$ be a left-orderable group, and $H$ a finite index subgroup.  Show that any left-ordering of $G$ which restricts to a bi-ordering of $H$ is Conradian.  In particular, if $|G : H | =n$, then for all pairs of positive elements $g, h \in G$ we have $g < hg^n$.

As an application of this fact, it follows that no bi-ordering of the pure braid group $P_n$ can be extended to a left-ordering of the braid group $B_n$ for $n \geq 5$.  What about for $n=3,4$?
\end{problem}

\begin{problem}
Show that the Dehornoy ordering  of $B_n$ is not Conradian for $n \geq 3$ by finding positive braids $\beta_1$ and $\beta_2$ that satisfy $\beta_1>_D \beta_2 \beta_1^2$.
\end{problem}

We see that $B_3$ and $B_4$ are examples of groups which admit non-Conradian left orderings, but also, because of Corollary 
\ref{conradian_iff_LI} below, they have Conradian orderings too, as they are locally indicable, according to Theorems \ref{B_3 LI} and \ref{B_4 LI}.

\section{Characterizations of Conradian left-orderability}

It is a remarkable fact that $n=2$ suffices in Theorem \ref{theorem:conradequiv} \cite{Navas10a}.  This can be deduced from Problem \ref{problem:conrad_homom_pos}, where one verifies that $\tau(g^{-1} h g^n) >1 $ for sufficiently large $n$.  Since $\tau$ is order-preserving and its image lies in $\R$, both $\tau(h)$ and $\tau(g)$ are non-negative real numbers so $n =2$ is all we need.  One can also argue directly from property (2) in Theorem \ref{theorem:conradequiv}  that $g<hg^2$ for all $g,h >1$.

\begin{proposition}
 \cite[Proposition 3.7]{Navas10a}
Suppose that $<$ is a Conradian left-ordering of a group $G$.  If $g, h >1$ then $g< hg^2$.
\end{proposition}
\begin{proof}
For contradiction, suppose that $hg^2 <g$.  Then $g^{-1} h g^2 = (g^{-1} h g) g < 1$, so $g^{-1} h g$ must be negative since $g$ is positive.  Thus we also have $g^{-1} h g <1$, or $hg <g$.
Now consider the positive elements $h$ and $x = hg$.  For all $n \geq 0$ we calculate
\begin{eqnarray}
hx^n &=&  h(hg)^{n-2}(hg)(hg)   \nonumber \\
   &< & h(hg)^{n-2}(hg)g  \nonumber \\
   & = &h(hg)^{n-2}(hg^2) \nonumber \\
 &<& h(hg)^{n-2}g  \nonumber \\
 & \cdots  &   \nonumber \\
 &<& hg =x  \nonumber
\end{eqnarray}
Thus the positive elements $h, x$ do not satisfy the Conradian condition, a contradiction.
\end{proof}

Using $n=2$, one can recharacterize Conradian left-orderability in way similar to the recharacterization of left-orderability given in Theorem \ref{loequiv}.   To begin with, suppose that $G$ is generated by a set $S \subset G$, and let $G_k$ denote all the elements of $G$ that can be expressed as a word of length $k$ or less in these generators.  Define a \textit{Conradian proper $k$-partition} of $G$ to be a proper $k$-partition $Q \subset G_k$ that also satisfies: if $g, h \in Q$ and $g^{-1} h g^2 \in G_k$, then $g^{-1} h g^2 \in Q$.

With the idea of a Conradian proper $k$-partition in place of a proper $k$-partition, one can prove Conradian versions of most of the theorems in Section 1.6.  With the exception of the Burns-Hale theorem, very few modifications are required.

\begin{theorem}
\label{conradian_k_partition}
If $G$ is generated by $S \subset G$, then $G$ is Conradian left-orderable if and only if there is a Conradian proper $k$-partition of $G_k$ for all $k \geq 1$.
\end{theorem}
\begin{problem}
\label{problem_conradian_k_partition}
Let $\mathcal{C}_k$ denote all the subsets of $G$ whose intersection with $G_k$ is a Conradian proper $k$-partition.  Prove that $\mathcal{C}_k$ is a closed subset of the power set $\mathcal{P}(G)$, and then prove Theorem \ref{conradian_k_partition} by following the proof of Theorem \ref{partition}.  
\end{problem}

It is in Problem \ref{problem_conradian_k_partition} that $n=2$ is really essential. If one tries to use the `standard' definition of a Conradian ordering, namely for every pair of positive elements $g, h \in G$ there exists $n>0$ such that $g<hg^n$, a problem arises when trying to show that $\mathcal{C}_k$ is closed.

\begin{problem}
Recall the topology on the power set $\mathcal{P}(G)$ introduced in Section ~\ref{topology on power set}, whose subbasis consists of sets $U_x$ and $U_x^c$. Show that a set $S$ in $\mathcal{P}(G)$ that violates the condition: 
\[ \mbox{for all pairs of positive elements $g, h \in S$ 
 there exists $n>0$ such that $g^{-1}hg^n \in S$}
\]
naturally belongs to an infinite intersection of subbasic open sets, by rewording the negation of the condition above in terms of the sets $U_x$ and $U_x^c$.  As an infinite intersection of subbasic open sets need not be open, this does not provide an open neighbourhood of $S$.
\end{problem}

In the case of Conradian left-orderings, as in the case of left-orderings, it is not necessary that the group under discussion be finitely generated.
\begin{theorem}
A group $G$ is Conradian left-orderable if and only if each of its finitely generated subgroups is Conradian left-orderable.
\end{theorem}
The proof of this theorem is more or less identical to the proof of Theorem \ref{fgLO}, so we omit it.  

For a subset $X \subset G$, let $C(X)$ denote the smallest subsemigroup of $G$ satisfying $X \subset C(X)$ and for all $x, y \in C(X)$, $x^{-1}yx^2 \in C(X)$.   The subsemigroups $C(X)$ of a group $G$ will replace the the standard subsemigroups $S(X)$ of $G$ in Theorems \ref{loequiv} and \ref{burnshale}, which we prove now in the Conradian case.

\begin{theorem} \cite{Navas10a}
\label{conradequiv}
A group $G$ admits a Conradian left-ordering if and only if for every finite subset  $ \{x_1, \dots , x_n \}$ of $G$ which does not contain the identity, there exist $\e_i = \pm 1$ such that $1 \not\in C( \{x_1^{\e_1}, \dots , x_n^{\e_n} \} )$.\end{theorem}
\begin{proof}
For one direction, suppose that $G$ admits a Conradian left-ordering with positive cone $P$.  Given a finite subset $ \{x_1, \dots , x_n \}$ of $G$ that does not contain the identity, choose exponents $\e_i = \pm 1$ so that $x_i^{\e_i} \in P$ for all $i$.  Then $P$ is a subsemigroup of $G$ containing $\{x_1^{\e_1}, \dots , x_n^{\e_n} \} $ but not $1$, and for all  $x, y \in P$, $x^{-1}yx^2 \in P$.  Since $C( \{x_1^{\e_1}, \dots , x_n^{\e_n} \} )$ is the smallest semigroup with these properties, we conclude $C( \{x_1^{\e_1}, \dots , x_n^{\e_n} \} ) \subset P$,  so that $1$ is not in $C( \{x_1^{\e_1}, \dots , x_n^{\e_n} \} )$.

For the other direction, consider the finite subset $G_k \setminus \{1 \}$ of  $G$, which we'll enumerate as $\{ x_1, \ldots, x_n \}$.  Choose exponents $\epsilon_i = \pm 1$ so that $1 \notin C(\{ x_1^{\epsilon_1}, \ldots, x_n^{\epsilon_n} \})$.  Then $C(\{ x_1^{\epsilon_1}, \ldots, x_n^{\epsilon_n} \}) \cap G_k$ is a Conradian proper $k$-partition.  It follows that $G$ is Conradian left-orderable, by Theorem \ref{conradian_k_partition}.
 \end{proof}

\index{Burns-Hale theorem}
\begin{theorem} {\rm(Burns-Hale for Conradian left-orderable groups, \cite{Navas10a})}
\label{theorem:conradian_burns_hale}
A group $G$ admits a Conradian left-ordering if and only if for every finitely generated subgroup $H \neq \{1 \}$ of $G$, there exists a Conradian-left-orderable group $K$ and a nontrivial homomorphism $H \rightarrow K$.
\end{theorem}
\begin{proof}
One direction is clear, since every subgroup of a Conradian left-ordered group is Conradian left-ordered.

For the other direction, we need to prepare a technical lemma in order to apply Theorem \ref{conradequiv}.

\begin{lemma}
\label{conrad_BH_lemma}
Let $X, Y$ denote subsets of $G$ that do not contain the identity.  If $h:\langle X, Y \rangle \rightarrow K$ is a homomorphism onto a Conradian left-ordered group $K$ satisfying $h(x) > 1$ for all $x \in X$ and $h(y) = 1$ for all $y \in Y$, then for all $g \in C(X \cup Y)$ we have $h(g ) \geq 1$, with $h(g) =1$ if and only if $g \in C(Y)$.
\end{lemma}
\begin{proof}
For any subset $W$ of $G$, define $S_C(W)$ to be the subsemigroup generated by elements of $W$, and elements of the form $v^{-1} u v^2$ with $u, v \in W$.   Now set $C_0(W) = S_C(W)$, and for $i > 0$ define $C_i = S_C( C_{i-1}(W))$.  The union $\bigcup_{i=0}^{\infty} C_i(W)$ is a subsemigroup of $G$ containing $W$, and for all $u, v \in \bigcup_{i=1}^{\infty} C_i(W)$, the element $v^{-1} u v^2$ is in the union as well.  Thus 
\[C(W) \subset \bigcup_{i=1}^{\infty} C_i(W)
\]
since $C(W)$ is the smallest subsemigroup of $G$ with these properties, and in fact $C(W) = \bigcup_{i=1}^{\infty} C_i(W)$.  We now proceed to prove the lemma by induction.

As a base case, let $g \in S_C(X \cup Y) = C_0(X \cup Y)$ be given.  By definition of $ S_C(X \cup Y)$, $g$ is a word in positive powers of the generating set $W$ where
\[ W = X \cup Y \cup \{ v^{-1} u v^2 : u, v \in X \cup Y \}.
\]
 Since $h(w) \geq 1$ for all $w \in W$, it follows that $h(g) \geq 1$.  Assuming $h(g) = 1$, then $g$ cannot contain any occurences of $w \in X$, or $w = v^{-1} u v^2$ with either $u \in X$ or $v \in X$, for in either of these cases $h(w) >1$.  Thus $g$ contains only occurences of $w \in Y$, and $w =  v^{-1} u v^2$ with $u, v \in Y$.  Therefore $g \in C(Y)$.

Now assume that for all $w \in C_i(X \cup Y)$, $h(w) \geq 1$ with $h(w) =1$ if and only if $w \in C_i(Y)$.   
Let $g \in C_{i+1}(X \cup Y)$ be given, $g$ is a word in terms $w \in C_i(X \cup Y)$ and $w = v^{-1} u v^2$ with $u, v \in C_i(X \cup Y)$.  In both cases, $h(w) \geq 1$, so $h(g) \geq 1$.

Assuming $h(g) =1$, we consider the two types of terms $w$ appearing in the expression for $g$.  First, if $w \in C_i(X \cup Y)$ then $h(w) = 1$ implies $w \in C_i(Y) \subset C_{i+1}(Y)$ by the induction assumption.  Second, if $w = v^{-1} u v^2$ with $u, v \in C_i(X \cup Y)$ then $h(w) =1$ implies $h(u) = h(v) =1$ so that $u, v \in  C_i(Y)$.   Thus $w = v^{-1} u v^2 \in C_{i+1}(Y)$, and it follows that $g \in C_{i+1}(Y)$. The lemma now follows by induction.
\end{proof}
Now we can complete the proof of Theorem \ref{theorem:conradian_burns_hale} by showing that the subgroup condition implies that $G$ is Conradian left-orderable, and using Theorem \ref{conradequiv}.  We proceed by induction on $n$, as in the proof of the Burns-Hale theorem.  The claim to prove is that for every finite subset $\{ x_1, \ldots , x_n \}$ of $G$, there are exponents $\epsilon_i = \pm 1$ such that $1 \notin C(\{ x_1, \ldots , x_n \})$.

If $n=1$ then $C(\{ x_1 \}) $ is equal to $S(\{ x_1 \})$, and this semigroup only contains the identity if $x_1$ has finite order.  This is not possible, because by assumption there is a nontrivial homomorphism from the cyclic group $\langle x_1 \rangle$ onto a Conradian left-orderable group.

Now suppose that the claim holds for all finite sets having fewer than $n$ elements and which do not contain the identity, and consider a set  $\{ x_1, \ldots , x_n \}$ of $n$ non-identity elements in $G$.  There exists a nontrivial homomorphism
\[ h : \langle x_1, \ldots, x_n \rangle \rightarrow K
\]
where $(K, \prec)$ is a Conradian left-orderable group.  Index the elements $\{ x_1, \ldots , x_n \}$ so that 
$$h(x_i) 
  \begin{cases}
   \ne 1 \text{ if } i= 1, \dots, r, \\
   = 1  \text{ if } r < i \le n.
   \end{cases}
$$   
and choose $\epsilon_i = \pm 1$ so that $1 \prec h(x_i^{\epsilon_i})$ for $i = 1 , \ldots, r$.  By the induction hypothesis, we can choose $\epsilon_i = \pm 1$ for $i = r+1, \ldots, n$ so that $1 \notin C(\{x_{r+1}^{\epsilon_{r+1}}, \ldots, x_{n}^{\epsilon_{n}} \})$.  Now choose $x \in  C(\{x_{1}^{\epsilon_{1}}, \ldots, x_{n}^{\epsilon_{n}} \})$, by Lemma \ref{conrad_BH_lemma} either $1 \prec h(x)$ or $h(x)=1$ and $x \in C(\{x_{r+1}^{\epsilon_{r+1}}, \ldots, x_{n}^{\epsilon_{n}} \})$.  In either case $x \neq 1$, so the claim follows by induction.

\end{proof}

Using the Conradian version of the Burns-Hale theorem, we can prove a deep theorem originally due to Brodskii \cite{Brodskii84}, with later proofs appearing in \cite{RR02} and \cite{Navas10b}.  First we observe that every nontrivial finitely generated Conradian left-orderable group has at least one torsion free abelian quotient.

\begin{proposition}
\label{prop:conradquotient}
If  $G$ is a nontrivial finitely generated Conradian left-ordered group, then there exists a proper normal convex subgroup $C \subset G$ such that $G/C$ is a finitely generated torsion free abelian group.
\end{proposition}
\begin{proof}
Suppose that $G$ has the Conradian left-ordering $<$ and is generated by $x_1, \ldots, x_n$.  We may suppose (by replacing a generator with its inverse if necessary) that $x_i >1$ for all $i$.  

Suppose without loss of generality that $x_n$ is the largest generator.  Then the subgroup
\[ C = \bigcup_{\substack{\text{$D$ convex}\\\text{$x_n \notin D$}}} D
\]
is convex, and in fact $(C,G)$ is a convex jump: any convex subgroup strictly larger than $C$ would have to contain $x_n$, and thus all of the generators of $G$ by convexity.
%
%
Since $<$ is a Conradian left-ordering the jump $(C, G)$ is Conradian, so the quotient $G/C$ admits an Archimedean ordering and there is an embedding
\[G/C \rightarrow (\R , +).
\]
Hence $G/C$ is torsion free and abelian.
\end{proof}

\begin{problem}
Let $G$ be a finitely generated Conradian left-orderable group, and suppose that the abelianization $G/[G, G]$ is infinite cyclic.  Show that $[G, G]$ is convex in every Conradian left-ordering of $G$.
\end{problem}

Recall that a group $G$ is said to be \textit{locally indicable} if every finitely generated nontrivial subgroup $H$ of $G$ admits a surjection $H \rightarrow \Z$.

\begin{corollary} \cite{Brodskii84,RR02,Navas10b}
\label{conradian_iff_LI}
\index{locally indicable}
A group $G$ admits a Conradian left-ordering if and only if $G$ is locally indicable.
\end{corollary}
\begin{proof}

Suppose $H$ is a finitely generated subgroup of the Conradian ordered group $G$.  Since $H$ is also Conradian, by Proposition \ref{prop:conradquotient} there exists a map of $H$ onto a nontrivial finitely generated torsion-free abelian group and hence onto $\Z$.
On the other hand, if there exists a surjective map $H \rightarrow \Z$ for every finitely generated subgroup $H$ of $G$, then $G$ admits a Conradian left-ordering by Theorem \ref{theorem:conradian_burns_hale}.
\end{proof}

Note that the obvious generalization of the Burns-Hale theorem to bi-orderable groups is not true: Locally indicable groups already satisfy the property that every finitely generated subgroup maps onto a bi-orderable group, yet not all locally indicable groups are  bi-orderable.


\chapter{Spaces of orderings}
\label{space of orderings chapter}

The space $LO(G)$ \index{LO@$LO(G)$} of all left-orderings of a group $G$ has already been introduced in Section \ref{space of orderings section}.  Recall that by identifying an ordering with its positive cone, we can regard $LO(G)$ as a subset of the power set $\mathcal{P}(G)$ of the set of elements of the group.  Since $\mathcal{P}(G)$ can be naturally identified with the product $\prod_{g \in G} \{ 0,1 \}$, we can equip it with the Tychonoff topology\index{Tychonoff topology}.  We saw that the Tychonoff topology has as a subbasis the sets
$$V_g = \{A \subset G : g \in A\} \quad \text{ and } \quad V_g^c = \{A \subset G : g \notin A\}$$
and so as a subspace of $\mathcal{P}(G)$, the topology on $LO(G)$ has as a subbasis all sets of the form
\[
\begin{split}
U_g = V_g \cap LO(G) =  \{P \in LO(G) : g \in P \} \quad  \\
\quad U_g^c=V_g^c \cap LO(G) = \{P \in LO(G) : g^{-1} \in P  \}
\end{split}
\]
 This natural choice of topology makes $LO(G)$ a closed subset, and therefore it is a compact space.  It is also totally disconnected since $\mathcal{P}(G)$ is totally disconnected.  Moreover, if $G$ is countably infinite then $LO(G)$ is metrizable, and so by a theorem of Brouwer is homeomorphic with the Cantor set if and only if it contains no isolated points.  A metric on $LO(G)$ in case $G$ is a countable group is constructed in Problem \ref{Tychonoff topology}.

%

 In this context, an isolated point in $LO(G)$ corresponds to a left-ordering that is the unique left-ordering satisfying some finite string of inequalities.  Equivalently, one can left-multiply the inequalities as necessary and see that an isolated ordering is the unique left-ordering of $G$ in which some finite set of elements $\{g_1, \ldots, g_n\}$ are all positive.  Thus an isolated ordering is one whose positive cone $P$ satisfies
 \[ \{ P \} =  \bigcap_{i=1}^n U_{g_i}
 \]
The existence or nonexistence of such positive cones is a central question when trying to determine the homeomorphism type of $LO(G)$.

\section{The natural actions on $LO(G)$}

There is a natural ``conjugation'' action of the group $G$ upon $LO(G)$ which we write as a right action.  Given an ordering $<$ of $G$, the right action of $g$ on $<$ will be written as $ <_g$, and is defined according to $x <_g y$ if and only if $xg^{-1} < yg^{-1}$.  For left-orderings the latter inequality is equivalent to $gxg^{-1} < gyg^{-1}$, hence the name conjugation.  The subset of $LO(G)$ which is fixed under all conjugations is clearly the set of all bi-invariant orderings, which we denote $O(G)$. 

\begin{problem}
Check that conjugation is really a right action, that is $(<_g)_h$ is the same as $<_{(gh)}$.
If $P$ is the positive cone of $<$, show that the positive cone for $<_g$ is $g^{-1}Pg$. 
Verify that, for any fixed $g \in G$, the mapping $< \to <_g$ is a homeomorphism of $LO(G)$. 
\end{problem}

In some cases, the structure of $LO(G)$ is straightforward, as is the action.  We already saw in Example \ref{Klein_bottle_orderings} that the group $K = \langle x, y : xyx^{-1} = y^{-1} \rangle$ has only four left-orderings, and all of them arise from the short exact sequence 
\[ 1 \rightarrow \langle  y  \rangle \rightarrow K \rightarrow K / \langle  y  \rangle \rightarrow 1
\]
In this case $LO(K)$ is a discrete space with four points, and the action of $K$ on $LO(K)$ is illustrated in Figure \ref{Klein_bottle_log_action}.  Note that every orbit of the action contains more than one point, because $K$ is not bi-orderable.
\begin{figure}
\begin{center}   
\includegraphics[height=4cm]{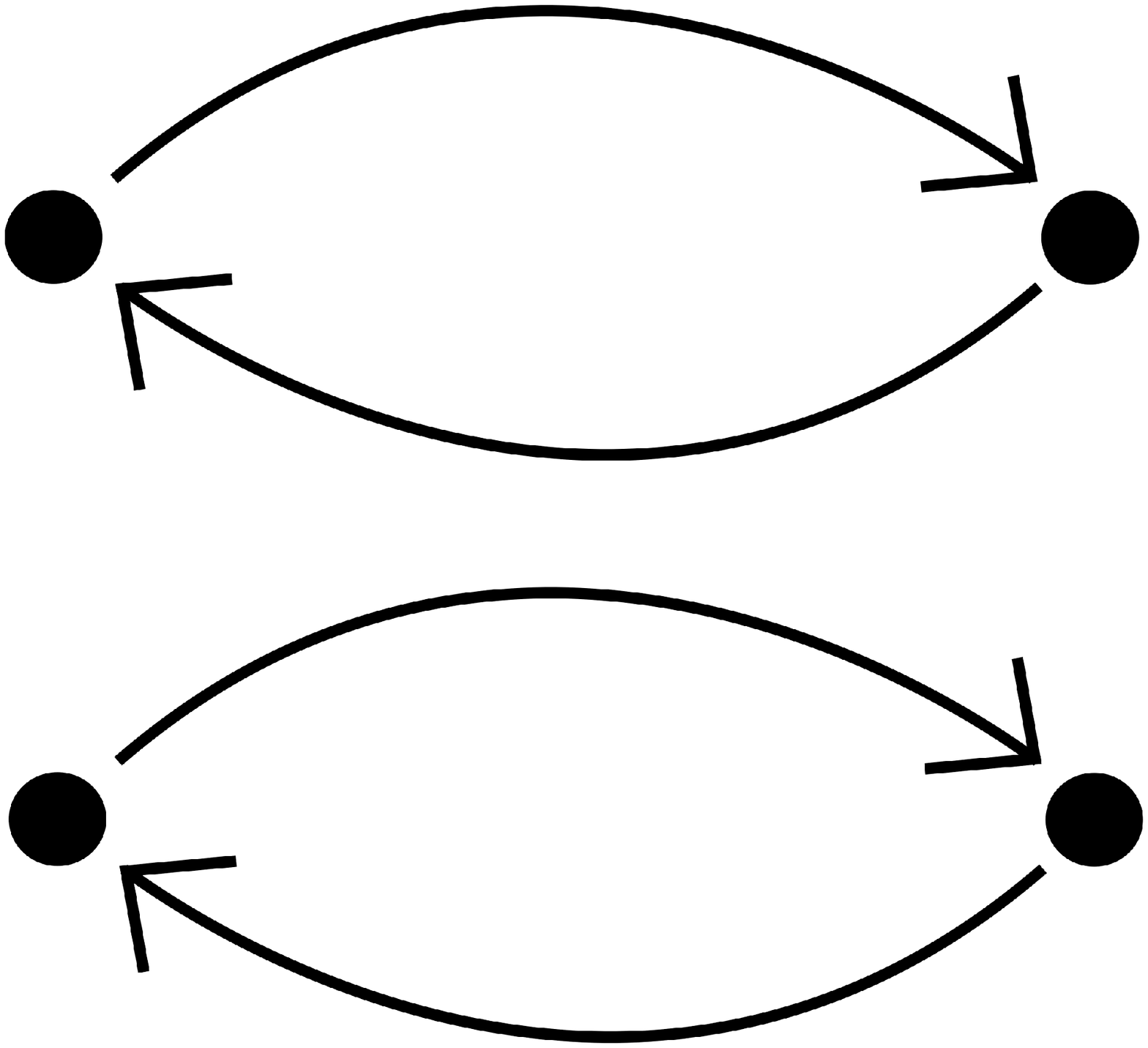}
\label{Klein_bottle_log_action}
\caption{The action of the generator $y$ is trivial, while the generator $x$ acts by swapping pairs of left-orderings.}
\end{center}
\end{figure}

\begin{problem}
Show that $O(G)$ is closed in $LO(G)$ by two different arguments.  For the first, observe that being the positive cone of a 2-sided ordering is a closed condition.  Alternatively, argue more generally that if a group $G$ acts on a topological space $X$ by homeomorphisms, the fixed point set of the action of any $g$ in $G$ is closed in $X$, and the same is true of the global fixed point set (points of $X$ fixed by all $g \in G$.)
\end{problem}

Thus we see that $O(G)$ is also a compact totally disconnected space.

Conjugation is a special case of an automorphism of $G$, namely an ``inner'' automorphism.  So, more generally we can say that any automorphism $\phi$ of $G$ gives rise to a mapping $< \longrightarrow <_\phi$  where the definition of $<_{\phi}$ is $x <_\phi y \iff \phi(x) < \phi(y)$.

\begin{problem}
Check that the mapping $< \to <_\phi$ defines a homeomorphism of $LO(G)$.  If  $\psi$ is a second automorphism, the relation $<_{(\phi\psi)}$ is the same as $(<_\phi)_\psi$.   If $P$ is the positive cone of $<$, then the positive cone of  $<_\phi$ is $\phi^{-1}(P)$.  (Here $\phi^{-1}$ denotes the inverse automorphism, and is not to be confused with, for example, $P^{-1}$, denoting inversion in the group $G$.) 
\end{problem}

Thus we have a homomorphism of groups
$$\mathrm{Aut}(G) \longrightarrow \mathrm{Homeo}(LO(G))$$
where $\mathrm{Aut}(G)$, of course, is the automorphism group of $G$ and $\mathrm{Homeo}(LO(G))$ is the group of homeomorphisms of the space $LO(G)$.  The group  $\mathrm{Inn}(G)$ of inner automorphisms is a normal subgroup of $\mathrm{Aut}(G)$, and the quotient $\mathrm{Aut}(G)/\mathrm{Inn}(G)$ is known as the {\em outer automorphism} group $\mathrm{Out}(G)$.

One checks that $\mathrm{Aut}(G)$ takes bi-invariant orderings to bi-invariant orderings, so that 
$\mathrm{Aut}(G)$ also acts on the subspace $O(G)$ by homeomorphisms.  Moreover, since $\mathrm{Inn}(G)$ acts trivially on $O(G)$ we have an action of $\mathrm{Out}(G)$, in other words a homomorphism
$$\mathrm{Out}(G) \longrightarrow \mathrm{Homeo}(O(G)) .$$

\begin{problem}
Verify the details of the above paragraph.
\end{problem}

\section{Orderings of $\Z^n$ and Sikora's theorem}

Of course for $G$ an abelian group, $LO(G)$ and $O(G)$ coincide.
It is easy to see that the additive integers $\Z$ can be given just two bi-orderings.  The goal of this section is to prove a fundamental theorem, due to A. Sikora \cite{Sikora04}.

\begin{theorem}\label{Sikora theorem}
For $n \ge 2$, the space $LO(\Z^n) = O(\Z^n)$ has no isolated points, and hence it is homeomorphic with the Cantor set.
\end{theorem}

\begin{proof} Before proving this in general, we invite the reader to prove it in the simplest case.

\begin{problem}  \label{case n=2}Prove Theorem \ref{Sikora theorem} for the case $n=2$ as follows.   Considering 
$\Z^2 \subset \R^2$ in the standard way, observe that, given a bi-ordering of $\Z^2$, there is a unique line $L$ through the origin in $\R^2$ such that all points of $\Z^2$ in one component of $\R^2 \setminus L$ are positive and all those in the other component are negative.  On the line itself, either $L \cap \Z^2 = \{0\}$ or else all positive points of $L \cap \Z^2$ lie on one component of $L \setminus \{0\}.$  In either case, show that if 
$g_1, \dots, g_k$ are in the positive cone of the ordering, then $L$ may be perturbed to define another ordering of $\Z^2$ in which all the $g_i$ remain positive. (Compare Example \ref{orderZ2})
\end{problem}

Now suppose $n \ge 3$ and that $g_1, \dots, g_k$ are positive elements of some given ordering of $\Z^n$.   Consider two cases.

    Case 1:  The ordering of $\Z^n$ is Archimedian.  Then, by Theorem \ref{theorem:Holder} there is an embedding $f: \Z^n \to \R$ which is an order-preserving homomorphism.  Considering $\Z^n \subset \R^n$, since $f$ is a linear map, there is a unique vector 
$v \in \R^n$ such that $f(x) = v \cdot x$ for all $x \in \Z^n$.  Moreover, the hyperplane $H$ perpendicular to $v \in \R^n$ must satisfy $H \cap \Z^n = \{0\}$, for otherwise $H \cap \Z^n$ would be a nontrivial convex subgroup of $\Z^n$, contradicting the assumption the ordering is Archimedian.  It follows that we have $v \cdot g_i  > 0$ for all $i$.
Consequently, one may perturb $v$ by a small vector perpendicular to $v$ to obtain $v'$ and define, via dot product, a different ordering on $\Z^n$ in which all the $g_i$ remain positive.

Case 2: The ordering on $\Z^n$ is non-Archimedian.  Then we have a nontrivial proper subgroup $K \subset \Z^n$ which is convex in the ordering.  This gives an exact sequence 
$$1 \to K \to \Z^n \stackrel{\phi}{\to} H \to 1$$  for which both the kernel $K$ and the cokernel $H$ are finitely-generated free abelian.
Moreover $K \cong \Z^p$ and $H \cong \Z^q$ and there are orderings $<_K$ and $<_H$ which lexicographically define the given ordering on $\Z^n$ as in Problem \ref{extension}.  Note that $p + q = n$, both $p$ and $q$ are less than $n$, and at least one of them is greater than 1.  We will proceed by induction, with Problem \ref{case n=2} as the base case.  If $p > 1$,
then we may assume inductively that $<_K$ is not isolated.  Therefore, there is an ordering $\prec_K$ distinct from 
$<_K$ such that whichever $g_i$ lie in $K$ (and therefore satisfy $g_i >_K 0$), will also satisfy $g_i \succ_K 0$.  Then the order of $\Z^n$ lexicographically defined by $\prec_K$ and $<_H$ will be a different ordering of $\Z^n$ in which all the $g_i$ remain positive.  On the other hand if $q > 1$, we may inductively assume there is an ordering
$\prec_H$ for which $\phi(g_i) \succ_H 0$ whenever  $\phi(g_i) >_H 0$, and similarly construct a different ordering of $\Z^n$ in which all the $g_i$ are positive.
\end{proof}

Recall that the \textit{rank} of an Abelian group $A$ is the dimension of $A \otimes \Q$ as a vector space over $\Q$.  Equivalently, a set $\{ a_{i} \}_{i \in I}$ of elements of $A$ is linearly independent if for every finite subset $J \subset I$
\[ \sum_{i \in J} n_i a_i = 0, n_i \in \Z
\]
implies $n_i =0$ for all $i \in J$.  The rank of an abelian group $A$ is the cardinality of a maximal linearly independent subset.  We can extend Sikora's theorem as follows.

\begin{theorem}\label{no isolated points}
If $A$ is any torsion-free abelian group of rank greater than one, then $LO(A) = O(A)$ has no isolated points.
\end{theorem}

The proof is left as the following exercises.

\begin{problem}  If $A$ is an abelian group and $H \subset A$ is a subgroup, define the \textit{isolator} of $H$ in $A$ to be
\[ I(H) = \{ a \in A \mid \exists k \in \Z \mbox{ such that } a^k \in H  \}
\]
Show that $I(H)$ is a subgroup, and that every bi-ordering of $H$ extends uniquely to a bi-ordering of $I(H)$.  Show the quotient $A/I(H)$ is torsion free, and conclude that every bi-ordering of $H$ extends to a bi-ordering of $A$.
\end{problem}

\begin{problem}
If $A$ is a torsion-free abelian group of rank greater than one, show that $LO(A)$ has no isolated points by applying Sikora's theorem to the finitely generated subgroups of $A$, and using the result of the previous problem (Hint:  If $S$ is a finite subset of $A$ and $\langle S \rangle \cong \Z$, then $I(\langle S \rangle ) = A$ implies that $A$ has rank one).
\end{problem}

\begin{problem}\label{rank one problem} Show that if $A$ is a torsion-free abelian group of rank one, then $LO(A) = O(A)$ has exactly two elements.
\end{problem}

\section{Examples of groups without isolated orderings}

We saw in the previous section that the crux of the analysis of $LO(\mathbb{Z}^n)$ lay in constructing new orderings of $\mathbb{Z}^n$ while keeping a certain set of elements positive.  The conjugation action and the action of $Aut(G)$ on $LO(G)$ both provide ways of creating new left-orderings, and in this section we'll see that sometimes these methods alone are sufficient to make a complete analysis of the structure of $LO(G)$.

Recall from Chapter \ref{braids chapter} that the braid groups $B_n$ can be defined by generators $\sigma_i$ together with the braid relations.  The definition of $B_n$ can be extended to yield an ``infinite strand braid group'' $B_{\infty}$, which is defined in terms of generators and relations by
\[ B_{\infty} = \left< \sigma_1, \sigma_2, \ldots
\begin{array}{|c}
\sigma_i \sigma_j = \sigma_j \sigma_i \mbox{ if $|i-j|>1$}\\
\sigma_i \sigma_j \sigma_i = \sigma_j \sigma_i \sigma_j \mbox{ if $|i-j|=1$} \end{array} 
 \right>
\]

\begin{problem}
Use the fact that $B_n \subset B_{\infty}$ for all $n>0$ to show that $B_{\infty}$ is left-orderable, but not bi-orderable or Conradian left-orderable.
\end{problem}

It turns out that the conjugation action on $B_{\infty}$ supplies us with enough new left-orderings to show that $LO(B_{\infty})$ has no isolated points \cite[Chapter XIV Proposition 2.10]{DDRW08}.  Consider a positive cone $P \in LO( B_ {\infty})$ and suppose that $S \subset P$ is a finite subset.   We can choose $n$ so large that the generators $\sigma_i$ for $i >n$ all commute with elements of $S$.  Thus every positive cone $\sigma_i^{-1} P \sigma_i$ contains $S$ for all $i >n$.

On the other hand, at least one of the positive cones $\sigma_i^{-1} P \sigma_i$  for $i>n$ has to be different than $P$.  If none of them are different, then $\sigma_i^{-1} P \sigma_i = P$ makes $P$ into the positive cone of a \textit{bi-ordering} of the subgroup $\langle \sigma_n, \sigma_{n+1}, \ldots \rangle \subset B_{\infty}$.  But the subgroup  $\langle \sigma_n, \sigma_{n+1}, \ldots \rangle $ is just a copy of $B_{\infty}$, so it is not bi-orderable.  This contradiction shows that $P$ can't be an isolated point in $LO(B_{\infty})$, and we have proved

\begin{proposition}
Each point in $LO(B_{\infty})$ is a limit point of its conjugates, and so $LO(B_{\infty})$ is homeomorphic to the Cantor set.
\end{proposition}

There are also times when the action of $Aut(G)$ is sufficient to determine the structure of $LO(G)$.

\begin{problem}
Let $F_{\infty}$ denote the free group on infinitely many generators $x_1, x_2, x_3, \ldots$.  Fix a positive cone $P \subset F_{\infty}$ and a finite subset $S \subset P$.  Construct an automorphism $\phi:F_{\infty} \rightarrow F_{\infty}$ satisfying $\phi(S) = S$, but $\phi(P) \neq P$, and conclude that $LO(F_{\infty})$ is homeomorphic to the Cantor set.  Show that the same construction can be adapted to show that $O(F_{\infty})$ is also a Cantor set.
\end{problem}

Both the argument in the case of $LO(B_{\infty})$ and the case of $LO(F_{\infty})$ take advantage of the fact that there are infinitely many generators of the group.  When the group in question is finitely generated, the problem can become much more difficult, and indeed, there may be isolated points (as we will see in the case of $B_n$).

\section{The space of orderings of a free product}

\index{free products of groups}
At the time of this writing, there are few general theorems available which predict when a group does (or does not) admit isolated points in its space of left-orderings.  However, we do have the following remarkable structure theorem due to Rivas.

\begin{theorem}
\cite{Rivas12}
Suppose that $G$ and $H$ are nontrivial left-orderable groups, so that $G * H$ is left-orderable.  Then $LO(G*H)$ has no isolated points.
\end{theorem}

In full generality, the proof is quite involved.  However we can demonstrate some of the core principles by considering the proof in the case $G \cong H \cong \mathbb{Z}$, so that $G*H =  F_2$. 

\begin{theorem}
\label{Rivas}
The space of left-orderings $LO(F_2)$ has no isolated points.
\end{theorem}
\begin{proof}
Suppose that $F_2$ has generators $a$ and $b$, let $P \in LO(F_2)$ be a positive cone.  Corresponding to the positive cone $P$ there is a left-ordering $<$ of $F_2$ and a \index{dynamic realization} dynamic realization $\rho:F_2 \rightarrow \mathrm{Homeo}_+(\mathbb{R})$ such that for all $g \in F_2$ we have  $\rho(g)(0) > 0$ if and only if $g>1$ (see Section \ref{dynamicalsection}).   Let $B_n(F_2) \subset F_2$ denote the set of elements of $F_2$ that are represented by reduced words of length less than or equal to $n$ relative to the generating set $\{ a, b \}$.  

\begin{problem}
\label{what we need}
 In order to prove that $P \in LO(F_2)$ is not isolated, show that it suffices to construct a representation $\rho_n : F_2 \rightarrow \mathrm{Homeo}_+(\mathbb{R})$ for each $n>0$ such that
\begin{enumerate}
\item  For all $w \in B_n(F_2)$, $ \rho_n(w)(0) =\rho(w)(0)$.
\item There exists $g \neq 1 \in F_2$ such that $\rho_n(g)(0)=0$.
\end{enumerate}
(Hint: For each $n$, the representations $\rho$ and $\rho_n$ can be used to construct two distinct orderings of $F_2$ which agree on $B_n(F_2)$ by ordering elements according to the orbit of zero).
\end{problem}

For each $n$, set
\[ g_n^+ = \max_< \{B_n(F_2)\}, \hspace{1em}  g_n^- = \min_< \{B_n(F_2)\}
\] 
Note that since $\rho$ satisfies $\rho(h)(0) > 0$ if and only if $h>1$, the map $h \mapsto \rho(h)(0)$ is order-preserving.  Thus for all $w \in B_n(F_2)$, we see that $\rho(w)(0)$ lies in the interval $[\rho(g_n^-)(0), \rho(g_n^+)(0)]$.

\begin{problem}
\label{only in a box}
  Show that if $\rho_n(a)(x) = \rho(a)(x)$ and $\rho_n(b)(x) = \rho(b)(x)$ for all $x \in  [\rho(g_n^-)(0), \rho(g_n^+)(0)]$, then $\rho_n(w)(0) = \rho(w)(0)$ for all $w \in B_n(F_2)$. (Hint: Use induction on the length of $w$)
\end{problem}

In other words, Problem \ref{only in a box} shows that in order to satisfy (1) of Problem \ref{what we need} the maps $\rho_n(a)$ and $ \rho_n(b) $ need only agree with $\rho(a)$ and $\rho(b)$ respectively on the interval $ [\rho(g_n^-)(0), \rho(g_n^+)(0)]$.  Here is how to construct maps $\rho_n(a)$ and $ \rho_n(b) $ with this property that also satisfy (2) of Problem \ref{what we need}.

To keep the notation simple, instead of writing $\rho(h)(0)$ below we will simply write $h(0)$ for all $h \in F_2$, and  in place of $g^{\pm}_n$ we'll write $g^{\pm}$.

Referring to Problem \ref{problem between} we may choose $c_1 \in \{ a^{\pm 1} \}$ and $c_2 \in \{ b^{\pm 1} \}$ such that $c_i g^+ >g^+$ for $i=1, 2$.  Set 
$$
f(x) = \left\{ \begin{array}{ll}
c_1(x) \mbox{ if $x\leq g^+(0)$} & \\[2ex]
 \left( \cfrac{c_1^2g^+(0)-c_1g^+(0)}{c_2g^+(0)-g^+(0)} \right)(x-g^+(0)) +c_1g^+(0) &\mbox{ otherwise.}
\end{array} \right.
$$
Note that the second part of the definition of $f(x)$ is just a straight line connecting $(g^+(0), c_1g^+(0))$ to $(c_2g^+(0), c_1^2g^+(0))$, which makes $f(x)$ both continuous and order-preserving.  Now define $\rho_n(c_1) = f$, and $\rho_n(c_2) = \rho(c_2)$.  Since $c_1$ and $c_2$ freely generate $F_2$ this defines a homomorphism $\rho_n : F_2 \rightarrow \mathrm{Homeo}_+(\mathbb{R})$, which clearly satisfies $\rho_n(a)(x) = \rho(a)(x)$ and $\rho_n(b)(x) = \rho(b)(x)$ for all $x \in  [g^-(0), g^+(0)]$.  Therefore (1) of Problem \ref{what we need} is satisfied, and (2) of Problem \ref{what we need} is checked in the following exercise.

\begin{problem}  Note that $c_i g^+$ is a reduced word of length $n+1$ for each $i$, since $c_i g^+(0)$ does not lie in $ [g^-(0), g^+(0)]$.  Therefore the product $(c_1g^+)^{-1}c_2g^+$ is a reduced word, and so is not the identity.  Check that $\rho_n((c_1g^+)^{-1}c_2g^+)(0) =0$, thus completing the proof of \ref{Rivas}. (Hint:  Rewrite $(c_1g^+)^{-1}c_2g^+$ as $(c_1g^+)^{-1}c_1^{-1} c_1c_2g^+$, and use the formula for $f$ above).
\end{problem}
\end{proof}

That the space $LO(F_n)$ admits no isolated points has, to date, been proved in many different ways \cite{Clay11a, Rivas12, Navas10a, Kopytov79}, although it was first proved by McCleary \cite{McCleary85}.  In fact, more is known:

\begin{theorem}\cite{Clay11a, Rivas12, Kopytov79}
Let $F_n$ denote the free group on $n>1$ generators.  Then the action of $F_n$ on $LO(F_n)$ admits a dense orbit.
\end{theorem}

Unfortunately very little is known in the case of bi-orderings, and at the time of this writing the following is still open for $n \ge 2$.

\begin{question} Does $O(F_n)$ have isolated points?
\end{question}

\section{Examples of groups with isolated orderings}

\index{isolated ordering}
As a first example of isolated orderings, we can revisit the Klein bottle group.  Since there are only four left orderings of the group $G = \langle x, y : xyx^{-1} = y^{-1} \rangle$, the space $LO(G)$ is finite and all four of its points are isolated.  So, groups with finitely many orderings exist, and obviously all of their orderings will be isolated.  Such groups are not a source of particularly interesting examples since groups admitting finitely many left-orderings are completely classified by the following theorem, attributed to Tararin.  

\begin{theorem} \cite[Proposition 5.2.1]{KM96}
\label{Tararin theorem}
\index{Tararin groups}
Let $G$ be a left-orderable group.  If $G$ admits only finitely many left-orderings, then $G$ admits a unique sequence of normal subgroups $\{ 1 \} = G_0 \triangleleft G_1 \triangleleft \cdots \triangleleft G_n =G$ such that each $G_{i+1} / G_i$ is torsion-free Abelian of rank one, and no quotient $G_{i+2}/G_i$ is bi-orderable.  Conversely if $G$ admits such a sequence of normal subgroups, then $LO(G)$ is finite and consists of exactly $2^n$ points.
\end{theorem}

\begin{problem}
For a group $G$ with a sequence of normal subgroups as in the preceding theorem, describe all of its left-orderings.
\end{problem}

For an example of a `genuine' isolated point, in the sense that the group admits infinitely many left-orderings (some of which are isolated), we turn to the braid groups. Recall from Section \ref{dehornoy introduction} that the braid groups are left-orderable, with the standard ordering (called the Dehornoy ordering \index{Dehornoy ordering}) defined in terms of $i$-positivity of representative words in the generators $\sigma_i$.  We saw that a  word $w$ in the generators $\sigma_1, \ldots, \sigma_{n-1}$ is called \textit{$i$-positive} (resp. \textit{$i$-negative}) if $w$ contains no generators $\sigma_j$ with $j<i$, and all occurences of $\sigma_i$ have positive (resp. negative) exponent.  A braid $\beta$ is called $i$-positive if it admits an $i$-positive representative word.  The positive cone of the Dehornoy ordering is the set of all braids that are $i$-positive for some $i$.  Thus if we let $P_D$ denote the positive cone of the Dehornoy ordering and use $P_i$ to denote the set of all $i$-positive braids, then 
\[ P_D = P_1 \cup P_2 \cup P_3 \cup \cdots \cup P_{n-1}.
\]
Closely related to the Dehornoy ordering is the Dubrovina-Dubrovin ordering of the braid group $B_n$, with positive cone $P_{DD}$ \cite{DD01}.    Its positive cone is 
\[P_{DD} = P_1 \cup P_2^{-1} \cup \cdots \cup P_{n-1}^{(-1)^n},
\]
that is, $P_{DD}$ contains all those braids which are $i$-positive for some odd number $i$, and all those braids which are $i$-negative for some even number $i$.  

\begin{problem}
Show that for each integer $i \in \{1, \dots , n-1\}$ the set $$\{1\} \cup P_i \cup  P_{i+1} \cup \cdots \cup P_{n-1} \cup P^{-1}_i \cup  P^{-1}_{i+1} \cup \cdots \cup P^{-1}_{n-1}$$ is a convex subgroup of $B_n$ with respect to the Dehornoy ordering.   Combine this with the result of Problem \ref{extension_generalization}, and the remarks following that problem, in order to prove that $P_{DD}$ is a positive cone.
\end{problem}

Set  $\beta_i = (\sigma_i \cdots \sigma_{n-1})^{(-1)^{i-1}},$ where $1 \leq i \leq n-1$.  The elements $\beta_i$ generate the group $B_n$, and in fact they are the key to describing a famous isolated point in the space $LO(B_n)$.

\begin{theorem} \cite{DD01}
\label{dehornoy isolated}
The positive cone $P_{DD}$ is generated as a semigroup by the elements $\beta_i$, and is thus an isolated point in $LO(B_n)$.
\end{theorem}

\begin{problem}
By appropriately modifying the proof of Theorem \ref{dehornoy in b3}, prove Theorem \ref{dehornoy isolated} in the case $n=3$.  That is, show that every element of $B_3$ can be represented by a word in which the generators $\beta_1 = \sigma_1 \sigma_2$ and $\beta_2 = \sigma_2^{-1}$ occur exclusively with positive exponents, or exclusively with negative exponents.
\end{problem}

\begin{problem}
Use the fact that $P_{DD}$ is equal to the subsemigroup of $B_n$ generated by $\beta_1, \ldots, \beta_{n-1} $ to show that $P_{DD}$ is an isolated point in $LO(B_n)$.  Argue more generally that if a positive cone of a left-ordered group $G$ is finitely generated as a subsemigroup, then it is an isolated point in $LO(G)$.
\end{problem}

By contrast, the Dehornoy ordering of $B_n$ is not isolated in $LO(B_n)$.  In fact, Navas observed that the following is true (see for example \cite{DDRW08} p.269):

\begin{theorem}
For $n \ge 3$, the Dehornoy ordering $P_D$ is a limit point of its conjugates in $LO(B_n)$.
\end{theorem}

\begin{corollary}
For $n \ge 3$, the positive cone  $P_D$ of the Dehornoy ordering is not finitely-generated as a subsemigroup of $B_n$.
\end{corollary}

For other examples of genuine isolated orderings, see \cite{Navas11a, Itopreprinta, Ito13, Ito2014preprint, Dehornoy14}.

\section{The number of orderings of a group}

In this section we will see a first application of compactness of $LO(G)$, by showing that $LO(G)$ cannot be countably infinite.  In contrast, $O(G)$ can be countably infinite \cite{Buttsworth71}.  We begin with the case of a Conradian left-ordered group.

\begin{theorem} \cite[Proposition 5.2.5]{KM96}
\label{uncountable Conradian}
If a group $G$ admits a Conradian left-ordering, then $LO(G)$ is either finite or uncountable.
\end{theorem}

The proof will follow by considering each way that $G$ can fail to be a `Tararin group', one of the groups described in Theorem \ref{Tararin theorem}.  In every case, one finds as a result that $G$ admits uncountably many left-orderings.

\begin{problem}
\label{first tararin}
Suppose that $G$ admits a Conradian left-ordering having infinitely many convex jumps.  Show that by `flipping' the orderings of the jumps, one can create uncountably many left-orderings of $G$.
\end{problem}

\begin{problem}
\label{second tararin}
Suppose that $G$ admits a Conradian left-ordering with convex jumps $\{ (C_i, D_i) \}_{i \in I}$.  Show that if there is a jump $(C_i, D_i)$ such that $D_i /C_i$ is torsion-free Abelian of rank greater than one, then $G$ admits uncountably many left-orderings.
\end{problem}

If $\{ 1 \} = G_0 \triangleleft G_1 \triangleleft \cdots \triangleleft G_n = G$ and $G_{i+2}/G_i$ is bi-orderable, then $G$ fails to be a Tararin group as well.  The following Proposition is relevant to that case. 

\begin{proposition}\cite[Lemma 2.1]{Rivas12a}
\label{uncountable example}
Suppose that $G$ is bi-orderable, and fits into a short exact sequence

\[ 1 \rightarrow K \rightarrow G \stackrel{q}{\rightarrow} H \rightarrow 1
\]
where both $K$ and $H$ are torsion-free Abelian of rank one.  Then $LO(G)$ is uncountable.
\end{proposition}

\begin{proof} If $G$ is Abelian, then $G$ admits uncountably many left-orderings, by Theorem \ref{no isolated points}.  So we suppose $G$ is not Abelian.   Identify $K$ with a subgroup of $\Q$; by Problem \ref{rank one problem} we may assume the ordering of $K$ as a subgroup of the ordered group $G$ agrees with the usual ordering of $\Q$.  Consider the map $\phi: H \rightarrow Aut(K)$ given by $\phi(h)(k) = gkg^{-1}$, where $q(g) = h.$  It is easy to check that this independent of the choice of $g \in q^{-1}(h)$.  It also preserves the ordering of $K$ as a subgroup of the bi-ordered group $G$, which may be assumed to coincide with the ordering of $K$ as a subgroup of $\Q$ ordered in the usual way.  By Lemma \ref{hion's lemma}, $\phi$ corresponds to multiplication by some positive rational number: $\phi(h)(k) = rk$, where $ r = r(h) \in \Q$.  Since $G$ is non-Abelian, $\phi$ is a nontrivial homomorphism.  If the kernel of $\phi$ were nontrivial, the image of $\phi$ would be torsion since $H$ is rank one Abelian.  However the only torsion element of $Aut(K)$ is multiplication by $-1$, and if this were the action of $H$ on $K$ then $G$ would not be bi-orderable.  Therefore $\phi :H \rightarrow Aut(K)$ is injective.

We claim that $H \cong \mathbb{Z}$.  Suppose not, then by Problem \ref{cyclic problem} one can choose $h \in H$ not equal to the identity and infinitely many $g_1, g_2, g_3, \ldots$ and $n_1, n_2, n_3, \ldots$ such that $g_i^{n_i} = h$ for all $i$.   Then if $\phi(h)$ acts on $k \in K$ according to $\phi(h)(k) = rk$, the elements $g_i$ must act according to $\phi(g_i^{n_i})(k) = r_i^{n_i} k$ for some rational numbers $r, r_1, r_2, r_3, \ldots$ satisfying $r_i^{n_i} = r$ for all $i$.  It is not possible to have infinitely many solutions to $r_i^{n_i} = r$ (see Problem \ref{finite roots}).

Thus $H \cong \mathbb{Z}$ and so the sequence splits, and $G \cong K \rtimes_{\phi} H$.  Therefore we can define a map $\rho: G \rightarrow \mathrm{Homeo}_+(\R)$ as follows: Fix $k \in K$ and set $\rho(k)(x) = x+1$.  For all other $k' \in K$, if $(k')^q = k^p$, set $\rho(k')(x) = x+p/q$.    For the group $H$, suppose that the generator is $h$ and that $\phi(h)(k) = rk$ for all $k \in K$, where $r = r(h) \in \Q$.  Then set $\rho(h)(x) = rx$ for all $x \in \R$.  A priori this defines a map $\rho : K *H \rightarrow  \mathrm{Homeo}_+(\R)$, but because $\rho(hkh^{-1}) = \rho(\phi(h)(k))$ for all $k \in K$, it descends to a map $\rho: G \rightarrow \mathrm{Homeo}_+(\R)$.  We can now define uncountably many positive cones $P_{\epsilon} \subset G$ by setting, for each irrational $\epsilon >0$, 
\[ P_{\epsilon} = \{ g \in G : \rho(g)(\epsilon) > \epsilon \}. \eqno{\square}
\]
\let\qed=\relax\end{proof}

\begin{lemma}[Hion's Lemma \cite{Hion54}] \label{hion's lemma} Suppose $A$ and $B$ are subgroups of the additive group of reals, $\R$ and that $f:A \to B$ is an isomorphism which preserves the natural order inherited from $\R$.  Then there is a positive real number $r$ such that $f(a) = ra$ for all $a$.  If $A$ and $B$ are subgroups of $\Q$, then $r \in \Q$.
\end{lemma}
\begin{proof}
Suppose there are $a, a' \in A$ such that $f(a)/f(a') \ne a/a'$, say \break $f(a)/f(a') < a/a'$.  Then we can find $p, q \in \Z$
with $f(a)/f(a') < p/q < a/a'$.  We may assume that $q$ and $a'$ (and hence $f(a')$) are positive, and conclude that $pa' < qa$ and $pf(a') > qf(a)$, a contradiction.  So $f(a)/a$ is constant for $a \ne 0$.  The last sentence of the lemma is obvious.
\end{proof}

\begin{problem}\label{cyclic problem}
If $H$ is a rank one torsion-free abelian group, then $H$ is cyclic if and only if for every $h \in H$ there are only finitely many $h_i \in H$ and $0 < n_i \in \Z$ such that $h_i^{n_i} = h$.
\end{problem}

\begin{problem}\label{finite roots} Verify that, given $r \in \Q$, the equation $r_i^{n_i} = r$ has only finitely many solutions with $n_i \in \Z, r_i \in \Q$. 
\end{problem}

\begin{problem}
Check the other details of the proof of Proposition \ref{uncountable example}.  In particular show that $P_{\epsilon} \ne P_{\epsilon'}$ for $\epsilon < \epsilon'$ by arguing that $K$ must be dense in $\Q$, as it is closed under multiplication by $r^k$ where $0 < r \ne 1$ corresponds to the action of the generator of $H$.  Then one can find appropriate $k \in K$ and $h \in H$ so that for the corresponding $g = (k,h) \in G$, the graph of the line $\rho(g)(x) = rx + p/q$ intersects the diagonal between $x = \epsilon$ and $x = \epsilon'$.
\end{problem}

\begin{problem}
\label{third tararin}
Use the result of Proposition \ref{uncountable example} to show that if $G$ admits a unique sequence of normal subgroups $\{ 1 \} = G_0 \triangleleft G_1 \triangleleft \cdots \triangleleft G_n =G$ with rank one Abelian quotients and $G_i /G_{i-2}$ is bi-orderable for some $i$, then $LO(G)$ is uncountable. 
\end{problem}

\begin{proof}[Proof of Theorem \ref{uncountable Conradian}]
Suppose that $G$ admits a Conradian left-ordering and that $LO(G)$ is infinite and countable.  By Problems \ref{first tararin} and \ref{second tararin}, countability of $LO(G)$ implies that our given Conradian ordering has finitely many convex jumps, each with a rank one Abelian quotient.   Let us name the convex subgroups
\[ \{ 1 \} = G_0 \triangleleft G_1 \triangleleft \cdots \triangleleft G_n =G
\]
with $G_{i+1}/G_{i}$ rank one Abelian for all $i$.  To establish uniqueness, suppose there is another series of normal subgroups
\[ \{ 1 \} = H_0 \triangleleft H_1 \triangleleft \cdots \triangleleft H_n =G
\]
with rank one Abelian quotients.  If $G_{n-1} \neq H_{n-1}$, set $K = G_{n-1} \cap H_{n-1}$ and argue that $G/K$ is torsion free Abelian as follows.  Since $G/G_{n-1}$ is Abelian, we have $[G,G] \subset G_{n-1}$ and similarly $[G,G] \subset H_{n-1}$, so $[G,G] \subset K$ and $G/K$ is Abelian.  To see it's torsion free, note that any nonidentity element of $G/K$ maps to a nonidentity element in (at least) one of $G/G_{n-1}$ or $H/H_{n-1}$ via the natural projections, so it has a torsion free image, and therefore itself has infinite order.
If the rank of $G/K$ is greater than one, then $G/K$ admits uncountably many left-orderings by Sikora's theorem, which implies that $G$ also has uncountably many left-orderings.
So the rank of $G/K$ must be one. 

Consider the exact sequence $$ 1 \to H_{n-1}/K \to G/K \to G/H_{n-1} \to 1 $$ of torsion-free Abelian groups.� By additivity of rank we see that the rank
of $H_{n-1}/K$ must be zero, so $H_{n-1} = K$.� Similarly $G_{n-1} = K$ and we conclude $H_{n-1} = G_{n-1}$.
   By induction we conclude that the two series of normal subgroups coincide.

Finally, we are in the case where $LO(G)$ is countable and infinite and $G$ must have a unique finite series of normal subgroups having rank one Abelian quotients.  By Theorem \ref{Tararin theorem}, this can only be possible if one of the quotients $G_i /G_{i-2}$ is bi-orderable.  In that case, Problem \ref{third tararin} shows that in fact $G$ has uncountably many left-orderings.  This contradiction finishes the proof.
\end{proof}

We can now prove the general result, due to Peter Linnell.

\begin{theorem}\cite{linnell11}
If $G$ is a left-orderable group, then $LO(G)$ is either finite or uncountable.
\end{theorem}
\begin{proof} 
We use the conjugation action of $G$ on $LO(G)$ to construct a \textit{minimal invariant set}.  Set
\[ \mathcal{S} = \{ A \subset LO(G) : A \hbox{ is nonempty, closed and $G$-invariant} \}
\]
The set $\mathcal{S}$ is nonempty since it contains $LO(G)$, and it is partially ordered by inclusion.  Moreover if we have any chain of nonempty, closed, $G$-invariant sets $\{ A_{i} \}_{i \in I}$ then this collection has the finite intersection property, so by compactness of $LO(G)$ we have $$\emptyset \neq \bigcap_{i \in I}A_i \in \mathcal{S}$$
Thus every chain has a lower bound, and so by Zorn's lemma there is a minimal element of $\mathcal{S}$, call it $M$.   The set $M$ is a minimal invariant set.

\begin{problem}\label{problem-minimal}
Show that a the minimal $G$-invariant closed subset $M \subset LO(G)$ satisfies:
\begin{enumerate} 
\item If $A$ is a closed, $G$-invariant subset of $LO(G)$ and $A \cap M \neq \emptyset$, then $M \subset A$.
\item Denote the orbit of $P \in LO(G)$ by $G(P)$.  For all $P \in M$, we have $\overline{G(P)} = M$.
\end{enumerate}
\end{problem}

Now we consider two cases, first the case where $M$ is finite.  In this case, every positive cone $P \in M$ has a finite orbit.  But then the stabilizer $G_P$ of such a cone $P$ is a finite index subgroup of $G$, and the left-ordering of $G$ corresponding to $P$ is a bi-ordering when restricted to $G_P$.  By Problem \ref{finite index Conradian}, $P$ must be the positive cone of a Conradian left-ordering.  By Theorem \ref{uncountable Conradian}, $LO(G)$ is then either finite or uncountable.

Suppose then that $M$ is infinite, and that there is $P \in M$ which is isolated (in $M$).  Since $M$ is infinite and $\overline{G(P)} =M$, the orbit $G(P)$ is infinite and thus has an accumulation point, say $Q$, since $M$ is compact.  Note that $Q$ is not in $G(P)$ since $Q$ is not an isolated point, and yet we have $\overline{G(Q)} = M = \overline{G(P)}$, meaning $P \in \overline{G(Q)}$, a contradiction.  Thus $M$ has no isolated points.  Now since $M$ is a compact Hausdorff space without isolated points, it is uncountable (see, for example \cite{Munkres2000}, Theorem 27.7), and thus $LO(G)$ is uncountable.
\end{proof}

\section{Recurrent orderings and a theorem of Witte-Morris}

Given that some left-orderable groups are not Conradian left-orderable, it is a natural question to ask what additional properties a left-orderable group $G$ must have in order to become a Conradian left-orderable group.     
\begin{conjecture}\cite{linnell01}
Suppose that $G$ is a left-orderable group that does not contain a non-Abelian free subgroup.  Then $G$ is Conradian left-orderable.
\end{conjecture}
A particular class of groups containing no non-Abelian free subgroups are \textit{amenable} groups.  There are many equivalent definitions of amenability, but we present the one that is most directly connected to our present setting.  Recall that a measure $\mu$ on a space $X$ is called a probability measure if $\mu (X) =1$.

\begin{definition}  A group $G$ is \textit{amenable} \index{amenable group} if whenever  $G$ acts by continuous maps on a compact Hausdorff space $X$, there exists a probability measure $\mu$ on $X$ such that $\mu(U) = \mu(g(U))$ for all $g\in G$ and all measurable sets $U \subset X$.
\end{definition}

The main ingredient we require is the Poincar\'{e} recurrence theorem, stated here in the form that we need.

\begin{theorem}
Let $X$ be a space with probability measure $\mu$, $A$ any measurable subset, and suppose that $f:X \rightarrow X$ is a homeomorphism which preserves ~$\mu$.  Then there is a subset $E \subset X$ such that $\mu(E) =0$ and for each $x \in A \setminus E$ there exists an infinite sequence of positive integers $n_1< n_2 <n_3< \ldots$ such that $f^{n_i}(x) \in A$ for every $i$.  (see for example \cite{Morris06}, Proposition 3.1).
\end{theorem}

In other words, as long as we avoid a few bad points (of which there are very few) we can be sure that the iterates $f^{n_i}(x)$ return very near to $x$ infinitely often.

\begin{theorem}\cite{Morris06}
Suppose that $G$ is a countable left-orderable group.  If $G$ is amenable, then it is Conradian left-orderable.
\end{theorem}
\begin{proof}
Assuming $G$ is left-orderable, amenability gives us a probability measure on $LO(G)$ that is invariant under the action of $G$.  Fix elements $g_1, g_2, g_3, \ldots, g_n$ of $G$, set 
\[ U_{g_1, \ldots, g_n} = \bigcap_{i=1}^{n-1} U_{g_i^{-1} g_{i+1}}
\]
(recall that $U_g =  \{P \in LO(G) : g \in P \}$).  Choose an element $g \in G$.
We apply the Poincar\'{e} recurrence theorem with $X = LO(G)$, the action of $g^{-1}$ in place of the map $f$, and the set $U_{g_1, \ldots, g_n}$ in place of $A$.  The conclusion of the theorem is that there exists a set $E_{g, g_1, \ldots, g_n}$ of measure zero such that for each $P \in U_{g_1, \ldots, g_n} \setminus E_{g, g_1, \ldots, g_n}$ there exists an increasing sequence of integers $\{n_i\}$ such that $g^{n_i}Pg^{-n_i} \in  U_{g_1, \ldots, g_n}$ for all $n_i$. This means there are infinitely many $n_i$ such that the inequalities
\[ g_1 g^{n_i} < g_2 g^{n_i} < \ldots < g_ng^{n_i}
\]
hold.  The union of all sets of the form $E_{g, g_1, \ldots, g_n}$ is again a set of measure zero since it is a countable union ($G$ is countable), so there is a positive cone $P$ that is not in any such set.  The ordering corresponding to $P$ satisfies: for every finite subset $\{ g_1, \ldots, g_n \}$ and every $g \in G$  there exists an increasing sequence of integers $\{n_i\}$ such that
\[g_1 g^{n_i} < g_2 g^{n_i} < \ldots < g_ng^{n_i}
\]
In particular, for every positive $g, h \in G$ there exist corresponding $n_i$ such that $g^{n_i} < hg^{n_i}$ for all $n_i$.  Since $g<g^2 < \ldots < g^{n_i}$, we find $g< hg^{n_i}$ and so the ordering is Conradian.
\end{proof}

It turns out that this argument actually yields orderings with a property that is strictly stronger than being Conradian, which is called \textit{recurrent for every cyclic subgroup} \index{recurrent ordering} (or simply a recurrent ordering).  

\begin{definition} 
A left-ordering $<$ of the group $G$ is {\em recurrent} if for every finite subset $\{ g_1, \ldots, g_n \}$ of $G$ satisfying $g_1<g_2< \ldots < g_n$ and every $g \in G$  there exists an increasing sequence of integers $\{n_i\}$ such that
\[g_1 g^{n_i} < g_2 g^{n_i} < \ldots < g_ng^{n_i}.
\]
\end{definition}
As we have already pointed out, a recurrent ordering is Conradian.  On the other hand, the converse does not hold.

\begin{problem}
\label{not recurrent problem}
Consider the group $G$ that fits into the short exact sequence
\[ 0 \rightarrow \Q^2 \rightarrow G \rightarrow \Z \rightarrow 0
\]
where $n \in \Z$  acts on $\Q^2$ by multiplication from the left by $A^n$, where $A = \left( \begin{array}{cc}
2 & 0  \\
0 & \frac{1}{2}  \end{array} \right)$.  Since the sequence above splits, $G = \Q^2 \times \Z$ as a set, with multiplication of $( \vec{v}_1, k_1), ( \vec{v}_2, k_2) \in \Q^2 \times \Z$ given by
\[  ( \vec{v}_1, k_1) \cdot ( \vec{v}_2, k_2) = ( A^{k_2} \vec{v_1} + \vec{v_2}, k_1 + k_2).
\]
Order $\Z$ in the usual way.  Order $\Q^2$ by choosing $\vec{w} \in \R^2$ in the second quadrant with irrational slope, and declaring $\vec{v}_1 < \vec{v}_2$ if and only if $\vec{w} \cdot \vec{v}_1 < \vec{w} \cdot \vec{v_2}$ as in  Problem \ref{orderZ2}.  Now left-order $G$ lexicographically as in Problem \ref{extension}.

Show that the resulting ordering on $G$ is Conradian, by not recurrent. (Hint:  To show that it is not recurrent, suppose that $\vec{u}$ lies in the first quadrant and satisfies $\vec{w} \cdot \vec{u} >0$, and consider the iterates $A^k \vec{u}$.
Take, as in the definition of recurrent, $g_1 = (\vec{0}, 0), g_2 = (\vec{u}, 0)$ and $g = (\vec{0}, 1)$ to arrive at a contradiction.)
\end{problem}

In fact, there are Conradian left-orderable groups which admit no recurrent orderings at all.  If one takes $F$ to be a free finite index subgroup of $\mathrm{SL}(2, \Z)$ and $G$ to be the natural semidirect product $F \ltimes \Z^2$, then $G$ admits no recurrent orderings although it is Conradian left-orderable \cite[Example 4.6]{Morris06}.
%
%
%


\backmatter
\bibliographystyle{amsplain}
\bibliography{OGTmaster}

\def\cprime{$'$}
\providecommand{\bysame}{\leavevmode\hbox to3em{\hrulefill}\thinspace}
\providecommand{\MR}{\relax\ifhmode\unskip\space\fi MR }
\providecommand{\MRhref}[2]{%
  \href{http://www.ams.org/mathscinet-getitem?mr=#1}{#2}
}
\providecommand{\href}[2]{#2}
\begin{thebibliography}{100}

\bibitem{Alexander28}
J.~W. Alexander, \emph{Topological invariants of knots and links}, Trans. Amer.
  Math. Soc. \textbf{30} (1928), no.~2, 275--306. \MR{MR1501429}

\bibitem{Artin25}
E.~Artin, \emph{Theorie der {Z\"opfe}}, Abh. Math. Sem. Univ. Hamberg
  \textbf{4} (1925), 47--72.

\bibitem{Artin47}
\bysame, \emph{Theory of braids}, Ann. of Math. (2) \textbf{48} (1947),
  101--126. \MR{MR0019087 (8,367a)}

\bibitem{Baumslag2010}
Gilbert Baumslag, \emph{Some reflections on proving groups residually
  torsion-free nilpotent. i}, Illinois J. Math. \textbf{54} (2010), no.~1,
  315--325.

\bibitem{Bergman91}
George~M. Bergman, \emph{Right orderable groups that are not locally
  indicable}, Pacific J. Math. \textbf{147} (1991), no.~2, 243--248.
  \MR{MR1084707 (92e:20030)}

\bibitem{birman74}
Joan~S. Birman, \emph{Braids, links, and mapping class groups}, Princeton
  University Press, Princeton, N.J., 1974, Annals of Mathematics Studies, No.
  82. \MR{MR0375281 (51 \#11477)}

\bibitem{MR77}
Roberta Botto~Mura and Akbar Rhemtulla, \emph{Orderable groups}, Marcel Dekker
  Inc., New York, 1977, Lecture Notes in Pure and Applied Mathematics, Vol. 27.
  \MR{MR0491396 (58 \#10652)}

\bibitem{BCpreprint}
Steven Boyer and Adam Clay, \emph{Foliations, orders, representations,
  {L}-spaces and graph manifolds}, Preprint, available via
  http://arxiv.org/abs/1401.7726.

\bibitem{BGW13}
Steven Boyer, Cameron~McA. Gordon, and Liam Watson, \emph{On {L}-spaces and
  left-orderable fundamental groups}, Math. Ann. \textbf{356} (2013), no.~4,
  1213--1245. \MR{3072799}

\bibitem{BRW05}
Steven Boyer, Dale Rolfsen, and Bert Wiest, \emph{Orderable 3-manifold groups},
  Ann. Institut Fourier (Grenoble) \textbf{55} (2005), 243--288.

\bibitem{Brodskii84}
S.~D. Brodski{\u\i}, \emph{Equations over groups, and groups with one defining
  relation}, Sibirsk. Mat. Zh. \textbf{25} (1984), no.~2, 84--103. \MR{MR741011
  (86e:20026)}

\bibitem{BH72}
R.~G. Burns and V.~W.~D. Hale, \emph{A note on group rings of certain
  torsion-free groups}, Canad. Math. Bull. \textbf{15} (1972), 441--445.
  \MR{MR0310046 (46 \#9149)}

\bibitem{Buttsworth71}
R.~N. Buttsworth, \emph{A family of groups with a countable infinity of full
  orders}, Bull. Austral. Math. Soc. \textbf{4} (1971), 97--104. \MR{MR0279013
  (43 \#4739)}

\bibitem{CD03}
Danny Calegari and Nathan~M. Dunfield, \emph{Laminations and groups of
  homeomorphisms of the circle}, Invent. Math. \textbf{152} (2003), no.~1,
  149--204. \MR{MR1965363 (2005a:57013)}

\bibitem{CR14}
Danny Calegari and Dale Rolfsen, \emph{Groups of {PL} homeomorphisms of cubes}.

\bibitem{CB88}
Andrew~J. Casson and Steven~A. Bleiler, \emph{Automorphisms of surfaces after
  {N}ielsen and {T}hurston}, London Mathematical Society Student Texts, vol.~9,
  Cambridge University Press, Cambridge, 1988. \MR{964685 (89k:57025)}

\bibitem{CL}
J.~C. Cha and C.~Livingston, \emph{Knotinfo: Table of knot invariants},
  http://www.indiana.edu/knotinfo.

\bibitem{Chehata52}
C.~G. Chehata, \emph{An algebraically simple ordered group}, Proc. London Math.
  Soc. (3) \textbf{2} (1952), 183--197. \MR{0047031 (13,817b)}

\bibitem{CGW14}
I.~M. Chiswell, A.~M.~W. Glass, and John~S. Wilson, \emph{Residual nilpotence
  and ordering in one-relator groups and knot groups}, Mathematical Proceedings
  of the Cambridge Philosophical Society \textbf{158} (2015), 275--288.

\bibitem{Clay11a}
Adam Clay, \emph{Free lattice-ordered groups and the space of left orderings},
  Monatshefte f{\"u}r Mathematik, 1--14, 10.1007/s00605-011-0305-5.

\bibitem{CDN14}
Adam Clay, Colin Desmarais, and Patrick Naylor, \emph{Testing bi-orderability
  of knot groups}, Preprint, available via http://arxiv.org/abs/1410.5774.

\bibitem{CR12}
Adam Clay and Dale Rolfsen, \emph{Ordered groups, eigenvalues, knots, surgery
  and {$L$}-spaces}, Math. Proc. Cambridge Philos. Soc. \textbf{152} (2012),
  no.~1, 115--129. \MR{2860419}

\bibitem{Conrad59}
Paul Conrad, \emph{Right-ordered groups}, Michigan Math. J. \textbf{6} (1959),
  267--275. \MR{MR0106954 (21 \#5684)}

\bibitem{CF77}
Richard~H. Crowell and Ralph~H. Fox, \emph{Introduction to knot theory},
  Springer-Verlag, New York, 1977, Reprint of the 1963 original, Graduate Texts
  in Mathematics, No. 57. \MR{MR0445489 (56 \#3829)}

\bibitem{DPT05}
Mieczyslaw~K. Dabkowski, Jozef~H. Przytycki, and Amir~A. Togha,
  \emph{Non-left-orderable 3-manifold groups}, Canad. Math. Bull. \textbf{48}
  (2005), no.~1, 32--40. \MR{2118761 (2005k:57001)}

\bibitem{Dehn10}
M.~Dehn, \emph{\"{U}ber die {T}opologie des dreidimensionalen {R}aumes}, Math.
  Ann. \textbf{69} (1910), no.~1, 137--168. \MR{1511580}

\bibitem{Dehornoy14}
Patrick Dehornoy, \emph{Monoids of o-type, sub word reverse, and ordered
  groups},  \textbf{17} (2014), no.~3, 465--524.

\bibitem{DDRW08}
Patrick Dehornoy, Ivan Dynnikov, Dale Rolfsen, and Bert Wiest, \emph{Ordering
  braids}, Surveys and Monographs, vol. 148, American Mathematical Society,
  Providence, RI, 2008.

\bibitem{DD01}
T.~V. Dubrovina and N.~I. Dubrovin, \emph{On braid groups}, Mat. Sb.
  \textbf{192} (2001), no.~5, 53--64. \MR{MR1859702 (2002h:20051)}

\bibitem{EHN81}
David Eisenbud, Ulrich Hirsch, and Walter Neumann, \emph{Transverse foliations
  of {S}eifert bundles and self-homeomorphism of the circle}, Comment. Math.
  Helv. \textbf{56} (1981), no.~4, 638--660. \MR{MR656217 (83j:57016)}

\bibitem{Epstein72}
D.~B.~A. Epstein, \emph{Periodic flows on three-manifolds}, Ann. of Math. (2)
  \textbf{95} (1972), 66--82. \MR{0288785 (44 \#5981)}

\bibitem{FM12}
Benson Farb and Dan Margalit, \emph{A primer on mapping class groups},
  Princeton Mathematical Series, vol.~49, Princeton University Press,
  Princeton, NJ, 2012. \MR{2850125 (2012h:57032)}

\bibitem{farrell76}
F.~Thomas Farrell, \emph{Right-orderable deck transformation groups}, Rocky
  Mountain J. Math. \textbf{6} (1976), no.~3, 441--447. \MR{MR0418078 (54
  \#6122)}

\bibitem{GMM}
David Gabai, Robert Meyerhoff, and Peter Milley, \emph{Minimum volume cusped
  hyperbolic three-manifolds}, J. Amer. Math. Soc. \textbf{22} (2009), no.~4,
  1157--1215. \MR{2525782 (2011a:57031)}

\bibitem{Glass99}
A.~M.~W. Glass, \emph{Partially ordered groups}, Series in Algebra, vol.~7,
  World Scientific Publishing Co. Inc., River Edge, NJ, 1999. \MR{MR1791008
  (2001g:06002)}

\bibitem{GL69}
E.~A. Gorin and V.~Ja. Lin, \emph{Algebraic equations with continuous
  coefficients, and certain questions of the algebraic theory of braids}, Mat.
  Sb. (N.S.) \textbf{78 (120)} (1969), 579--610. \MR{0251712 (40 \#4939)}

\bibitem{Greenepreprint}
Joshua Greene, \emph{Alternating links and left-orderability}, Preprint,
  available via http://arxiv.org/abs/1107.5232.

\bibitem{Hatchernotes}
Allen Hatcher, \emph{Notes on basic 3-manifold topology}, available from the
  author's website, 2007, http://www.math.cornell.edu/~hatcher/3M/3M.pdf.

\bibitem{Higman40}
G.~Higman, \emph{The units of group rings}, Proc. London Math. Soc. \textbf{46}
  (1940), no.~2, 231--248.

\bibitem{HNN49}
Graham Higman, B.~H. Neumann, and Hanna Neuman, \emph{Embedding theorems for
  groups}, Journal of the London Mathematical Society \textbf{s1-24} (1949),
  no.~4, 247--254.

\bibitem{MLM83}
Hugh~M. Hilden, M.~T. Lozano, and Jos{\'e}~Mar{\'{\i}}a Montesinos,
  \emph{Universal knots}, Bull. Amer. Math. Soc. (N.S.) \textbf{8} (1983),
  no.~3, 449--450. \MR{693959 (84g:57002)}

\bibitem{HMM85}
Hugh~M. Hilden, Mar{\'{\i}}a~Teresa Lozano, and Jos{\'e}~Mar{\'{\i}}a
  Montesinos, \emph{On knots that are universal}, Topology \textbf{24} (1985),
  no.~4, 499--504. \MR{816529 (87a:57010)}

\bibitem{Hion54}
Ya.~V. Hion, \emph{Archimedean ordered rings}, Uspehi Mat. Nauk \textbf{4}
  (1954), no.~9, 237--242.

\bibitem{HY61}
John~G. Hocking and Gail~S. Young, \emph{Topology}, Addison-Wesley Publishing
  Co., Inc., 1961.

\bibitem{Holder01}
O.~H\"{o}lder, \emph{{D}ie {A}xiome der {Q}uantit{\"a}t und die {L}ehre vom
  {M}ass}, Ber. Verh. Sachs. Ges. Wiss. Leipzig Math. Phys. Cl. \textbf{53}
  (1901), 1--64.

\bibitem{HTW}
Jim Hoste, Morwen Thistlethwaite, and Jeff Weeks, \emph{The first 1,701,936
  knots}, Math. Intelligencer \textbf{20} (1998), no.~4, 33--48. \MR{1646740
  (99i:57015)}

\bibitem{Howie82}
James Howie, \emph{On locally indicable groups}, Math. Z. \textbf{180} (1982),
  no.~4, 445--461. \MR{MR667000 (84b:20036)}

\bibitem{Howie00}
\bysame, \emph{A short proof of a theorem of {B}rodski\u\i}, Publ. Mat.
  \textbf{44} (2000), no.~2, 641--647. \MR{1800825 (2001i:20066)}

\bibitem{HS85}
James Howie and Hamish Short, \emph{The band-sum problem}, J. London Math. Soc.
  (2) \textbf{31} (1985), no.~3, 571--576. \MR{812788 (87c:57004)}

\bibitem{Huntington55}
Edward~V. Huntington, \emph{The continuum and other types of serial order.
  {W}ith an introduction to {C}antor's transfinite numbers}, Dover
  Publications, Inc., New York, 1955, 2d ed. \MR{0067953 (16,804d)}

\bibitem{Itopreprinta}
Tetsuya Ito, \emph{Construction of isolated left orderings via partially
  central cyclic amalgamation}, Preprint, available via
  http://arxiv.org/abs/1107.0545.

\bibitem{Ito2014preprint}
\bysame, \emph{Isolated orderings on amalgamated free products}, Preprint,
  available via http://arxiv.org/abs/1405.1163.

\bibitem{Ito11a}
\bysame, \emph{Braid ordering and knot genus}, J. Knot Theory Ramifications
  \textbf{20} (2011), no.~9, 1311--1323. \MR{2844810}

\bibitem{Ito13}
\bysame, \emph{Dehornoy-like left orderings and isolated left orderings}, J.
  Algebra \textbf{374} (2013), 42--58. \MR{2998793}

\bibitem{khoi03}
Vu~The Khoi, \emph{A cut-and-paste method for computing the {S}eifert volumes},
  Math. Ann. \textbf{326} (2003), no.~4, 759--801. \MR{2003451 (2004h:57029)}

\bibitem{KR03}
Djun~Maximilian Kim and Dale Rolfsen, \emph{An ordering for groups of pure
  braids and fibre-type hyperplane arrangements}, Canad. J. Math. \textbf{55}
  (2003), no.~4, 822--838. \MR{MR1994074 (2004e:20063)}

\bibitem{KK74}
Ali~Ivanovi{\v{c}} Kokorin and Valeri{\={\i}}~Matveevi{\v{c}} Kopytov,
  \emph{Fully ordered groups}, Halsted Press [John Wiley\thinspace \&\thinspace
  Sons], New York-Toronto, Ont., 1974, Translated from the Russian by D.
  Louvish. \MR{MR0364051 (51 \#306)}

\bibitem{Kopytov79}
V.~M. Kopytov, \emph{Free lattice-ordered groups}, Algebra i Logika \textbf{18}
  (1979), no.~4, 426--441, 508. \MR{582096 (81i:06018)}

\bibitem{KM96}
Valeri{\u\i}~M. Kopytov and Nikola{\u\i}~Ya. Medvedev, \emph{Right-ordered
  groups}, Siberian School of Algebra and Logic, Consultants Bureau, New York,
  1996. \MR{MR1393199 (97h:06024a)}

\bibitem{KM04}
P.~B. Kronheimer and T.~S. Mrowka, \emph{Witten's conjecture and property {P}},
  Geom. Topol. \textbf{8} (2004), 295--310 (electronic). \MR{2023280
  (2004m:57023)}

\bibitem{Levi42}
F.~W. Levi, \emph{Ordered groups}, Proc. Indian Acad. Sci., Sect. A.
  \textbf{16} (1942), 256--263. \MR{0007779 (4,192b)}

\bibitem{Levi43}
\bysame, \emph{Contributions to the theory of ordered groups}, Proc. Indian
  Acad. Sci., Sect. A. \textbf{17} (1943), 199--201. \MR{MR0008807 (5,58r)}

\bibitem{Lickorish62}
W.~B.~R. Lickorish, \emph{A representation of orientable combinatorial
  {$3$}-manifolds}, Ann. of Math. (2) \textbf{76} (1962), 531--540. \MR{0151948
  (27 \#1929)}

\bibitem{Lickorish63}
\bysame, \emph{Homeomorphisms of non-orientable two-manifolds}, Proc. Cambridge
  Philos. Soc. \textbf{59} (1963), 307--317. \MR{0145498 (26 \#3029)}

\bibitem{Lickorish65}
\bysame, \emph{A foliation for {$3$}-manifolds}, Ann. of Math. (2) \textbf{82}
  (1965), 414--420. \MR{0189061 (32 \#6488)}

\bibitem{linnell01}
Peter~A. Linnell, \emph{Left ordered groups with no non-abelian free
  subgroups}, J. Group Theory \textbf{4} (2001), no.~2, 153--168. \MR{1812322
  (2002h:06018)}

\bibitem{linnell11}
\bysame, \emph{The space of left orders of a group is either finite or
  uncountable}, Bull. Lond. Math. Soc. \textbf{43} (2011), no.~1, 200--202.
  \MR{2765563 (2011k:06039)}

\bibitem{Malyutin04}
{\relax A.V}.~Malyutin and N.Yu. Netstvetaev, \emph{Dehornoy's ordering on the
  braid group and braid moves}, St. Peterburg Math. J. \textbf{15} (2004),
  no.~3, 437--448.

\bibitem{McCleary85}
Stephen~H. McCleary, \emph{Free lattice-ordered groups represented as
  {$o$}-{$2$} transitive {$l$}-permutation groups}, Trans. Amer. Math. Soc.
  \textbf{290} (1985), no.~1, 69--79. \MR{MR787955 (86m:06034a)}

\bibitem{Milnor62}
J.~Milnor, \emph{A unique decomposition theorem for {$3$}-manifolds}, Amer. J.
  Math. \textbf{84} (1962), 1--7. \MR{0142125 (25 \#5518)}

\bibitem{Mineyev12}
Igor Mineyev, \emph{Submultiplicativity and the {H}anna {N}eumann conjecture},
  Annals of Mathematics \textbf{175} (2012), no.~1, 393--414.

\bibitem{Morris06}
D.~Morris, \emph{Amenable groups that act on the line}, Algebr. Geom. Topol.
  \textbf{6} (2006), 2509--2518.

\bibitem{Munkres2000}
James~R. Munkres, \emph{Topology, second edition}, Prentice-Hall, 2000.

\bibitem{Murasugi61}
Kunio Murasugi, \emph{Remarks on knots with two bridges}, Proc. Japan Acad.
  \textbf{37} (1961), 294--297. \MR{0139162 (25 \#2599)}

\bibitem{Navas10b}
Andr{\'e}s Navas, \emph{A finitely generated, locally indicable group with no
  faithful action by {$C^1$} diffeomorphisms of the interval}, Geom. Topol.
  \textbf{14} (2010), no.~1, 573--584. \MR{2602845 (2011d:37045)}

\bibitem{Navas10a}
Andr\'{e}s Navas, \emph{On the dynamics of (left) orderable groups}, Annales de
  l'institut Fourier \textbf{60} (2010), no.~5, 1685--1740.

\bibitem{Navas11a}
Andr{\'e}s Navas, \emph{A remarkable family of left-ordered groups: central
  extensions of {H}ecke groups}, J. Algebra \textbf{328} (2011), 31--42.
  \MR{2745552 (2011m:20094)}

\bibitem{neumann49a}
B.~H. Neumann, \emph{On ordered division rings}, Trans. Amer. Math. Soc.
  \textbf{66} (1949), 202--252. \MR{MR0032593 (11,311f)}

\bibitem{Neumann49b}
\bysame, \emph{On ordered groups}, Amer. J. Math. \textbf{71} (1949), 1--18.
  \MR{MR0028312 (10,428a)}

\bibitem{Ni07}
Yi~Ni, \emph{Knot {F}loer homology detects fibred knots}, Invent. Math.
  \textbf{170} (2007), no.~3, 577--608. \MR{MR2357503 (2008j:57053)}

\bibitem{Novikov64}
S.~P. Novikov, \emph{Foliations of co-dimension {$1$} on manifolds}, Dokl.
  Akad. Nauk SSSR \textbf{155} (1964), 1010--1013. \MR{0165540 (29 \#2821)}

\bibitem{OS05}
Peter Ozsv{\'a}th and Zolt{\'a}n Szab{\'o}, \emph{On knot {F}loer homology and
  lens space surgeries}, Topology \textbf{44} (2005), no.~6, 1281--1300.
  \MR{MR2168576 (2006f:57034)}

\bibitem{OS05double}
\bysame, \emph{On the {H}eegaard {F}loer homology of branched double-covers},
  Adv. Math. \textbf{194} (2005), no.~1, 1--33. \MR{2141852 (2006e:57041)}

\bibitem{Papakyriakopoulos57}
C.~D. Papakyriakopoulos, \emph{On {D}ehn's lemma and the asphericity of knots},
  Ann. of Math. (2) \textbf{66} (1957), 1--26. \MR{0090053 (19,761a)}

\bibitem{PR03}
Bernard Perron and Dale Rolfsen, \emph{On orderability of fibred knot groups},
  Math. Proc. Cambridge Philos. Soc. \textbf{135} (2003), no.~1, 147--153.
  \MR{MR1990838 (2004f:57015)}

\bibitem{RR02}
Akbar Rhemtulla and Dale Rolfsen, \emph{Local indicability in ordered groups:
  braids and elementary amenable groups}, Proc. Amer. Math. Soc. \textbf{130}
  (2002), no.~9, 2569--2577 (electronic). \MR{MR1900863 (2003b:20058)}

\bibitem{Rivas12}
Crist{\'o}bal Rivas, \emph{Left-orderings on free products of groups}, J.
  Algebra \textbf{350} (2012), 318--329. \MR{2859890}

\bibitem{Rivas12a}
\bysame, \emph{On groups with finitely many {C}onradian orderings}, Comm.
  Algebra \textbf{40} (2012), no.~7, 2596--2612. \MR{2948849}

\bibitem{Rolfsen90}
Dale Rolfsen, \emph{Knots and links}, Mathematics Lecture Series, vol.~7,
  Publish or Perish Inc., Houston, TX, 1990, Corrected reprint of the 1976
  original. \MR{MR1277811 (95c:57018)}

\bibitem{RW01}
Dale Rolfsen and Bert Wiest, \emph{Free group automorphisms, invariant
  orderings and topological applications}, Algebr. Geom. Topol. \textbf{1}
  (2001), 311--320 (electronic). \MR{MR1835259 (2002g:20073)}

\bibitem{RZ98}
Dale Rolfsen and Jun Zhu, \emph{Braids, orderings and zero divisors}, J. Knot
  Theory Ramifications \textbf{7} (1998), no.~6, 837--841. \MR{MR1643939
  (99g:20072)}

\bibitem{RW00}
Colin Rourke and Bert Wiest, \emph{Order automatic mapping class groups},
  Pacific J. Math. \textbf{194} (2000), no.~1, 209--227. \MR{1756636
  (2001f:57003)}

\bibitem{Sch}
Rob Scharein, \emph{Knotplot}, www.knotplot.com.

\bibitem{Schubert56}
Horst Schubert, \emph{Knoten mit zwei {B}r\"ucken}, Math. Z. \textbf{65}
  (1956), 133--170. \MR{0082104 (18,498e)}

\bibitem{Scott73}
G.~P. Scott, \emph{Compact submanifolds of {$3$}-manifolds}, J. London Math.
  Soc. (2) \textbf{7} (1973), 246--250. \MR{0326737 (48 \#5080)}

\bibitem{Scott83b}
Peter Scott, \emph{The geometries of {$3$}-manifolds}, Bull. London Math. Soc.
  \textbf{15} (1983), no.~5, 401--487. \MR{705527 (84m:57009)}

\bibitem{ST80}
Herbert Seifert and William Threlfall, \emph{{Seifert and Threlfall, A textbook
  of topology, Volume 89 (Pure and Applied Mathematics)}}, Academic Press, June
  1980.

\bibitem{Serre80}
Jean-Pierre Serre, \emph{Trees}, Springer Monographs in Mathematics,
  Springer-Verlag, Berlin, 2003, Translated from the French original by John
  Stillwell, Corrected 2nd printing of the 1980 English translation.
  \MR{1954121 (2003m:20032)}

\bibitem{SW00}
Hamish Short and Bert Wiest, \emph{Orderings of mapping class groups after
  {T}hurston}, Enseign. Math. (2) \textbf{46} (2000), no.~3-4, 279--312.
  \MR{MR1805402 (2003b:57003)}

\bibitem{Shpilrain01}
W.~Shpilrain, \emph{Representing braids by automorphisms}, Internat. J. Algebra
  and Comput. \textbf{11} (2001), no.~6, 773--777.

\bibitem{Sikora04}
{\relax A.S}.~Sikora, \emph{Topology on the spaces of orderings of groups},
  Bull. London Math. Soc. \textbf{36} (2004), 519--526.

\bibitem{smythe67}
N.~Smythe, \emph{Trivial knots with arbitrary projection}, J. Austral. Math.
  Soc. \textbf{7} (1967), 481--489. \MR{0220271 (36 \#3337)}

\bibitem{wada92}
M.~Wada, \emph{Group invariants of links}, Topology \textbf{31} (1992), no.~2,
  399--406.

\bibitem{Wallace60}
Andrew~H. Wallace, \emph{Modifications and cobounding manifolds}, Canad. J.
  Math. \textbf{12} (1960), 503--528. \MR{0125588 (23 \#A2887)}

\bibitem{Wood69}
John~W. Wood, \emph{Foliations on {$3$}-manifolds}, Ann. of Math. (2)
  \textbf{89} (1969), 336--358. \MR{0248873 (40 \#2123)}

\end{thebibliography}

\printindex

\end{document}